\documentclass[reqno, 11pt]{amsart}
\usepackage{amsmath, amssymb}

\usepackage[margin=1.3in]{geometry}

\usepackage[colorlinks=true, linkcolor=black]{hyperref}

\newtheorem{thm}{Theorem}[section]
\newtheorem{prop}{Proposition}[section]
\newtheorem{lem}{Lemma}[section]
\theoremstyle{remark}
\newtheorem{rem}{Remark}[section]
\newtheorem*{rem*}{Remark}
\newtheorem{cor}{Corollary}[section]

\numberwithin{equation}{section}

\newcommand{\op}[1]{\operatorname{#1}}
\newcommand{\R}{\mathbb R}

\begin{document}

\title[Effective maximum principles]{Effective maximum principles
for spectral methods}

\author[D. Li]{Dong Li}
\email{mpdongli@gmail.com}



\begin{abstract}
Many physical problems such as Allen-Cahn flows have natural maximum principles which yield strong point-wise control of the physical solutions in terms of the boundary data,
the initial
conditions and the operator coefficients. Sharp/strict maximum principles insomuch of fundamental importance for the continuous problem
 often do not persist under numerical  discretization. 
A lot of past research concentrates on designing fine numerical schemes
which preserves the sharp maximum principles especially for nonlinear problems. 
However these sharp principles  not only sometimes introduce unwanted stringent conditions on the numerical schemes but also completely
  leaves  many powerful
 frequency-based methods unattended and rarely analyzed directly in the sharp maximum
 norm topology.
 A prominent example is the spectral methods in the family of weighted residual methods.

In this work we introduce and develop a new framework of almost sharp maximum principles 
which allow the numerical solutions to deviate from the sharp bound by a controllable
discretization error: we call them effective maximum principles. 
We showcase the analysis 
for the classical Fourier spectral methods including Fourier Galerkin and Fourier collocation in space
with forward Euler in time or second order Strang splitting.  The model equations include
the Allen-Cahn equations with double well potential, the Burgers equation and the
Navier-Stokes equations. We give a comprehensive proof of the effective maximum principles
under very general parametric conditions.
\end{abstract}
\maketitle
\tableofcontents
\section{Introduction}
In solving  physical problems such as Allen-Cahn flows in interfacial dynamics, the 
maximum principle plays an important role since it gives strong point-wise control of the physical solutions in terms of the boundary data,
the initial
conditions and the operator coefficients. For practical numerical simulations, it is often the case
that sharp/strict maximum principles  for the continuous problem often do not persist under numerical  discretization. 
A lot of past research is centered on designing fine numerical schemes which preserves
the maximum in a sharp way especially for nonlinear problems. For linear parabolic equations,
it is well known that central finite difference in space with backward Euler time stepping
can preserve the sharp maximum principle (cf. Chapter 9 of \cite{LarTh03}
for a textbook analysis of 1D homogeneous heat equation).  This is also the case 
if one employs lumped mass linear finite element in space using acute simplicial triangulation. 
Although preserving the sharp maximum principle is highly desirable for numerical simulations,
 these  often  introduce unwanted stringent conditions on the numerical schemes. Moreover it completely
  leaves  out many powerful
 $L^2$-based methods unattended and rarely analyzed directly in the sharp maximum
 norm topology.
 In this respect a prominent example is the spectral methods in the family of weighted residual methods. In this work we introduce and develop a new framework of almost sharp maximum principles 
which allow the numerical solutions to deviate from the sharp bound by a controllable
discretization error: we call them effective maximum principles. Our main models are Allen-Cahn
equations in physical dimensions $d\le 3$, but we also discuss related models such
as Burgers equations, Navier-Stokes equations.  All these will be discussed in this introduction.

We begin by considering the Allen-Cahn equation in physical dimensions $d=1,2,3$:
\begin{align} \label{r1}
\begin{cases}
\partial_t u = \nu^2 \Delta u -f(u),  \quad t>0; \\
u \Bigr|_{t=0} =u_0, \\
\text{periodic boundary conditions,}
\end{cases}
\end{align}
where $u$ is a scalar function which typically represents the concentration of one of the
two metallic components of the alloy. For simplicity we consider the periodic boundary condition and assume the function to have period $1$ in each spatial coordinate axis. The parameter $\nu>0$ controls the interfacial width which
is small compared with the system size under study. The nonlinear term has the usual
double well form:
\begin{align*}
f(u)=F^{\prime}(u)=u^3-u, \quad F(u)=\frac14 {(u^2-1)^2}.
\end{align*}
Introduce the energy functional
\begin{align*}
\mathcal E(u) = \int ( \frac 12 \nu^2 |\nabla u|^2 + F(u) ) dx.
\end{align*}
The system \eqref{r1} can be regarded as a gradient flow of $\mathcal E(\cdot)$ in the $L^2$ metric.
On the other hand if one changes the topology to $H^{-1}$ then we obtain the usual Cahn-Hilliard
system. Due to the gradient flow structure 
\begin{align*}
u_t =- \frac{\delta \mathcal E}{ \delta u},
\end{align*}
one has the energy law
\begin{align*}
\frac {d} {dt} \mathcal E(u(t)) 
= - \| \partial_t u \|_2^2 = - \int | \nu^2 \Delta u -f(u)|^2 dx.
\end{align*}
Thus for smooth solutions, we have the energy decay
\begin{align*}
\mathcal E(u(t)) \le \mathcal E (u(s)), \quad \forall\, 0\le s \le t<\infty.
\end{align*}
Similar energy laws also exists for other phase field models such as the Cahn-Hilliard system.
However for the Allen-Cahn system due to its particular structure one has an additional maximum
principle which asserts that the $L^{\infty}$ norm of the 
smooth solution is bounded by $1$ if the initial data is bounded by $1$. Verification of these two fundamental conservation laws are of pivotal role in designing robust and stable numerical
schemes for the Allen-Cahn system.

We begin with  an implicit-explicit Fourier Galerkin discretization of \eqref{r1}. For simplicity
consider the system in one dimension, i.e. the torus $\mathbb T=[0,1)$.  For a periodic function
$f$ with Fourier coefficients $\widehat f (k)$, we define its projection to the first $N$-modes
($N\ge 2$ is an integer)
as:
\begin{align*}
(\Pi_N f )(x) = \sum_{|k| \le N} \widehat f (k) e^{2\pi i k  \cdot x}.
\end{align*}
In yet other words $\Pi_N$ is the projection to the space
\begin{align*}
S_N=\operatorname{span} \{ e^{2\pi i k\cdot x}: \;  \text{ $k\in \mathbb Z$ and  $|k| \le N$} \}.
\end{align*}
A prototypical implicit-explicit Fourier spectral scheme has the form:
\begin{align} \label{r3}
\begin{cases}
\frac{u^{n+1} -u^n}{ \tau} = \nu^2 \partial_{xx} u^{n+1} -\Pi_N ( (u^n)^3- u^n),
\quad n\ge 0, \\
u^0 =\Pi_N u_0,
\end{cases}
\end{align}
where $\tau>0$ is the size of the time step, and $u^n$ denotes the numerical solution at the time
step $t=n\tau$.  Note that the linear part is treated implicitly whereas the nonlinear part is explicit
which makes the practical computation very convenient.  Thanks to the frequency projection
the Fourier modes of $u^n$ are  all trapped in the window $[-N,N]$.  In later sections  we also consider the full collocation  case and aliasing errors.  
On the other hand, the system \eqref{r3} in some sense
captures the essential difficulties of the numerical analysis for the Fourier Galerkin method.

As was already mentioned we are concerned with the $L^{\infty}$ maximum principle
for the approximation system \eqref{r3}. For this purpose it is convenient to recast it as
\begin{align} \label{r4}
\begin{cases}
u^{n+1}= T_{N, \tau \nu^2}  \Bigl( f_{\tau}(u^n) \Bigr), \quad n\ge 0;\\
u^{0}= \Pi_N u_0, 
\end{cases}
\end{align}
where $T_{N, \tau\nu^2} = (\operatorname{Id}-\tau \nu^2 \partial_{xx})^{-1} \Pi_N$, and
\begin{align*}
f_{\tau}(z) = (1+\tau) z - \tau z^3.
\end{align*}
Similar reformulation can also be written down for other discretization
methods such as the finite difference scheme. One immediate problem which makes
the analysis of \eqref{r4} nontrivial in the $L^{\infty}$ setting is the lack of $L^{\infty}$
preservation due to the frequency truncation $\Pi_N$. Indeed, even for $n=0$ and generic
initial data $u_0$ with $\|u_0\|_{\infty}\le 1$, one can have
\begin{align*}
\|u^0\|_{\infty} >1
\end{align*}
which is caused by the lack of positivity of the Dirichlet kernel. For this and similar other
technical obstructions (cf. page 219 of \cite{YDZ18}), there has been no discussion of $L^{\infty}$ maximum principle
for the Fourier spectral method in the literature prior to this work. The very purpose of this
work is to settle this issue and develop a new framework for spectral methods. 

Our first result is concerned with a detailed analysis of the operator $T_{N,\tau \nu^2}$. 
Albeit classical 
the main novelty here is the quantification of the parameters and the sharpness of the 
involved constants. To allow some generality we denote for
$N\ge 1$, $\beta>0$, $T_{N,\beta}= (\operatorname{Id}-\beta \partial_{xx})^{-1}
\Pi_N$. Also denote $\Pi_{>N}=\operatorname{Id}-\Pi_N$ and
$K_{>N, \beta}=(\operatorname{Id}-\beta\partial_{xx})^{-1} \Pi_{>N}$.

\begin{thm}[Maximum principle for $T_{N,\beta}$, $1$D] \label{thm1.1_00}
Let $\beta>0$ and $N\ge 1$. 
\begin{itemize}
\item \underline{Strict maximum principle}: If $N\ge N_0(\beta)= \frac 1 {2\pi^2 
\sqrt{\beta} } e^{\frac 1 {2\sqrt{\beta}}}$, then we have the
strict maximum principle:
\begin{align*}
\| T_{N,\beta} f \|_{L^{\infty}(\mathbb T)} \le \| f \|_{L^{\infty}(\mathbb T)}.
\end{align*}
On the other hand if $N< N_0$ there are counterexamples.
\item \underline{Effective maximum principle}: For any $N\ge 1$, we have
\begin{align*}
\|T_{N,\beta} f \|_{L^{\infty}(\mathbb T)} 
\le (1+\frac {\alpha_1}{ 1+\beta N^2}  \log (N+2) ) \| f\|_{L^{\infty}(\mathbb T)},
\end{align*}
where $\alpha_1>0$ is an absolute constant.
\item \underline{Sharp bounds on $K_{>N,\beta}$}:  there are absolute constants 
$c_1>0$, $c_2>0$, such that
\begin{align*}
\frac {c_1}{1+\beta N^2} \log (N+2)
\le \|K_{>N, \beta} \|_{L^1(\mathbb T)} 
\le \frac {c_2} {1+\beta N^2} \log (N+2).
\end{align*}
\end{itemize}

These results are essentially sharp with respect to the dependence on the parameters.
 See Theorem \ref{thm2.14old} in Section 2 for some results concerning lower bounds and 
 more
definite and precise statements.

\end{thm}

\begin{rem*}
For the strict maximum principle of $T_{N,\beta}$ to hold we needed $N\ge N_0(\beta)
=O(e^{C/\sqrt{\beta}})$ which is large in the regime $0<\beta\ll 1$.  The sharpness
of such bounds is established in Proposition \ref{prop_tmp211}. But heuristically 
there are several ways to see why this is needed. One way is as follows. 
Denote the kernel of $T_{N,\beta}$ as $K_N=D_N*g_{\beta}$, where
\begin{align*}
g_{\beta}(x)=\sum_{n\in \mathbb Z} g_{\beta}^{(0)}(x+n), \quad g_{\beta}^{(0)}
(x)= \frac 1 {2\sqrt{\beta}} \exp(-\frac{|x|}{\sqrt{\beta}}),
\end{align*}
and $*$ is the usual convolution on the torus.  Take the fundamental domain for the torus
as $[-\frac 12,\frac 12]$ and consider $K_N(x_0)= \int_{|y|\le \frac 12} D_N(y)
g_{\beta}(x_0-y) dy$ with $x_0$ near $\frac 12-\frac {0.5}{2N+1}$.  Note that
the Dirichlet kernel is sharply peaked at $y=0$ with value $(2N+1)$ followed by a minimum
at $y=\frac 32 \cdot \frac 1{2N+1}$ with value approximately $-\frac 2{3\pi} (2N+1) \approx -0.212*(2N+1)$ and so on. 
 Clearly for
$y=O(\frac 1{2N+1})$ the contribution is approximately $N\cdot e^{-\frac 1{2\sqrt{\beta}}}$
and this has to beat the $O(1)$ fluctuation near $x_0$. Thus we need $N\gtrsim
\exp(c/\sqrt{\beta})$. Another way to see it is to observe that $g_{\beta}(y)$ has minimum
value around $\frac 1 {2\sqrt{\beta}} \exp(-c/\sqrt{\beta})$. Now we take the Dirichlet kernel
and convolve it with $g_{\beta}$ around this minimum. Clearly we need $N\gg \exp(c/\sqrt{\beta})$
in order for the first peak of $D_N$ to beat the first negative trough.

\end{rem*}

Theorem \ref{thm1.1_00} is a special case of Theorem \ref{thm2.14old}
in Section 2.  In Section 2 we give two proofs of the sharp characterization that $\|K_{>N, \beta}\|_{L^1(\mathbb T)}
 \sim \frac 1 {1+N^2 \beta} \log (N+2)$. The first proof (Proposition 
 \ref{prop_KN_lu_0001}) is essentially based on the Poisson summation whereas the second proof
 (Proposition \ref{prop_upperlower1}) uses the Fejer kernel which exploits the convexity of the Fourier
 coefficients in certain regimes. The proof therein naturally generalizes to certain convex
 trigonometric series.

 Our next result  generalizes Theorem \ref{thm1.1_00} to dimensions $d\ge 1$. 
In numerical computations, we usually work with the projection operators 
\begin{align*}
\Pi_N=\Pi_N^{(d)} = \Pi_{N,k_1} \cdots \Pi_{N, k_d},
\end{align*}
where $\Pi_{N,k_j}$ refers to projection into each frequency coordinate $k_j$. In
yet other words $\Pi_N^{(d)}$ is the frequency truncation to the
domain  $\{k:\, |k|_{\infty}\le N\}$, i.e., 
\begin{align*}
\widehat{\Pi_N^{(d)} } (k) =1_{|k|_{\infty} \le N}, \qquad k=(k_1,\cdots,k_d) \in 
\mathbb Z^d.
\end{align*} 
For $\beta>0$, $N\ge 2$, we define 
\begin{align*}
&T_{>N, \beta}= (\operatorname{Id}- \beta \Delta)^{-1} (\operatorname{Id}
-\Pi_{N}^{(d)}),  \qquad \widehat{T_{>N,\beta}}(k) = \frac 
1{1+\beta(2\pi |k| )^2} \cdot 1_{|k|_{\infty}>N}, \\
&T_{N,\beta}
= (\operatorname{Id}-\beta \Delta)^{-1} \Pi_N^{(d)},
\qquad \widehat{T_{N, \beta}}(k) = \frac 1 {1+\beta(2\pi |k|)^2} 
\cdot 1_{|k|_{\infty} \le N}.
\end{align*}

 \begin{thm}[Maximum principle for $T_{N,\beta}$ and $T_{>N,\beta}$,
 $d\ge 1$]  \label{thm_1.2_00}
Let $d\ge 1$, $N\ge 2$, $\beta>0$. Then the following hold for $T_{>N, \beta}= (\operatorname{Id}- \beta \Delta)^{-1} (\operatorname{Id}
-\Pi_{N}^{(d)})$ with the kernel $K_{>N,\beta}$, and $T_{N,\beta}
= (\operatorname{Id}-\beta \Delta)^{-1} \Pi_N^{(d)}$.
\begin{enumerate}
\item \underline{Sharp $L_x^1$ bound}. 
\begin{align*}
\frac {c_1} {1+\beta N^2}
(\log (N+2) )^d \le 
\| K_{>N,\beta} \|_{L_x^1(\mathbb T^d)} \le \frac {c_2} {1+\beta N^2}
(\log (N+2))^d,
\end{align*}
where $c_1>0$, $c_2>0$ depend only on the dimension $d$. 
\item \underline{Effective maximum principle}. For all $N\ge 2$, we have 
\begin{align*}
\| T_{N,\beta} f \|_{L_x^{\infty}(\mathbb T^d)}
\le (1+ \frac {c_3} { 1+\beta N^2} (\log (N+2))^d) \| f \|_{L_x^{\infty}(\mathbb T^d)},
\end{align*}
where $c_3>0$ depends only on the dimension $d$. 

\item \underline{Sharp maximum principle}.  Let $d=1,2$.  Denote the kernel of $T_{N,\beta}$ as
$K_{N,\beta}$. For all $N\ge N_0= N_0(\beta,d)$ (i.e. $N_0$ is constant depending only
on $\beta$ and $d$), it holds that
$K_{N,\beta}$ is a strictly positive function on $\mathbb T^d$ with unit $L^1_x$ mass, and
consequently
\begin{align*}
\| T_{N, \beta} f \|_{L_x^{\infty}(\mathbb T^d)} \le \|f \|_{L_x^{\infty}(\mathbb T^d)}.
\end{align*}
\end{enumerate}
\end{thm}
Theorem \ref{thm_1.2_00} is a restatement of Theorem \ref{thm2.18_00} in Section $2$. 
The corresponding analysis and proof can be found therein. From a practical point of
view, the statement (2) in Theorem \ref{thm_1.2_00} is most useful and effective in
the regime $0<\beta \ll 1$ since we only need to take $N$ moderately large in order
to obtain an almost sharp maximum principle. The $\log^d (N+2)$ factor in the upper
bound reflects the fact that the corresponding kernel is a tensor product of
one-dimensional Dirichlet kernels.

A natural generalization of the operator $(\op{Id}- \beta \Delta)^{-1} \Pi_N$ is the
truncated Bessel type operators $(\op{Id}- \beta \Delta)^{-s} \Pi_N$, where $s>0$. 
We denote the kernel function corresponding to $(\op{Id}- \beta \Delta)^{-s} \Pi_N$ 
as $F_{N,s}$ and note that $F_{N,s} \in C^{\infty}(\mathbb T^d)$ for each finite $N$. 
Let $s_*\in [0,1]$ be the unique solution to the equation
\begin{align*}
\int_0^{\frac {3}2 \pi} x^{1-s} \sin x dx =0, \quad 0\le s \le 1.
\end{align*}
In Lemma \ref{2015lem2.12} of Section $2$, we show that $s_*\in (0.308443, 0.308444)$.
The following group of results is proved in Section $2$. Perhaps a bit surprisingly, for 
the sharp maximum principle to hold, the transition threshold occurs at $s=s_*<1$. 
\begin{thm}[Maximum principle for general Bessel case]
Let $d\ge 1$, $N\ge 2$, $\beta>0$ and $s>0$. Then for all $N\ge 2$, we have
the \underline{effective maximum principle}: 
\begin{align*}
\| F_{N,s} \|_{L_x^1(\mathbb T^d)}
\le 1 + \frac {\alpha_{s,d} } {(1+\beta N^2)^{\frac s2} } ( \log (N+2) )^d,
\end{align*}
where $\alpha_{s,d}>0$ depend only on ($d$, $s$).

Concerning sharp maximum principles, we have the following.
\begin{enumerate}
\item If $s>d$ and $N\ge c_{s,d} (\beta^{\frac {d-1}2} e^{\frac 2 {\sqrt{\beta}} }
)^{\frac 1 {s-d} }$ ($c_{s,d}>0$ depends only on ($s$, $d$)), then $F_{N,s}$ is
a positive function and $\|F_{N,s} \|_{L^1(\mathbb T^d)} =1$.

\item Let $d=1$ and $s=1$. If $N\ge \alpha \beta^{-\frac 12} e^{\frac 1 {\sqrt{\beta} } }$
($\alpha>0$ is an absolute constant), then $F_{N,1}$ is  positive and has
unit $L^1$ mass.
\end{enumerate}
 Let $d=1$, $\beta>0$ and $0<s<1$. We identify $\mathbb T=[-\frac 12, \frac 12)$.  For some $s_* \in (0.308443, 0.308444)$, the following hold.
\begin{enumerate}
\item \underline{Lack of positivity for $0<s<s_*$}.  There are positive constants 
$\gamma_2(s)>\gamma_1(s)>0$ and $\gamma(s)>0$,  such that for any $y\in[\gamma_1(s),\, \gamma_2(s)]$
and $N\gg_s 1+\beta^{-\frac 1 {2(1-s)} }$, we have
\begin{align*}
\frac 1 {N^{1-s}} F_{N,s}(\frac y N) = - \gamma(s) + (1+ \frac 1 {\sqrt{\beta}} )\cdot O_s(N^{-(1-s)})
<-\frac 12 \gamma(s),
\end{align*}
In particular this shows that for $0<s<s_*$ the kernel function $F_{N,s}$ must be negative on an interval of length $O_s(N^{-1})$ for
all large $N$.

\item \underline{Positivity for $s_*<s<1$}. For any $s_*<s<1$ there are constants
$c_1(s)>0$ and $c_2(s)>0$ such that if $$N\ge \max\{
c_1(s) \beta^{-\frac 12} e^{\frac 1 {\sqrt{\beta}} }, c_2(s)\cdot(
1+ \beta^{-\frac 1 {2(1-s)} } ) \},$$
 then $F_N(x)$ is positive and hence has unit
$L_x^1$ mass.

\item \underline{Lack of positivity for $s=s_*$ and $0<\beta<\beta_*$}. 
Let $s=s_*$ and $y=\frac 34$. Let $0<\beta<\beta_*$ where $\beta_*$
is the same absolute constant as in Lemma \ref{lem2.14beta}.
 Then for $N\ge \frac 1 {\sqrt{\beta}}$, we have
\begin{align*}
F_{N,s}(\frac y N) = \frac 1 {2\pi \sqrt{\beta} }
A_{\beta}(s) +O_{s,\beta} ( N^{-s}  ),
\end{align*}
where $A_{\beta}(s)$ is given by the expression
\begin{align*}
A_{\beta}(s)
=\int_{\mathbb R} (\langle x \rangle^{-s} - |x|^{-s} )dx
-s 2\pi \sqrt{\beta}\int_{\mathbb R}
\langle x \rangle^{-s-2} x 
\{\frac x {2\pi \sqrt{\beta}} \} dx,
\end{align*}
where $\langle x \rangle =(1+x^2)^{\frac 12}$, and for $y \in \mathbb R$,
$\{y \}=y-[y]$ with $[y]$ being the smallest integer less than or equal to $y$.
Furthermore for $0<\beta<\beta_*$, we have
\begin{align*}
0<\alpha_1 < -A_{\beta}(s) <\alpha_2<\infty,
\end{align*}
where $\alpha_1>0$, $\alpha_2>0$ are absolute constants.
\end{enumerate}
\end{thm}
\begin{rem}
For the sharp maximum principle, further interesting cases are 
\begin{enumerate}
\item $d\ge 2$, $0<s\le d$; 
\item $d=1$, $s=s_*$, $\beta>\beta_*$. 
\end{enumerate}
Some partial results are available but we will not dwell on this issue here and
will investigate it elsewhere.
\end{rem}

We now return to \eqref{r4}. For $\beta= \nu^2 \tau >0$
Theorem \ref{thm1.1_00} shows
that in general we can only expect the effective maximum principle:
\begin{align*}
\| T_{N,\nu^2 \tau} f \|_{L^{\infty}(\mathbb T)}
\le (1 + \frac {\operatorname{const} } {1+ N^2 \nu^2 \tau} 
\log (N+2) ) \| f \|_{L^{\infty}(\mathbb T)}, \qquad\forall\, N\ge 2.
\end{align*}
Moreover by Theorem \ref{thm2.14old}, this upper bound is optimal for 
$0<\beta=\nu^2\tau\ll 1$ (which is typically the case 
in practical numerical simulations) and moderately
large $N$, i.e. $\frac {\operatorname{const} } {\nu \sqrt{\tau} }
\le N \le \frac {\operatorname{const} } { \nu \sqrt{\tau} }
\exp(\frac 1 {6\nu^2 \tau} )$.  The sharp maximum does hold for
large $N\gtrsim \exp( \frac {\operatorname{const} } {\nu^2 \tau} )$ but it is
computationally unfeasible especially when $0<\nu^2 \tau \ll 1$. 
Thus in order for \eqref{r4} to admit any sort of maximum principle in some
reasonable generality, we must have some strong contractive estimates on
the nonlinear map $u^n \to f_{\tau} (u^n)$ which will compensate for the loss
in the linear estimate.  As it turns out, this is indeed possible under some mild constraints
on the time step. The heart of the matter is encapsulated in the following
elementary iterative system.

\begin{prop}   \label{2015prop1}
Let $\tau>0$ and consider the cubic polynomial $f_{\tau}(x)=(1+\tau) x -\tau x^3$
for $x\in \mathbb R$.
Then the following hold:
\begin{itemize}

\item If $0<\tau\le \frac 12$, then $\max_{|x|\le 1} |f_{\tau}(x)| =1$. Actually for any
$1\le \alpha\le \sqrt{1+\frac 2 {\tau}}$, we have
\begin{align*}
\max_{|x|\le \alpha} |f_{\tau} (x)|  \le \alpha.
\end{align*}
Furthermore for any
$1<\alpha<\sqrt{1+\frac 2 {\tau}}$, we have the \underline{strict inequality}
\begin{align*}
\max_{|x|\le \alpha} |f_{\tau} (x)| <\alpha.
\end{align*}

\item If $\frac 12 \le \tau \le 2$, then for any
$\frac {(1+\tau)^{\frac 32}} {\sqrt{3\tau}}
\cdot \frac 23 \le \alpha \le \sqrt{\frac {2+\tau}{\tau}} $, we have
\begin{align*}
\max_{|x|\le \alpha} |f_{\tau} (x)| \le \alpha.
\end{align*}
Furthermore if $\frac {(1+\tau)^{\frac 32}} {\sqrt{3\tau}}
\cdot \frac 23 <\alpha < \sqrt{\frac {2+\tau}{\tau}} $, then we have the 
\underline{strict inequality}
\begin{align*}
\max_{|x|\le \alpha} |f_{\tau} (x)| < \alpha.
\end{align*}

\item For $\tau>2$, define $x_0 = \sqrt{ \frac {1+\tau} {3\tau}} \in (0,1)$
 and $x_{n+1} = f_{\tau} (x_n)$. Then $|x_n|\to \infty$
as $n\to \infty$. 
\end{itemize}
\end{prop}
\begin{rem*}
We have
\begin{align*}
\frac {(1+\tau)^{\frac 32}} {\sqrt{3\tau}}
\cdot \frac 23 
=\begin{cases}
1, \quad \tau=\frac 12; \\
\sqrt 2, \quad \tau =2.
\end{cases}
\qquad
\sqrt{\frac {2+\tau} {\tau} }
= \begin{cases}
\sqrt 5, \quad \tau=\frac 12;\\
\sqrt 2, \quad \tau =2.
\end{cases}
\end{align*}
\end{rem*}

The most useful result in Proposition \ref{2015prop1} is the strict inequality for the
maximum. In particular for $0<\tau \le \frac 12$ and $\alpha =1+\delta_0$ with
$\delta_0>0$ suitably small, we have
\begin{align} \label{2015e199a}
\frac 1 {\alpha} \max_{|x|\le \alpha} |f_{\tau}(x) | \le \theta_0(\alpha,\tau)<1,
\end{align}
where $\theta_0(\alpha,\tau)$ depends only on ($\alpha$, $\tau$). 

Now we return to \eqref{r4}.  Denote $\alpha_n = \|u^n \|_{\infty}$. 
Assume $\alpha_n>1$ and $\alpha_n-1$ is suitably small. 
By Theorem \ref{thm1.1_00} and \eqref{2015e199a}, for $0<\tau\le \frac 12$  we deduce
\begin{align*}
\alpha_{n+1} \le 
(1+ O(\frac 1 {1+N^2 \nu^2 \tau} \log (N+2) ) )
\cdot \theta_0(\alpha_n, \tau)\cdot  \alpha_n. 
\end{align*}

Before we proceed further, we should point out one subtle technical difficulty with
the above simplified system in the regime $0<\tau\ll 1$. 
Note that 
$\theta_0(\alpha,\tau)\to 1$ as $\tau \to 0$. Since the  the damping factor contains $N^2 \tau$, it
follows that the threshold $N$ must be taken $\tau$-dependent in order to obtain contractive estimates which is hardly desirable in practice. In order to retain stability 
for $\tau\to 0$  and build a stability analysis for $N$ moderately large and \emph{independent of $\tau$},  a different line of argument is needed and indeed we develop a refined
analysis in Section 3 (and later sections) to cover the regime $0<\tau \ll 1$. 

To simplify the discussion, we now consider the case $\tau_* \le \tau \le \frac 12$ where
$\tau_*>0$ is fixed. In this case, it is not difficult to see that for $N\ge N_0(\tau_*, \nu)$
sufficiently large, we have
\begin{align*}
\alpha_{n+1} \le  \max_{|x|\le \alpha_n} |f_{\tau}(x) | 
+ \eta_n,
\end{align*}
where $\eta_n$ accounts for the spectral error and can be made sufficiently small. 
The next proposition quantifies the desired strong stability. 
\begin{prop}[Strong stability of the prototype iterative system]  \label{2015prop2}
Let $\tau>0$ and  $f_{\tau}(x)=(1+\tau) x -\tau x^3$ for $x\in \mathbb R$. Consider the recurrent relation
 \begin{align*}
\alpha_{n+1}= \max_{|x|\le \alpha_n} |f_{\tau}(x)| + \eta,\quad n\ge 0, 
\end{align*}
where $\eta>0$.

\begin{enumerate}
\item Case $0<\tau \le \frac 12$.  Let $\alpha_0=2$.  There exists an absolute constant $\eta_0>0$ sufficiently small, such
that for all $0 \le \eta\le \eta_0$, we have  $1\le \alpha_n \le 2$ for all $n$.  Furthermore
for all $n\ge 1$,
\begin{align} \notag
1+\eta \le \alpha_n \le 1+ \theta^n +\frac {1-\theta^n}{1-\theta} \eta,
\end{align}
where $\theta =1-2\tau$. 
\item Case $\frac 12 \le \tau \le 2-\epsilon_0$ where $0<\epsilon_0\le 1$. 
Let $\alpha_0 = \frac 12 (\frac {(1+\tau)^{\frac 32}} {\sqrt{3\tau}}\cdot \frac 23+
\sqrt{\frac {2+\tau}{\tau}})$. 
 Then there
exists a constant $\eta_0>0$ depending only on $\epsilon_0$, such that if 
$0\le \eta \le \eta_0$,  then for all $n\ge 1$, we have 
\begin{align*}
\frac {(1+\tau)^{\frac 32}} {\sqrt{3\tau}}\cdot \frac 23+\eta
\le \alpha_n \le \alpha_0.
\end{align*}
\end{enumerate}
\end{prop}

Both Proposition \ref{2015prop1} and Proposition \ref{2015prop2} are proved in Section 3
 (see in particular Lemma \ref{lem_polycubic3} and \ref{lem_cubic_iteration} therein). 
For the one-dimensional system \eqref{r4}, a complete theory of $L^{\infty}$-stability and
instability  is worked out in Section 3 for all $0<\tau <\infty$ and $N\ge 2$. We shall
not reproduce all the details here and turn now to  the general theory for dimensions
$1\le d \le 3$ developed in Section 4. Consider
\begin{align} \label{r3md}
\begin{cases}
\frac{u^{n+1} -u^n}{ \tau} = \nu^2 \Delta u^{n+1} -\Pi_N ( (u^n)^3- u^n),
\quad n\ge 0, \qquad (t,x) \in (0,\infty) \times \mathbb T^d; \\
u^0 =\Pi_N u_0.
\end{cases}
\end{align}
Concerning \eqref{r3md}, the following group of stability results is proved in Section 4.
For the first time we are able to establish effective maximum principles for the Fourier spectral
methods applied on the nonlinear system.

\begin{thm}[Effective maximum principles for \eqref{r3md}] 
Consider \eqref{r3md} on $\mathbb T^d=[0,1)^d$ with $1\le d\le 3$ and $\nu >0$. 
Then the following hold.

\begin{enumerate}
\item \underline{Energy stability for $0<\tau \le 0.86$}. 
Let  $0<\tau\le 0.86$. Assume  $u_0 \in H^s(\mathbb T^d)\cap H^1(\mathbb T^d)$, $s>\frac d2$ and $$\|u^0\|_{\infty} \le \sqrt{\frac 53}. \qquad (
\text{note that }
\sqrt{\frac 53} \approx 1.29099).$$
 Then we have energy stability for any $N\ge N_0=N_0(s, d,\nu, u_0)$:
\begin{align*}
E(u^{n+1}) \le E(u^n), \quad \forall\, n\ge 0,
\end{align*}
where $E(u) = \int_{\mathbb T^d} ( \frac 12 \nu^2 |\nabla u|^2 + \frac 14 (u^2-1)^2) dx$. 
Furthermore,
\begin{align*}
\sup_{n\ge 0} \|u^n\|_{H^s} \le U_1<\infty,
\end{align*}
where $U_1>0$ is a constant depending only on ($s$, $d$, $\nu$, $u_0$).

\item \underline{Effective maximum principle for $0<\tau\le \frac 12$}.
Let $0<\tau\le \frac 12$. Assume $u_0 \in H^s(\mathbb T^d)\cap H^1(\mathbb T^d)$, $s>\frac d2$ and $\|u^0\|_{\infty} =\|\Pi_N u_0\|_{\infty} \le 1+\alpha_0$ for some $0\le \alpha_0 \le \sqrt{\frac 53}-1$ 
(note that $\sqrt{\frac 53} -1 \approx 0.29099$). Then for $N\ge N_1=N_1(s,d,\nu, u_0)$, we have
\begin{align*}
\|u^n\|_{\infty} \le 1+ (1-2\tau)^n \alpha_0+  N^{-(s-\frac d2)} C_{s,d,\nu, u_0}, \qquad\forall\, n\ge 1, 
\end{align*}
where $C_{s, d, \nu,  u_0}>0$ depends only on ($s$, $d$, $\nu$, $u_0$).

\item \underline{$L^{\infty}$-stability for $\frac 12 < \tau <2$}. 
Assume $\frac 12 <\tau < 2$. Then the following hold:
\begin{enumerate}
\item 
Let $\frac 12 <\tau<2-\epsilon_0$ for some $0<\epsilon_0\le 1$.  Denote $M_0=
\frac 12 (\frac {(1+\tau)^{\frac 32}} {\sqrt{3\tau}}\cdot \frac 23+
\sqrt{\frac {2+\tau}{\tau}})$. 
If $\|u^0\|_{\infty} \le M_0$ and $N\ge N_2=N_2(\nu)=
C(\epsilon_0) \cdot \nu^{-1} (|\log {\nu} |)^{\frac d2}$ when $0<\nu\ll 1$ ($C(\epsilon_0)$
is a constant depending only on $\epsilon_0$), then
\begin{align*}
\sup_{n\ge 0} \|u^n\|_{\infty} \le M_0.
\end{align*}

\item Let $d=1,2$. 
 For any
$\frac {(1+\tau)^{\frac 32}} {\sqrt{3\tau}}
\cdot \frac 23 \le \alpha \le \sqrt{\frac {2+\tau}{\tau}} $, if $\|u^0\|_{\infty} \le \alpha$ and
$N\ge N_3=N_3(\nu)=O(e^{\frac {\operatorname{const}} {\nu} }) $, then
\begin{align*}
\sup_{n\ge 0} \|u^n\|_{\infty} \le \alpha.
\end{align*}

\end{enumerate}

\end{enumerate}
\end{thm}
\begin{rem}
In Statement (1), our stability region is much wider than the result obtained by
Tang and Yang \cite{TY16}. For the finite difference case the analysis therein
requires $0<\tau \le \frac 12$. Our analysis here covers both the finite difference case
and the spectral case, and produces energy stability  for $0<\tau \le 0.86$ under 
much less stringent assumptions on the initial data. Also we should point out that
for $0<\nu \ll 1$, the dependence of $N_0$ is only power like which is a very mild 
constraint in practice. 
\end{rem}
\begin{rem}
For $\tau=2$ and $N$ not large, there are counterexamples as shown in Proposition
\ref{prop_tau2_tmp001}.
\end{rem}

To understand the effect of pure spectral truncation on $L^{\infty}$-stability, we now
consider the following model system 
\begin{align} \label{2015eJ1}
\begin{cases}
\partial_t u = \nu^2 \Delta u -\Pi_N ( u^3- u),
\quad (t,x) \in (0,\infty) \times \mathbb T^d, \\
u\Bigr|_{t=0} =\Pi_N u_0,
\end{cases}
\end{align}
where $\nu>0$ and $d\ge 1$.  The following results are proved
in Section 4. 

\begin{thm}[Maximum principle for the continuous in time system with spectral
truncation] 
Consider \eqref{2015eJ1} with $\nu>0$, $d\ge 1$.  Then the following hold.
\begin{enumerate}
\item \underline{Effective maximum principle for regular initial 
data}. Assume $\|u_0\|_{\infty} \le 1$ and 
$u_0 \in H^s(\mathbb T^d)$, $s>\frac d2$.
Then for $N\ge N_0=N_0(s, d,\nu,  u_0)$, we have
\begin{align*}
\sup_{0\le t<\infty} \|u(t, \cdot)\|_{\infty} 
\le  1+N^{-\gamma_0} C_1,
\end{align*}
where $\gamma_0=\min\{s-\frac d2, 1\}$, and $C_1>0$ depends only on
($s$, $d$, $\nu$, $u_0$).

\item \underline{Effective maximum principle for $L^{\infty}$ initial data}.
 Assume $\|u_0\|_{\infty} \le 1$.
Then for $N\ge N_1=N_1( d,\nu,  u_0)$, we have
\begin{align*}
\sup_{0\le t<\infty} \|u(t, \cdot)\|_{\infty} 
\le  \max\{\|\Pi_N u_0\|_{\infty}, 1\}+N^{-\frac 12} C_2,
\end{align*}
where $C_2>0$ depends only on
($d$, $\nu$, $u_0$).

\end{enumerate}
\end{thm}

We now turn to the Fourier collocation method which is widely used in practical numerical simulations. 
To illustrate the theory we consider the 1D Allen-Cahn on the periodic torus 
$\mathbb T=[0,1)$ using discrete Fourier transform in space. We discretize the domain 
$[0,1)$ using $x_j= \frac {j} N$, $j=0$, $1$, $\cdots$, $N-1$, where $N$ is usually taken to be an even
number.  In typical FFT simulations, $N$ is usually taken to be a dyadic number.
 We use $u^n=(u^n_0,
\cdots, u^n_{N-1})^T$ to
denote the approximation of $u(t,x_j)=u(t,\frac j N)$, $ t= n\tau$, $0\le j\le N-1$. We shall adopt the following
convention for discrete Fourier transform: 
\begin{align*}
& \tilde U_k = \frac 1 N \sum_{j=0}^{N-1} U_j e^{-2\pi i k \cdot \frac jN}; \\
& U_j = \sum_{k=-\frac N2+1}^{\frac N2} \tilde U_k e^{2\pi i k \cdot \frac jN},
\end{align*}
where $\tilde U=(\tilde U_0,\cdots, \tilde U_{N-1})^T$ is the (approximate) Fourier coefficient 
vector of
 the input data $U=(U_0,\cdots, U_{N-1})^T$.   Note that $\tilde U_{k\pm N} =\tilde U_k$ for any $k\in \mathbb Z$. The discrete Laplacian operator $\Delta_h$ 
 corresponds to the Fourier multiplier $-(2\pi k)^2$ (when $k$ is restricted to $-\frac N2<k \le\frac N2$).

To understand how Fourier collocation affects $L^{\infty}$-stability, we
consider the following  ODE system which is continuous-in-time with Fourier-collation-in-space 
discretization of Allen-Cahn.  Here $U(t)= (U(t)_0, \cdots, U(t)_{N-1})^T$ is the numerical
approximation of the exact node values $u(t, x_j)_{j=0}^{N-1}$. 
\begin{align} \label{2015FCe1}
\begin{cases}
\frac  d {dt} U = \nu^2 \Delta_h U + U - U^{.3}, \\
U\Bigr|_{t=0}=U^0 \in \mathbb R^N,
\end{cases}
\end{align}
 where $U^{.3}=((U_0)^3, \cdots, (U_{N-1})^3)^T$.
 
 Concerning \eqref{2015FCe1}, the following results are proved in Section \ref{S:DFT1}.
\begin{thm}
Consider \eqref{2015FCe1} with $\nu>0$. Then the following hold.

\begin{enumerate}
\item \underline{Effective maximum principle for regular initial data}. Suppose $u_{\op{init}}:\, \mathbb T\to \mathbb R$ satisfies $u_{\op{init}} \in H^{\frac 32} (\mathbb T)$ and
$\|u_{\op{init}}\|_{\infty} \le 1$. Take the initial data
$U^0$ such that $U^0_j= u_{\op{init}}(\frac j N)$ for $0\le j \le N-1$. 
 Then for $N\ge N_1=N_1(u_{\op{init}}, \nu)>0$, we have
\begin{align*}
\|U(t)\|_{\infty}  \le  1 + C_1 \cdot N^{-\frac 12},
\end{align*}
where $C_1>0$ depends only on ($u_{\op{init}} $, $\nu$). 

\item \underline{Effective maximum principle: version 2}. Suppose $u_{\op{init}}:\, \mathbb T\to \mathbb R$ has Fourier support in 
$\{|k|\le \frac N2 \}$ with $\widehat{u_{\op{init}}}(\frac N2)
=\widehat{u_{\op{init}}}(-\frac N2)$ and
$\|u_{\op{init}}\|_{\infty} \le 1$. 
 Take the initial data
$U^0$ such that $U^0_j= u_{\op{init}}(\frac j N)$ for $0\le j \le N-1$. 
 Then for $N\ge N_2=N_2(u_{\op{init}}, \nu)>0$, we have
\begin{align*}
\|U(t)\|_{\infty}  \le  1 + C_2 \cdot N^{-\frac 14},
\end{align*}
where $C_2>0$ depends only on ($u_{\op{init}} $, $\nu$).

\item \underline{Effective maximum principle for rough initial data}.
Suppose $U^0 \in \mathbb R^N$ satisfies $\|U^0\|_{\infty} \le 1$.
If $N\ge N_3(\nu)>0$,  it holds that
\begin{align*}
\sup_{0\le t <\infty} \| U(t) \|_{\infty} \le 1+\epsilon_1,
\end{align*}
where $0<\epsilon_1<10^{-2}$ is an absolute constant. More precisely the following hold.
(Below we shall write $X=O(Y)$ if $|X| \le C Y$ where the constant $C$ only depends on $\nu$.)
\begin{enumerate}
\item For $t\ge T_0= T_0(\nu)>0$, 
\begin{align*}
& \| U(t) \|_{\infty}  \le 1 + O(N^{-\frac 18} (\log N)^2); 
 \quad E( U(t)) \le O(1),
\end{align*}
where  $E(U)$ was defined in \eqref{Nov18EU1}.

\item For $C_1=C_1(\nu)>0 $ and $C_1\cdot N^{-\frac 34} \le t \le T_0$, it holds that
\begin{align*}
\| U(t) \|_{\infty}  \le 1+ O(N^{-\frac 18} (\log N)^2 ).
\end{align*}

\item For $C_2=C_2(\nu)>0$ and $ C_2 \cdot N^{-2} \log N \le t \le C_1 \cdot N^{-\frac 34}$,
\begin{align*}
\| U(t) \|_{\infty}  \le 1+ O(N^{-\frac 3 {16}} ).
\end{align*}

\item For $0<t \le C_2 N^{-2} \log N$, we have
\begin{align*}
& \|U(t) \|_{\infty} \le (1+\frac 12 \epsilon_1) \|U^0\|_{\infty}
+ O(N^{-\frac 12} (\log N)^{\frac 14} ).
\end{align*}

\item If $\|U^0\|_{\infty} \le \frac 1 {1+\frac 12 \epsilon_1} $ (note that 
$\frac 1 {1+\frac 12 \epsilon_1} >0.995$), then for all $t\ge 0$, we have
\begin{align*}
\|U(t) \|_{\infty} \le 1 + O(N^{-\frac 18} (\log N)^2).
\end{align*}

\item Suppose $f:\, \mathbb T\to \mathbb R$ is continuous and $\|f\|_{\infty} \le 1$. 
If $(U^0)_l= f(\frac l N)$ for all $0\le l\le N-1$,  then
\begin{align*}
\sup_{t\ge 0} \|U(t) \|_{\infty} 
\le 1 + O(\omega_f(N^{-\frac 23}) ) + O(N^{-c}),
\end{align*}
where $c>0$ is an absolute constant, and $\omega_f$ is defined in \eqref{N18:1b.0}.

\item  Suppose $f:\, \mathbb T\to \mathbb R$ is $C^{\alpha}$-continuous (see
\eqref{N18:1b.1}) for some $0<\alpha<1$ and $\|f\|_{\infty} \le 1$.  
If $(U^0)_l= f(\frac l N)$ for all $0\le l\le N-1$,  then
\begin{align*}
\sup_{t\ge 0}\|U(t)\|_{\infty}
\le 1 + O(N^{-c_1}),
\end{align*}
where $c_1>0$ is a constant depending only on $\alpha$. 
\end{enumerate}
Moreover we have the following result which shows the sharpness of our estimates above.
There exists a  function $f$: $\mathbb T \to \mathbb R$, continuous at all of $\mathbb T
\setminus \{x_*\}$ for some $x_*\in \mathbb T$ (i.e. continuous at all $x\ne x_*$) and  has the bound $\|f\|_{\infty} \le 1$
 such that 
the following hold:
for a sequence of even numbers $N_m \to \infty$, $t_m =\nu^{-2}N_m^{-2}$ and $x_m
=j_m/N_m$ with $0\le j_m\le N_m-1$ ,  if  $\tilde U_m(t) \in \mathbb R^{N_m}$ solves \eqref{2015FCe1}
with $\tilde U_m(0) =v_m \in R^{N_m}$ satisfying $(v_m)_j= f(\frac j {N_m})$ for all $0\le j \le N_m-1$.  Then
\begin{align*}
|\tilde U_m(t_m,x_m) | \ge 1+ \eta_*, \qquad\forall\, m\ge 1,
\end{align*}
where $\eta_*>0.001$ is an absolute constant.
\end{enumerate}
\end{thm}
\begin{rem}
For $0<\nu \ll 1$ the dependence of $N_i (\nu)$, $i=1,2,3$ is only power like which is
a mild constraint in practice.
\end{rem}

We now consider the fully discrete system.
\begin{align} \label{2015FCe4}
\frac {U^{n+1} - U^n}{\tau} = \nu^2 \Delta_h U^{n+1} + U^n-(U^n)^{.3},
\end{align}
where $\nu>0$, $\tau>0$ and $(U^n)^{.3}=((U_0^n)^3, \cdots, (U_{N-1}^n)^3)^T$. 
This is a first order IMEX method applied to Allen-Cahn with Fourier collocation in space.
The following results are proved in Section \ref{S:DFT1}.

\begin{thm}[Effective maximum principles for \eqref{2015FCe4}]
Consider \eqref{2015FCe4} with $\nu>0$, $\tau>0$. Then the following hold.

\begin{enumerate}

\item \underline{Sharp maximum principle for $0<\tau\le \frac 12$ and very large $N$}.  Assume $0<\tau \le \frac 12$ and
 $\|U^0\|_{\infty} \le 1$. If $N\ge4+\frac 1 {\pi^2 \nu\sqrt{\tau}} e^{\frac 1 {2\nu \sqrt{\tau} } }$, then  \begin{align*}
\|U^n\|_{\infty} \le 1, \qquad\forall\, n\ge 1.
\end{align*}

\item \underline{$L^{\infty} $ stability for 
$\frac 12 \le \tau \le 1.99$}.  Assume $\frac 12 \le \tau\le 1.99$ and
$\| U^0\|_{\infty} \le 1$. If $N\ge N_4(\nu)>0$ (the dependence of $N_4$ on $\nu$ is only 
power-like), then 
\begin{align} \notag
&\sup_{n\ge 0} \| U^n \|_{\infty} \le M_a = 
\frac 12 ( \frac 23 \cdot \frac {(1+\tau)^{\frac 32} } {\sqrt{3\tau}}
+\sqrt{\frac  {2+\tau}{\tau} } \Bigr). 
\end{align}
Moreover if $\frac 12 \le \tau \le 0.86$, then 
\begin{align}
&\sup_{n\ge 0} \| U^n\|_{\infty} \le M_b = \frac 23 \cdot \frac {(1+\tau)^{\frac 32} } {\sqrt{3\tau}}
+ \eta_0, \qquad \eta_0=10^{-5};  \notag \\
& E(U^{n+1}) \le E(U^n),    \qquad \forall\, n\ge 0; \\
& E(U^n) \le E(U^1) \le \mathcal E_1,  \qquad\forall\, n\ge 1;
\end{align}
where $\mathcal E_1>0$ depends only on $\nu$.

\item \underline{Effective maximum principle for 
$ N^{-0.3} \le \tau \le \frac 12$}.
Suppose $N^{-0.3} \le \tau \le \frac 12$ and
$\| U^0\|_{\infty} \le 1$. If $N\ge N_5(\nu)>0$ (the dependence of $N_5$ on $\nu$ is only 
power-like), then 
\begin{align} 
&\| U^n \|_{\infty} \le 1+ O(N^{-\frac 13}), \qquad\forall\, n\ge 0;  \notag \\
& E(U^{n+1}) \le E(U^n), \qquad\forall\, n\ge 0;\notag \\
& E(U^n) \le \mathcal E_2, \qquad \forall\, n\ge \frac 1 \tau; \notag 
\end{align}
where $\mathcal E_2>0$ depends only on $\nu$.

\item \underline{Effective  maximum principle for
$ 0<\tau \le N^{-0.3} $}. 
 Assume $
0<\tau\le N^{-0.3} $ and $\|U^0\|_{\infty} \le 1$. If $N\ge N_6(\nu)>0$ (the dependence of $N_6$ on $\nu$ is only 
power-like), then it holds that 
\begin{align*}
\sup_{n\ge 0} \|U^n \|_{\infty} \le 1 +\epsilon_1,
\end{align*}
where $0<\epsilon_1 <10^{-2}$ is an absolute constant. More precisely the following hold.
(Below we shall write $X=O(Y)$ if $|X|\le CY$ and the constant $C$ depends only
on $\nu$.)
\begin{enumerate}
\item $E(U^{n+1}) \le E(U^n)$ for all $n\ge 0$.
\item  For all $n\ge 2N^{-0.3}/\tau$, we have
\begin{align*}
\| U^n\|_{\infty} \le 1+O(N^{-\frac 3 {40}}).
\end{align*}
\item For $1\le n \le 2N^{-0.3}/\tau$, we have 
\begin{align*}
& \|U^n\|_{\infty} \le (1+\frac 12 \epsilon_1)
\|U^0\|_{\infty} + O(N^{-\frac 3 {40}}).
\end{align*}
\item If $\|U^0\|_{\infty} \le \frac 1 {1+\frac 12 \epsilon_1}$ 
(note that $1/(1+\frac 12 \epsilon_1)>0.995$), then 
\begin{align*}
\sup_{n\ge 0} \|U^n\|_{\infty} \le 1 + O(N^{-\frac 3{40}}).
\end{align*}

\item Suppose $f:\, \mathbb T\to \mathbb R$ is continuous and $\|f\|_{\infty} \le 1$. 
If $(U^0)_l= f(\frac l N)$ for all $0\le l\le N-1$,  then
\begin{align*}
\sup_{n\ge 0} \|U^n\|_{\infty} 
\le 1 + O(\omega_f(N^{-\frac 23}) ) + O(N^{-c}),
\end{align*}
where $c>0$ is an absolute constant, and $\omega_f$ is defined in \eqref{N18:1b.0}.

\item  Suppose $f:\, \mathbb T\to \mathbb R$ is $C^{\alpha}$-continuous for some $0<\alpha<1$ and $\|f\|_{\infty} \le 1$.  
If $(U^0)_l= f(\frac l N)$ for all $0\le l\le N-1$,  then
\begin{align*}
\sup_{n\ge 0}\|U^n\|_{\infty}
\le 1 + O(N^{-c_1}),
\end{align*}
where $c_1>0$ is a constant depending only on $\alpha$. 
\end{enumerate}
Moreover we have the following result which shows the sharpness of our estimates above.
There exists a  function $f$: $\mathbb T \to \mathbb R$, continuous at all of $\mathbb T
\setminus \{x_*\}$ for some $x_*\in \mathbb T$ (i.e. continuous at all $x\ne x_*$) and  has the bound $\|f\|_{\infty} \le 1$
 such that 
the following hold:
for a sequence of even numbers $N_m \to \infty$, $\tau_m =\frac 14 \nu^{-2}N_m^{-2}$ and $x_m
=j_m/N_m$ with $0\le j_m\le N_m-1$ , if $\tilde U_m \in \mathbb R^{N_m}$ solves 
\begin{align*}
(\op{Id} -\tau \nu^2 \Delta_h ) \tilde U_m = V_m + \tau_m f(V_m),
\end{align*}
where $(V_m)_l = f(\frac l {N_m})$ for all $0\le l\le N_m-1$. Then 
\begin{align*}
|\tilde U_m(x_m) | \ge 1+ \eta_*, \qquad\forall\, m\ge 1,
\end{align*}
where $\eta_*>0.001$ is an absolute constant.
\end{enumerate}
\end{thm}
\begin{rem}
In Statement (1) the cut-off $\tau=1.99$ is for convenience only. One can replace it by any number less than $2$ but  $N_4$ will have to be adjusted correspondingly.
\end{rem}

Concerning time discretization, all the numerical methods we discussed so far are only
first order in time. With further work our effective maximum principles can be generalized to higher order
in time methods. To showcase the theory, we consider Strang splitting with Fourier collocation
for Allen-Cahn on $\mathbb T$. The exact model equation is 
\begin{align*}
\partial_t u = \nu^2 \partial_{xx} u + f(u),\qquad f(u)=u-u^3, \qquad
(t,x) \in (0,\infty) \times \mathbb T.
\end{align*}
We slightly abuse the notation and denote for $U\in \mathbb R^N$
\begin{align*}
f(U)= U - U^{\cdot 3}.
\end{align*}
We consider the time splitting
as follows. Let $\tau>0$ be the time step.  Consider the ODE
\begin{align*}
\begin{cases}
\frac d {dt} U = f(U), \\
U\Bigr|_{t=0} = a \in \mathbb R^N.
\end{cases}
\end{align*}
We define the solution operator $\mathcal N_{\tau}:\, \mathbb R^N
\to \mathbb R^N$ as the map $a \to U(\tau)$.  Thanks to
the explicit form of $f(U)$, we have ( $a=(a_0,\cdots, a_{N-1})^T$ and with no loss
we shall assume that $N\ge 2$ )
\begin{align*}
(\mathcal N_{\tau} a)_j = U(\tau)_j = \frac {a_j}
{ \sqrt{ e^{-2\tau} + (1-e^{-2\tau} ) (a_j)^2} },
\qquad j=0, 1,\cdots, N-1.
\end{align*}

Define $\mathcal S_{\tau}= e^{\frac {\tau}2 \nu^2 \Delta_h}$. 
Then $U^{n+1}$, $U^n \in \mathbb R^N$
are related via the relation:
\begin{align} \label{2015STe1}
U^{n+1}= 
\mathcal S_{\tau} \mathcal N_{\tau} \mathcal S_{\tau} U^n, 
\qquad n\ge 0.
\end{align}
It is not difficult to check that this particular Strang splitting method is 
second order in time. The following results are established in Section 6. 
\begin{thm}[Effective maximum principle for Strang splitting of Allen-Cahn with Fourier
Collocation] 
Consider \eqref{2015STe1} with $\nu>0$, $\tau>0$.  Then the following hold.

\begin{enumerate}
\item \underline{Effective maximum principle for $ N^{-0.4} \le \tau <\infty$}. 
 Assume $N^{-0.4} \le \tau<\infty$ and 
$\| U^0\|_{\infty} \le 1$. If $N\ge N_1(\nu)>0$ (the dependence of $N_1$ on $\nu$ is only 
power-like), then  for all $n\ge 0$, we have
\begin{align}  \notag
&\| U^n \|_{\infty} \le 1+ O(N^{-0.5}).
\end{align}

\item \underline{Effective maximum principle for $0<\tau < N^{-0.4} $}. 
 Assume $0< \tau <N^{-0.4}$ and 
$\| U^0\|_{\infty} \le 1$.
If $N\ge N_2(\nu)>0$ (the dependence of $N_2$ on $\nu$ is only 
power-like), then 
\begin{align}  \notag
\sup_{n\ge 0} \| U^n \|_{\infty} \le 1+ \epsilon_1,
\end{align}
where $0<\epsilon_1<10^{-2}$ is an absolute constant. More precisely the following hold.
\begin{enumerate}
\item For some $T_0=O(N^{-0.3})$ sufficiently small and for all $n\ge T_0/\tau$, we have
\begin{align*}
& \| U^n \|_{\infty} \le 1+ O(N^{-0.29}),
\end{align*}
\item For $1\le n \le T_0/\tau$, we have 
\begin{align*}
& \|U^n - \mathcal S_{2n\tau} U^0 \|_{\infty}
\le O(N^{-0.3}); \\
& \|
 \mathcal S_{2n\tau} U^0 
\|_{\infty} \le (1+\frac 1 2 \epsilon_1) \|U^0\|_{\infty};  \\
& \|U^n\|_{\infty} \le (1+\frac 12 \epsilon_1)
\|U^0\|_{\infty} + O(N^{-0.3}).
\end{align*}
\item If $\|U^0\|_{\infty} \le \frac 1 {1+\frac 12 \epsilon_1}$ 
(note that $1/(1+\frac 12 \epsilon_1)>0.995$), then 
\begin{align*}
\sup_{n\ge 0} \|U^n\|_{\infty} \le 1 + O(N^{-0.3}).
\end{align*}

\item Suppose $f:\, \mathbb T\to \mathbb R$ is continuous and $\|f\|_{\infty} \le 1$. 
If $(U^0)_l= f(\frac l N)$ for all $0\le l\le N-1$,  then
\begin{align*}
\sup_{n\ge 0} \|U^n\|_{\infty} 
\le 1 + O(\omega_f(N^{-\frac 23}) ) + O(N^{-c}),
\end{align*}
where $c>0$ is an absolute constant, and $\omega_f$ is defined in \eqref{N18:1b.0}.

\item  Suppose $f:\, \mathbb T\to \mathbb R$ is $C^{\alpha}$-continuous for some $0<\alpha<1$ and $\|f\|_{\infty} \le 1$.  
If $(U^0)_l= f(\frac l N)$ for all $0\le l\le N-1$,  then
\begin{align*}
\sup_{n\ge 0}\|U^n\|_{\infty}
\le 1 + O(N^{-c_1}),
\end{align*}
where $c_1>0$ is a constant depending only on $\alpha$. 
\end{enumerate}
Moreover we have the following result which shows the sharpness of our estimates above.
There exists a  function $f$: $\mathbb T \to \mathbb R$, continuous at all of $\mathbb T
\setminus \{x_*\}$ for some $x_*\in \mathbb T$ (i.e. continuous at all $x\ne x_*$) and  has the bound $\|f\|_{\infty} \le 1$
 such that 
the following hold:
for a sequence of even numbers $N_m \to \infty$, $\tau_m =\frac 14 \nu^{-2}N_m^{-2}$
, $n_m\ge 1$,  and $x_m
=j_m/N_m$ with $0\le j_m\le N_m-1$, if $\tilde U_m \in \mathbb R^{N_m}$ satisfies
\begin{align*}
\tilde U_m = \underbrace{e^{\frac 1 2 \tau_m \Delta_h} \mathcal N_{\tau_m} e^{\frac 12
\tau_m \Delta_h} \cdots 
e^{\frac 1 2 \tau_m \Delta_h} \mathcal N_{\tau_m} e^{\frac 12
\tau_m \Delta_h}}_{\text{iterate $n_m$ times}}  a_m,
\end{align*}
where $a_m \in \mathbb R^{N_m} $ satisfies 
$(a_m)_l = f(\frac l {N_m})$ for all $0\le l\le N_m-1$. Then 
\begin{align*}
|\tilde U_m(x_m) | \ge 1+ \eta_*, \qquad\forall\, m\ge 1,
\end{align*}
where $\eta_*>0.001$ is an absolute constant.
\end{enumerate}
\end{thm}

Our analysis is certainly not restricted to the Allen-Cahn equation and be generalized
and developed in many other directions.  In Section $7$ and $8$ we introduce effective maximum principles for the 1D Burgers equation and the 2D Navier-Stokes equation with periodic boundary conditions. Quite interestingly, these almost sharp
maximum principles exhibit similar strong stability properties much as the Allen-Cahn case, although they are derived using slightly different mechanisms. In the following we give a short summary of the 
results obtained in Section $7$ and $8$.

Consider the 1D Burgers equation:
\begin{align} \label{20151DB1}
\begin{cases}
\partial_t u + \Pi_N ( u u_x) = \nu^2 u_{xx}, \quad (t,x) \in (0,\infty) \times \mathbb T; \\
u\Bigr|_{t=0} = u^0= \Pi_N f,
\end{cases}
\end{align}
where $f \in L^{\infty}(\mathbb T)$. 

\begin{thm}[Effective maximum principle for continuous in time Burgers with spectral
truncation]
Let $\nu>0$ and $u$ be the solution to \eqref{20151DB1}. 
We have for all $N\ge 2$, 
\begin{align*}
\sup_{t\ge 0} \|u(t) \|_{\infty}
\le \|u^0\|_{\infty} + \alpha \cdot N^{-c}, 
\end{align*}
where $c>0$ is an absolute constant, and $\alpha>0$ depends only on 
($\|f\|_{\infty}$, $\nu$). 
\end{thm}
\begin{rem}
The dependence of $\alpha$ on $1/\nu$ is at most power-like for $0<\nu \lesssim 1$. 
Also if one work with $u^0=\tilde P_N f$, where $\tilde P_N$ is a nice Fourier projector
such as the Fejer's kernel which satisfies $\|\tilde P_N f \|_{\infty} \le \|f\|_{\infty}$ and
$\operatorname{supp}(\tilde P_N f ) \subset \{ k:\, -\frac N2 <k \le \frac N2\}$,  then
we obtain in this case 
\begin{align*}
\sup_{t\ge 0} \|u(t) \|_{\infty}
\le \|f\|_{\infty} + O(N^{-c}).
\end{align*}
Such pre-processing of initial data is quite easy to implement in practice.
\end{rem}

Consider the following time-discretized version of \eqref{20151DB1}.
\begin{align} \label{20151DB2}
\begin{cases}
\frac{u^{n+1}-u^n} 
{\tau}  + \Pi_N ( u^n \partial_x u^n) = \nu^2 
\partial_{xx} u^{n+1}, \quad (t,x) \in (0,\infty) \times \mathbb T; \\
u\Bigr|_{t=0} = u^0= \Pi_N f,
\end{cases}
\end{align}
where $f \in L^{\infty}(\mathbb T)$. 

\begin{thm}[Effective maximum principle, Burgers with forward Euler and spectral truncation]
Let $\nu>0$ and $u^n$ be the solution to \eqref{20151DB2}. 
We have for all $0<\tau<\tau_0=\tau_0(\nu, \|f\|_{\infty})$ and $N\ge 2$, 
\begin{align*}
\sup_{n\ge 0} \|u^n \|_{\infty}
\le \|u^0\|_{\infty} + \alpha_1 \cdot N^{-\frac 12} +\alpha_2 \tau^{0.49}, 
\end{align*}
where $\alpha_1, \alpha_2>0$ depend only on 
($\|f\|_{\infty}$, $\nu$). 
\end{thm}
\begin{rem}
One should observe that the above $L^{\infty}$ bound has a very nice pattern:
\begin{align*}
\sup_{n\ge 0} \|u^n \|_{\infty} \le ``\text{$L^{\infty}$ of initial data}" + ``
\text{Spectral error}" + ``\text{Time discretization
error}".
\end{align*}
The decay rate in $N$ and $\tau$ is probably not optimal but we will not dwell on it here.
\end{rem}

We now consider the Fourier collocation method applied to the 1D Burgers. 
Let $U(t) \in \mathbb R^N$, $N\ge 2$ solve
\begin{align} \label{20151DB8}
\begin{cases}
\frac d {dt} U + \frac 12 \partial_h ( U^{\cdot 2} )=\Delta_h U,  \quad
t>0, \\
U\Bigr|_{t=0} = U^0 \in \mathbb R^N,
\end{cases}
\end{align}
where for convenience we have set the viscosity coefficient $\nu=1$, and the operators
$\partial_h$, $\Delta_h$ correspond to the Fourier multipliers
$2\pi i k$, $-4\pi^2 k^2$ for $-\frac N2<k\le \frac N2$ respectively in the DFT formula.  

\begin{thm}[Effective maximum principle, 1D Burgers with Fourier collocation and regular
data] 
Consider \eqref{20151DB8} with $N\ge 2$. Suppose $f$ is a real-valued function on $\mathbb T$
having $\hat f$ supported in $\{ |k| \le \frac N2 \}$ with
$\hat f (\frac N2) = \hat f( -\frac N2)$ (so that $f=Q_Nf$, where $Q_N$ is defined in
the beginning part of Section \ref{S:DFT1}).
Let $(U^0)_j= f(\frac j N)$ for all $0\le j\le N-1$.
Then for all $N\ge N_0=N_0(\|f\|_{2})$, we have
\begin{align*}
\sup_{t\ge 0} \|U(t) \|_{\infty}
 \le \|f \|_{\infty} + \gamma_1 \cdot N^{-\frac 12} 
\cdot (1+ \| \partial_x f \|_2^2),
\end{align*}
where $\gamma_1>0$ depends only on $\|f\|_{2}$.
\end{thm}
\begin{rem}
We stress that the condition on the initial function $f$ can certainly be removed, see
Theorem \ref{thm7.4} for a much more general and technical result for rough
initial data. Our results give
a clear explanation and justification of almost sharp $L^{\infty}$ bounds observed in practical numerical simulations.  By using the machinery developed in this work, it is also possible
to give a comprehensive analysis of the fully discrete scheme such as forward Euler in time with
implicit treatment of the dissipation term, and Fourier collocation in space. All these will be addressed
elsewhere.
\end{rem}

Consider on the torus $\mathbb T^2 = [0,1)^2$ the two dimensional Navier-Stokes
system expressed in the vorticity form:
\begin{align} \label{2015NS1}
\begin{cases}
\partial_t \omega + \Pi_N( u \cdot \nabla \omega) = \Delta \omega, \\
\omega \Bigr|_{t=0} = \omega^0 = \Pi_N f,
\end{cases}
\end{align}
where $f\in L^{\infty}(\mathbb T^2)$ and has mean zero.
We shall work with $\omega$ having zero mean which is clearly preserved in time. The
velocity $u$ is connected to the vorticity $\omega$ through the Biot-Savart law:
$u=\nabla^{\perp} \Delta^{-1} \omega$.

\begin{thm}[Effective maximum principle for 2D Navier-Stokes with spectral Galerkin
truncation] 
Let  $\omega$ be the solution to \eqref{2015NS1}. 
We have for all $N\ge 2$, 
\begin{align*}
\sup_{t\ge 0} \|\omega(t) \|_{\infty}
\le \|\omega^0\|_{\infty} + \alpha \cdot N^{-\frac 12}, 
\end{align*}
where  $\alpha>0$ depends only on 
$\|f\|_{\infty}$. 
\end{thm}

For the time-discretized case, consider
\begin{align} \label{2015NS2}
\begin{cases}
\frac{\omega^{n+1}-\omega^n} 
{\tau}  + \Pi_N ( u^n \cdot \nabla \omega^n) = 
\Delta \omega^{n+1}, \quad (t,x) \in (0,\infty) \times \mathbb T^2; \\
\omega\Bigr|_{t=0} = \omega^0= \Pi_N f,
\end{cases}
\end{align}
where $f \in L^{\infty}(\mathbb T^2)$ and has mean zero. 

\begin{thm}[Effective maximum principle for 2D Navier-Stokes, Forward Euler with spectral
truncation]
Let  $\omega^n$ be the solution to \eqref{2015NS2}. 
We have for all $0<\tau<\tau_0=\tau_0(\|f\|_{\infty})$ and $N\ge 2$, 
\begin{align*}
\sup_{n\ge 0} \|\omega^n \|_{\infty}
\le \|\omega^0\|_{\infty} + \alpha_1 \cdot N^{-\frac 12} +\alpha_2 \tau^{0.4}, 
\end{align*}
where $\alpha_1, \alpha_2>0$ depend only on  $\|f\|_{\infty}$. 
\end{thm}

\subsection*{Notation and preliminaries}

For any $k=(k_1,\cdots,k_d) \in \mathbb R^d$, we denote
\begin{align*}
&|k| = |k|_2= \sqrt{k_1^2+\cdots+k_d^2}, \quad
|k|_{\infty} =\max_{1\le j\le d} |k_j|; \\
&\langle k \rangle  = (1+|k|^2)^{\frac 12}.
\end{align*}

For any two positive quantities $X$ and $Y$, we shall write $X\lesssim Y$ or $Y\gtrsim X$ if
$X \le  CY$ for some  constant $C>0$ whose precise value is unimportant. 
We shall write $X\sim Y$ if both $X\lesssim Y$ and $Y\lesssim X$ hold.
We write $X\lesssim_{\alpha_1,\cdots, \alpha_k}Y$ if the
constant $C$ depends on some parameters $\alpha_1,\cdots, \alpha_k$. We shall 
write $X=O(Y)$ if $|X| \lesssim Y$ and $X=O_{\alpha_1,\cdots,\alpha_k}(Y)$ if $|X| \lesssim_{\alpha_1,\cdots, \alpha_k} Y$. 

We shall denote $X\ll Y$ if
$X \le c Y$ for some sufficiently small constant $c$. The smallness of the constant $c$ is
usually clear from the context. The notation $X\gg Y$ is similarly defined. Note that
our use of $\ll$ and $\gg$ here is \emph{different} from the usual Vinogradov notation
in number theory or asymptotic analysis.

We adopt the following convention for Fourier transforms. Denote for
$f:\; \mathbb R^d \to \mathbb C$,  $g:\mathbb R^d \to \mathbb C$, 
\begin{align*}
(\mathcal F f)(\xi) = \int_{\mathbb R^d} f(x) e^{-2\pi i x \cdot \xi} dx, \\
(\mathcal F^{-1} g)(x) = \int_{\mathbb R^d}  g(\xi ) e^{ 2\pi i \xi \cdot x } d\xi.
\end{align*}

For any function $f$ defined on the periodic torus $\mathbb T^d = \mathbb R^d/\mathbb Z^d$
which we identify as $[0,1)^d$, denote
\begin{align*}
\widehat f (k) = \int_{[0,1)^d} f(x) e^{-2\pi i k \cdot x} dx.
\end{align*}
Occasionally we identify $\mathbb T^d=[-\frac 12, \frac 12)^d$ in order to
isolate the singularity near $x=0$. 

For $0<\gamma<1$ and $f:\mathbb T^d \to \mathbb R$, we define 
\begin{align*}
&\| f \|_{\dot C^{\gamma}(\mathbb T^d) } = \sup_{x \ne y} \frac {|f(x) - f(y) |} {
|x-y|^{\gamma} }, \\
& \| f \|_{C^{\gamma}(\mathbb T^d)} = \| f \|_{L^{\infty}(\mathbb T^d)}
+ \| f \|_{\dot C^{\gamma}(\mathbb T^d) }.
\end{align*}

Recall the Dirichlet kernel and unscaled Fejer kernel:
\begin{align*}
&D_j(x) =\frac{\sin(2j+1)\pi x} {\sin \pi x} = \sum_{|l| \le j} e^{2\pi i l x}; \quad j\ge 0; \\
&\tilde F_k =(\frac {\sin (k+1) \pi x } {\sin \pi x} )^2= \sum_{j=0}^k D_j (x), \quad
k\ge 0; \\
\end{align*}
We also define $\tilde F_{-1}=0$, $\tilde F_{-2}=1$ and note that
\begin{align*}
2 \cos 2 \pi k x = \tilde F_k -2\tilde F_{k-1} +\tilde F_{k-2}, \qquad\forall\, k\ge 0.
\end{align*}

Now we recall the usual Poisson summation formula which will be used sometimes without
explicit mentioning.
\begin{lem}[Poisson summation] 
Let $\delta$ be the usual Dirac distribution. Then 
\begin{align*}
& \ \sum_{k \in \mathbb Z^d} e^{-2\pi i k \cdot x}=\sum_{n \in \mathbb Z^d} \delta(x-n).
\end{align*}
For $f\in L^1(\mathbb R^d)$ with $|f(x)|+|(\mathcal F f)(x)| \lesssim (1+|x|)^{-d-\epsilon}$ for some $\epsilon>0$,
we have
\begin{align*}
& \sum_{k \in \mathbb Z^d} (\mathcal F f)(k)=\sum_{n \in \mathbb Z^d} f(n), \\
&   \sum_{k \in \mathbb Z^d}  (\mathcal Ff)(k) e^{2\pi ik\cdot y}
=\sum_{n \in \mathbb Z^d} f(y+n), \quad\forall\, y \in \mathbb R^d.
\end{align*}
\end{lem}

\begin{rem*}
This is just saying that under suitable decay assumptions the natural periodization of
the original function inherits its Fourier coefficients. 
An immediate useful estimate is: if $|f(x)|+|(\mathcal F f)(x)| \lesssim (1+|x|)^{-d-\epsilon}$ for some $\epsilon>0$, then
\begin{align*}
\|   \sum_{k \in \mathbb Z^d}  (\mathcal Ff)(k) e^{2\pi ik\cdot y}
\|_{L^1_y(\mathbb T^d)} \le \| f \|_{L^1(\mathbb R^d)},
\end{align*}
where equality holds when $f$ has a definite sign.
\end{rem*}
\begin{rem*}
Another useful corollary is as follows. Suppose $f \in L^1(\mathbb T^d)$ and 
has absolutely converging Fourier series expansion with Fourier coefficients
$\hat f(k) $. Then
\begin{align*}
\| \sum_{k} \hat f(k) (\mathcal F g)(k) e^{2\pi i k\cdot x} \|_{L^1_x(\mathbb T^d)}
\le \|f \|_{L_x^1(\mathbb T^d)} \| g \|_{L_x^1(\mathbb R^d)},
\end{align*}
where $g$ and $\mathcal Fg$ is assumed to have sufficient decay.
\end{rem*}

\begin{rem*}
Observe that for $w \in \mathbb R$, one has
\begin{align*}
e^{-|w|} = \frac 1 {\pi} 
\int_{\mathbb R} \frac 1 {1+y^2} e^{i w \cdot y} dy
= \int_0^{\infty}
\frac 1 {\sqrt{\pi s}} e^{-s} e^{-\frac {|w|^2}  {4s} } ds. 
\end{align*}
Then for $x \in \mathbb R^d$,  $\beta\ge 0$,
\begin{align*}
e^{-\beta |x|} = \int_0^{\infty}
\frac 1 {\sqrt{\pi s}} e^{-s} e^{-\beta^2\frac {|x|^2}  {4s} } ds. 
\end{align*}
Thus
\begin{align*}
\int_{\mathbb R^d} e^{-2\pi \beta |x|} e^{-2\pi i\xi \cdot x } dx
=& \pi^{-\frac {d+1}2} \cdot \beta^{-d}\cdot (1+\frac{|\xi|^2} {\beta^2})^{-\frac{d+1}2}
\Gamma(\frac{d+1}2),
\end{align*}
where $\Gamma$ is the usual Gamma function. It follows that
for $\beta>0$, $y \in \mathbb R^d$, 
\begin{align*}
\sum_{k \in \mathbb Z^d}
(1+\frac{|k|^2} {\beta^2})^{-\frac{d+1}2}  e^{2\pi i k\cdot y}
=\pi^{\frac {d+1}2} \cdot \beta^{d}\cdot \frac 1 {
\Gamma(\frac{d+1}2)} \sum_{n \in \mathbb Z^d} e^{-2\pi \beta|y+n|}.
\end{align*}
Interestingly, for dimension $d=1$, one can take the limit $\beta\to 0$ and derive
\begin{align*}
\sum_{0\ne k \in \mathbb Z}  \frac 1 {k^2} e^{2\pi i k |y|} =
2\pi^2 (y^2-|y|+\frac 16), \quad \forall\, y \in [-1,1].
\end{align*}
Denote for $0\le y <1$,
\begin{align*}
F_2(y)= \sum_{0\ne k \in \mathbb Z}  \frac 1 {k^2} e^{2\pi i k y} =
2\pi^2 (y^2-y+\frac 16),
\end{align*}
and extend $F_2$ to the whole $\mathbb R$ via $F_2(y)=F_2(1+y)$.  It is easy
to check that $F_2$ coincides with $2\pi^2(y^2-|y|+\frac 16)$ in the domain $[-1,1]$ and
there is no inconsistency in the definition. 
One should recognize that the function $F_2(y)$ is the usual Bernoulli polynomial.
\end{rem*}

\section{Linear setting: 1D torus continuous case}
Consider the periodic 1D torus $\mathbb T= \mathbb R/\mathbb Z$ which can be identified
as $[0,1)$. For periodic function $f:\, \mathbb T \to \mathbb C$ and integer $N\ge 1$, 
recall
\begin{align*}
\Pi_N f = \sum_{|k|\le N} \hat f(k) e^{2\pi i k \cdot x}.
\end{align*}
Note that $\Pi_N f = D_N * f$, where $D_N$ is the usual Dirichlet kernel given by
\begin{align*}
D_N(x) = \sum_{|k|\le N} e^{2\pi i k \cdot x} = \frac {\sin(2N+1)\pi x} {\sin \pi x}.
\end{align*}

\begin{prop}
Let $N\ge 1$. Then
\begin{align*}
\| \Pi_N f \|_{L^{\infty}(\mathbb T)} = \| D_N* f \|_{L^{\infty}(\mathbb T)}
\lesssim \log (N+2) \| f\|_{L^{\infty}(\mathbb T)}.
\end{align*}
The bound is sharp in the following sense. For each $N\ge 1$, there exists
$f_N \in C^{\infty}(\mathbb T)$, such that
\begin{align*}
| (D_N* f_N)(0) | > c_1 \log (N+2) \cdot \| f_N\|_{L^{\infty}(\mathbb T)},
\end{align*}
where $c_1>0$ is an absolute constant.
\end{prop}
\begin{proof}
This is rather standard. For the lower bound define
\begin{align*}
\tilde f_N =
\begin{cases} \sum_{j=0}^N \chi_{[\frac{2j}{2N+1}, \frac{2j+1}{2N+1})}
- \sum_{j=0}^{N-1} \chi_{[\frac{2j+1}{2N+1}, \frac{2j+2}{2N+1} )}, \quad 0\le x<1,\\
\text{periodic extension.}
\end{cases}
\end{align*}
Mollifying $\tilde f_N$ suitably then yields the lower bound.
\end{proof}
\begin{prop} \label{prop_Omegapp}
Let  $N\ge 1$. Define 
\begin{align*}
\Omega_{N}^+=\{ x\in [0,1]:\;\; D_N(x)>0\} = \bigcup_{j=0}^N (\frac {2j}{2N+1},
\, \frac{2j+1}{2N+1} ), \\
\Omega_N^-=\{ x\in [0,1]:\;\; D_N(x)<0\}
=\bigcup_{j=0}^{N-1}(\frac{2j+1}{2N+1},\, \frac{2j+2}{2N+1}).
\end{align*}
Then $\operatorname{Leb}(\Omega_N^+) =\frac{N+1}{2N+1}$, $\operatorname{Leb}
(\Omega_N^-) =\frac{N}{2N+1}$, and
\begin{align*}
& \int_{\Omega_N^+} D_N(x) dx \ge \frac 1 {\pi^2} \log (2N+3);\\
&\int_{\Omega_N^-} (-D_N(x) ) dx \ge \frac 1 {\pi^2} \log (N+1).
\end{align*}

\end{prop}
\begin{proof}
Clearly
\begin{align*}
 \sum_{j=0}^{N}
\int_{\frac{2j}{2N+1}}^{\frac{2j+1}{2N+1}}
\frac {\sin (2N+1)\pi x} {\sin \pi x} dx
&=\sum_{j=0}^N
\int_0^{\frac 1{2N+1} }
\frac{\sin(2N+1)\pi y} {\sin( \frac{2j\pi}{2N+1} + \pi y)} dy
\ge \sum_{j=0}^N \frac 1 {\frac{2j+1}{2N+1} \pi}
\int_0^{\frac 1 {2N+1}}
\sin( 2N+1)\pi y dy \notag \\
&\ge \frac  2 {\pi^2}  \sum_{j=0}^N 
\frac 1 {2j+1} \ge \frac 2{\pi^2} \int_0^{N+1} \frac 1 {2x+1} dx=\frac 1 {\pi^2}
\log(2N+3).
\end{align*}
Similarly
\begin{align*}
- \sum_{j=0}^{N-1}
\int_{\frac{2j+1}{2N+1}}^{\frac{2j+2}{2N+1}}
\frac {\sin (2N+1)\pi x} {\sin \pi x} dx
\ge \frac 1 {\pi^2} \sum_{j=0}^{N-1}
\frac 1 {j+1} \ge \frac 1 {\pi^2} \log(N+1).
\end{align*}

\end{proof}

Let $\nu>0$, $N\ge 1$, and consider
\begin{align*}
&K_{N}(x) = \sum_{|k|\le N} \frac 1 {1+4\pi^2 \nu^2 k^2} e^{2\pi i k\cdot x}, \quad
x \in \mathbb T;\\
&K_{>N}(x) = \sum_{|k|> N} \frac 1 {1+4\pi^2 \nu^2 k^2} e^{2\pi i k\cdot x}, \quad
x \in \mathbb T.\\
\end{align*}
Note that $\|K_{>N}\|_{1} \le \|K_{>N}\|_2= (\sum_{|k|>N}
\frac 1 {(1+4\pi^2 \nu^2 k^2)^2} )^{\frac 12} 
\lesssim \nu^{-2} \cdot N^{-\frac 32}.$ A better bound is available. See below.

\begin{prop} \label{prop_KN_lu_0001}
Let $\nu>0$. We have
\begin{align*}
\frac {c_1} {1+ (\pi N\nu)^2} \log (N+2)
\le \| K_{>N} \|_{L^1(\mathbb T)} \le \frac {c_2} {1+ (\pi N\nu)^2} \log (N+2), \quad
\forall\, N\ge 1,
\end{align*}
where $c_1>0$, $c_2>0$ are absolute constants.
\end{prop}

\begin{proof}
We first show the upper bound.

Choose $\phi \in C_c^{\infty}(\mathbb R)$ such that $\phi (x) =1$ for $|x| \le 1$ 
and $\phi(x)=0$ for $|x| \ge 2$. Choose $\phi_1 \in C_c^{\infty}(\mathbb R)$
such that $\phi_1(x)=1$ for $1\le |x| \le 2$ and $\phi_1(x)=0$ for $
x\le 0.9$ or $x\ge 2.1$.  Then
\begin{align*}
K_{>N}(x)& = \sum_{|k|>N} \frac 1 {1+4\pi^2 \nu^2 k^2}
\phi(\frac k N) e^{2\pi i k\cdot x} + \sum_{|k|>N}
\frac 1 {  1+ 4\pi^2 \nu^2 k^2} (1- \phi(\frac k N) ) e^{2\pi i k\cdot x}\\
& = \sum_{N<|k|\le 2N} \frac 1 {1+4\pi^2 \nu^2 k^2}
\phi(\frac k N) e^{2\pi i k\cdot x} + \sum_{k \in \mathbb Z}
\frac 1 {  1+ 4\pi^2 \nu^2 k^2} (1- \phi(\frac k N) ) e^{2\pi i k\cdot x} \\
& = \sum_{N<|k|\le 2N} \frac 1 {1+4\pi^2 \nu^2 k^2}
\phi(\frac k N) \phi_1(\frac kN) e^{2\pi i k\cdot x} + \sum_{k \in \mathbb Z}
\frac 1 {  1+ 4\pi^2 \nu^2 k^2} (1- \phi(\frac k N) ) e^{2\pi i k\cdot x} \notag \\
&=: A(x) + B(x).
\end{align*}
For the first piece, denote 
\begin{align*}
K_1(x) = \mathcal F^{-1} ( \frac 1 { 1+(2\pi \nu \xi)^2} \phi(\frac {\xi} N)
\phi_1( \frac {\xi} N) )
\end{align*}
and note that (this is where the localization $\phi_1$ is needed)
\begin{align*}
\| K_1 \|_{L^1(\mathbb R)} \lesssim \frac 1 { 1+ (\pi \nu N)^2}.
\end{align*}
Observe 
\begin{align*}
A(x) &= \int_{\mathbb R} ( \frac 1 {1+(2\pi \nu \xi)^2} 
\phi(\frac {\xi } N) \phi_1( \frac {\xi} N) ) 
\cdot ( \sum_{N<|k|\le 2N} \delta (\xi - k) ) e^{2\pi i \xi \cdot x} d\xi \notag \\
& = \sum_{n \in \mathbb Z} \int_0^1 K_1(x-y-n) \tilde D_N (y) dy,
\end{align*}
where 
\begin{align*}
\tilde D_N (y) = \sum_{N<|k|\le 2N} e^{2\pi i k\cdot y}.
\end{align*}
Thus
\begin{align*}
\| A\|_{L^1(\mathbb T)} \lesssim \| K_1 \|_{L^1(\mathbb R)} \| \tilde D_N \|_{L^1(\mathbb T)}
\lesssim \frac 1 {1+ (\pi \nu N)^2} \log N.
\end{align*}

For the second piece we discuss two cases. If $N \le \frac 1 {\pi \nu}$, then we simply have
\begin{align*}
\| B\|_{L^1(\mathbb T)} \lesssim 1.
\end{align*}
If $N> \frac 1 { \pi \nu}$, then one can make use of the identity
\begin{align*}
 & \frac 1 {1+4\pi^2 \nu^2 \xi^2} \cdot (1-\phi(\frac {\xi} N)) \notag \\
 =&\; \frac 1 {4\pi^2 \nu^2 N^2}
 \cdot ( 1- \frac 1 { 1+ 4\pi^2 \nu^2 \xi^2} ) \cdot \frac { 1- \phi( \frac {\xi} N)}
 {( \frac {\xi} N)^2}.
 \end{align*}
 Thus in this case
 \begin{align*}
 \| B \|_{L^1(\mathbb T)} \lesssim \frac 1 { \pi^2 \nu^2 N^2} \lesssim
 \frac 1 {1+ (\pi \nu N)^2}.
 \end{align*}

Now we turn to the lower bound.  We first consider the situation that $N$ is sufficiently large (larger
than an absolute constant). This amounts to showing
\begin{align*}
\| A \|_{L^1(\mathbb T)} \gtrsim \frac 1 { 1+ (\pi \nu N)^2} \log N.
\end{align*}
Denote (below $\phi$ and $\phi_1$ are the same functions used in the definition
of $A(x)$)
\begin{align*}
\tilde A(x)  & = \sum_{N<|k| \le 2N}
\phi( \frac k N) \phi_1( \frac k N) e^{2\pi i k \cdot x} \notag \\
& = \sum_{N<|k| \le 2N} \phi( \frac k N) e^{2\pi i k \cdot x} 
= \sum_{|k|>N} \phi( \frac k N) e^{2\pi i k \cdot x} \notag \\
& =- \sum_{|k|\le N} \phi( \frac k N) e^{2\pi i k \cdot x}  + \sum_{k \in \mathbb Z}
\phi( \frac k N) e^{2\pi i k \cdot x} \notag \\
& =- \sum_{|k|\le N}  e^{2\pi i k \cdot x}  + \sum_{k \in \mathbb Z}
\phi( \frac k N) e^{2\pi i k \cdot x} = - D_N(x) + \sum_{k \in \mathbb Z}
\phi( \frac k N) e^{2\pi i k \cdot x}.
\end{align*}
Clearly for $N$ large,
\begin{align*}
\| \tilde A \|_{L^1(\mathbb T)} \gtrsim \log N.
\end{align*}
Now observe that on $\mathbb T$:
\begin{align*}
\tilde A = (1-  \nu^2 \partial_{xx} ) A.
\end{align*}

Let $\psi \in C_c^{\infty}(\mathbb R)$ be such that $\psi (\xi) =1$ for $0.5\le |\xi| \le 3$ and
$\psi(\xi)=0$ for $|\xi|\le 0.4$ or $|\xi| >4$. Let 
\begin{align*}
m(x) = \int_{\mathbb R} \frac {1+ (2\pi \nu \xi)^2} {1+ (2\pi \nu N)^2} \psi(\frac {\xi} N)
e^{2\pi i \xi \cdot x} d\xi.
\end{align*}
It is easy to check that $\| m\|_{L^1(\mathbb R)} \lesssim 1$.  Now since
\begin{align*}
\tilde A(x) = ( { 1+(2\pi \nu N)^2} ) \sum_{n \in 
\mathbb Z} \int_0^1 m(x-y)  A(y) dy,
\end{align*}
we clearly have
\begin{align*}
\| \tilde A \|_{L^1(\mathbb T)} \lesssim (1+(2\pi \nu N)^2) \cdot \| A \|_{L^1(\mathbb T)}.
\end{align*}
Thus the lower bound for $\| A\|_{L^1(\mathbb T)} $ is shown when $N$ is large.


Now if $N$ is of constant order, we can use the interpolation inequality
\begin{align}
\| K_{>N} \|_2 \lesssim \| K_{>N}  \|_1^{\frac 12} \| K_{>N} \|_{\infty}^{\frac 12}. \notag
\end{align}
Note that for $N\sim 1$, 
\begin{align}
\| K_{>N}  \|_2^2 \sim \sum_{k\gtrsim 1}  \frac 1 {(1+k^2 \nu^2)^2}
\sim \begin{cases}
\frac 1 {\nu}, \quad 0<\nu<1; \\
\nu^{-4}, \quad \nu\ge 1.
\end{cases}
\end{align}
Also
\begin{align}
\| K_{>N} \|_{\infty} 
\lesssim \sum_{k\gtrsim 1}
\frac 1 { 1+ k^2 \nu^2} \lesssim 
\begin{cases}
\frac 1 {\nu}, \quad 0<\nu <1; \\
\nu^{-2}, \quad \nu \ge 1.
\end{cases}
\end{align}
Thus for $N\sim 1$,
\begin{align}
\| K_{>N} \|_1 \gtrsim \frac 1 {1+\nu^2}. \notag
\end{align}
\end{proof}

In Proposition \ref{prop_KN_lu_0001}, the lower bound on $K_{>N}$ can also be obtained from
a more general result, see Proposition \ref{prop_phidN1} below. 

We first need a simple lemma.

\begin{lem} \label{theta_complex_1}
For any $\theta_0\in [0,1)$, there exits $\alpha_1$ ($\alpha_1$ may depend on
$\theta_0$) with $0<10^{-6}< \alpha_1<1-10^{-6}$, such that
 \begin{align*}
 \operatorname{Im}( e^{2\pi i (\theta_0 +x) }) \ge 10^{-9}, \quad
 \forall\, x \in [\alpha_1,\,\alpha_1+10^{-7} ].
 \end{align*}
\end{lem}
\begin{proof}
Obvious. One just need to discuss separately the cases $\theta_0 $ in different intervals
 and choose the corresponding shifts suitably. 
\end{proof}

Let $\phi \in \mathcal S(\mathbb R)$ and note that
$\widehat \phi(1)$ may be complex-valued. 
If $\widehat{\phi}(1) \ne 0$, then by Lemma \ref{theta_complex_1}, we have for some 
$0<10^{-6}< \alpha_1<1 -10^{-6}$,
\begin{align*}
\operatorname{Im}(e^{2\pi i y} \frac { \widehat{\phi}(1) }
 {|\widehat \phi (1)| } ) \ge 10^{-9}, \quad \forall\,
 y\in [\alpha_1,\, \alpha_1+10^{-7} ].
\end{align*}

Define 
\begin{align*}
Q_N ( D_N) (x) 
& = \int_{\mathbb R} N \phi(Ny ) \frac {\sin( (2N+1) \pi (x-y) )} 
{ \sin ( \pi (x-y) )} dy \notag \\
& =\int_{\mathbb R} \phi(y) \frac 
{ \sin (  (2N+1)  \pi (x- \frac 1 Ny) )} { \sin ( \pi (x-\frac 1 Ny) ) } dy.
\end{align*}
One typical case is that $\widehat \phi (\xi) $ is localized to $\{|\xi| \sim 1 \}$ so that 
$Q_N(D_N)$ is the frequency projection of the Dirichlet kernel to the frequency annulus
$\{ |\xi| \sim N \}$.  Note that if $x=\frac {2j_0+\epsilon} {2N+1}$, where $j_0$ is an  integer, 
$N^{\frac 1{10} } \le j_0 \le N^{\frac 12} $, 
$\alpha_1 \le \frac {\epsilon}2
\le \alpha_1+10^{-7} $ with $\alpha_1$ the same as in Lemma \ref{theta_complex_1}, 
then  as $N\to \infty$, 
\begin{align*}
Q_N(D_N) (x) \cdot x & = \int_{\mathbb R} 
\phi(y) \frac { \sin ( ( 2+\frac 1 N) \pi ( \frac  N{2N+1} \epsilon -y) ) }
{ \sin ( \pi ( \frac {2j_0+\epsilon} {2N+1} - \frac y N) )}  
( \frac {2j_0+\epsilon} {2N+1 } )dy \notag \\
& \to   \frac 1 {\pi} \int_{\mathbb R} \phi(y) 
\sin (2\pi (\frac 12 \epsilon-y)) dy  \notag \\
& = \frac 1 { \pi } \operatorname{Im}( e^{\pi i \epsilon} \widehat \phi (1) )
\gtrsim \; |\widehat{\phi}(1)|.
\end{align*}
This heuristic computation then leads to the following proposition.
\begin{prop} \label{prop_phidN1}
If $\widehat{\phi}(1)=0$, then
\begin{align*}
\lim_{N\to \infty} \frac{\| Q_N(D_N)\|_{L^1(\mathbb T)} }
{\log N} =0.
\end{align*}
If $\widehat{\phi}(1) \ne 0$, then
there exists some $N_0 =N_0(\phi)$ such that if $N\ge N_0$, then 
\begin{align*}
c_1 |\widehat\phi(1)| \log N \le \| Q_N (D_N) \|_{L^1(\mathbb T)} \le c_2
|\widehat \phi(1) |  \log N,
\end{align*}
where $c_1>0$, $c_2>0$ are absolute constants.

\end{prop}

\begin{proof}
Consider the case $\widehat{\phi}(1) \ne 0$. 
 We first show the lower bound.
Let $x_0=\frac {2j_0+\epsilon} {2N+1}$, where $j_0$ is an odd integer (this is to ensure
that the intervals corresponding to different $j_0$ are non-overlapping), 
$N^{\frac 1{10} } \le j_0 \le N^{\frac 12} $, 
$\alpha_1 \le \frac {\epsilon}2
\le \alpha_1+10^{-7} $ with $\alpha_1$ the same as in Lemma \ref{theta_complex_1}, 
then  
\begin{align*}
Q_N(D_N) (x_0)  & = \int_{\mathbb R} 
\phi(y) \cdot \chi_{|y| \le N^{\frac 1 {20}}} 
\frac { \sin ( ( 2+\frac 1 N) \pi ( \frac  N{2N+1} \epsilon -y) ) }
{ \sin ( \pi ( \frac {2j_0+\epsilon} {2N+1} - \frac y N) )}  
dy \notag \\
& \quad + \int_{\mathbb R } 
\phi(y) \cdot \chi_{ |y| > N^{\frac 1 {20} }   } 
\frac { \sin ( ( 2+\frac 1 N) \pi ( \frac  N{2N+1} \epsilon -y) ) }
{ \sin ( \pi ( \frac {2j_0+\epsilon} {2N+1} - \frac y N) )}  
 dy.
\end{align*}
Since $\phi$ is a Schwartz function, the second piece above can be easily bounded by
$N^{-10}$ which is negligible as $N$ tends to infinity.  In the computation below we shall
completely ignore this piece. Then
\begin{align*}
Q_N(D_N) (x_0) \cdot x_0 & \approx \int_{\mathbb R} 
\phi(y) \chi_{|y| \le N^{\frac 1{20}} } \frac { \sin ( ( 2+\frac 1 N) \pi ( \frac  N{2N+1} \epsilon -y) ) }
{ \sin ( \pi ( \frac {2j_0+\epsilon} {2N+1} - \frac y N) )}  
( \frac {2j_0+\epsilon} {2N+1 } )dy \notag \\
& = \int_{\mathbb R} 
\phi(y)\chi_{|y| \le N^{\frac 1{20}} } \frac { \sin ( ( 2+\frac 1 N) \pi ( \frac  N{2N+1} \epsilon -y) ) }
{ \sin ( \pi ( \frac {2j_0+\epsilon} {2N+1} - \frac y N) )}  
( \frac {2j_0+\epsilon} {2N+1 }  - \frac y N)dy \notag \\
&\quad + \frac 1 N \int_{\mathbb R} 
y\phi(y) \chi_{|y| \le N^{\frac 1{20}} }\frac { \sin ( ( 2+\frac 1 N) \pi ( \frac  N{2N+1} \epsilon -y) ) }
{ \sin ( \pi ( \frac {2j_0+\epsilon} {2N+1} - \frac y N) )}   dy.
\end{align*}
Note that the second piece is bounded by $O(N^{-\frac 1 {10}})$ and it is negligible.
For the first piece, one notes that $|\frac {x} {\sin x} - 1| \lesssim x^2$ for $|x| \ll 1$. Thus
\begin{align*}
Q_N(D_N) (x_0) \cdot x_0 & = \frac 1 {\pi} \int_{\mathbb R} 
\phi(y) \chi_{|y| \le N^{\frac 1{20}} } { \sin ( ( 2+\frac 1 N) \pi ( \frac  N{2N+1} \epsilon -y) ) }
dy + O(N^{-\frac 1 {10}}) \notag \\
& =\frac 1 {\pi} \underbrace{\int_{\mathbb R} 
\phi(y)  { \sin ( ( 2+\frac 1 N) \pi ( \frac  N{2N+1} \epsilon -y) ) }
dy}_{=:I_N} + O(N^{-\frac 1 {10}}). \notag 
\end{align*}
Note that in $I_N$ there is no dependence on $j_0$. As $N\to \infty$, 
 we have
\begin{align*}
&I_N \to   \int_{\mathbb R} \phi(y) 
\sin (2\pi (\frac 12 \epsilon-y)) dy  
 =  \operatorname{Im}( e^{\pi i \epsilon} \widehat \phi (1) )
\gtrsim \; |\widehat{\phi}(1)|,
\end{align*}
where in the last step we used Lemma \ref{theta_complex_1}. It follows that
\begin{align*}
\sum_{\substack{N^{\frac 1{10}} \le j_0 \le N^{\frac 12}
\\ \text{$j_0$ is odd} } }
\int_{\frac{2j_0+2\alpha_1}{2N+1}}^{  \frac{2j_0+2\alpha_1+2\cdot 10^{-7}} {2N+1}}
Q_N(D_N)(x) dx \gtrsim\, |\widehat{\phi}(1)|
\sum_{\substack{N^{\frac 1{10}} \le j_0 \le N^{\frac 12}
\\ \text{$j_0$ is odd} } } \frac 1 { \frac {j_0} N} \cdot \frac 1 N \gtrsim  (\log N) |\widehat{\phi}(1)|.
\end{align*}

Now we turn to the upper bound.  For $x=\frac{j+\epsilon}{2N+1}$
with $0\le \epsilon \le 1$ and $j$ being an integer in $[0,2N]$, we have
\begin{align*}
Q_N(D_N)(x)  
& =\int_{\mathbb R} \phi(y) \frac 
{ \sin (  (2N+1)  \pi (x- \frac 1 Ny) )} { \sin ( \pi (x-\frac 1 Ny) ) } dy \notag \\
&=(-1)^j \int_{\mathbb R} \phi(y) \frac 
{ \sin (  (2N+1)  \pi (\frac {\epsilon}{2N+1}- \frac 1 Ny) )} { \sin ( \pi (
\frac {j+\epsilon}{2N+1} -\frac 1 Ny) ) } dy.
\end{align*}
For the regime $ j\le \alpha {\log N}$ where
$\alpha= |\widehat{\phi}(1)|/ \| \phi\|_{L^1(\mathbb R)}$, 
we note that  $\| Q_N(D_N)\|_{\infty}
\lesssim \| \phi\|_{L^1(\mathbb R)} \cdot N$, and
\begin{align*}
\| Q_N(D_N) \|_{L^1([\frac j {2N+1}, \frac {j+1} {2N+1} ])} \lesssim  \| \phi\|_{L^1(\mathbb R)}.
\end{align*}
Thus
\begin{align*}
\sum_{j\le \alpha {\log N}} 
\| Q_N(D_N) \|_{L^1([\frac j {2N+1}, \frac {j+1} {2N+1} ])} \lesssim  
\| \phi\|_{L^1(\mathbb R)} \alpha {\log N} \lesssim |\widehat{\phi}(1)|
\log N.
\end{align*}
Next consider the regime $\frac N { \log N} \le j \le 2N$.
Clearly we have
\begin{align*}
\sum_{\frac N{ \log N} 
\le j \le 2N} 
\int_{\frac {j}{2N+1}}^{\frac {j+1} {2N+1}} 
|Q_N(D_N)(x)| dx 
\lesssim \| \phi\|_{L^1(\mathbb R)} \sum_{\frac N{ \log N} 
\le j \le 2N}  \frac 1 j  \lesssim\;
\| \phi\|_{L^1(\mathbb R)} \log\log N 
\lesssim \; |\widehat{\phi}(1)| \log N,
\end{align*}
for $N$ sufficiently large.  Finally consider the regime 
$\alpha {\log N}\le j \le \frac N {\log N}$. 
For $x= \frac {j+\epsilon}{2N+1}$, we have
\begin{align*}
Q_N(D_N)(x) \cdot x
&= (-1)^j \int_{\mathbb R} \phi(y) \frac 
{ \sin (  (2N+1)  \pi (\frac {\epsilon}{2N+1}- \frac 1 Ny) )} { \sin ( \pi (
\frac {j+\epsilon}{2N+1} -\frac 1 Ny) ) } 
(\frac {j+\epsilon}{2N+1} -\frac 1 Ny) dy \notag \\
&\quad+ \frac {(-1)^j} N
\int_{\mathbb R}  y\phi(y) \frac 
{ \sin (  (2N+1)  \pi (\frac {\epsilon}{2N+1}- \frac 1 Ny) )} { \sin ( \pi (
\frac {j+\epsilon}{2N+1} -\frac 1 Ny) ) } dy.
\end{align*}
Note that the second piece above can be easily bounded by $(\log N)^{-1}$. For the first
piece one can use the cut-off $\chi_{|y| \le \frac 1 {100} \alpha \log N}$ and the
inequality $|\frac x {\sin x} -1 | \ll x^2$ for $|x| \ll 1$ to extract the main order. It follows
easily that
\begin{align*}
|Q_N(D_N)(x)| \cdot |x| \lesssim |\widehat {\phi}(1)| + O( (\log N)^{-1} ) 
\lesssim |\widehat{\phi}(1)|,
\end{align*}
for $N$ sufficiently large. Collecting all estimates, we obtain 
\begin{align*}
\| Q_N(D_N) \|_{L^1(\mathbb T)} \lesssim |\widehat{\phi}(1)| \cdot \log N.
\end{align*}

Finally we turn to the case $\widehat{\phi}(1)=0$. The analysis is similar. First
the regime $\frac N{\log N} \le j\le 2N$ is acceptable since it gives at most $\log\log N$
growth.  The other two regimes are $j\le \sqrt {\log N}$ and 
$\sqrt{\log N} \le j \le \frac N {\log N}$. We omit the details.
\end{proof}

In the next few propositions, we outline an alternative approach to obtain the $L^1$-norm
bound on $K_{>N}$. The advantage is that it can be used on more general trigonometric 
series whose coefficients satisfy certain convexity properties.

\begin{prop} \label{prop_upperlower}
Let $\nu>0$. Let $N_0= e^{3\pi^2}$. If $N\ge \max\{N_0, \,\frac 1 {2\sqrt 3 \pi \nu}\}$,
then
\begin{align*}
\frac {c_1} {(N\nu)^2} \log N 
\le \| K_{>N} \|_{L^1(\mathbb T)} \le  \frac {c_2} {(N\nu)^2} \log N,
\end{align*}
where $c_1>0$, $c_2>0$ are absolute constants. 

\end{prop}
\begin{rem*}
In the regime $N\ll \frac 1{\nu}$ (for $0<\nu\ll 1$), the $L^1$ bound of $K_{>N}$ and $K_N$ 
can also be quite
large. For example, fix $N=N_0$ independent of $\nu$ and take $\nu\to0$, then the main
order of $K_{N_0}$ is given by the Dirichlet kernel, and we have (below we use $L^2$ to estimate
the error piece)
\begin{align*}
\| K_{N_0} \|_{L^1(\mathbb T)} &\ge  \| \frac {\sin (2N_0+1) \pi x}{ \sin \pi x}
\|_{L^1(\mathbb T)} - \operatorname{const} \cdot \nu^2 \cdot N_0^{2.5} \notag \\
& \ge \; \operatorname{const}
\cdot \log N_0 - \operatorname{const} \cdot \nu^2 \cdot N_0^{2.5} \notag \\
& \gtrsim \log N_0\, ,
\end{align*}
where $\nu>0$ is sufficiently small. 
 Noting that
$\| K_{N_0} + K_{>N_0} \|_{L^1(\mathbb T)} = 1$, we  obtain a similar lower
bound for $K_{>N_0}$.

\end{rem*}

\begin{proof}
Denote 
\begin{align*}
&c_k = \frac 1 {1+4\pi^2 \nu^2 k^2}, \\
&\tilde F_k (x)= \sum_{j=0}^k D_j (x) = \left( \frac {\sin (k+1) \pi x} {\sin \pi x} \right)^2.
\end{align*}
Note that $\|\tilde F_k\|_1=k+1$ and $\tilde F_k -\tilde F_{k-1} =D_k$. Then
\begin{align*}
K_{>N}(x)=\sum_{|k|\ge N+1}
\frac 1 {1+4\pi^2 \nu^2 k^2} e^{2\pi i k\cdot x} 
& = \sum_{k\ge N+1}  c_k ( \tilde F_k - 2 \tilde F_{k-1} +\tilde F_{k-2}) \notag \\
& = \sum_{k\ge N+1} c_k \tilde F_k
-2 \sum_{k\ge N} c_{k+1} \tilde F_k + \sum_{k\ge N-1} c_{k+2} \tilde F_k \notag \\
& =\sum_{k\ge N} (c_k-2c_{k+1} +c_{k+2}) \tilde F_k
-c_N \tilde F_N + c_{N+1} \tilde F_{N-1} \notag \\
& =\sum_{k\ge N} (c_k-2c_{k+1} +c_{k+2}) \tilde F_k
-(c_N-c_{N+1})\tilde F_{N-1} -c_{N} D_N.
\end{align*}
Note that if $12 k^2 \pi^2 \nu^2-1\ge 0$, then
\begin{align*}
c_k -2 c_{k+1} +c_{k+2} \ge 0.
\end{align*}
It is easy to check that
\begin{align*}
(c_N-c_{N+1}) \| \tilde F_{N-1}\|_{L^1(\mathbb T)}  =
\frac{4\pi^2 (2N+1) \nu^2 N} {(1+4\pi^2 N^2 \nu^2) (1+
4\pi^2 (N+1)^2 \nu^2) } \le \frac 2 {1+ 4\pi^2 N^2 \nu^2}.
\end{align*}
By Proposition \ref{prop_Omegapp}, we have
\begin{align*}
\int_{\Omega_N^{-}} K_{>N}(x) dx \ge
\frac {1} {1+ 4\pi^2 N^2 \nu^2}  \frac 1 {\pi^2}
\log (N+1)-\frac 2 {1+ 4\pi^2 N^2 \nu^2}
\gtrsim  \frac 1 {(N\nu)^2} \log N.
\end{align*}
Thus 
\begin{align*}
\| K_{>N} \|_{L^1(\mathbb T)} &\ge \|K_{>N} \|_{L^1(\Omega_N^-)}
\ge \int_{\Omega_N^-} K_{>N}(x) dx  \\
&\gtrsim  \frac 1 {(N\nu)^2} \log N.
\end{align*}

For the upper bound, one notes that $\int_{\mathbb T} K_{>N}(x) dx=0$, and 
\begin{align*}
\| K_{>N} \|_{L^1(\mathbb T)} 
\le 2 N(c_N-c_{N+1}) + 2c_N \| D_N \|_{L^1(\mathbb T)} \lesssim 
\frac 1 {(N\nu)^2} \log N.
\end{align*}

\end{proof}

\begin{prop} \label{prop_upperlower1}
Let $\nu>0$.  There exists an absolute constant $N_1>0$ such that if $
N_1\le N$ and $N\le   \frac 1 {2 \sqrt 3 \pi \nu}$,
then
\begin{align*}
 &\| K_{N} \|_{L^1(\mathbb T)} \sim  \log N, \\
 &\| K_{>N} \|_{L^1(\mathbb T)} \sim  \log N.
\end{align*}
Consequently by Proposition \ref{prop_upperlower}, for all $N \ge  \max\{N_1, e^{3\pi^2}\}$,
we have
\begin{align*}
\| K_{>N} \|_{L^1(\mathbb T)} \sim \frac 1 {1+(N\nu)^2} \log N.
\end{align*}
\end{prop}
\begin{proof}
We only need to show that $\| K_{N} \|_{L^1(\mathbb T)} \gtrsim \log N$. 
Denote again $c_k = \frac 1 {1+4\pi^2 \nu^2 k^2}$.  Then
\begin{align*}
K_{N}(x) &=\sum_{|k|\le 4} c_k e^{2\pi i k \cdot x} 
+ \sum_{4<k\le N} c_k (\tilde F_{k}- 2\tilde F_{k-1} + \tilde F_{k-2}) \notag \\
& = \sum_{|k|\le 4} c_k e^{2\pi i k \cdot x}  +
\sum_{k=4}^{N-1} (c_k -2 c_{k+1} +c_{k+2}) \tilde F_k
+c_{N} D_{N}  \notag \\
& \qquad +(c_{N}-c_{N+1}) \tilde F_{N-1} + (c_5-c_{4}) \tilde F_{3}
-c_4 D_4.
\end{align*}
Note that for $k \le N-1$, we have
\begin{align*}
c_k -2 c_{k+1} +c_{k+2}
= \frac {\beta}{1+\beta(k+1)^2}
\cdot \frac { 2\beta(3k^2+6k+2) -2} { (1+\beta k^2) (1+\beta (k+2)^2) } \le 0,
\end{align*}
where $\beta = (2\pi \nu)^2$.  Here we used the assumption $N\le \frac 1 {2\sqrt 3 \pi \nu}$. 

Recall $\Omega_N^- = \{ x\in \mathbb T:\, D_N (x) <0\}$. Then
\begin{align*}
\| K_N\|_{L^1(\mathbb T)} \ge \| K_N \|_{L^1(\Omega_N^-)} \ge
c_N \| D_N\|_{L^1(\Omega_N^-)} - \operatorname{const}
- (c_N-c_{N+1}) N \gtrsim \log N.
\end{align*}

\end{proof}

%
%
%

\begin{lem} \label{lem_FNlower1}
Let 
\begin{align*}
F_N(x) = - \frac{\sin (2N+1)\pi x} {\sin \pi x} - \frac 2 N
\cdot \left( \frac {\sin N \pi x} {\sin \pi x } \right)^2.
\end{align*}
There exists an absolute constant $N_0>0$ sufficiently large, 
such that  the following hold for any $N\ge N_0$, $\epsilon \in [\frac 1 {10},\frac 16]$:

If $x= \frac {j+1-\epsilon}{2N}$, $N^{\frac 1{10}} <j <N^{\frac 12}$, and $j$ is odd, then
\begin{align*}
F_N(x) \ge \beta_1 \frac N j,
\end{align*}
where $\beta_1$ is an absolute constant.
\end{lem}
\begin{proof}
Note that 
\begin{align*}
&-\sin (2N+1)\pi x= (-1)^{j+1} \sin (\epsilon \pi - \frac {j+1-\epsilon}{2N} \pi)
=\sin (\epsilon \pi - \frac {j+1-\epsilon}{2N} \pi) \ge \frac 12 \epsilon \pi, \\
& 2\sin^2 N\pi x= 1-\cos 2N\pi x=1-(-1)^{j+1} \cos (\epsilon \pi)=
\frac 12 \epsilon^2 \pi^2 + O(\epsilon^4),
\end{align*}
where in the first inequality above we have used the assumptions on $N$ and $\epsilon$. 
Then 
\begin{align*}
&- \frac{\sin (2N+1)\pi x} {\sin \pi x} \ge 
 \operatorname{const} \cdot \epsilon \cdot \frac N j,\\
& \frac 2 N
\cdot \left( \frac {\sin N \pi x} {\sin \pi x } \right)^2
\le \operatorname{const} \cdot \frac N {j^2} \cdot \epsilon^2.
\end{align*}
The desired result then clearly follows.

\end{proof}

Lemma \ref{lem_FNlower1} leads to an interesting point-wise lower bound on $K_{>N}$.
It also yields another proof for $\|K_{>N}\|_{L^1(\mathbb T)}$. We record it here
for the sake of completeness.

\begin{prop}
Let $\nu>0$. If $N\ge \max\{N_0, \, \frac 1 {2\sqrt 3 \pi \nu} \}$ where $N_0$ is the
same absolute constant as in Lemma \ref{lem_FNlower1}, then for any
$0\le x<1$, 
\begin{align*}
K_{>N}(x) \ge \beta_1 \sum_{\substack{N^{\frac 1{10}} <j<N^{\frac 12}\\
\text{$j$ is odd} }} 
\frac 1 {1+4\pi^2 \nu^2 N^2}
\cdot \frac N j \cdot 1_{[\frac {j+\frac 56}{2N}, \, \frac{j+\frac 9 {10}} {2N} ]},
\end{align*}
where $\beta_1$ is the same absolute constant as in Lemma \ref{lem_FNlower1}. 
Consequently
\begin{align*}
\| K_{>N}\|_{L^1(\mathbb T)} \gtrsim \frac 1 {1+4\pi^2 \nu^2 N^2}
\log (N+2).
\end{align*}
\end{prop}
\begin{proof}
Clearly
\begin{align*}
K_{>N}(x)=\sum_{|k|\ge N+1}
\frac 1 {1+4\pi^2 \nu^2 k^2} e^{2\pi i k\cdot x} 
& =\sum_{k\ge N} (c_k-2c_{k+1} +c_{k+2}) \tilde F_k - c_N D_N
-(c_N-c_{N+1})\tilde F_{N-1},
\end{align*}
where we use again the kernel
\begin{align*}
\tilde F_k (x)= \sum_{j=0}^k D_j (x) = \left( \frac {\sin (k+1) \pi x} {\sin \pi x} \right)^2
\end{align*}
and 
\begin{align*}
c_k = \frac 1 {1+4\pi^2 \nu^2 k^2}. 
\end{align*}

Note that for $k\ge N\ge \frac 1 {2\sqrt 3 \pi \nu}$, we have
\begin{align*}
c_k -2 c_{k+1} +c_{k+2} \ge 0.
\end{align*}
Noting that $c_N-c_{N+1} \le c_N\cdot \frac 2N$, we obtain
\begin{align*}
\sum_{|k|\ge N+1}
\frac 1 {1+4\pi^2 \nu^2 k^2} e^{2\pi i k\cdot x} 
&\ge  \frac 1 {1+4\pi^2 \nu^2 N^2} \cdot
\left[- \frac {\sin (2N+1)\pi x} {\sin \pi x}
- \frac 2 N \cdot \left(
\frac {\sin N\pi x} {\sin \pi x } \right)^2\right].
\end{align*}
The point-wise lower bound then easily follows from Lemma \ref{lem_FNlower1}. The 
$L^1$ lower bound follows from direct integration.
\end{proof}
\begin{prop} \label{prop_tmp211}
Let $\nu>0$. If $N\ge  \frac 1 {2\pi^2 \nu}
 e^{\frac 1 {2\nu}}$, then $K_N(x) >0$ for all
$x \in \mathbb T$ and $\| K_N\|_{L^1(\mathbb T)} =1$.

The dependence of $N$ on $\nu$
is almost sharp in the following sense: for $0<\nu\ll 1$, if $
\frac 1{2\sqrt 3  \pi \nu} \le N \lesssim \frac 1 {\sqrt{\nu}} e^{\frac 1{6\nu}}$, then 
\begin{align*}
K_N<- \frac {0.3} {1+4 \pi^2 \nu^2 (N+1)^2}
\end{align*}
 on a set of measure $O(1)$ and
 \begin{align*}
 \|K_N\|_{L^1(\mathbb T)} \ge 1+ \frac {\gamma_1}{1+4 \pi^2 \nu^2 (N+1)^2},
 \end{align*}
 where $\gamma_1$ is an absolute constant.
\end{prop}


\begin{proof}
For the first result, note that 
\begin{align*}
K_N(x) &\ge \frac 1 {2\nu} \sum_{n\in \mathbb Z}
e^{-\frac 1 {\nu} |x+n|} - \sum_{|k|>N}  \frac 1 {1+4\pi^2 \nu^2 k^2} \notag \\
&> \frac 1 {\nu} e^{-\frac 1 {2\nu}} -  
\frac 1 {4\pi^2 \nu^2}
\frac 2 {N} \ge 0, 
\end{align*}
if $N\ge  \frac 1 {2\pi^2 \nu}
 e^{\frac 1 {2\nu}}$.

For the last result we recall the identity for $K_{>N}(x)$:
\begin{align*}
\sum_{|k|\ge N+1}
\frac 1 {1+4\pi^2 \nu^2 k^2} e^{2\pi i k\cdot x} 
& =\sum_{k\ge N} (c_k-2c_{k+1} +c_{k+2}) \tilde F_k 
-c_N \tilde F_N + c_{N+1} \tilde F_{N-1}.
 \end{align*}

Note that for $k\ge N\ge \frac 1 {2\sqrt 3 \pi \nu}$, we have
\begin{align*}
c_k -2 c_{k+1} +c_{k+2} \ge 0.
\end{align*}

Now if we take $x_0= \frac {j_0}{N+1}$, where $j_0$ is an integer in $[1,N]$, then 
$\tilde F_N=0$ and 
\begin{align*}
\tilde F_{N-1} = \biggl( \frac 
{ \sin ( N\pi \frac {j_0} {N+1} ) } { \sin (\frac {j_0 \pi} {N+1} ) } \biggr)^2 =1.
\end{align*}
It follows that for $x_0= \frac {j_0}{N+1}$, $N\ge  \frac 1 {2\sqrt 3 \pi \nu}$,
\begin{align*}
K_{>N}(x_0)
\ge  \frac 1 { 1+ 4\pi^2 \nu^2 (N+1)^2}.
\end{align*}
Thus 
\begin{align*}
K_N(x_0) &\le \frac 1 {2\nu} \sum_{n\in \mathbb Z}
e^{-\frac 1 {\nu} |x_0+n|} - \frac 1 {1+4 \pi^2 \nu^2 (N+1)^2} \notag\\
&< \frac 1 {\nu} \cdot 2\frac { e^{-\frac 1{\nu}  \min\{x_0,1-x_0\}  }} {1-e^{-\frac 1{\nu}}} 
- \frac 1 {1+4 \pi^2 \nu^2 (N+1)^2}. 
\end{align*}
Assume $j_0\in [\frac 13 (N+1), \frac 23 (N+1)]$ such that $x_0 \in [\frac 13, \frac 23]$, then
for $
\frac 1{2\sqrt 3  \pi \nu} \le N \lesssim \frac 1 {\sqrt{\nu}} e^{\frac 1{6\nu}}$, we have
\begin{align*}
K_N(x_0) < \frac 1 {\nu} \cdot \frac 2{e^{\frac 1 {3\nu}} (1-e^{-\frac 1 {\nu} }) } 
- \frac 1 {1+4 \pi^2 \nu^2 (N+1)^2} 
<- \frac {0.5} {1+4 \pi^2 \nu^2 (N+1)^2}.
\end{align*}
One may then pick $O(N)$ of  $j_0$ from the interval $[\frac 13 (N+1), \frac 23 (N+1)]$,
and choose $x_0 =\frac {j_0+\epsilon} {N+1}$, with $\epsilon \in [0,\epsilon_*]$
($\epsilon_*$ is a sufficiently small absolute constant), such that
\begin{align*}
K_N(x_0)<- \frac {0.3} {1+4 \pi^2 \nu^2 (N+1)^2}.
\end{align*}
Note that the total measure of such $x_0$ is $O(1)$. Finally we should point out that since $\int_{\mathbb T}
K_N dx=1$, the above point-wise bound then yields the (slightly inferior) $L^1$ lower
bound of $K_N$ as:
\begin{align*}
\|K_N\|_{L^1(\mathbb T)} \ge 1+ \frac {\gamma_1}{1+4 \pi^2 \nu^2 (N+1)^2},
\end{align*}
where $\gamma_1$ is an absolute constant. Note that we lose a logarithm here compared
to the optimal bound.
\end{proof}


\subsection{Generalization to convex sequences}

The simple method used in the proof of Proposition \ref{prop_upperlower} is quite robust and can
be generalized. For example, consider a sequence of  real
numbers $(c_k)_{k\ge 0}$ such that  $\sup_{k\ge 0} |c_k| <\infty$ and the following hold:
\begin{enumerate}
\item for some $k_0\ge 0$, 
\begin{align*}
c_k -2c_{k+1}+c_{k+2} \ge 0, \quad \forall\, k\ge k_0.
\end{align*}

\item $ \lim_{k \to \infty} (c_{k+1}-c_k) k =0$. 

\item $\lim_{k\to \infty} c_k \log k =0$.

\end{enumerate}
\begin{rem}\label{rem12implies}
Condition (1) and (2) implies that $\sum_{k=0}^{\infty} | c_{k} -2 c_{k+1} +
c_{k+2} | \cdot (k+1) <\infty$.
\end{rem}

\begin{rem*}
Condition (3) cannot be deduced from (1) and (2). For example, let $c_k = \frac 1 {\log\log k}$, then condition (1) and (2) are satisfied, but
not condition (3).
\end{rem*}
\begin{rem*}
Condition (2) and (3) is convenient for extracting convergence rates. These conditions can be
weakened further provided one works with other norms such as total variational norms and so
on.
However, we do not dwell on this issue here.
\end{rem*}

Now let
\begin{align*}
G_N(x) = \sum_{|k| \le N} c_{|k|} e^{2\pi i k\cdot x} = c_0+ 2 \sum_{k=1}^N 
c_k \cos (2 \pi k \cdot x).
\end{align*}

\begin{thm}
Suppose the sequence $(c_k)_{k\ge 0}$ satisfies the conditions (1), (2) and (3). Then
$G_N$ converges in $L^1$ to a function $G_{\infty}$ as $N$ tends to infinity.
Furthermore the following upper and lower estimates hold:
\begin{align*}
& \| G_N-G_{\infty} \|_{L^1(\mathbb T)}
\le 2 |c_N-c_{N+1}| \cdot N + \alpha_1|c_N| \log N, \quad \forall\, N\ge 4; \\
&\| G_N - G_{\infty} \|_{L^1(\mathbb T)} \ge c_N \frac 1 {\pi^2}
\log (N+1) - |c_N-c_{N+1}| \cdot N, \quad \forall\, N\ge 4,
\end{align*}
where $\alpha_1>0$ is an absolute constant. 
\end{thm}
\begin{rem*}
This theorem is quite handy in practical applications. For example, one can take
$c_k =k^{-\alpha}$, $k\ge 1$, where $\alpha>0$ (note that we do not require $\alpha>1$!). 
It is easy to check that conditions (1), (2) and (3) are satisfied. Then corresponding to
this sequence $(c_k)_{k\ge 0}$ we have
\begin{align*}
\| G_N - G_{\infty} \|_{L^1(\mathbb T)} \sim N^{-\alpha} \log N, 
\end{align*}
for all large $N$. 
\end{rem*}

\begin{proof}
Note that for $4\le N<M$, 
\begin{align*}
G_M - G_N & = \sum_{N<k\le M} c_k (\tilde F_k - 2 \tilde F_{k-1} + \tilde F_{k-2} ) \notag \\
& = \sum_{k=N+1}^M c_k \tilde F_k -2 \sum_{k=N}^{M-1} c_{k+1} \tilde F_{k} +
\sum_{k=N-1}^{M-2} c_{k+2} \tilde F_{k} \notag \\
& =\sum_{k=N}^{M-2} (c_k -2 c_{k+1} +c_{k+2}) \tilde F_k
+c_M \tilde F_M +c_{M-1} \tilde F_{M-1} -2 c_M \tilde F_{M-1} -c_N \tilde F_N
+c_{N+1} \tilde F_{N-1} \notag \\
&=\sum_{k=N}^{M-2} (c_k -2 c_{k+1} +c_{k+2}) \tilde F_k
+c_M D_M +(c_{M-1}-c_{M}) \tilde F_{M-1} - (c_N-c_{N+1}) \tilde F_{N-1}
-c_N D_N.
\end{align*}
Since (by Remark \ref{rem12implies}) $\sum_{k} |c_k -2 c_{k+1} +c_{k+2}| (k+1)<\infty$, it follows easily that $G_N$ is Cauchy
in $L^1$ and converges to an $L^1$ function as $N\to \infty$.  The upper and lower
bounds are similarly estimated as before. We omit the details.
\end{proof}
\begin{rem*}
One can also change the metric and obtain more precise point-wise convergence results with suitable weights.
For example, fix any $\delta_0>0$, one can show (say identify $\mathbb T=[-\frac 12,
\frac 12)$)
\begin{align*}
\|G_N-G_{M}\|_{L_x^{\infty}(\delta_0<|x|<\frac 12)} \to 0,
\end{align*}
as $N<M$ tends to infinity.  
\end{rem*}

\subsection{PDE results}
Let $\nu>0$ and consider the following linear elliptic PDE posed on the 1D torus $\mathbb 
T=[0,1)$:
\begin{align} \label{simplest_elpde}
u - \nu^2 \partial_{xx} u = \Pi_N f,
\end{align}
where we recall $\Pi_N$  defined as
\begin{align*}
\Pi_N f = \sum_{|k|\le N} \hat f(k) e^{2\pi i k \cdot x}.
\end{align*}
This equation is perhaps one of the simplest cases for the spectral Galerkin method. By using
the results derived earlier, we have the following theorem.
\begin{thm} \label{thm2.14old}
Let $f\in L^{\infty}(\mathbb T)$. Then the unique solution $u$ to \eqref{simplest_elpde}
satisfies the following:

\begin{itemize}
\item Strict maximum principle. If $N\ge  \frac 1 {2\pi^2 \nu} e^{\frac 1 {2\nu}}$, then
\begin{align*}
\| u\|_{L^{\infty}(\mathbb T)} \le \| f \|_{L^{\infty}(\mathbb T)}.
\end{align*}
\item Generalized maximum principle. For all $N\ge 1$, 
\begin{align*}
\| u \|_{L^{\infty}(\mathbb T)} \le
\left( 1+ \frac{\alpha_1} { 1+ (\pi \nu N)^2} \log (N+2) \right) \| f\|_{L^{\infty}(\mathbb T)},
\end{align*}
where $\alpha_1>0$ is an absolute constant.

\item Lower bound. For each $N\ge 1$, there exists smooth $f^{(N)}$ with $
\|f^{(N)}\|_{L^{\infty}(\mathbb T)}=1$, such that the corresponding solution
$u^{(N)}$ satisfies
\begin{align*}
\| u^{(N)} \|_{L^{\infty}(\mathbb T)} \ge 
\frac {\alpha_2} { 1+ (\pi N\nu)^2} \log (N+2) -1,
\end{align*}
where $\alpha_2>0$ is an absolute constant.

\item Lower bound for small $\nu$.  Assume $0<\nu\ll 1$.  For each $N$ satisfying 
$\frac 1{2\sqrt 3  \pi \nu} \le N \lesssim \frac 1 {\sqrt{\nu}} e^{\frac 1{6\nu}}$, 
there exists smooth $f^{(N)}$ with $
\|f^{(N)}\|_{L^{\infty}(\mathbb T)}=1$, such that the corresponding solution
$u^{(N)}$ satisfies
 \begin{align*}
 \|u^{(N)} \|_{L^{\infty}(\mathbb T)} \ge 1+ \frac {\alpha_3}{1+4 \pi^2 \nu^2 (N+1)^2},
 \end{align*}
 where $\alpha_3>0$ is an absolute constant.
 
\end{itemize}

\end{thm}

\begin{proof}
The first and fourth result follow from Proposition \ref{prop_tmp211}.
The second and third results follow from Proposition \ref{prop_KN_lu_0001}.

\end{proof}

\subsection{The case $d\ge 2$}
We now discuss some higher dimensional analogues of the previous results.
In practical numerical computations, we usually use the projection 
\begin{align*}
\Pi_N=\Pi_N^{(d)} = \Pi_{N,k_1} \cdots \Pi_{N, k_d},
\end{align*}
where $\Pi_{N,k_j}$ refers to projection into each frequency coordinate $k_j$. In
yet other words $\Pi_N^{(d)}$ is the frequency truncation with $|k|_{\infty}\le N$. 

We begin with a technical lemma.

\begin{lem} \label{lem2.15_00}
Let $d\ge 1$, $\beta>0$, and $N\ge 2$.
Suppose $\phi \in C_c^{\infty}(\mathbb R^d)$ is such that
$\phi(\xi)=1$ for $|\xi| \le \frac 32\sqrt d$, and 
$\phi(\xi )=0$ for $|\xi|>2\sqrt d$. Then
\begin{align*}
\| 
\sum_{|k|_{\infty}>N}
\frac 1 {1+4\pi^2 \beta |k|^2}
\phi(\frac k N) e^{2\pi i k\cdot x }
\|_{L_x^1(\mathbb T^d)} \sim
\frac 1 {1+\beta N^2} ( \log (N+2))^d,
\end{align*}
where the implied constants depend only on $d$ and the function $\phi$.
\end{lem}
\begin{proof}
First we show the upper bound. Let $\phi_1 \in C_c^{\infty}(\mathbb R^d)$
be such that $\phi_1(\xi)=1$ when $1\le |\xi| \le 2\sqrt d$,
and $\phi_1(\xi)=0$ for $|\xi|<\frac 12$ or $|\xi|>3\sqrt d$. Then clearly we have
\begin{align*}
 &\|  \sum_{|k|_{\infty}>N}
\frac 1 {1+4\pi^2 \beta |k|^2}
\phi(\frac k N) e^{2\pi i k\cdot x }
\|_{L_x^1(\mathbb T^d)} \notag \\
=&\;
\| \sum 1_{N<|k|_{\infty} \le 2\sqrt d N} 
\cdot \frac 1 {1+4\pi^2 \beta |k|^2} 
\phi(\frac k N) \phi_1(\frac  k N) e^{2\pi i k \cdot x} \|_{L_x^1(\mathbb T^d)} 
\notag \\
\le & \| \tilde D_N \|_{L_x^1(\mathbb T^d)} \cdot \| H\|_{L_x^1(\mathbb R^d)
} \lesssim ( \log (N+2) )^d \cdot \frac 1 {1+\beta N^2},
\end{align*}
where $\tilde D_N$ is a Dirichlet kernel corresponding to the part
$1_{N<|k|_{\infty} \le 2\sqrt d N}=
1_{|k|_{\infty}\le 2\sqrt d N} - 1_{|k|_{\infty} \le N}$ (Note that both summands 
naturally split as product
of one-dimensional kernels) and $H=\mathcal F^{-1} ( \frac 1 {1+4\pi^2 \beta |\xi|^2} 
\phi(\frac {\xi} N) \phi_1(\frac  {\xi} N) )$. 

Next we show the lower bound. Note that 
\begin{align*}
\tilde A(x) = \sum_{|k|_{\infty}>N}
\phi(\frac k N) e^{2\pi i k \cdot x}
= - \sum_{|k|_{\infty} \le N} e^{2\pi i k \cdot x} + 
\sum_{k \in \mathbb Z^d} \phi(\frac k N) e^{2\pi i k\cdot x}
\end{align*}
which implies
\begin{align*}
\| \tilde A \|_{L^1_x(\mathbb T^d)} \gtrsim
\; (\log N)^d, \quad \text{for $N$ large}.
\end{align*}
Observe 
\begin{align*}
\tilde A(x) = (1+4\pi^2 \beta N^2)
\sum_{|k|_{\infty} >N}
\frac 1 {1+4\pi^2 \beta |k|^2}
\phi(\frac k N) \phi_1(\frac k N)
\frac {1+4\pi^2 \beta |k|^2} {1+4\pi^2 \beta N^2} e^{2\pi i k \cdot x}.
\end{align*}
The desired lower bound (for $N$ large) then follows since
\begin{align*}
\| \mathcal F^{-1} ( 
\phi_1(\frac {\xi} N)
\frac {1+4\pi^2 \beta |\xi|^2} {1+4\pi^2 \beta N^2})
\|_{L_x^1(\mathbb R^d)} \lesssim 1.
\end{align*}

Finally for $N\sim 1$ one can note that $A(x)=\sum_{|k|_{\infty}>N}
\frac 1 {1+4\pi^2 \beta |k|^2}
\phi(\frac k N) e^{2\pi i k\cdot x }$ has the bounds:
\begin{align*}
&\| A\|_{L_x^{\infty}(\mathbb T^d)} \lesssim \frac  1 {1+\beta}, \quad
\| A\|_{L_x^2(\mathbb T^d)} \sim \frac 1 {1+\beta}.
\end{align*}
One can then use interpolation to get the lower bound. 

\end{proof}

Recall the Japanese bracket notation $\langle x \rangle = (1+|x|^2)^{\frac 12}$
for $x \in \mathbb R^d$. Let $\epsilon>0$ and consider 
\begin{align*}
g_N(x) &= \sum_{|k|_{\infty} \le N} \langle k \rangle^{-\epsilon}
e^{2\pi i k \cdot x}, \\
\tilde g_N(x) &= \sum_{k \in \mathbb Z^d}
\langle k \rangle^{-\epsilon}
\psi(\frac k N) e^{2\pi i k\cdot x},
\end{align*}
where $\psi \in C_c^{\infty}(\mathbb R^d)$ satisfies $\psi(\xi)=0$
for $|\xi| \le \sqrt d$, and $\psi(\xi)=0$ for $|\xi|>2\sqrt d$.

\begin{lem} \label{lem2.16_00}
Both $(g_N)_{N\ge 2}$ and $(\tilde g_N)_{N\ge 2}$ are Cauchy in $L^1(\mathbb T^d)$, and
\begin{align*}
\| g_N -\tilde g_N \|_{L_x^1(\mathbb T^d)} \lesssim N^{-\epsilon} \cdot (\log N)^d
\to 0,\qquad \text{as $N\to \infty$}.
\end{align*}
\end{lem}
\begin{proof}
Due to the assumption on the support of $\psi$, we see that $\psi(\frac k N)=1$ when 
$|k|_{\infty} \le N$, and $\psi(\frac k N)=0$ when $|k|_{\infty} >2N$. 
Choose $\psi_1 \in C_c^{\infty}$ with support in $\{ |\xi| \sim 1\}$ such that
$\psi_1(\xi)=1$ on the support on $\psi$.  Then
\begin{align*}
\| g_N -\tilde g_N\|_{L_x^1(\mathbb T^d)}
& = \| \sum_{N<|k|_{\infty}\le 2N} 
\langle k \rangle^{-\epsilon} \psi(\frac k N)  \psi_1(\frac k N) e^{2\pi i k\cdot x}
\|_{L_x^1(\mathbb T^d)} \notag \\
& \lesssim \; \| \tilde D_N
\|_{L_x^1(\mathbb T^d)} \|
\int \langle \xi \rangle^{-\epsilon}
\psi(\frac {\xi} N) \psi_1(\frac {\xi} N) e^{2\pi i \xi \cdot x} d
\xi \|_{L_x^1(\mathbb R^d)} \lesssim 
(\log N)^d \cdot N^{-\epsilon},
\end{align*}
where $\tilde D_N$ is the Dirichlet kernel corresponding to $1_{N<|k|_{\infty} \le 2N}$.

It then suffices for us to show that $(\tilde g_N)_{N\ge 2}$ is Cauchy in $L^1_x(\mathbb
T^d)$. Observe that for $M>N\ge 2$, we have
\begin{align*}
\| \tilde g_N -\tilde g_M\|_{L_x^1(\mathbb T^d)}
\le \| \int_{\mathbb R^d}
\langle \xi \rangle^{-\epsilon}
(\psi (\frac {\xi} N) - \psi(\frac {\xi} M) ) e^{2\pi i \xi \cdot x}
d\xi \|_{L_x^1(\mathbb R^d) } \to 0,
\end{align*}
as we take $M>N$ to infinity.

\end{proof}

Lemma \ref{lem2.16_00} can be used to show the following folklore identity that for $\epsilon>0$,
\begin{align*}
\sum_{k \in \mathbb Z^d} \langle 2\pi k \rangle^{-\epsilon}
e^{2\pi i k \cdot x} = \sum_{n \in \mathbb Z^d}
f(x+n),
\end{align*}
where $f=\mathcal F^{-1}( \langle 2\pi \xi \rangle^{-\epsilon})$ is the usual Bessel potential.
Note that even though $f$ has exponential decay, the decay of the Fourier coefficient
is not enough to apply the classic version of the Poisson summation formula.
To resolve this the natural idea is to truncate the LHS and in some sense Lemma
\ref{lem2.16_00} assures the equivalence of different two truncations. The convergence of the Fourier series on the LHS can then be understood in  $L_x^1(\mathbb T^d)$ sense. Stronger point-wise asymptotics
are also available but we will not dwell on this issue here.

 Consider the operator
$T_{>N, \beta}= (\operatorname{Id}- \beta \Delta)^{-1} (\operatorname{Id}
-\Pi_{N}^{(d)})$.  Note that it is spectrally localized to the sector $\{k:\, |k|_{\infty}>N\}$.
Denote the corresponding kernel as $K_{>N, \beta}$.
\begin{thm} \label{thm2.17_00}
Let $N\ge 2$. 
$K_{>N, \beta}$ is an $L^1$ function  on $\mathbb T^d$ and 
\begin{align*}
\| K_{>N,\beta} \|_{L^1(\mathbb T^d)} \sim \frac 1 {1+\beta N^2}
\log^d (N+2),
\end{align*}
where the implied constants depend only on $d$.
\end{thm}
\begin{proof}
Choose $\phi$ as in Lemma \ref{lem2.15_00} and note that
\begin{align*}
K_{>N,\beta}(x)
&=\sum_{|k|_{\infty}>N}
\frac 1 {1+4\pi^2 \beta |k|^2}
\phi(\frac k N) e^{2\pi i k \cdot x}
+\sum_{k \in \mathbb Z^d}
\frac 1 {1+4\pi^2 \beta |k|^2}
(1-\phi(\frac k N) ) e^{2\pi i k \cdot x} \notag \\
&=: \,A_1(x) +A_2(x).
\end{align*}
Here we used the fact that when $1-\phi(\frac k N)=0$ we must have
$|k|\ge \frac 32 N \sqrt d$ which implies $|k|_{\infty} \ge \frac 32 N$
(so that the condition $|k|_{\infty}>N$ can be dropped). 

To bound $A_2$ we can use Lemma \ref{lem2.16_00} and the discussion afterwards which gives
\begin{align*}
\| A_2\|_{L_x^1(\mathbb T^d)
}
\le\; \| \mathcal F^{-1}( \frac 1 {1+4\pi^2 \beta |\xi|^2}
(1- \phi(\frac {\xi} N) ) ) \|_{L_x^1(\mathbb R^d) } \lesssim 
\frac 1 {1+ \beta N^2}.
\end{align*}
By Lemma \ref{lem2.15_00} we have
\begin{align*}
\|A_1\|_{L_x^1(\mathbb T^d)}\sim
\frac 1 {1+\beta N^2} ( \log (N+2))^d,
\end{align*}
and thus the desired bound  for $K_{>N, \beta}$ follows when $N$ is large.

Note that the desired upper bound for $K_{>N, \beta}$ holds for all $N\ge 2$ and
we only need to show the lower bound when $N$ is of order $1$.  To this end consider
\begin{align*}
&B(x) = \sum_{|k|_{\infty}>N}
\frac 1 { |k|^{d+1}} \cdot \frac 1 {1+4\pi^2 \beta |k|^2} e^{2\pi i k \cdot x}, \\
&B_1(x) =\sum_{|k|_{\infty}>N} \frac 1 {|k|^{d+1}} e^{2\pi i k\cdot x}.
\end{align*}
Easy to check that for $N\sim 1$, 
\begin{align*}
\| B\|_{L_x^2(\mathbb T^d)} \sim \frac 1 {1+\beta}, \quad
\|B_1\|_{L_x^2(\mathbb T^d)} \sim 1.
\end{align*}
Since $B= B_1*K_{>N,\beta}$ (here $*$ denotes the usual convolution on the torus),
we then have
\begin{align*}
\frac 1 {1+\beta} \lesssim \|B\|_{L_x^2(\mathbb T^d)} \le
 \|B_1\|_{L_x^2(\mathbb T^d)} \cdot \|K_{>N,\beta} \|_{L_x^1(\mathbb T^d)}
 \lesssim \| K_{>N,\beta} \|_{L_x^1(\mathbb T^d)}.
 \end{align*}
Thus the desired lower bound holds also when $N$ is of order $1$.

\end{proof}

\begin{thm} \label{thm2.18_00}
Let $d\ge 1$, $N\ge 2$, $\beta>0$. Then the following hold for $T_{>N, \beta}= (\operatorname{Id}- \beta \Delta)^{-1} (\operatorname{Id}
-\Pi_{N}^{(d)})$ with the kernel $K_{>N,\beta}$, and $T_{N,\beta}
= (\operatorname{Id}-\beta \Delta)^{-1} \Pi_N^{(d)}$.
\begin{enumerate}
\item \underline{Sharp $L_x^1$ bound}. 
\begin{align*}
\frac {c_1} {1+\beta N^2}
(\log (N+2) )^d \le 
\| K_{>N,\beta} \|_{L_x^1(\mathbb T^d)} \le \frac {c_2} {1+\beta N^2}
(\log (N+2))^d,
\end{align*}
where $c_1>0$, $c_2>0$ depend only on the dimension $d$. 
\item \underline{Generalized maximum principle}. For all $N\ge 2$, we have 
\begin{align*}
\| T_{N,\beta} f \|_{L_x^{\infty}(\mathbb T^d)}
\le (1+ \frac {c_3} { 1+\beta N^2} (\log (N+2))^d) \| f \|_{L_x^{\infty}(\mathbb T^d)},
\end{align*}
where $c_3>0$ depend only on the dimension $d$. 

\item \underline{Strict maximum principle}.  Let $d=1,2$.  Denote the kernel of $T_{N,\beta}$ as
$K_{N,\beta}$. For all $N\ge N_0= N_0(\beta,d)$ (i.e. $N_0$ is constant depending only
on $\beta$ and $d$), it holds that
$K_{N,\beta}$ is a strictly positive function on $\mathbb T^d$ with unit $L^1_x$ mass, and
consequently
\begin{align*}
\| T_{N, \beta} f \|_{L_x^{\infty}(\mathbb T^d)} \le \|f \|_{L_x^{\infty}(\mathbb T^d)}.
\end{align*}
\end{enumerate}

\end{thm}
\begin{rem*}
An interesting problem is to investigate the maximum principle for the spherical operator
$\tilde T_{N,\beta}= (\operatorname{Id}-\beta \Delta)^{-1} \tilde \Pi_N$ where
$\tilde \Pi_N$ corresponds to the Fourier cut-off $|k|\le N$ (i.e. $|k|_{\infty}$ is replaced by
the usual $l^2$ norm $|k|$). 
\end{rem*}
\begin{proof}
The first two follow from Theorem \ref{thm2.17_00}. 
We only need to show (3). Note that the case $d=1$ is already settled before with explicit dependence
of constants.  The case $d=2$ is proved in the appendix. 
\end{proof}

\subsection{Truncation of general Bessel case}

We first recall the usual Bessel potential on $\mathbb R^d$: for $s>0$,
\begin{align*}
G_s= (\operatorname{Id}-\Delta)^{-\frac s2} \delta_0,
\qquad \widehat{G_s}(\xi)= \langle 2\pi \xi \rangle^{-s}.
\end{align*}
At the origin, we have for $|x|\to 0$,
\begin{align*}
G_s(x)= \begin{cases}
\frac{\Gamma(\frac{d-s}2) } {2^s \pi^{\frac s2} |x|^{d-s}}  (1+o(1)),
\quad 0<s<d;\\
\frac 1 {2^{d-1} \pi^{\frac d2}}
\ln \frac  1 {|x|} (1+o(1)), \quad s=d; \\
\frac {\Gamma(\frac{s-d}2)} {2^s \pi^{\frac s2}}
(1+o(1)), \quad s>d.
\end{cases}
\end{align*}
When $|x|\to \infty$,
\begin{align*}
G_s(x)
= \frac {e^{-|x|}}
{2^{\frac{d+s-1}2}
\pi^{\frac {d-1}2} 
\Gamma(\frac s2)
|x|^{\frac{d+1-s}2}
}
(1+o(1)).
\end{align*}
It follows that  for $s>d$, we have the rough bound
\begin{align*}
G_s(x) \gtrsim \; e^{-2|x|}, \qquad\forall\, x \in \mathbb R^d.
\end{align*}
For $s=d$, we have
\begin{align*}
G_s(x) \gtrsim\; 
\begin{cases}
e^{-2|x|}, 
\quad \forall\, |x| \ge \frac 1 {5}; \\
-\ln |x|, \quad \forall\, 0<|x|<\frac 1 5.
\end{cases}
\end{align*}
For $0<s<d$, we have
\begin{align*}
G_s(x) \gtrsim\; 
\begin{cases}
e^{-2|x|}, 
\quad \forall\, |x| \ge \frac 1 {5}; \\
|x|^{-(d-s)}, \quad \forall\, 0<|x|<\frac 1 5.
\end{cases}
\end{align*}

Denote for $\beta>0$, $N\ge 2$,
\begin{align*}
F_{N,s}(x)= \sum_{|k|_{\infty} \le N} \frac 1
{(1+4\pi^2 \beta |k|^2)^{\frac s2} } e^{2\pi i k \cdot x}, \qquad
x \in \mathbb T^d. 
\end{align*}
We suppress the dependence of $F_{N,s}$ on $\beta$ since we focus only on
studying other parameters while keeping $\beta>0$ fixed. 
It is easy to check that the $L_x^1$ limit of $F_{N,s}$ is given by 
\begin{align*}
F_{\infty,s} (x) =\beta^{-\frac d2}
\sum_{n \in \mathbb Z^d} G_s(\frac {x+n}{\sqrt{\beta}} )
\ge \beta^{-\frac d2} G_s(\frac x {\sqrt{\beta}}),
\qquad \forall\, x \in \mathbb T^d=[-\frac 12,\frac 12)^d.
\end{align*}

\begin{lem}
Let $d\ge 1$ and $s>d$. There exists $N_0=c_{s,d}\cdot (\beta^{\frac {d-1}2} e^{\frac 2 {\sqrt{\beta}} } )^{\frac 1 {s-d}}
$ ($c_{s,d}>0$ depends only on $(s,d)$) such that if $N\ge N_0$,
then each $F_{N,s}$ is positive and hence has unit $L^1_x$ mass.
\end{lem}
\begin{rem*}
Note that for $d=1$ and $s=2$, our $N_0 \sim e^{\frac 2{\sqrt{\beta}}} $ is slightly worse
than the constant $O(\frac 1 {\sqrt{\beta}} e^{\frac 1 {\sqrt {\beta}}})$ 
obtained before. This is due to the use of the
 rough lower bound on $G_s$.
\end{rem*}
\begin{proof}
There is no issue with point-wise convergence since $s>d$. Observe that
\begin{align*}
\sum_{|k|_{\infty}>N} \frac 1
{(1+4\pi^2 \beta |k|^2)^{\frac s2} }\lesssim \beta^{-\frac s2} N^{-(s-d)}.
\end{align*}
It follows that
\begin{align*}
F_{N,s}(x) \gtrsim \; \beta^{-\frac d2}
e^{-2 \frac 1 {\sqrt{\beta}} }  - \beta^{-\frac s2 } N^{-(s-d)}> 0,
\end{align*}
if $N\ge N_0 \sim  (\beta^{\frac {d-s}2} e^{\frac 2 {\sqrt{\beta}} } )^{\frac 1 {s-d}}$.
\end{proof}

The complementary regime $0<s\le d$ requires more delicate analysis. 

First we make an easy note. For all $s>0$, $d\ge 1$, and $|x|\le \frac 1 {6 N\sqrt{d}}$,
we have 
\begin{align*}
2 \pi |k\cdot x | \le \frac {\pi} 3, \quad \forall\, \text{$k\in \mathbb Z^d$
with $|k|_{\infty} \le N$}.
\end{align*}
Clearly then for all  $|x|\le \frac 1 {6 N\sqrt{d}}$, we have
\begin{align*}
F_{N,s}(x) \ge \frac 12 \sum_{|k|_{\infty} \le N}
\frac 1 {(1+4\pi^2 \beta |k|^2)^{\frac s2} } >0.
\end{align*}
Thus to analyze positivity we only need to focus on the regime
$\{x \in [-\frac 12, \frac 12]^d:\; |x|\ge \frac 1 {6N \sqrt{d}}\} $.  In the rest of the analysis we shall only
treat $x$ in this regime.

\subsubsection{Complete analysis for $1D$ case}

We first investigate the 1D case in complete detail. For convenience throughout this
subsection we shall identify
\begin{align*}
\mathbb T=[-\frac 12, \frac 12).
\end{align*}
This is to single out the possible singularity near $0$. 

To continue it is useful to denote (by a slight abuse of notation)
\begin{align*}
c_k =c(k)= \langle 2\pi \sqrt{\beta} k \rangle^{-s}= \frac
1 {(1+4\pi^2 \beta |k|^2)^{\frac s2} }.
\end{align*}
Note that $(\langle x \rangle^{-s})^{\prime\prime}=s \langle x \rangle^{-(s+4)}
\cdot((s+1)x^2-1)$. Thus $c(k)$ is convex whenever $k\ge \frac 1 {2\pi \sqrt{s+1} \sqrt{\beta}}$.
We will use the following identity for any $4\le N<M$:
\begin{align} \label{id847_0_01}
F_{M,s}-F_{N,s}
&=\sum_{k=N}^{M-2} (c_k -2 c_{k+1} +c_{k+2}) \tilde F_k
+c_M D_M +(c_{M-1}-c_{M}) \tilde F_{M-1} - (c_N-c_{N+1}) \tilde F_{N-1}
-c_N D_N,
\end{align}
where $D_N(x) =\frac{\sin(2N+1)\pi x}{\sin \pi x}$ and
$\tilde F_N(x) = ( \frac{\sin(N+1) \pi x} {\sin \pi x} )^2$. 
From this it is not difficult to check that $(F_{N,s})_{N\ge 4}$ converges uniformly on
$\mathbb T_{\delta}=\{x \in [-\frac 12,\frac 12):\, |x|>\delta\}$ for any fixed
small $\delta>0$.  Sending $M$ to infinity and using convexity of $c_k$ (for $k\ge N$)
then yields for $N\ge \frac 1 {2\pi \sqrt{s+1} \sqrt{\beta}}$,
\begin{align} \label{id847_0_02}
F_{N,s}(x) \ge F_{\infty,s}(x)  -h_N(x)+c_N D_N(x), \quad\text{for any $\delta<|x|\le \frac 12$}.
\end{align}
where 
\begin{align*}
h_N(x) = \sum_{k=N}^{\infty} (c_k-2c_{k+1}+c_{k+2}) \tilde F_k.
\end{align*}

\begin{lem}
Let $d=1$ and $0<s\le 1$.  For all $0<s\le 1$ and 
$\frac 15 \sqrt{\beta} \le |x| \le \frac 12$, we have
\begin{align*}
F_{\infty,s}(x)
\gtrsim \, \beta^{-\frac 12} e^{-\frac 1 {\sqrt{\beta}}}.
\end{align*}
For all $0<|x|<\min \{\frac 15 \sqrt{\beta},\, \frac 12\}$, we have
\begin{align*}
F_{\infty,s}(x)
\gtrsim
\begin{cases}
\beta^{-\frac 12} \ln | \frac{\sqrt{\beta}}{|x|} |,
\qquad s=1,\\
\beta^{-\frac 12} ( \frac{\sqrt{\beta}} {|x|} )^{1-s},
\qquad 0<s<1.
\end{cases}
\end{align*}
For any $N\ge 
\frac 1 {2 \pi \sqrt{\beta}}$ and $
\frac 1 {6N}<|x| \le \frac 12$, we have
\begin{align*}
|c_N| + | h_N(x)|
\lesssim \; \beta^{-\frac s2} N^{-s}  \frac 1 {|x|}.
\end{align*}
\end{lem}
\begin{proof}
Direct computation. Note that
\begin{align*}
h_N(x) \lesssim\; |x|^{-2} \sum_{k=N}^{\infty}(c_k-2c_{k+1}+c_{k+2})
=|x|^{-2}(c_N-c_{N+1}) \lesssim |x|^{-2} \beta^{-\frac s2} N^{-s-1}
\lesssim |x|^{-1} \beta^{-\frac s2} N^{-s}.
\end{align*}
\end{proof}
\begin{lem}
Let $d=1$ and $s=1$. Then for any $N\ge N_0(\beta)= 
\alpha \cdot \beta^{-\frac 12}
e^{\frac 1{\sqrt{\beta}}}$ ($\alpha>0$ is some absolute
constant), $F_{N,1}$ is positive 
and has unit $L_x^1$ mass.
\end{lem}
\begin{proof}
Consider first the regime $\frac 15 \sqrt{\beta} \le |x| \le \frac 12$ (note that
in this case for such $x$ to exist we must have $0<\beta\le (\frac 52)^2$). We have
\begin{align*}
F_{N,s}(x) & \ge  \alpha_1 \beta^{-\frac 12} e^{-\frac 1 {\sqrt{\beta}}}
-\alpha_2 \beta^{-\frac 12} \frac 1 {N|x|}\notag \\
& \ge \alpha_1 \beta^{-\frac 12} e^{-\frac 1 {\sqrt{\beta}}}
-5\alpha_2 \beta^{-\frac 12} \frac 1 {N\sqrt{\beta}},
\end{align*}
where $\alpha_1>0$, $\alpha_2>0$ are absolute constants. Clearly then
$N_0 \sim \beta^{-\frac 12} e^{\frac 1 {\sqrt{\beta}}}$ suffices.

Next consider the regime $\frac 1 {6N}\le |x| \le \min\{ \frac 15 \sqrt{\beta}, \, \frac 12 \}$. 
We have 
\begin{align*}
F_{N,s}(x) 
& \ge \alpha_3 \beta^{-\frac 12} \ln
| \frac {\sqrt{\beta}} {|x| } | - \alpha_2 \beta^{-\frac 12} \cdot \frac 1 {N|x|},
\end{align*}
where $\alpha_3>0$ is also an absolute constant. Now if
$\ln| \frac {\sqrt{\beta}} {|x| } |  \ge 2\frac {\alpha_2} {\alpha_3}$, 
clearly we have positivity.  On the other hand if $\ln| \frac {\sqrt{\beta}} {|x| } |  <2\frac {\alpha_2} {\alpha_3}$, then 
\begin{align*}
F_{N,s}(x) \ge \alpha_3 \beta^{-\frac 12}
\ln 5 - \alpha_2 \beta^{-1} \cdot \frac 1 N \cdot e^{2\frac{\alpha_2}{\alpha_3}}>0,
\end{align*}
if we take $N_0 \gg \beta^{-\frac 12}$. 

\end{proof}

We now focus on $0<s<1$.

\begin{lem}
Let $d=1$ and $0<s<1$. If $\frac 15 \sqrt {\beta}
\le |x| \le \frac 12$ and $N\ge N_0(\beta,s)
=\alpha(s) \cdot \beta^{-\frac 12} e^{\frac 1{s\sqrt{\beta}}}$ ($\alpha(s)>0$
depends only on $s$), then $F_{N,s}(x)>0$.
\end{lem}
\begin{proof}
Observe that 
\begin{align*}
F_{N,s}(x)
& \ge \alpha_1(s) \cdot \beta^{-\frac 12} e^{-\frac 1{\sqrt{\beta}}}
-\alpha_2(s) 
\cdot \beta^{-\frac s2} \cdot N^{-s} \cdot \frac 1{\sqrt{\beta}},
\end{align*}
where $\alpha_1(s)>0$, $\alpha_2(s)>0$ depend only on $s$. 
\end{proof}

\begin{lem}
Let $d=1$ and $0<s<1$. Assume $N\ge \frac 1 {2\pi \sqrt{\beta}}$. There exist a constant $\tilde\gamma_1(s)>0$ depending only
on $s$ such that if $\frac {{\tilde \gamma_1(s)}} N \le |x| 
\le \min\{ \frac 15 \sqrt{\beta},\, \frac 12 \}$, then $F_{N,s}(x)>0$.
\end{lem}
\begin{proof}
Observe that 
\begin{align*}
F_{N,s}(x)
&\ge \,
\alpha_3(s) \beta^{-\frac 12}
( \frac{\sqrt{\beta}} {|x|} )^{1-s}
-\alpha_4(s)
\cdot \beta^{-\frac s2}
N^{-s} \cdot \frac 1{|x|} \notag \\
&= \,
\alpha_3(s) \beta^{-\frac s 2} |x|^{-1+s}
-\alpha_4(s)
\cdot \beta^{-\frac s2} 
N^{-s} |x|^{-1},
\end{align*}
where $\alpha_3>0$, $\alpha_4>0$ depend only on $s$. 
Matching $|x|^{-1+s}$ with the term $N^{-s} |x|^{-1}$  then determines the constant
 $\tilde \gamma_1(s)$.
\end{proof}

\begin{rem*}
Yet another way to obtain the positivity for $F_{N,s}(x)$ in the regime 
$\frac 1 N \ll |x| \le \frac 12$ is as follows. For simplicity consider $\beta=
\frac 1 {(4\pi)^2}$ and
\begin{align*}
K_{>N}(x) = \sum_{|k|>N} \langle k \rangle^{-s} e^{2\pi i k \cdot x}
=\sum_{k \in \mathbb Z} \langle k \rangle^{-s}
\eta_N(k) e^{2\pi i k\cdot x},
\end{align*}
where $\eta_{N}$ is a linear interpolation such that
$\eta_N(k)=0$ for $|k|\le N$  and $\eta_N(k)=1$ for $|k|\ge N+1$. 
In particular $\eta_N^{\prime}(k)= \operatorname{sgn}(k)1_{|k| \in[N,N+1]}$. We only
need to show
\begin{align*}
|K_{>N}(x)| \lesssim\; N^{-s} \frac 1 {|x|}.
\end{align*}
It then 
suffices to consider
\begin{align*}
G(x)&= \int_{\mathbb R} \langle \xi \rangle^{-s} \eta_N(\xi) e^{2\pi i \xi
\cdot x} d\xi \notag \\
&= -\frac 1 {2\pi i x} \int_{\mathbb R}
(-s) \langle \xi \rangle^{-s-2} \xi \eta_N(\xi) e^{2\pi i \xi \cdot x} d\xi 
 -\frac 1 {2\pi i x} \int_N^{N+1} \langle \xi \rangle^{-s}
2i \sin 2\pi \xi \cdot x d\xi.
\end{align*}
One can then obtain for $0<|x|\le \frac 12$,
\begin{align*}
|G(x)| \lesssim \; |x|^{-1} \cdot N^{-s},
\end{align*}
and for $|x| \ge \frac 12$,
\begin{align*}
|G(x) | \lesssim |x|^{-2} \cdot N^{-s}.
\end{align*}
The desired estimate then follows. One should observe that this line of computation 
is in some sense analogous to \eqref{id847_0_02} whereas the latter can be viewed as discrete integration by parts.
\end{rem*}
\begin{rem*}
In the preceding remark, one may also take $\eta_N$ to be a smooth function as
\begin{align*}
\eta_N(\xi) =\eta_{1} * (1_{|\xi|>N+\frac 12}),
\end{align*}
where $*$ is the usual convolution on $\mathbb R$, and $\eta_{1}  \in C^{\infty}_c(\mathbb R)$ is an even function such 
that $\eta_{1} (x)=1$ for $|x|\le 0.01$ and $\eta_{1}(x)=0$ for $|x| >0.02$. This way one does not
need to worry about boundary terms when doing integration by parts.
\end{rem*}

Since we have shown the positivity of $F_{N,s}$ for $|x|<\frac 1 {6N}$ and $\frac 1 N\ll |x|\le \frac 12$, it then suffices for us to study the positivity or non-positivity of $F_{N,s}(x)$ for $0<s<1$ and
 $|x| \in (\frac 1 {6N},
\frac{\tilde \gamma_1(s)} N)$. To this end let $x= \frac y N$ and assume that
$|y| \in [\gamma_1, \, \gamma_2]$ for some $0<\gamma_1(s)<\gamma_2(s) $. 
Here
we allow some flexibility and the values of $\gamma_1$ and $\gamma_2$ will become clear later. 
Consider the expression
\begin{align*}
H_{N,s}(y) =N^{-(1-s)}  F_{N,s}(\frac y N)
=\frac 1 {N^{1-s}} \sum_{|k|\le N} \frac 1 {(1+4\pi^2 \beta |k|^2)^{\frac s2} }
e^{2\pi i \frac k N \cdot y}
\end{align*}
and note that
\begin{align*}
H_{\infty,s}(y) =(2\pi \sqrt{\beta})^{-s}  \int_{-1}^1 |t|^{-s} e^{2\pi it\cdot y} dt.
\end{align*}

\begin{lem} \label{lem2.25_00}
We have
\begin{align*}
\max_{|y| \in [\gamma_1(s),\,
\gamma_2(s)]}
|H_{N,s}(y) - H_{\infty,s}(y) | \lesssim_s\; 
N^{-(1-s)} ( 1+\frac 1 {\sqrt{\beta}}).
\end{align*}
\end{lem}

\begin{proof}
Utterly standard and thus we only sketch the details (so as to check the constants).
First note that the tail piece is OK:
\begin{align*}
N^{-(1-s)}
\Bigl|\sum_{|k|\le \frac 1 {\sqrt{\beta}}}
\frac 1 {(1+4\pi^2 \beta |k|^2)^{\frac s2} }
e^{2\pi i \frac k N \cdot y}
\Bigr| \lesssim_s \; (1+\frac 1{\sqrt{\beta}}) N^{-(1-s)}.
\end{align*}
Here we have the factor $1+\frac 1{\sqrt{\beta}}$ instead of purely $\frac 1 {\sqrt {\beta}}$ due
to the $k=0$ term. This is needed especially when $\beta >1$.

Next observe that 
\begin{align*}
& N^{-(1-s)} \Bigl|\sum_{\frac 1 {\sqrt{\beta}} <|k|\le N} (
\frac 1 {(1+4\pi^2 \beta |k|^2)^{\frac s2} } -
\frac 1 {(4\pi^2 \beta |k|^2)^{\frac s2} } )
e^{2\pi i \frac k N \cdot y} \Bigr|
\notag \\
\lesssim_s\; &
N^{-(1-s)} \sum_{\frac 1 {\sqrt{\beta}} <|k|\le N}
\frac 1 {(1+4\pi^2 \beta |k|^2)^{\frac s2+1} } 
 \lesssim_s \; (1+\frac 1{\sqrt{\beta}})N^{-(1-s)}.
\end{align*}
Finally to bound the difference between the integral and the summation, it suffices
to consider the function $g(\xi)= |\xi|^{-s} e^{2\pi i \xi \cdot y}$. Note that on
the mesh $[\frac k N, \frac {k+1} N]$ (say WLOG $k\ge 2$),  we  have
\begin{align*}
\Bigl|g(\xi) -g (\frac k N)\Bigr| \lesssim_s  (\frac k N)^{-s-1} \cdot \frac 1 N= k^{-1-s} \cdot N^{s}.
\end{align*}
Since the mesh size is $1/N$, the total error for the function
$\beta^{-\frac s 2} g(\xi)$ can then be estimated by
\begin{align*}
\beta^{-\frac s2} \sum_{|k|>\frac 1 {\sqrt{\beta}}} |k|^{-1-s} \cdot N^s\cdot \frac 1 N 
\lesssim \; N^{-(1-s)} \cdot (1+\frac 1 {\sqrt{\beta}}).
\end{align*}

\end{proof}

Thanks to Lemma \ref{lem2.25_00}, we have for all $y\in [\gamma_1(s),
\gamma_2(s)]$, 
\begin{align*}
F_{N,s}(\frac yN) = N^{1-s} H_{\infty,s}(y) + O_s(1) \cdot 
(1+\frac 1 {\sqrt{\beta}}),
\end{align*}
where for any quantity $X$ we denote $X=O_s(1)$ if $|X|\le C_s$ for some
constant $C_s>0$ depending only on $s$. Clearly
\begin{align*}
H_{\infty,s}(y) &=2 (2\pi \sqrt{\beta})^{-s}  \int_{0}^1 t^{-s}  \cos  ({2\pi  t y}) dt \notag \\
&= 2 (2\pi \sqrt{\beta})^{-s}
(2\pi y)^{-(1-s)} \int_0^{2\pi y} t^{-s} \cos t dt.
\end{align*}

Consider for $0<s<1$, and for $y\in[0, \infty)$,
\begin{align*}
h_s(y) =\int_0^y  t^{-s} \cos t d t.
\end{align*}
\begin{lem}
For any $0<\delta_1<\frac 1{10}$, we have
\begin{align*}
\inf_{\delta_1\le y <\infty} h_s(y) =\min\{h_s(\delta_1),\;
 h_s(\frac{3\pi}2)\}.
\end{align*}
\end{lem}
\begin{proof}
Note that
\begin{align*}
h_s(\infty) = s(1+s) \int_0^{\infty} t^{-s-2}  (1-\cos t )d t>0.
\end{align*}
Clearly $h_s$ only has interior minimums at $y=(n+\frac 12)\pi$ with
$n\ge 1$ being any odd integer. The result then follows from the fact that
for odd $n\ge 1$, 
\begin{align*}
\int_{(n+\frac 12) \pi}^{(n+2+\frac 12) \pi}
t^{-s} \cos t dt 
=\int_{0}^{2 \pi} ( (n+\frac 12)\pi+t)^{-s} \sin  t dt>0.
\end{align*}
\end{proof}

Now note that
\begin{align*}
h_s(\frac 32 \pi)=
\frac 1 {1-s} \int_0^{\frac 32 \pi} \xi^{1-s} \sin \xi d\xi.
\end{align*}

\begin{lem} \label{2015lem2.12}
Consider for $0\le s\le 1$ 
\begin{align*}
f_1(s)= \int_0^{\frac 32 \pi} \xi^{1-s} \sin \xi d\xi.
\end{align*}
Then $f_1(0)=-1$ and $f_1(1)=1$, and
\begin{align*}
\inf_{0< s <1}   f_1^{\prime}(s)   \ge \frac 12.
\end{align*}
For some $s_* \in (0.308443, 0.308444)$, we have $f_1(s_*)=0$, and $f_1(s)<0$ for $0\le s<s_*$,
and $f_1(s)>0$ for $s_*<s\le 1$. 
\end{lem}
\begin{rem*}
One can check that $f_1^{\prime\prime}(s) <0$ and
\begin{align}
\inf_{0<s<1} f_1^{\prime}(s) = f_1^{\prime}(1) \approx 0.77. \notag
\end{align}
\end{rem*}

\begin{rem*}
For $0\le s \le1$, the function $f_1(s)$ is given by a hypergeometric expression:
\begin{align*}
\frac 1 {3-s} \cdot \Bigl( \frac {3\pi} 2 \Bigr)^{3-s}
\operatorname{HypergeometricPFQ}[\frac 32-\frac s2, \{\frac 32, \frac {5-s}2 \},
-\frac {9\pi^2}{16} ].
\end{align*}
\end{rem*}
\begin{proof}
We have
\begin{align*}
&f_1^{\prime} (s)  = \int_0^{\frac 32\pi} \xi^{1-s} (-\log \xi)\sin \xi d\xi,
\qquad f_1^{\prime\prime} (s)  = \int_0^{\frac 32\pi} \xi^{1-s} (\log \xi)^2\sin \xi d\xi.
\end{align*}
Clearly we have uniform bound on $\|f_1^{\prime\prime} \|_{L_s^\infty([0,1])}$. A rigorous
numerical computation together with error estimates then shows that the desired lower
bound on $f_1^{\prime}$ indeed holds true.  Since $f_1$ is strictly monotonically increasing,
the existence of $s_*$ follows from the intermediate value theorem. The precise numerical 
value is found via rigorous numerical integration. In particular, we have
\begin{align*}
f_1(0.308443)\approx -1.92202\times 10^{-6}<0, 
\qquad f_1(0.308444) \approx 5.43492\times 10^{-7}>0.
\end{align*}
\end{proof}

The following theorem follows immediately from the preceding series of lemmas and computations.

\begin{thm}[Sharp characterization of $F_{N,s}$ when $0<s<1$]
For some $s_* \in (0.308443, 0.308444)$, the following hold.
\begin{enumerate}
\item \underline{Lack of positivity for $0<s<s_*$}.  There are positive constants 
$\gamma_2(s)>\gamma_1(s)>0$ and $\gamma(s)>0$,  such that for any $y\in[\gamma_1(s),\, \gamma_2(s)]$
and $N\gg_s 1+\beta^{-\frac 1 {2(1-s)} }$, we have
\begin{align*}
\frac 1 {N^{1-s}} F_{N,s}(\frac y N) = - \gamma(s) + (1+ \frac 1 {\sqrt{\beta}} )\cdot O_s(N^{-(1-s)})
<-\frac 12 \gamma(s),
\end{align*}
In particular this shows that for $0<s<s_*$ the kernel function $F_{N,s}$ must be negative on an interval of length $O_s(N^{-1})$ for
all large $N$.

\item \underline{Positivity for $s_*<s<1$}. For any $s_*<s<1$ there are constants
$c_1(s)>0$ and $c_2(s)>0$ such that if $$N\ge \max\{
c_1(s) \beta^{-\frac 12} e^{\frac 1 {\sqrt{\beta}} }, c_2(s)\cdot(
1+ \beta^{-\frac 1 {2(1-s)} } ) \},$$
 then $F_N(x)$ is positive and hence has unit
$L_x^1$ mass.
\end{enumerate}
\end{thm}
\begin{rem}
It is also possible to characterize the lack of positivity of $F_{N,s}$ via real-space representation.
For example, for $0<s\ll 1$ we can consider (take $\beta=1$) 
\begin{align*}
F_{N,s}(x)= \sum_{n \in \mathbb Z} \int_{|y|<\frac 12}
g_0(x-y+n) \frac {\sin M \pi y} {\sin \pi y} dy,
\end{align*}
where $M=2N+1$ and $g_0=\mathcal F^{-1} ( (1+4\pi^2 \xi^2)^{-\frac s2} )$. 
Note that $g_0(z) \sim |z|^{-(1-s)} (1+o(1))$ as $|z|\to 0$ and has exponential decay as
$|z|\to \infty$. 
By a discussion similar to that in the proof of Statement (3) of Theorem \ref{thm2.18_00}
in the appendix, we can consider only the main piece (below $f_0(y)=\frac y {\sin \pi y}$)
\begin{align*}
\sum_{|n|\le 3} \int_{|y|<\frac M2} |x+n-\frac y M|^{-(1-s)} f_0(\frac yM) \frac {\sin \pi y} y dy
&=
M^{1-s}\Bigl( \int_{|y|<\frac M2} |Mx- y|^{-(1-s)} f_0(\frac yM)\frac {\sin \pi y} y dy \notag \\
&\quad +\sum_{1\le |n|\le 3} \int_{|y|<\frac M2} |Mx+Mn- y|^{-(1-s)}
f_0(\frac yM) \frac {\sin \pi y} y dy
\Bigr).
\end{align*}
Setting $x= \frac {3 } {2M}$ and taking the limit $M\to \infty$ for the first term inside
the bracket, we  obtain
\begin{align*}
&\int_{\mathbb R} |y-\frac {3}2|^{-(1-s)} \frac {\sin \pi y} y dy \notag \\
=&\int_{|y-\frac {3}2 |<0.01}  |y-\frac {3}2|^{-(1-s)} \frac {\sin \pi y} y dy  +\int_{|y-\frac {3}2 |>0.01}  |y-\frac {3}2|^{-(1-s)} \frac {\sin \pi y} y dy.
\end{align*}
Clearly the first term diverges as $-\frac {\operatorname{C_1}} s$ ($C_1>0$ is
a constant) as $s\to 0$.  On the other hand, it is easy to check that
\begin{align*}
\Bigl| \sum_{1\le |n|\le 3} \int_{|y|<\frac M2} |Mx+Mn- y|^{-(1-s)}
f_0(\frac yM) \frac {\sin \pi y} y dy\Bigr| \lesssim M^{-(1-s)} \log M.
\end{align*}
One can then see that $F_{N,s}(\frac 3{4N+2}) $ must be negative
for $0<s\ll 1$ and $N$ sufficiently large.
\end{rem}

Now we turn to the critical case $s=s_*$ which requires a more refined analysis
since the main integral term vanishes in the limit.
\begin{lem}
\begin{align*}
\sum_{i=m}^n f(i) &=
\int_{m}^n f(x) dx
+ \frac{f(m)+f(n)} 2+ \int_m^n f^{\prime}(x) (\{x\} -\frac 12) dx \notag \\
&= 
\int_{m}^n f(x) dx
+ {f(m)} + \int_m^n f^{\prime}(x) \{x\} dx,
\end{align*}
where $\{x\}=x-[x]$ denotes the usual fractional part (i.e., $[x]$ 
denotes the largest integer less than or equal to $x$ so that $\{x\} \in [0,1)$).
For example
$\{-3.1\}=0.9$ and $\{3.1\}=0.1$.
\end{lem}
\begin{proof}
This is the usual Euler-Maclaurin formula. See
for example \cite{EMformula}.
\end{proof}

\begin{lem} \label{lem2.14beta}
Recall $\langle x \rangle =(1+x^2)^{\frac 12}$.
For $0<s<1$ let
\begin{align*}
A_{\beta} (s) =\int_{-\infty}^{\infty}
( \langle x \rangle^{-s} - |x|^{-s}) dx
- s 2\pi \sqrt{\beta}\int_{-\infty}^{\infty}
 \langle x\rangle^{-s-2}  x \{ \frac x {2\pi \sqrt{\beta}} \} dx.
 \end{align*}
 Then for  $s \in [0.29, 0.31]$ and $0<\beta <\beta_*$ ($\beta_*$ is an absolute constant), we have
 \begin{align*}
 0<\alpha_1<   - A_{\beta}(s) <\alpha_2<\infty,
 \end{align*}
 where $a_1$ and $\alpha_2$ are absolute constants.
 \end{lem}
 \begin{rem*}
 One can compute numerically that for $\beta=\frac 1 {4\pi^2}$ and $s=0.3$
 \begin{align*}
 -A_{\beta}(s) \in [1.13-0.53, \;1.13+0.53].
 \end{align*}
 A further interesting question is to investigate the case $\beta \gtrsim 1$. 
 \end{rem*} 
 \begin{proof}
 Observe that for $s\in [0.29,0.31]$, 
 \begin{align*}
 s |\int_{-\infty}^{\infty}
 \langle x\rangle^{-s-2}  x \{ \frac x {2\pi \sqrt{\beta}} \} dx| 
 \le c_1 \cdot \beta^{\frac 12},
 \end{align*}
 where $c_1>0$ is an absolute constant.
 
 On the other hand for $s\in [0.29, 0.31]$, we have
 \begin{align*}
 \int_{-\infty}^{\infty} (|x|^{-s} - \langle x \rangle^{-s} ) dx
 = -\sqrt{\pi} 
 \frac{\Gamma (\frac {-1+s}2) } {\Gamma(\frac s 2) } 
 \ge c_2>0,
 \end{align*}
 where $\Gamma(\cdot)$ is the usual Gamma function and
  $c_2>0$ is another absolute constant. 
 The desired result then follows easily.
 \end{proof}

\begin{thm}[Lack of positivity for $s=s_*$ and $0<\beta<\beta_*$]
Let $s=s_*$ and $y=\frac 34$. Let $0<\beta<\beta_*$ where $\beta_*$
is the same constant as in Lemma \ref{lem2.14beta}.
 Then for $N\ge \frac 1 {\sqrt{\beta}}$, we have
\begin{align*}
F_{N,s}(\frac y N) = \frac 1 {2\pi \sqrt{\beta} }
A_{\beta}(s) +O_{s,\beta} ( N^{-s}  ),
\end{align*}
where $A_{\beta}(s) <0$ is given by the expression
\begin{align*}
A_{\beta}(s)
=\int_{\mathbb R} (\langle x \rangle^{-s} - |x|^{-s} )dx
-s 2\pi \sqrt{\beta}\int_{\mathbb R}
\langle x \rangle^{-s-2} x 
\{\frac x {2\pi \sqrt{\beta}} \} dx.
\end{align*}
\end{thm}
\begin{rem}
A further interesting issue is to investigate the case  $|y-\frac 34|\ll 1$. However
we shall not dwell on it here.
\end{rem}
\begin{proof}
Denote $\beta_1= 4\pi^2 \beta$ and $y_1=2\pi y $. 
Clearly for $f(k) =(1+\beta_1 k^2)^{-\frac s2} e^{i \frac k N y_1}$, we have
\begin{align*}
F_{N,s}(\frac y N) &
= \sum_{k=-N}^{N} f(k)
= \int_{-N}^N f(x) dx 
+f(-N) + \int_{-N}^N ( (1+\beta_1 x^2)^{-\frac s2} )^{\prime}
e^{i \frac x N y_1} \{ x \} dx  \notag \\
&\qquad \qquad \qquad + 
\int_{-N}^N
(1+\beta_1 x^2)^{-\frac s2} \frac {iy_1}{N}
e^{i \frac x N y_1} \{x \} dx.
\end{align*}
Note that for $s=s_*$, $y=\frac 3 4$, we have
\begin{align*}
\int_{-N}^N ( \beta_1 k^2)^{-\frac s2} e^{2\pi i \frac k N y} dk =0.
\end{align*}
Also it is easy to check that (here and below we need $N\sqrt{\beta}\ge 1$):
\begin{align*}
|f(-N)| + \int_{-N}^N
(1+\beta_1 x^2)^{-\frac s2} \frac {1}{N}dx =O_s( \beta^{-\frac s2} N^{-s}).
\end{align*}
Thus we obtain
\begin{align*}
F_{N,s}(\frac yN)
&= \int_{-N}^N
( (1+\beta_1 x^2)^{-\frac s 2}
- (\beta_1 x^2)^{-\frac s2} )
e^{i \frac x N y_1} dx 
-s \int_{-N}^N
(1+\beta_1 x^2)^{-\frac s2-1}
\cdot \beta_1 x e^{i\frac x N y_1}
\{ x \} dx + O_s( \beta^{-\frac s2} N^{-s} ) \notag \\
&=  \frac 1 {\sqrt{\beta_1}} \Bigl( \int_{-N_1 }^{N_1}
( (1+ x^2)^{-\frac s 2}
- |x|^{-s} )
e^{i \frac x {N_1} y_1} dx \Bigr)
-s \int_{-N_1}^{N_1}
(1+ x^2)^{-\frac s2-1}
\cdot  x e^{i\frac x {N_1} y_1}
\{ \frac {x} {\sqrt{\beta_1}} \} dx \notag 
\notag \\
&\qquad + O_s( \beta^{-\frac s2} N^{-s} ),
\end{align*}
where we have denoted $N_1=N\sqrt{\beta_1}$. 
The result then follows from the error estimates:
\begin{align*}
&\int_{|x|\le N_1}  |(1+ x^2)^{-\frac s 2}
- |x|^{-s} | \cdot
|1-e^{i\frac x {N_1} y_1}| dx  = O(N_1^{-1}), \notag \\
& \int_{|x|>N_1}  |(1+ x^2)^{-\frac s 2}
- |x|^{-s} | dx = O(N_1^{-1-s}), \notag \\
&\int_{|x|>N_1} (1+x^2)^{-\frac s2-1} |x| dx = O(N_1^{-s}), \notag \\
&\int_{|x|\le N_1} (1+x^2)^{-\frac s2-1}
\cdot |x|\cdot |e^{i\frac x {N_1} y_1} -1| dx = O(N_1^{-s}).
\end{align*}
Here we used the fact  that for $|x|\gg 1$, 
$(1+x^2)^{-\frac s2} = |x|^{-s} + O(|x|^{-s-2})$. 
\end{proof}

\begin{rem}
Let $\beta>0$, $N\ge 2$. Let $d\ge 2$, $0<s\le d$. One can consider the
general Bessel cases: 
\begin{align*}
&F^{\operatorname{square} }_{N,s} (x) = \sum_{|k|_{\infty} > N}
(1+4\pi^2 \beta |k|^2)^{-\frac s2} e^{2\pi i k\cdot x}, \\
&F^{\operatorname{sphere} }_{>N,s} (x) = \sum_{|k| > N}
(1+4\pi^2 \beta |k|^2)^{-\frac s2} e^{2\pi i k\cdot x}.
\end{align*}
These will be investigated elsewhere. 
\end{rem}

We summarize the results obtained in this section as the following theorem.

\begin{thm}[Sharp maximum principle]
Let $d\ge 1$ and $\beta>0$. Then the following hold.
\begin{enumerate}
\item If $s>d$ and $N\ge c_{s,d} (\beta^{\frac {d-1}2} e^{\frac 2 {\sqrt{\beta}} }
)^{\frac 1 {s-d} }$ ($c_{s,d}>0$ depends only on ($s$, $d$)), then $F_{N,s}$ is
a positive function and $\|F_{N,s} \|_{L^1(\mathbb T^d)} =1$.

\item Let $d=1$ and $s=1$. If $N\ge \alpha \beta^{-\frac 12} e^{\frac 1 {\sqrt{\beta} } }$
($\alpha>0$ is an absolute constant), then $F_{N,1}$ is  positive and has
unit $L^1$ mass.
\end{enumerate}
 Let $d=1$, $\beta>0$ and $0<s<1$. We identify $\mathbb T=[-\frac 12, \frac 12)$.  For some $s_* \in (0.308443, 0.308444)$, the following hold.
\begin{enumerate}
\item \underline{Lack of positivity for $0<s<s_*$}.  There are positive constants 
$\gamma_2(s)>\gamma_1(s)>0$ and $\gamma(s)>0$,  such that for any $y\in[\gamma_1(s),\, \gamma_2(s)]$
and $N\gg_s 1+\beta^{-\frac 1 {2(1-s)} }$, we have
\begin{align*}
\frac 1 {N^{1-s}} F_{N,s}(\frac y N) = - \gamma(s) + (1+ \frac 1 {\sqrt{\beta}} )\cdot O_s(N^{-(1-s)})
<-\frac 12 \gamma(s),
\end{align*}
In particular this shows that for $0<s<s_*$ the kernel function $F_{N,s}$ must be negative on an interval of length $O_s(N^{-1})$ for
all large $N$.

\item \underline{Positivity for $s_*<s<1$}. For any $s_*<s<1$ there are constants
$c_1(s)>0$ and $c_2(s)>0$ such that if $$N\ge \max\{
c_1(s) \beta^{-\frac 12} e^{\frac 1 {\sqrt{\beta}} }, c_2(s)\cdot(
1+ \beta^{-\frac 1 {2(1-s)} } ) \},$$
 then $F_N(x)$ is positive and hence has unit
$L_x^1$ mass.

\item \underline{Lack of positivity for $s=s_*$ and $0<\beta<\beta_*$}. 
Let $s=s_*$ and $y=\frac 34$. Let $0<\beta<\beta_*$ where $\beta_*$
is the same constant as in Lemma \ref{lem2.14beta}.
 Then for $N\ge \frac 1 {\sqrt{\beta}}$, we have
\begin{align*}
F_{N,s}(\frac y N) = \frac 1 {2\pi \sqrt{\beta} }
A_{\beta}(s) +O_{s,\beta} ( N^{-s}  ),
\end{align*}
where $A_{\beta}(s)$ is given by the expression
\begin{align*}
A_{\beta}(s)
=\int_{\mathbb R} (\langle x \rangle^{-s} - |x|^{-s} )dx
-s 2\pi \sqrt{\beta}\int_{\mathbb R}
\langle x \rangle^{-s-2} x 
\{\frac x {2\pi \sqrt{\beta}} \} dx.
\end{align*}
Furthermore for $0<\beta<\beta_*$, we have
\begin{align*}
0<\alpha_1 < -A_{\beta}(s) <\alpha_2<\infty,
\end{align*}
where $\alpha_1>0$, $\alpha_2>0$ are absolute constants.
\end{enumerate}
\end{thm}
\begin{thm}[Effective maximum principle] \label{2015thm9}
Let $d\ge 1$, $N\ge 2$, $\beta>0$ and $s>0$. Then for all $N\ge 2$, we have
\begin{align*}
\| F_{N,s} \|_{L_x^1(\mathbb T^d)}
\le 1 + \frac {c_{s,d} } {(1+\beta N^2)^{\frac s2} } ( \log (N+2) )^d,
\end{align*}
where $c_{s,d}>0$ depends only on ($d$, $s$).
\end{thm}
\begin{proof}[Proof of Theorem \ref{2015thm9}]
The proof is a simple adaptation of the proof of Theorem \ref{thm2.17_00}. We omit
the details.

\end{proof}

\section{1D torus case for Allen-Cahn: semi-discretization}
 Consider the scheme for Allen-Cahn on the torus $\mathbb T=[0,1)$:
\begin{align*}
\begin{cases}
\frac{u^{n+1} -u^n}{ \tau} = \nu^2 \partial_{xx} u^{n+1} -\Pi_N ( (u^n)^3- u^n),
\quad n\ge 0, \\
u^0 =\Pi_N u_0.
\end{cases}
\end{align*}
Here $\tau>0$ is the size of the time step.

The following lemma will be used later which we will often refer to as a standard discrete
energy estimate.
\begin{lem}[Energy estimate]
For any $n\ge 0$, we have
\begin{align*}
E(u^{n+1}) - E(u^n) &+ \frac 12 \nu^2 \| \partial_x (u^{n+1}-u^n) \|_2^2 
+ (\frac 1 {\tau}+\frac 12) \|u^{n+1}-u^n\|_2^2 \notag \\
& \le \frac 32 \max\{ \|u^n\|_{\infty}^2, \; \|u^{n+1}\|_{\infty}^2
\}\cdot \| u^{n+1}-u^n\|_2^2,
\end{align*}
where 
$F(x)=\frac 14 (x^2-1)^2$, and
\begin{align*}
E(u) = \frac 12 \nu^2 \| \partial_x u \|_{L^2(\mathbb T)}^2 + \int_{\mathbb T}
F(u) dx.
\end{align*}
\end{lem}

\begin{proof}
Denote by $(,)$ the usual $L^2$ pairing of functions.  Then 
\begin{align*}
\frac {\| u^{n+1} - u^n \|_2^2} {\tau}
&=-\nu^2 \frac 12 ( \| \partial_x u^{n+1} \|_2^2
-\| \partial_x u^n \|_2^2 + \| \partial_x (u^{n+1}-u^n)\|_2^2)
- (\Pi_N ( f(u^n)), u^{n+1}-u^n) \notag \\
&=-\nu^2 \frac 12 ( \| \partial_x u^{n+1} \|_2^2
-\| \partial_x u^n \|_2^2 + \| \partial_x (u^{n+1}-u^n)\|_2^2)
- ( f(u^n), u^{n+1}-u^n),
\end{align*}
where we have denoted $f(z)= F^{\prime}(z) = z^3-z$. 

Now note that
\begin{align*}
F(u^{n+1}) =F(u^n) +f(u^n) \cdot (u^{n+1}-u^n) +
\frac 12 f^{\prime}(\xi) (u^{n+1}-u^n)^2,
\end{align*}
where $\xi$ lies between $u^n$ and $u^{n+1}$. The desired inequality then
easily follows since $f^{\prime}(z)=3z^2-1$.

\end{proof}

\begin{prop} \label{prop_c3_1}
Assume $u_0 \in H^1(\mathbb T)$ and $\| u_0\|_{L^{\infty}(\mathbb T)} \le 1$. 
Let $0<\tau_{*}\le \tau \le \frac 12$. Let $0<\alpha_0\le 1$. If 
$N\ge N_0=N_0(\tau_*,\nu, \|u_0\|_{H^1},\alpha_0)$, then
\begin{align*}
\sup_{n\ge 0} \| u^n \|_{\infty} \le 1+\alpha_0.
\end{align*}
\end{prop}
\begin{rem*}
The condition $u_0\in H^1(\mathbb T)$ may be weakened to $u_0\in H^s(\mathbb T)$
$s>\frac 12$.
\end{rem*}
\begin{rem*}
The dependence of $N_0$ on $\nu$ (neglecting $\|u_0\|_{H^1}$ and 
$\alpha_0$) is rather mild especially in the regime $\nu\ll 1$.
From the proof below, one can see that it suffices to take $N_0$ with
$N_0^{-2} \log N_0 \ll \nu^2$ which is roughly $N_0 \gg \nu^{-1}\sqrt{ |\log \nu| }$. 
One should note that from a scaling heuristic the typical mesh size should resolve 
$\Delta x \approx \nu$ which corresponds to the Fourier cut-off $1/\nu$ in the frequency
space. Thus up to a logarithm $N_0$ is  optimal.
\end{rem*}
\begin{rem*}
The constraint $\alpha_0>0$  already comes from the fact that $u^0=\Pi_N u_0$ in general
does not preserve the strict upper bound as the Dirichlet kernel changes its sign. On the
other hand, even if we assume $\|u^0\|_{\infty} \le 1$ (note $u^0=\Pi_N u_0$!),
this maximum is still not preserved  (for small $N$) 
as one can see from the counter-example
in Proposition \ref{prop_c3_0}.
\end{rem*}

\begin{proof}
Step 1. Initial data. Write $\Pi_{>N}= I-\Pi_N$. Then 
\begin{align*}
\| \Pi_N u_0\|_{\infty} 
& \le \|u_0\|_{\infty} + \| \Pi_{>N} u_0\|_{\infty} \notag \\
& \le 1 + N^{-\frac 12 }\|u_0\|_{H^1(\mathbb T)} \le 1+\alpha_0,
\end{align*}
if $N\ge N_0(\alpha_0, \|u_0\|_{H^1})$.  Here one should note that we must have
$\alpha_0>0$. 

Step 2. Induction.  Rewrite 
\begin{align*}
(1-\tau \nu^2 \partial_{xx}) u^{n+1}
= \Pi_N \bigl( (1+\tau) u^n - \tau (u^n)^3 \bigr).
\end{align*}
The inductive assumption is 
\begin{align*}
\|u^n \|_{\infty} \le 1+\alpha_0.
\end{align*}
In the following we shall show $u^{n+1} \le 1+\alpha_0$. By repeating the argument for
$-u^{n+1}$ we also get the lower bound. 

Denote $u^n=1+\eta^n$. Then clearly by the inductive assumption
\begin{align*}
-(2+\alpha_0) \le \eta^n \le \alpha_0.
\end{align*}

For $\eta^{n+1}=u^{n+1}-1$ we have the equation
\begin{align*}
(1-\tau \nu^2 \partial_{xx}) \eta^{n+1}
&= \Pi_N \bigl(  (1+\tau) (1+\eta^n) - \tau (1+\eta^n)^3 -1 \bigr) \notag \\
& = \Pi_N \bigl( (1-2\tau) \eta^n - 3 \tau (\eta^n)^2 -\tau (\eta^n)^3 \bigr) \notag\\
&= (1-2\tau) \eta^n -\tau (3+\eta^n) (\eta^n)^2 + \tau 
\Pi_{>N}  \bigl(  3 (\eta^n)^2 + (\eta^n)^3 \bigr).
\end{align*}

Here we used the fact that $\Pi_N \eta^n = \eta^n$. Note that 
$3+\eta^n \ge 1-\alpha_0 \ge 0$ (here we need $\alpha_0\le 1$). By using the maximum
principle, we then get
\begin{align*}
\max \eta^{n+1}
&\le (1-2\tau) \alpha_0 + \tau \| (1-\tau \nu^2 \partial_{xx})^{-1} \Pi_{>N}
\bigl( 3 (\eta^n)^2 + (\eta^n)^3 \bigr) \|_{\infty} \notag \\
& \le (1-2\tau) \alpha_0  + \tau \cdot \frac{\operatorname{const}}{1+4\pi^2 \tau \nu^2 N^2}
\log(N+2)\cdot \operatorname{const} \notag \\
& \le (1-2\tau) \alpha_0 + \operatorname{const} \cdot \frac 1 {\nu^2 N^2} \log (N+2)
\le \alpha_0,
\end{align*}
if $N\ge N_0$ (here note that in the last inequality we need the lower bound on $\tau$ so that $N_0$ remains independent
of $\tau$). Note that in the second inequality above we have used the
$L^1$ bound for the operator $(1-\tau \nu^2 \partial_{xx})^{-1} \Pi_{>N}$ which
was established in the previous section.

\end{proof}

\begin{prop}[Lack of strict maximum principle] \label{prop_c3_0}
There exist $\nu>0$, $\tau>0$, $N\ge 1$, $u^0$ with $\|u^0\|_{\infty}\le 1$ such that
\begin{align*}
\|u^1\|_{\infty} >1.
\end{align*}

\end{prop}
\begin{rem*}
More elaborate (covering more regimes of the parameters) construction of counter-examples is possible but we shall not dwell on this issue
here.
\end{rem*}
\begin{proof}
Take $\tau=\frac 12$, $N=2$, and 
\begin{align*}
u^0 = 1+\eta^0 = 1-2 \delta \cdot \cos^2(2\pi x),
\end{align*}
where $\delta>0$ will be taken sufficiently small.  Write $u^1= 1+\eta^1$.  Then
\begin{align*}
(1-\tau \nu^2 \partial_{xx}) \eta^1 = -\tau 
\Pi_N \bigl(  3 (\eta^0)^2 + (\eta^0)^3 \bigr).
\end{align*}
Thus 
\begin{align*}
\eta^1 &= - \tau \Pi_N ( 3 (\eta^0)^2) -\tau^2 \frac{\nu^2 \partial_{xx}} 
{1-\tau \nu^2 \partial_{xx} } \Pi_N ( 3 (\eta^0)^2) 
+ \frac {-\tau}{ 1-\tau \nu^2 \partial_{xx} }  \Pi_N ( (\eta^0)^3)  \notag \\
&= -3\tau \Pi_N ( (\eta^0)^2) + O(\nu^2) + O(\delta^3).
\end{align*}
Now note that $\eta^0= -\delta( 1+\cos 4\pi x)$, and
\begin{align*}
& (\eta^0)^2 = \delta^2 ( \frac 32 + 2\cos 4\pi x+ \frac 12 \cos 8\pi x), \\
& \Pi_N( (\eta^0)^2) = \delta^2( \frac 32 +2 \cos 4\pi x).
\end{align*}
Clearly then
\begin{align*}
\max  \bigl( -3\tau \Pi_N ( (\eta^0)^2)  \bigr) \ge \frac 3 4 \delta^2.
\end{align*}
By taking $\delta$ small and $\nu$ small, we then obtain that
\begin{align*}
\max \eta^1 >\frac 38 \delta^2.
\end{align*}

\end{proof}

\begin{prop} \label{prop_tmp_prop3.3a}
Assume $\|u^0\|_{\infty} \le 1+\alpha_0$ for some $0\le \alpha_0\le 1$. 
Let $0<\tau_{*}\le \tau \le \frac 12$.  If 
$N\ge N_0=N_0(\tau_*, \nu)$ (note that $N_0$ is independent of $\alpha_0$), then
for all $n\ge 1$,
\begin{align*}
\|u^n\|_{\infty} \le 1+ \theta^{n} \alpha_0+ \frac {\beta}{
\nu^2 N^2} \log (N+2) \cdot \frac {1-\theta^{n}}{1-\theta},
\end{align*}
where $\theta=1-2\tau$, and $\beta>0$ is an absolute constant.  
\end{prop}
\begin{rem*}
It follows that
\begin{align*}
\limsup_{n\to\infty} \|u^n\|_{\infty} \le 1+ \frac {\beta}{2\tau
\nu^2 N^2} \log (N+2).
\end{align*}
In yet other words the true maximum $1$ is asymptotically preserved up to the spectral
truncation error! One should also note that the decay is exponential in $n$.
\end{rem*}
\begin{proof}
We adopt the same notation as in Proposition \ref{prop_c3_1} and the proof is a minor 
variation. First by using the proof of Proposition \ref{prop_c3_1} with $\alpha_0=1$, one
has the weak bound 
\begin{align} \label{weak_bdtmp1_1}
\sup_{n\ge 0} \|u^n\|_{\infty} \le 2.
\end{align}
We should emphasize here that $N_0$ can be taken to be independent of $u^0$ since we
assumed $\|u^0\|_{\infty} \le 2$. 

Define
\begin{align*}
\alpha_n = \max \eta_n.
\end{align*}
Then by repeating the derivation in Proposition \ref{prop_c3_1}, we obtain
\begin{align*}
\alpha_{n+1}
& \le (1-2\tau) \alpha_n  + \tau \cdot \frac{\operatorname{const}}{1+4\pi^2 \tau \nu^2 N^2}
\log(N+2)\cdot \operatorname{const} \notag \\
& \le (1-2\tau) \alpha_n + \frac {\beta} {\nu^2 N^2} \log (N+2),
\end{align*}
where $\beta>0$ is an absolute constant. Note that in deriving the above bound we have
used the fact that $\| \eta^n \|_{\infty} \lesssim 1$. This is where
\eqref{weak_bdtmp1_1} is needed.

Iterating in $n$ then gives the  upper bound. The proof for the lower bound for
$\tilde \eta_n= u^n+1$ is similar and thus we obtained the desired upper bound for 
$\|u^n\|_{\infty}$.
\end{proof}

\begin{prop}[Energy stability and $L^{\infty}$-stability for $\tau_*\le \tau
\le \frac 12$]  \label{prop3.5a}
Assume $u_0 \in H^1(\mathbb T)$ and 
$$
\| u^0\|_{L^{\infty}(\mathbb T)} \le 2.$$
Let $N_0$ be the same as in Proposition \ref{prop_tmp_prop3.3a}, and assume $N\ge N_0$
satisfies
\begin{align*}
\frac {\beta}{
\nu^2 N^2} \log (N+2) \cdot \frac {1}{2\tau_*} \le \frac 1{10},
\end{align*}
where the constant $\beta>0$ is the same as in Proposition \ref{prop_tmp_prop3.3a}. Let
\begin{align*}
n_0 = - \frac {\log (100)} { \log (1-2\tau_*)}.
\end{align*}

Then for any $n\ge n_0$, we have 
\begin{align*}
&\| u^n \|_{L^\infty(\mathbb T)} \le 1.11, \\
& E(u^{n+1}) \le E(u^n), 
\end{align*}
where 
$F(x)=\frac 14 (x^2-1)^2$, and
\begin{align*}
E(u) = \frac 12 \nu^2 \| \partial_x u \|_{L^2(\mathbb T)}^2 + \int_{\mathbb T}
F(u) dx.
\end{align*}
\end{prop}

\begin{proof}
Easy to check that by using Proposition \ref{prop_tmp_prop3.3a}, we have for $n\ge n_0$,
\begin{align*}
&\| u^n \|_{L^\infty(\mathbb T)} \le 1.11.
\end{align*}
Now by a standard discrete energy estimate, we have 
\begin{align*}
E(u^{n+1}) - E(u^n) &+ \frac 12 \nu^2 \| \partial_x (u^{n+1}-u^n) \|_2^2 
+ (\frac 1 {\tau}+\frac 12) \|u^{n+1}-u^n\|_2^2 \notag \\
& \le \frac 32 \max\{ \|u^n\|_{\infty}^2, \; \|u^{n+1}\|_{\infty}^2
\}\cdot \| u^{n+1}-u^n\|_2^2.
\end{align*}
Note that $\frac 1 {\tau}+\frac 12 \ge \frac 52$, and for $n\ge n_0$,
\begin{align*}
\| u^n \|_{\infty}^2 \le 1.11^2 =1.2321< 5/3.
\end{align*}
It clearly follows that
\begin{align*}
E(u^{n+1}) \le E(u^n).
\end{align*}

\end{proof}

\begin{lem} \label{lem_polycubic3}
Let $\tau>0$ and consider the cubic polynomial $p(x)=(1+\tau) x -\tau x^3$.
Then the following hold:
\begin{itemize}

\item If $0<\tau\le \frac 12$, then $\max_{|x|\le 1} |p(x)| =1$. Actually for any
$1\le \alpha\le \sqrt{1+\frac 2 {\tau}}$, we have
\begin{align*}
\max_{|x|\le \alpha} |p(x)|  \le \alpha.
\end{align*}
Furthermore for any
$1<\alpha<\sqrt{1+\frac 2 {\tau}}$, we have the strict inequality
\begin{align*}
\max_{|x|\le \alpha} |p(x)| <\alpha.
\end{align*}

\item If $\frac 12 \le \tau \le 2$, then for any
$\frac {(1+\tau)^{\frac 32}} {\sqrt{3\tau}}
\cdot \frac 23 \le \alpha \le \sqrt{\frac {2+\tau}{\tau}} $, we have
\begin{align*}
\max_{|x|\le \alpha} |p(x)| \le \alpha.
\end{align*}
Furthermore if $\frac {(1+\tau)^{\frac 32}} {\sqrt{3\tau}}
\cdot \frac 23 <\alpha < \sqrt{\frac {2+\tau}{\tau}} $, then we have the strict inequality
\begin{align*}
\max_{|x|\le \alpha} |p(x)| < \alpha.
\end{align*}

\item For $\tau>2$, define $x_0 = \sqrt{ \frac {1+\tau} {3\tau}} \in (0,1)$
 and $x_{n+1} = p(x_n)$. Then $|x_n|\to \infty$
as $n\to \infty$. 
\end{itemize}

\end{lem}

\begin{rem}
For $\tau=\frac 12$, $\frac {(1+\tau)^{\frac 32}} {\sqrt{3\tau}}\cdot \frac 23=1$, and 
for $\tau=2$, $\frac {(1+\tau)^{\frac 32}} {\sqrt{3\tau}}\cdot \frac 23=\sqrt 2$.

\end{rem}

\begin{proof}
Straightforward. For the first two results, note that a critical point of $p$ is $x_c= \sqrt{ \frac {1+\tau}{ 3\tau} }$ such that
\begin{align*}
p(x_c) = \frac { (1+\tau)^{\frac 32}} {\sqrt{3\tau}} \cdot \frac 23.
\end{align*}
Solving $p(x)=-x$ gives $x_1=\sqrt{\frac {2+\tau}{\tau}}$. It is easy to verify that if
$\tau <2$, then $p(x_c) < -p(x_1)$ and when $\tau=2$ we have $p(x_c)=-p(x_1)$. 

For the last result, one notes that $x_1= p(x_0)=\frac {(1+\tau)^{\frac 32}} {\sqrt{3\tau}}
\cdot \frac 23 > \sqrt{\frac {2+\tau}{\tau}}$. It is then easy to verify that
$x_n$ oscillates to infinity. 
\end{proof}

Loosely speaking, with the help of Lemma \ref{lem_polycubic3}, one can reduce
the problem of estimating $\|u^{n+1}\|_{\infty}$ in terms of $\|u^n\|_{\infty}$
(when $\tau\ge \tau_*$ is not close to zero)
 to
the recurrent relation
\begin{align*}
\alpha_{n+1}= \max_{|x|\le \alpha_n} |p(x)| + \eta,
\end{align*}
where $\eta>0$ denotes the spectral error.  The following Lemma plays a key role in the
proof of maximum principle later. From a practical point of view, we do not state the results
in its most general form. For example, below $\alpha_0$ corresponds to the upper bound
of  $\|u^0\|_{\infty}$
in the original numerical scheme, and it is convenient to assume $\alpha_0$ to be around $1$
in practice. 

\begin{lem}[Prototype iterative system for the maximum principle] \label{lem_cubic_iteration}
Let $\tau>0$ and  $p(x)=(1+\tau) x -\tau x^3$. Consider the recurrent relation
 \begin{align*}
\alpha_{n+1}= \max_{|x|\le \alpha_n} |p(x)| + \eta,\quad n\ge 0, 
\end{align*}
where $\eta>0$. 

\begin{enumerate}
\item Case $0<\tau \le \frac 12$.  Let $\alpha_0=2$.  There exists an absolute constant $\eta_0>0$ sufficiently small, such
that for all $0 \le \eta\le \eta_0$, we have  $1\le \alpha_n \le 2$ for all $n$.  Furthermore
for all $n\ge 1$,
\begin{align} \label{eq_cubic_iteration_01}
1+\eta \le \alpha_n \le 1+ \theta^n +\frac {1-\theta^n}{1-\theta} \eta,
\end{align}
where $\theta =1-2\tau$.

\item Case $\frac 12 \le \tau \le 2-\epsilon_0$ where $0<\epsilon_0\le 1$. 
Let $\alpha_0 = \frac 12 (\frac {(1+\tau)^{\frac 32}} {\sqrt{3\tau}}\cdot \frac 23+
\sqrt{\frac {2+\tau}{\tau}})$. 
 Then there
exists a constant $\eta_0>0$ depending only on $\epsilon_0$, such that if 
$0\le \eta \le \eta_0$,  then for all $n\ge 1$, we have 
\begin{align*}
\frac {(1+\tau)^{\frac 32}} {\sqrt{3\tau}}\cdot \frac 23+\eta
\le \alpha_n \le \alpha_0.
\end{align*}
\end{enumerate}

\end{lem}

\begin{rem*}
For $\tau=2$, even if $\alpha_0=1$, we have $\alpha_n \to \infty$ as $n\to \infty$.
Also note that in Case (2) we do not discuss the decay estimate at all since in practice
we are only interested in the regime where the maximum is around $1$.

\end{rem*}

\begin{proof}
(1) Observe that for $\eta>0$ sufficiently small the values of $\alpha_n$ are trapped in
$[1,2]$ by a simple induction argument. By another induction we obtain
$\alpha_n \ge 1+\eta$ for all $n\ge 1$.

 For the decay estimate, we discuss three cases.

Case 1: $1\le \alpha_n \le \sqrt{\frac {1+\tau}{3\tau}}$. Then
$\max_{|x|\le \alpha_n}|p(x)|= p(\alpha_n)$, and 
\begin{align*}
\alpha_{n+1} &= 1+p(\alpha_n)-1+\eta  \notag \\
& = 1+ p^{\prime}(\xi_n) (\alpha_n-1) + \eta 
\quad \xi_n \in (1,\alpha_n), \quad p^{\prime}(\xi_n)
=1+\tau-3\tau (\xi_n)^2, \\
&\le 1+ (1-2\tau) (\alpha_n-1)+\eta.
\end{align*}

Case 2: $\sqrt{\frac {1+\tau}{3\tau}}
<\alpha_n \le \sqrt{\frac {1+\tau} {\tau}}$. Then $\max_{|x|\le \alpha_n}
|p(x)| = p(\sqrt{\frac {1+\tau}{3\tau}})$, and
\begin{align*}
\alpha_{n+1}& = 1+ p(\sqrt{\frac {1+\tau}{3\tau}}) -1 +\eta \notag \\
& \le 1+(1-2\tau) ( \sqrt{\frac {1+\tau}{3\tau}} -1) + \eta \notag \\
& \le 1+(1-2\tau) (\alpha_n-1)+\eta.
\end{align*}

Case 3: $\sqrt{\frac {1+\tau}{\tau}} \le \alpha_n \le 2$. Observe that
$\max_{|x|\le \alpha_n}|p(x)| = \max\{ p( \sqrt{\frac {1+\tau}{3\tau}}), 
|p(\alpha_n)| \}$.  Let $\tilde p(x) = |p(x)|$. Note that in the interval
$[\sqrt{\frac {1+\tau}{\tau}}, \sqrt{\frac {2+\tau} {\tau}}]$, the function
$\tilde p(x)$ monotonically increases from $0$ to $\sqrt{\frac {2+\tau}{\tau}}$. 
It is easy to check that $x_*=2\sqrt{\frac {1+\tau}{3\tau}}$ is the unique point
in this interval such that
\begin{align*}
\tilde p(x_*) =p(\sqrt{\frac {1+\tau}{3\tau}}).
\end{align*}
Note that $x_* \ge 2$ since $0<\tau\le \frac 12$.  Thus $|p(\alpha_n)|
\le p(\sqrt{\frac {1+\tau}{3\tau}})$ and we obtain
\begin{align*}
\alpha_{n+1}& = 1+ p(\sqrt{\frac {1+\tau}{3\tau}}) -1 +\eta \notag \\
& \le 1+(1-2\tau) (\alpha_n-1)+\eta.
\end{align*}
Concluding from all three cases, an iteration in $n$ then gives 
\eqref{eq_cubic_iteration_01}.

(2) First it is easy to check that  for all $n\ge 0$ the values of $\alpha_n$ are 
trapped in $[1, \alpha_0]$. Note that $\max_{|x| \le 1}  |p(x)|
=p(\sqrt{\frac {1+\tau}{3\tau}}) = \frac {(1+\tau)^{\frac 32}} {\sqrt{3\tau}}\cdot \frac 23$.
Clearly then the lower bound for $\alpha_n$, $n\ge 1$ holds.

\end{proof}

A direct corollary of Lemma \ref{lem_polycubic3} is the following.
\begin{cor}[No stability for $\tau>2$] \label{cor3.9a}
Let $\tau>2$. Define $u_0(x)\equiv \sqrt{ \frac {1+\tau} {3\tau}} \in (0,1)$.
Then for any $N\ge 1$, we have $\|u^n\|_{\infty} \to \infty$ as $n\to \infty$.
\end{cor}

Observe that for $\tau=2$ we have $\frac {(1+\tau)^{\frac 32}} {\sqrt{3\tau}}
\cdot \frac 23 = \sqrt{\frac {2+\tau}{\tau}}=\sqrt 2 $. Interestingly we shall show
that for $\tau=2$, the maximum
 $\alpha=\sqrt 2$ cannot be preserved even for moderately large $N$. 

\begin{prop}[No maximum principle for $\tau=2$ and $N$ not large] \label{prop_tau2_tmp001}
Let $\tau=2$. Let $0<\nu\ll 1$. For any
$\frac 1{2\sqrt 3  \pi \nu} \le N \lesssim \frac 1 {\sqrt{\nu}} e^{\frac 1{12\nu}}$,
we can find $u_0$ with $\|u_0\|_{\infty} \le \sqrt 2$, but
\begin{align*}
\|u^1\|_{\infty} > \sqrt 2.
\end{align*}
\end{prop}
\begin{rem}
A further interesting issue is to consider $u_0$ spectrally localized with
 $\|u^0\|_{\infty} \le \sqrt 2$ and $\|u_0\|_{\infty} \le \sqrt 2$. But we do not
 dwell on this here.
\end{rem}
\begin{rem}
Similar results hold for any $\tau_*\le \tau \le 2$, $\alpha =\sqrt{1+\frac 2 {\tau}}$. 
Recall that $p(x)=(1+\tau) x -\tau x^3$, $p(\alpha)=-\alpha$, and $p^{\prime}(\alpha)<0$.
For the argument below, the main point is that the kernel $K_N$ has a nontrivial
negative part. One can find $u_0=\alpha+\eta$ with $\|u_0\|_{\infty} \le \alpha$, but
\begin{align*}
u^1(0) &= p(\alpha) + \bigl( K_N * ( p^{\prime}(\alpha) \eta) \bigr)(0)+ O(\eta^2) \notag \\
&=-\alpha+ p^{\prime}(\alpha) \bigl(K_N *(\eta) \bigr)(0)+O(\eta^2) <-\alpha,
\end{align*}
i.e. $\|u^1\|_{\infty}>\alpha$.  Thus for moderately large $N$, the upper bound
$\alpha= \sqrt{1+\frac 2{\tau}}$ cannot be preserved. 
\end{rem}
\begin{proof}
Recall that
\begin{align*}
u^{1}
= (1-\tau \nu^2 \partial_{xx} )^{-1} \Pi_N \bigl( (1+\tau) u^0 - \tau (u^0)^3 \bigr).
\end{align*}
Let $u_0= \sqrt 2 + \delta \eta$, where the function $\eta$ will be chosen
momentarily and $\delta>0$ will be taken sufficiently small.  Since
$u^0=\Pi_N u_0$, we have
\begin{align*}
u^{1}&= -\sqrt 2 + (1-\tau \nu^2 \partial_{xx} )^{-1} \Pi_N 
\bigl( -9 \delta \cdot  \eta) \notag  + O(\delta^2)\\
&=-\sqrt 2 - 9 \delta (1-\tau \nu^2 \partial_{xx} )^{-1} \Pi_N  (\eta)
+ O(\delta^2).
\end{align*}

From Proposition \ref{prop_tmp211}, we observe that the kernel $K_N$ corresponding
to the operator $(1-\tau \nu^2 \partial_{xx} )^{-1} \Pi_N$ can be written as
$K_N=K_N^{+}-K_N^{-}$ where $K_N^{+}=\max\{K_N,0\}$, $K_N^-=
\max\{-K_N,0\}$. Both $K_N^+$ and $K_N^-$ have nontrivial $L^1$ mass. Now note that
$K_N$ is an even function, and by the definition of convolution, we have
\begin{align*}
-\Bigl( (1-\tau \nu^2 \partial_{xx} )^{-1} \Pi_N (\eta) \Bigr)(0)
& = - \int_{0}^{1} K_N(-y) \eta (y) dy \notag \\
& = - \int_0^{1} K_N (y) \eta (y) dy = \int_0^{1} K_N^-(y) \eta (y) dy- 
\int_0^{1} K_N^+ (y) \eta (y) dy.
\end{align*}
One can then choose $\eta(x) = - K_N^-(x)$ and let $\delta$ be sufficiently small. 
Obviously $ \|u_0\|_{\infty} \le \sqrt 2$ and 
\begin{align*}
-u^1(0) >\sqrt 2.
\end{align*}
\end{proof}

\begin{prop}[No loss in spectral cut-off when $N$ is large] \label{prop3.13a}
Let $0<\tau_{*}\le \tau \le \frac 12$.  Assume $\|u^0\|_{\infty} \le 1$. If 
$N\ge N_1=N_1(\tau_*,\nu)$, then
\begin{align*}
\sup_{n\ge 0} \|u^n\|_{\infty} \le 1.
\end{align*}
More generally, for any
$1\le \alpha\le \sqrt{1+\frac 2 {\tau}}$, if $\|u^0\|_{\infty} \le \alpha$ and
$N\ge N_2=N_2(\tau_*, \nu)$, then
\begin{align*}
\sup_{n\ge 0} \|u^n\|_{\infty} \le \alpha.
\end{align*}

\end{prop}
\begin{rem*}
Neglecting the dependence on $\tau_*$, we have 
$N_i(\tau_*, \nu)=O(e^{\frac {\operatorname{const}} {\nu} }) $ ($i=1,2$) which is of $O(1)$ for
large $\nu$. 
\end{rem*}

\begin{proof}
Recall that
\begin{align*}
u^{n+1}
= (1-\tau \nu^2 \partial_{xx} )^{-1} \Pi_N \bigl( (1+\tau) u^n - \tau (u^n)^3 \bigr).
\end{align*}
The result can then be easily proved using Lemma \ref{lem_polycubic3}, Theorem
\ref{thm2.14old} and induction on $n$.
Note that the condition $\tau\ge \tau_*$ is needed for controlling  the
operator $(1-\tau \nu^2 \partial_{xx} )^{-1} \Pi_N$ by taking $N$ large to be independent
of $\tau$. 
\end{proof}

Similarly we have the following proposition. The proof is omitted since it is similar.
\begin{prop} \label{prop3.14a}
Let $\frac 12 < \tau \le 2$. For any
$\frac {(1+\tau)^{\frac 32}} {\sqrt{3\tau}}
\cdot \frac 23 \le \alpha \le \sqrt{\frac {2+\tau}{\tau}} $, if $\|u^0\|_{\infty} \le \alpha$ and
$N\ge N_3=N_3(\nu)=O(e^{\frac {\operatorname{const}} {\nu} }) $, then
\begin{align*}
\sup_{n\ge 0} \|u^n\|_{\infty} \le \alpha.
\end{align*}
\end{prop}
\begin{rem*}
Note that the bound $N_3(\nu)=O(e^{\frac {\operatorname{const}} {\nu} }) $
 is not good when $\nu \to 0$. This will be improved below by a sligthly different method
for $\frac 12 <\tau<2$. Note that for $\tau=2$ such an improvement
is impossible by Proposition \ref{prop_tau2_tmp001}.
\end{rem*}

\begin{prop} \label{prop3.14_tmp00a}
Let $\frac 12 <\tau<2-\epsilon_0$ for some $0<\epsilon_0\le 1$.  Denote $M_0=
\frac 12 (\frac {(1+\tau)^{\frac 32}} {\sqrt{3\tau}}\cdot \frac 23+
\sqrt{\frac {2+\tau}{\tau}})$. 
if $\|u^0\|_{\infty} \le M_0$ and $N\ge N_4=N_4(\nu)=
C(\epsilon_0) \cdot \nu^{-1} \sqrt{|\log {\nu} |}$ when $0<\nu\ll 1$ ($C(\epsilon_0)$
is a constant depending only on $\epsilon_0$), then
\begin{align*}
\sup_{n\ge 0} \|u^n\|_{\infty} \le M_0.
\end{align*}
\end{prop}
\begin{rem*}
The bound $M_0= \frac 12 (\frac {(1+\tau)^{\frac 32}} {\sqrt{3\tau}}\cdot \frac 23+
\sqrt{\frac {2+\tau}{\tau}})$ can be replaced by any number $\widetilde M_0 \in
(\frac {(1+\tau)^{\frac 32}} {\sqrt{3\tau}}\cdot \frac 23, \sqrt{\frac {2+\tau}{\tau}})$.
In this general case $N_4$ will also depend on $\widetilde M_0$, or more precisely, its
distance to the end-points 
$\frac {(1+\tau)^{\frac 32}} {\sqrt{3\tau}}\cdot \frac 23$ and
$ \sqrt{\frac {2+\tau}{\tau}}$. 
\end{rem*}

\begin{proof}
Write
\begin{align*}
u^{n+1}
&= (1-\tau \nu^2 \partial_{xx} )^{-1} \Pi_N \bigl( (1+\tau) u^n - \tau (u^n)^3 \bigr) 
\notag \\
& = (1-\tau \nu^2 \partial_{xx} )^{-1}  \bigl( (1+\tau) u^n - \tau (u^n)^3 \bigr)
+(1-\tau \nu^2 \partial_{xx} )^{-1} \Pi_{>N} \bigl( (1+\tau) u^n - \tau (u^n)^3 \bigr).
\end{align*}
Recall $p(x) = (1+\tau) x - \tau x^3$. Then
\begin{align*}
\| u^{n+1} \|_{\infty}
\le  \| p(u^n )\|_{\infty} + \frac {\operatorname{const}} { 1+ \nu^2 N^2}
\log (N+2) \cdot \|p(u^n) \|_{\infty}.
\end{align*}
The result then follows from Lemma \ref{lem_cubic_iteration} and induction.
\end{proof}

\subsection{The case $0<\tau <\tau_*$}
We now consider the analogue of Proposition \ref{prop_c3_1} for the case $0<\tau<
\tau_*$. 
\begin{prop}[Energy stability and $L^{\infty}$-stability for small time step] \label{prop_c3_1_small}
Assume $u_0 \in H^1(\mathbb T)$ and $\| u_0\|_{L^{\infty}(\mathbb T)} \le 1$. For
some $\tau_*=\tau_*(\nu, \|u_0\|_{H^1})$ the following holds for any
$0<\tau<\tau_*$: Let $0<\alpha_0\le 1$. If 
$N\ge N_0=N_0(\nu, \|u_0\|_{H^1},\alpha_0)$, then
\begin{align*}
&\sup_{n\ge 0} \| u^n \|_{\infty} \le 1+\alpha_0, \\
&  E(u^{n+1}) \le E(u^n), \qquad\forall\, n\ge 0,
\end{align*}
where 
$F(x)=\frac 14 (x^2-1)^2$, and
\begin{align*}
E(u) = \frac 12 \nu^2 \| \partial_x u \|_{L^2(\mathbb T)}^2 + \int_{\mathbb T}
F(u) dx.
\end{align*}

\end{prop}

\begin{proof}
It is necessary to modify the proof of Proposition \ref{prop_c3_1} and we shall
outline the main changes.

The main inductive assumption is:
\begin{align*}
& \|u^n\|_{\infty} \le 1+\alpha_0;\\
& E(u^{n}) \le E(u^{n-1}).
\end{align*}
We tacitly assume that $u^{-1} =u^0$ and start the induction for $n\ge 0$. 
Note that $\|u^0\|_{H^1} \le \|u_0\|_{H^1}$ and
$E(u^0) \lesssim E(u_0)$ since $u^0=\Pi_N u_0$.  

We now justify the
inductive assumption for $u^{n+1}$. This is divided into the following steps.

Step 1: Energy stability. By a standard discrete energy estimate, we have 
\begin{align*}
E(u^{n+1}) - E(u^n) &+ \frac 12 \nu^2 \| \partial_x (u^{n+1}-u^n) \|_2^2 
+ (\frac 1 {\tau}+\frac 12) \|u^{n+1}-u^n\|_2^2 \notag \\
& \le \frac 32 \max\{ \|u^n\|_{\infty}^2, \; \|u^{n+1}\|_{\infty}^2
\}\cdot \| u^{n+1}-u^n\|_2^2.
\end{align*}
Recall $p(x) = (1+\tau) x - \tau x^3$ and 
\begin{align*}
u^{n+1} = (1-\tau \nu^2 \partial_{xx})^{-1} \Pi_N ( p(u^n ) ).
\end{align*}
Then
\begin{align*}
\| u^{n+1} \|_{\infty} & \lesssim \| u^{n+1} \|_{H^1}  \notag \\
& \le  \|p(u^n)\|_{H^1} \lesssim_{ \nu, \|u_0\|_{H^1} } 1.
\end{align*}
By inductive assumption we have $\|u^n\|_{\infty} \le 2$. Then by taking $\tau_*>0$ 
sufficiently small (depending only on $\nu$ and $\|u_0\|_{H^1}$), we clearly have
$E(u^{n+1}) \le E(u^n)$.

Step 2: $L^{\infty}$-estimate.  This step is similar to that in Proposition \ref{prop_c3_1}.
Denote $u^n=1+\eta^n$. Then from the inductive assumption
\begin{align*}
-(2+\alpha_0) \le \eta^n \le \alpha_0.
\end{align*}

For $\eta^{n+1}=u^{n+1}-1$, we then have
\begin{align*}
\max \eta^{n+1}
&\le (1-2\tau) \alpha_0 + \tau \| (1-\tau \nu^2 \partial_{xx})^{-1} \Pi_{>N}
\bigl( 3 (\eta^n)^2 + (\eta^n)^3 \bigr) \|_{\infty} \notag \\
& \le (1-2\tau) \alpha_0 + \tau \| \Pi_{>N}
\bigl( 3 (\eta^n)^2 + (\eta^n)^3 \bigr) \|_{\infty} \notag \\
\end{align*}
By the inequality $\|\Pi_{>N} g\|_{\infty} \lesssim N^{-\frac 12} \|g\|_{H^1}$, we have
\begin{align*}
\| \Pi_{>N}( (\eta^n)^2) \|_{\infty} \lesssim
N^{-\frac 12} \| (\eta^n)^2 \|_{H^1} \lesssim_{ \nu, \|u_0\|_{H^1}} N^{-\frac 12}
\le \frac 12 \alpha_0, \qquad \text{for $N\ge N_0$}.
\end{align*}
Similar estimates hold for the term corresponding to $(\eta^n)^3$.  We then clearly
obtain
\begin{align*}
\|u^{n+1} \|_{\infty} \le 1+\alpha_0.
\end{align*}
\end{proof}

Now we discuss the decay estimate for the case $0<\tau<
\tau_*$.  Note that in the following $N_0$ is independent of $\alpha_0$. 
\begin{cor} \label{cor_c3_1_small}
Assume $u_0 \in H^1(\mathbb T)$ and $\| u^0\|_{L^{\infty}(\mathbb T)} \le 1+\alpha_0$
for some $0<\alpha_0\le 1$. For
some $\tau_*=\tau_*(\nu, \|u_0\|_{H^1})>0$ the following holds for any
$0<\tau<\tau_*$: If 
$N\ge N_0=N_0(\nu, \|u_0\|_{H^1})$, then for any $n\ge 1$, 
\begin{align*}
\|u^n\|_{\infty} \le  1+ \theta^{n} \alpha_0+
\frac {1-\theta^{n}}{1-\theta} \cdot \tau \cdot N^{-\frac 12} \cdot C_{\nu, \|u_0\|_{H^1}}.
\end{align*}
where $\theta=1-2\tau$, and
the constant $C_{\nu, \|u_0\|_{H^1}}>0$ depends only on $\nu$ and $\|u_0\|_{H^1}$.
Consequently
\begin{align*}
\limsup_{n\to \infty}
\|u^n\|_{\infty} \le 1+ \frac 12 \cdot  N^{-\frac 12} \cdot C_{\nu, \|u_0\|_{H^1}}.
\end{align*}
\end{cor}
\begin{proof}
First by taking $\alpha_0=1$ in Proposition \ref{prop_c3_1_small} we can achieve 
for $N\ge N_0=N_0(\nu, \|u_0\|_{H^1})$ that
\begin{align*}
&\sup_{n\ge 0} \| u^n \|_{\infty} \le 2, \\
& \sup_{n\ge 0} E(u^n) \le E(u^0).
\end{align*}

Now denote $\alpha_n = \max \eta^n=\max (u^n-1)$. Then by repeating the proof in Proposition \ref{prop_c3_1_small},
we obtain
\begin{align*}
\alpha_{n+1} \le (1-2\tau) \alpha_n +\tau \cdot N^{-\frac 12} \cdot C_{\nu, \|u_0\|_{H^1}},
\end{align*}
where the constant $C_{\nu, \|u_0\|_{H^1}}>0$ depends only on $\nu$ and $\|u_0\|_{H^1}$.
Similar estimates also hold for $\tilde \alpha_n= \max (-1-u^n)$.  Thus
for $\theta=1-2\tau$. 
\begin{align*}
\|u^n\|_{\infty} \le  1+ \theta^{n} \alpha_0+
\frac {1-\theta^{n}}{1-\theta} \cdot \tau \cdot N^{-\frac 12} \cdot C_{\nu, \|u_0\|_{H^1}}.
\end{align*}
\end{proof}

\subsection{New energy stability results for $\frac 12 <\tau <\tau_1
\approx 0.86$ }

Consider the equation
\begin{align*}
\frac 12 + \frac 1 x =\frac 32
\cdot \Bigl(  \frac 23 \cdot \frac {(1+x)^{\frac 32}} { \sqrt{3x} } \Bigr)^2.
\end{align*}
Easy to check that 
\begin{align*}
x= \tau_1= \frac 12 (-2 + (9 - 3 \sqrt{6} )^{\frac 13} + (9 + 3\sqrt 6)^{\frac 13})
\approx 0.860018
\end{align*}
is the unique real-valued solution. 

It follows that if $\frac 12 <\tau \le \tau_1-\epsilon_0$ where $0<\epsilon_0\le 0.1$, then
\begin{align} \label{3apr10:32-19000}
\frac 12 +\frac 1 {\tau} 
\ge \frac 3 2 \cdot \Bigl( \frac {(1+\tau)^{\frac 32}} {\sqrt{3\tau}}\cdot \frac 23
+ \eta(\epsilon_0) \Bigr)^2,
\end{align}
where $\eta(\epsilon_0)>0$ depends only on $\epsilon_0$.  For example, it is not difficult
to check that for all $\frac 12 \le \tau \le 0.86$,
\begin{align*}
\frac 12 +\frac 1 {\tau} 
\ge \frac 3 2 \cdot \Bigl( \frac {(1+\tau)^{\frac 32}} {\sqrt{3\tau}}\cdot \frac 23
+ \eta_0 \Bigr)^2, \qquad \eta_0=10^{-5}.
\end{align*}

\begin{prop}[Energy stability for $\frac 12 <\tau<\tau_1$] \label{prop3.18a}
Let $\frac 12 <\tau \le \tau_1-\epsilon_0$ for some $0<\epsilon_0\le 0.1$. 
Suppose $\|u^0\|_{\infty} \le \frac {(1+\tau)^{\frac 32}} {\sqrt{3\tau}}\cdot \frac 23$.
Denote
\begin{align*}
M_3=\frac {(1+\tau)^{\frac 32}} {\sqrt{3\tau}}\cdot \frac 23+\eta(\epsilon_0)
\end{align*}
where $\eta(\epsilon_0)$ is the same as in \eqref{3apr10:32-19000}.
Then for $N\ge N_0(\nu, \epsilon_0)$, we have 
\begin{align*}
&\sup_{n\ge 0} \|u^n\|_{\infty} \le M_3;  \\
&E(u^{n+1}) \le E(u^n), \quad \forall\, n\ge 0.
\end{align*}
\end{prop}

\begin{proof}
The $L^{\infty}$-stability estimate is established  in the same way as in Proposition
\ref{prop3.14_tmp00a} (see also the remark after).  Now
recall the discrete energy estimate
\begin{align*}
E(u^{n+1}) - E(u^n) &+ \frac 12 \nu^2 \| \partial_x (u^{n+1}-u^n) \|_2^2 
+ (\frac 1 {\tau}+\frac 12) \|u^{n+1}-u^n\|_2^2 \notag \\
& \le \frac 32 \max\{ \|u^n\|_{\infty}^2, \; \|u^{n+1}\|_{\infty}^2
\}\cdot \| u^{n+1}-u^n\|_2^2.
\end{align*}
Thus it suffices to check 
\begin{align*}
\frac 12 +\frac 1 {\tau} \ge \frac 32 \cdot M_3^2 
=\frac 32 \cdot ( \frac {(1+\tau)^{\frac 32}} {\sqrt{3\tau}}\cdot \frac 23+\eta(\epsilon_0)
)^2
\end{align*}
which is exactly \eqref{3apr10:32-19000}. 

\end{proof}

\section{General theory for dimension $d\ge 1$}
 Consider the scheme for Allen-Cahn on the torus $\mathbb T=[0,1)^d$ in the physical dimensions $d\le 3$:
\begin{align} \label{17O1}
\begin{cases}
\frac{u^{n+1} -u^n}{ \tau} = \nu^2 \Delta u^{n+1} -\Pi_N ( (u^n)^3- u^n),
\quad n\ge 0, \\
u^0 =\Pi_N u_0,
\end{cases}
\end{align}
where $\nu>0$.

\begin{thm}[Energy stability for $0<\tau \le 0.86$] \label{17O2}
Consider \eqref{17O1} with $u_0 \in H^s(\mathbb T^d)\cap H^1(\mathbb T^d)$, $s>\frac d2$ and $d\le 3$. 
Let $0<\tau \le 0.86$ and assume  $$\|u^0\|_{\infty} \le \sqrt{\frac 53}. \qquad (
\text{note that }
\sqrt{\frac 53} \approx 1.29099).$$
 Then we have energy stability for any $N\ge N_0=N_0(s, d,\nu, u_0)$:
\begin{align*}
E(u^{n+1}) \le E(u^n), \quad \forall\, n\ge 0.
\end{align*}
Furthermore,
\begin{align*}
\sup_{n\ge 0} \|u^n\|_{H^s} \le U_1<\infty,
\end{align*}
where $U_1>0$ is a constant depending only on ($s$, $d$, $\nu$, $u_0$).

\end{thm}
\begin{rem}
For $0<\nu \ll 1$,
the dependence of $N_0$ on $\nu$ is only power like which is a very mild constraint in practice. 
In dimension $d=1$ we have $N_0= O( \nu^{-1} (\log (\frac 1\nu +1))^{\frac 12})$ 
for $\tau = O(1)$ as shown
in the previous section. 
For $0<\tau\ll 1$, by using concurrently the energy conservation, 
a rough estimate gives  $\|u^n\|_{\infty} =O(\nu^{-\frac 14})$ which yields
$\tau_*= \nu^{\frac 12}$. This in turn renders the estimate
 $N_0= O( \nu^{-1.25} (\log (\frac 1 \nu+1))^{\frac 12})$
for $d=1$.
In dimension $d\ge 2$ this dependence gets a bit worse since we have
to use discrete smoothing estimates.
\end{rem}
\begin{proof}[Proof of Theorem \ref{17O2}]
We proceed in several steps.

1. By  using the theory developed in \cite{LTang20a},  it is not difficult to check that for some
$0<\tau_*=\tau_*(s,\nu, u_0,d)\le \frac 12$, we have energy stability $E(u^{n+1})\le E(u^n)$ for all 
$0<\tau \le \tau_*$ and $n\ge 0$.
Furthermore the uniform in time $H^s$ estimate follows from discrete smoothing estimates.

2. We now consider the regime $\tau_*\le \tau \le \frac 12$.  Since we assumed
$\|u^0\|_{\infty} \le \sqrt{\frac 53}$ and $\tau\ge \tau_*$, we can adapt the proof in Proposition \ref{prop_c3_1}. 
Observe that by Theorem \ref{thm2.18_00},
 the kernel $(1-\tau \nu^2\Delta )^{-1} \Pi_{>N}$ has a bound 
$N^{-2} \nu^{-2}  \log^d (N+2)$. We then clearly have (for $N$ sufficiently large)
\begin{align*}
\sup_{n\ge 0} \|u^n\|_{\infty} \le \sqrt{\frac 53}.
\end{align*}
It is obvious that
\begin{align*}
\frac 1 {\tau} + \frac 12 \ge 2.5 \ge \frac 32 \sup_{n\ge 0} \|u^n\|_{\infty}^2.
\end{align*}
Thus we have $E(u^{n+1}) \le E(u^n)$ for all $n\ge 0$.

3. Finally we consider the regime $\frac 12 \le \tau \le 0.86$. In this case we essentially repeat
the proof of Proposition \ref{prop3.18a}. Note that
\begin{align*}
\frac 12 (-2 + (9 - 3 \sqrt{6} )^{\frac 13} + (9 + 3\sqrt 6)^{\frac 13})
\approx 0.860018>0.86.
\end{align*}
The desired result clearly follows.

4. The uniform $H^s$ estimate for $\tau_*\le \tau\le 0.86$ follows easily from smoothing estimates.
We omit further details.
\end{proof}

\begin{thm}[Almost sharp maximum principle for $0<\tau\le \frac 12$] \label{17O3}
Consider \eqref{17O1} with $u_0 \in H^s(\mathbb T^d)\cap H^1(\mathbb T^d)$, $s>\frac d2$ and $d\le 3$. Let $0<\tau \le \frac 12$.
Assume $\|u^0\|_{\infty}=\|\Pi_N u_0\|_{\infty} \le 1+\alpha_0$ for some $0\le \alpha_0 \le \sqrt{\frac 53}-1$ 
(note that $\sqrt{\frac 53} -1 \approx 0.29099$). Then for $N\ge N_1=N_1(s,d,\nu, u_0)$, we have
\begin{align*}
\|u^n\|_{\infty} \le 1+ (1-2\tau)^n \alpha_0+  N^{-(s-\frac d2)} C_{s,d,\nu, u_0}, \qquad\forall\, n\ge 1, 
\end{align*}
where $C_{s, d, \nu,  u_0}>0$ depends only on ($s$, $d$, $\nu$, $u_0$).
\end{thm}
\begin{proof}[Proof of Theorem \ref{17O3}]
First by using Theorem \ref{17O2}, we clearly have (in this part we determine $N_0$)
\begin{align*}
\sup_{n\ge 0} \|u^n\|_{H^s(\mathbb T^d)} \lesssim 1.
\end{align*}
We then adapt the $L^{\infty}$-part of the proof of Proposition \ref{prop_c3_1_small} and also
the proof of Corollary \ref{cor_c3_1_small}. The
key is to observe that
\begin{align*}
\| \Pi_{>N} g \|_{L^{\infty}(\mathbb T^d)} \lesssim N^{-(s-\frac d2)}
\| g\|_{H^s(\mathbb T^d) }.
\end{align*}
Thus 
\begin{align*}
\max \eta^{n+1} &\le (1-2\tau)\max \eta^n + \tau \| \Pi_{>N} (3 (\eta^n )^2 + (\eta^n)^3 ) \|_{\infty}
\notag \\
&\le (1-2\tau) \max \eta^n + \tau\cdot N^{-(s-\frac d2)} \cdot \operatorname{const}.
\end{align*}
The desired result clearly holds. Note that we allow the full range
$\alpha_0 \in [0,\sqrt{\frac 53} -1]$ (i.e. we do not insist $\alpha_0>0$) since we record explicitly
the spectral error term. 
\end{proof}

\begin{thm}[$L^{\infty}$-stability for $\frac 12 < \tau \le 2$] \label{17O4}
Consider \eqref{17O1} with $d\le 3$ and $\frac 12 <\tau \le 2$. Then the following hold:
\begin{enumerate}
\item 
Let $\frac 12 <\tau<2-\epsilon_0$ for some $0<\epsilon_0\le 1$.  Denote $M_0=
\frac 12 (\frac {(1+\tau)^{\frac 32}} {\sqrt{3\tau}}\cdot \frac 23+
\sqrt{\frac {2+\tau}{\tau}})$. 
If $\|u^0\|_{\infty} \le M_0$ and $N\ge N_2=N_2(\nu)=
C(\epsilon_0) \cdot \nu^{-1} (|\log {\nu} |)^{\frac d2}$ when $0<\nu\ll 1$ ($C(\epsilon_0)$
is a constant depending only on $\epsilon_0$), then
\begin{align*}
\sup_{n\ge 0} \|u^n\|_{\infty} \le M_0.
\end{align*}

\item Let $d=1,2$. 
 For any
$\frac {(1+\tau)^{\frac 32}} {\sqrt{3\tau}}
\cdot \frac 23 \le \alpha \le \sqrt{\frac {2+\tau}{\tau}} $, if $\|u^0\|_{\infty} \le \alpha$ and
$N\ge N_3=N_3(\nu)=O(e^{\frac {\operatorname{const}} {\nu} }) $, then
\begin{align*}
\sup_{n\ge 0} \|u^n\|_{\infty} \le \alpha.
\end{align*}

\end{enumerate}
\end{thm}
\begin{rem}
For $\tau=2$ and $N$ not large, there are counterexamples as shown in Proposition
\ref{prop_tau2_tmp001}.
\end{rem}

\begin{proof}[Proof of Theorem \ref{17O4}]
For the first statement, we repeat the proof of Proposition \ref{prop3.14_tmp00a}. Note that
for general dimension $d\le 3$, we use Theorem \ref{thm2.18_00} to control the operator
$(\operatorname{I}-\nu^2 \tau \Delta)^{-1} (\operatorname{Id}-\Pi_N)$. The second statement follows from a
similar argument as in Proposition \ref{prop3.14a}. Here we also use the sharp
maximum principle established in  Theorem \ref{thm2.18_00} for dimension $d=1,2$.
\end{proof}

We now consider the pure spectral truncation for Allen-Cahn on the torus $\mathbb T=[0,1)^d$ in dimensions $d\ge 1$:
\begin{align} \label{17O5}
\begin{cases}
\partial_t u = \nu^2 \Delta u -\Pi_N ( u^3- u),
\quad t>0, \\
u\Bigr|_{t=0} =v_0=\Pi_N u_0,
\end{cases}
\end{align}
where $\nu>0$.
\begin{thm}[Maximum principle for the continuous in time system with spectral
truncation] \label{17O6}
Consider \eqref{17O5} with $\nu>0$, $d\ge 1$. Assume $\|u_0\|_{\infty} \le 1$ and 
$u_0 \in H^s(\mathbb T^d)$, $s>\frac d2$.
Then for $N\ge N_0=N_0(s, d,\nu,  u_0)$, we have
\begin{align*}
\sup_{0\le t<\infty} \|u(t, \cdot)\|_{\infty} 
\le  1+N^{-\gamma_0} C_1,
\end{align*}
where $\gamma_0=\min\{s-\frac d2, 1\}$, and $C_1>0$ depends only on
($s$, $d$, $\nu$, $u_0$).
\end{thm}
\begin{rem}
The assumption $u_0 \in H^s$ is needed so that we can have a uniform choice of the spectral
cut-off $N_0$.  One can also use Theorem \ref{17O3} and  take $\tau\to 0$ to derive the 
result for $d\le 3$. 
\end{rem}
\begin{proof}[Proof of Theorem \ref{17O6}]

First observe that for some $T_0=T_0(s,d,\nu, \|u_0\|_{H^s})>0$, we can construct a unique
local solution which satisfies
\begin{align*}
&\max_{0\le t \le T_0} \|u (t) \|_{H^s} \le 2 \|u_0\|_{H^s}, \\
&\max_{0\le t \le T_0} \|u (t) \|_{\infty} \le  B_0 \|u_0\|_{H^s},
\end{align*}
where $B_0>0$ depends only on ($s$, $d$). By using Energy conservation and taking advantage
of the Fourier projection, it is clear that the local solution can be extended for all time 
(the bounds on various norms will depend on $N$ without further estimates). 
Note that $u \in C_t^0 L_x^{\infty}$, and in particular $\|u(t)\|_{\infty}$ is a continuous function
of $t$.  We first show that if $N\ge N_0=N_0(s,d,\nu, u_0)$ (the choice
of $N_0$ will become clear momentarily)
\begin{align} \label{17O7}
\sup_{0\le t <\infty} \| u(t) \|_{\infty} \le B_0 \| u_0 \|_{H^s} +2 =:\beta_0.
\end{align}
Indeed assume that $T_*>T_0$ is the first moment that $\|u(t) \|_{\infty}$ achieves
the upper bound. If $T_*$ does not exist, we are done. Otherwise $T_0<T_*<\infty$. 
By a simple bootstrapping estimate, we have
\begin{align*}
\max_{T_0\le t \le T_*} \| \nabla u(t) \|_{\infty} \le B_1,
\end{align*}
where $B_1>0$ depends only on ($s$, $d$, $\nu$, $u_0$). 
It follows that 
\begin{align*}
\max_{T_0\le t\le T_*} \| \Pi_{>N} (u^3-u) \|_{\infty} \le N^{-1} B_2,
\end{align*}
where $B_2>0$ depends only on ($s$, $d$, $\nu$, $u_0$).  Now without loss of generality
we can assume $u(T_*,x_*) =\beta_0$ for some $x_*\in \mathbb T^d$. Clearly
we have $(\partial_t u)(T_*,x_*) = 0$. On the other hand, we have
\begin{align*}
 & \nu^2 (\Delta u)(T_*,x_*) + \Pi_N ( u -u^3) \Bigr|_{(T_*, x_*) } \notag \\
 \le &\;   \beta_0 -\beta_0^3 + N^{-1} B_2  \le -6 +N^{-1} B_2 <0,
 \end{align*}
if $N\ge B_2$. 

Thus \eqref{17O7} holds. Furthermore by a bootstrapping estimate, we have
\begin{align*}
&\sup_{T_0\le t<\infty} \| \Pi_{>N} (u^3-u) \|_{\infty} \le N^{-1} B_3, \notag \\
&\sup_{0\le t \le T_0} \| \Pi_{>N} (u^3-u) \|_{\infty} \le N^{-(s-\frac d2)} B_4, 
\end{align*}
where $B_3>0$, $B_4>0$ depend only on ($s$, $d$, $\nu$, $u_0$). Also observe that
\begin{align*}
\| v_0\|_{\infty} \le 1+ O(N^{-(s-\frac d2)}).
\end{align*}
Our desired result then follows by running again the maximum principle argument
(similar to the proof of \eqref{17O7}) and choosing $N$ sufficiently large. 
\end{proof}

We now drop the assumption $u_0 \in H^s$ and derive another form of the maximum principle.

\begin{thm}[Maximum principle for the continuous in time system with spectral
truncation] \label{17O8}
Consider \eqref{17O5} with $\nu>0$, $d\ge 1$. Assume $\|u_0\|_{\infty} \le 1$.
Then for $N\ge N_1=N_1( d,\nu,  u_0)$, we have
\begin{align*}
\sup_{0\le t<\infty} \|u(t, \cdot)\|_{\infty} 
\le  \max\{\|v_0\|_{\infty}, 1\}+N^{-\frac 12} C_2,
\end{align*}
where $C_2>0$ depends only on
($d$, $\nu$, $u_0$).
\end{thm}
\begin{proof}[Proof of Theorem \ref{17O8}]
The proof is similar to that in Theorem \ref{17O6} and we shall only sketch the needed modifications.
Since $\Pi_N$ is the product of one-dimensional Fourier multipliers, it is not difficult to verify that
\begin{align*}
\sup_{N\ge 1} \| \Pi_N g \|_q \le C_{q,d} \| g\|_q,  \qquad\forall\, 1<q<\infty,
\end{align*}
where $C_{q,d}$ depends only on ($q$, $d$).  Fix $4d<q<\infty$ sufficiently large (the needed
largeness will become clear momentarily.) It is not difficult to construct a local solution in 
$C_t^0 L_x^q$. Moreover the solution can be extended globally in time.
By using the local wellposedness, we can find $t_0=t_0(\|u_0\|_q,d,\nu)>0$
sufficiently small such that  
\begin{align*}
& \sup_{0\le t \le t_0}  \|u (t) \|_q \le 2 \| v_0 \|_q \le C_{q,d} \|u_0\|_q \le C_{q,d}, \\
& \sup_{0\le t \le t_0} t^{-0.99} \| u(t) -e^{t \nu^2 \Delta} v_0\|_{\infty} \le B_1 , \\
& \sup_{0\le t \le t_0} t^{0.51} \| \nabla u(t) \|_q \le B_2, 
\end{align*}
where $B_1>0$, $B_2>0$ are constants depending on ($d$, $\nu$, $u_0$, $q$).
Clearly for $t_1<t_0$ ($t_1$ will be chosen later) and $q$ sufficiently large, we have
\begin{align} \label{17O9}
\sup_{t_1\le t\le t_0} \| \Pi_{>N} ( u(t)-u(t)^3) \|_{\infty} \le N^{-0.99} t_1^{-0.51} B_3,
\end{align}
where $B_3>0$ depends on ($d$, $\nu$, $u_0$). Note that after this step is finished we can fix the value
of $q$.   By using the smoothing estimates and similar maximum principle estimates as
used in Theorem \ref{17O6} we also have 
\begin{align} \label{17O10}
\sup_{t \ge  t_0} \| \Pi_{>N} ( u(t)-u(t)^3) \|_{\infty} \le N^{-1} B_4,
\end{align}
where $B_4>0$ depends on ($d$, $\nu$, $u_0$).

Now note that for $0\le t\le t_1$, we have 
\begin{align*}
\| u(t) \|_{\infty} \le \|v_0\|_{\infty} + B_1 t_1^{0.9}.
\end{align*}
Choosing $t_1= N^{-\frac {49}{51}}$ and running a simple maximum principle argument using
\eqref{17O9}, \eqref{17O10} then yields the result.
\end{proof}

\section{1D Allen-Cahn: fully discrete case using DFT } \label{S:DFT1}
In this section we consider the analysis of the 1D Allen-Cahn on the periodic torus
$\mathbb T=[0,1)$ using discrete Fourier transform. We discretize the domain 
$[0,1)$ using $x_j= \frac {j} N$, $j=0$, $1$, $\cdots$, $N-1$, where $N$ is taken to be an even
number.  In typical FFT simulations, $N$ is usually taken to be a dyadic number.
 We use $u^n=(u^n_0,
\cdots, u^n_{N-1})^T$ to
denote the approximation of $u(t,x_j)=u(t,\frac j N)$, $ t= n\tau$, $0\le j\le N-1$. We shall adopt the following
convention for discrete Fourier transform: 
\begin{align*}
& \tilde U_k = \frac 1 N \sum_{j=0}^{N-1} U_j e^{-2\pi i k \cdot \frac jN}; \\
& U_j = \sum_{k=-\frac N2+1}^{\frac N2} \tilde U_k e^{2\pi i k \cdot \frac jN},
\end{align*}
where $\tilde U=(\tilde U_0,\cdots, \tilde U_{N-1})^T$ is the (approximate) Fourier coefficient 
vector of
 the input data $U=(U_0,\cdots, U_{N-1})^T$.   Note that $\tilde U_{k\pm N} =\tilde U_k$ for any $k\in \mathbb Z$. The discrete Laplacian operator $\Delta_h$ 
 corresponds to the Fourier multiplier $-(2\pi k)^2$ (when $k$ is restricted to $-\frac N2<k \le\frac N2$).

 For $U=(U_0,\cdots, U_{N-1})^T \in \mathbb R^N$, $V=(V_0,\cdots, V_{N-1})^T\in \mathbb R^N$, we define the inner product
 \begin{align*}
 \langle U, V \rangle = \frac 1 N \sum_{j=0}^{N-1} U_j V_j.
 \end{align*}
 It is easy to check that 
 \begin{align*}
 \langle U, V \rangle = \sum_{-\frac N 2<k\le \frac N2} \tilde U_k \tilde V_{-k}
 \end{align*}
 which corresponds to the Plancherel formula.

 For any real-valued $g\in H^s(\mathbb T)$ with $s>\frac 12$, we have
 \begin{align*}
 \tilde g_k = \frac 1N \sum_{j=0}^{N-1} g(\frac jN) e^{-2\pi i k\cdot \frac jN}
 = \sum_{l\in \mathbb Z} \hat g(k+l N), \qquad -\frac N2 <k\le \frac N2.
 \end{align*}
 Note that the convergence of the series is not a problem thanks to Cauchy-Schwartz.
 We define 
 \begin{align*}
 (Q_N g)(x) &=\op{Re}\Bigl( \sum_{-\frac N2 <k\le \frac N2} \tilde g_k e^{2\pi i k \cdot x} \Bigr) 
 \notag \\
 &= \sum_{|k|<\frac N2} \tilde g_k e^{2\pi i k\cdot x}
 + \frac 12 \tilde g_{\frac N2} \cdot ( e^{2\pi i\cdot \frac N2x} + e^{-2\pi i \cdot \frac N2 x}).
 \end{align*}
 Here in the above we used the fact that  $\tilde g_k^*= \tilde g_{-k}$ 
 (here $*$ denotes complex conjugate) since $g$ is real-valued.
 In particular 
 \begin{align*}
 \tilde g_{-\frac N2}=\tilde g_{\frac N2} = \frac 1 N \sum_{j=0}^{N-1} g(\frac j N) (-1)^{j}
 \end{align*}
 is real-valued. Note that 
 \begin{align} \label{Nov1.001}
& \| Q_N g \|_2^2 = \sum_{|k|<\frac N2} |\tilde g_k|^2 + \frac 12 |\tilde g_{\frac N2} |^2, \notag\\
 &\frac 1 {4\pi^2} \| \partial_x Q_N g\|_2^2= 
 \sum_{|k|<\frac N2} |k|^2 |\tilde g_k|^2+ \frac 12
 \cdot \frac {N^2} 4 \cdot |\tilde g_{\frac N2} |^2.
 \end{align}
 
 On the Fourier side,
 \begin{align*}
 \widehat{Q_N g}(k) =\begin{cases}
 \sum_{l\in \mathbb Z} \hat g(k+l N), \quad -\frac N2 < k < \frac N2;\\
\frac 12 \sum_{l \in \mathbb Z} \hat g(k+lN), \quad k = \pm \frac N2; \\
 0, \qquad \text{otherwise}.
 \end{cases}
 \end{align*}
 
 In yet other words, for smooth $g:\, \mathbb T \to \mathbb R$ we have
 \begin{align*}
 & g(x) = \sum_{l \in \mathbb Z} \sum_{-\frac N2 <k\le \frac N2}
 \hat g(k+l N) e^{2\pi i (k+lN) x}, \notag \\
 & (Q_N g)(x) =\sum_{|k|<\frac N2}
 \Bigl( \sum_{l \in \mathbb Z} \hat g(k+l N)  \Bigr)e^{2\pi i k\cdot x}
 + \Bigl( \sum_{l \in \mathbb Z}
 \hat g(\frac N2+ lN) \Bigr) \cos (2\pi \cdot \frac N2 x).
 \end{align*}

 In particular if $\hat g$ is supported in $\{ -\frac N2\le k \le \frac N2 \}$
 with $\hat g(\frac N2)=\hat g(-\frac N2)$, then clearly $Q_N g =g$ and one
 can perfectly reconstruct $g$ from its node values $g(\frac jN)$, $0\le j\le N-1$. 
 One should note, However,  that the operator $Q_N$ in general \emph{does not} commute with the usual
 differentiation operators or frequency projection operators.

 We shall be using the following estimates
 without explicit mentioning:
 \begin{align*}
 & \|(Q_N - \op{Id} )g \|_{\infty} = \| Q_N g - g \|_{\infty} 
 \lesssim   \sum_{|k|\ge \frac N2} |\widehat{g}(k)| \lesssim N^{-\frac 12} \| \partial_x g \|_{L^2(\mathbb T)}.
 &
 \end{align*}

 \begin{rem}\label{Nov1re1}
 Note that the operator $Q_N$ is also defined for bounded $g:$ $\mathbb T\to \mathbb R$ by using only
 the expression
 \begin{align} \label{Nov1.01}
 (Q_N g)(x)  &=\op{Re}\Bigl( \sum_{-\frac N2 <k\le \frac N2}
 \Bigl( \frac 1 N\sum_{j=0}^{N-1} g(\frac j N) e^{-2\pi i k \frac jN} \Bigr)
 e^{2\pi i k \cdot x} \Bigr) \notag \\
 &= \sum_{j=0}^{N-1} g(\frac jN)  G_N(x-\frac jN),
 \end{align}
 where
 \begin{align*}
 G_N(y)= \frac {\sin N \pi y} {N\tan \pi y}.
 \end{align*}
 In general given $U=(U_0,\cdots,U_{N-1})^T \in \mathbb R^N$, we can define
 \begin{align} \label{Nov1.01b}
 (Q_NU)(x)
 &=\op{Re}\Bigl( \sum_{-\frac N2 <k\le \frac N2}
 \Bigl( \frac 1 N\sum_{j=0}^{N-1} U_j e^{-2\pi i k \frac jN} \Bigr)
 e^{2\pi i k \cdot x} \Bigr) \notag \\
 &= \sum_{j=0}^{N-1} U_j G_N(x -\frac jN).
 \end{align}
 Clearly if $\|U\|_{\infty} \le 1$, then
 \begin{align*}
 \| Q_N U \|_{L^{\infty}}\le \sup_{x\in \mathbb T}\sum_{j=0}^{N-1} |G_N(\frac jN-x)|
 \lesssim \log N.
 \end{align*}
 On the other hand,  by taking $x= \frac {\epsilon_0} N$ with $0<\epsilon_0<1$,  $U_j=1$
 for $\frac N{100} \le j\le \frac N{10}$ with $j$ being even, $U_j=0$ otherwise, one can 
 easily check that
 \begin{align*}
 |(Q_N U)(x)| \gtrsim \log N.
 \end{align*}
 Thus the mere condition $\| U\|_{\infty} \le 1$ cannot guarantee a good bound on $Q_N U$. 
 In a similar vein, one can also consider the function
 \begin{align*}
 \phi_N(x)= \op{Re} \Bigl( \sum_{-\frac N2 <k \le \frac N2}
 \frac 1 {1+\nu^2 (2\pi k)^2 \tau} e^{2\pi i k \cdot x } \Bigr).
 \end{align*}
 For $\tau>0$ sufficiently small ($\tau$ can depend on $N$), we have 
 \begin{align*}
 \| \phi_N\|_{L^1(\mathbb T)} \gtrsim \log N.
 \end{align*}
Note that as $\tau\to 0$, $\phi_N(\frac j N) \to 0$ for all $1\le |j|\le N-1$, whereas $\phi_N(0) \to 1$.
 \end{rem}
 
\medskip

Now consider the fully discrete system:
\begin{align} \label{Nov1a}
\frac {U^{n+1} - U^n}{\tau} = \nu^2 \Delta_h U^{n+1} + U^n-(U^n)^{.3},
\end{align}
where $(U^n)^{.3}=((U_0^n)^3, \cdots, (U_{N-1}^n)^3)^T$.
We recast it as
\begin{align} \label{Nov1}
U^{n+1} = (\operatorname{I}-\nu^2\tau \Delta_h)^{-1} \Bigl( (1+\tau)U^n - \tau (U^n)^{.3} \Bigr).
\end{align}
One can view \eqref{Nov1a} as the time discretization of the continuous-in-time ODE system:
\begin{align} \label{Nov1.1}
\begin{cases}
\frac  d {dt} U = \nu^2 \Delta_h U + U - U^{.3}, \\
U\Bigr|_{t=0}=U^0 \in \mathbb R^N.
\end{cases}
\end{align}
For small $\tau>0$ one can study the nearness of the solutions to \eqref{Nov1} and 
\eqref{Nov1.1}. However we shall not explore it here. On the other hand, we can consider the following system 
\begin{align} \label{Nov1.2}
\begin{cases}
\partial_t u = \nu^2 \Delta u  + Q_N ( u -u^3), \\
u\Bigr|_{t=0} = Q_N (u_{\operatorname{init} }),
\end{cases}
\end{align}
where $u_{\operatorname{init}}$ is a bounded real-valued function on $\mathbb T$.  If we require
\begin{align*}
U^0_j = u_{\operatorname{init}}(\frac jN), \quad \forall\, 0\le j\le N-1,
\end{align*}
then it is not difficult to check that the system \eqref{Nov1.1} and \eqref{Nov1.2} are equivalent,
i.e.
\begin{align*}
U(t)_j = u(t, \frac j N), \qquad\forall\, 0\le j \le N-1.
\end{align*}
Thanks to the spectral localization, the function $u(t,x)$ is uniquely determined by its node values
$u(t, \frac jN)$ and thus the above equivalence is not a problem.

\begin{lem} \label{Nov2}
Let  $N\ge 4+\frac 1 {\pi^2 \nu\sqrt{\tau}} e^{\frac 1 {2\nu \sqrt{\tau} } }$. 
Suppose $f\in \mathbb R^N$, then
\begin{align*}
\| (\operatorname{I}- \nu^2 \tau\Delta_h)^{-1} f \|_{\infty} \le \| f\|_{\infty}.
\end{align*}
\end{lem}
\begin{proof}
It suffices to examine the kernel $K_h$ corresponding to $(I-\nu^2\Delta_h )^{-1}$. 
Denote $\beta = \nu \sqrt{\tau}$. Note that for $x= \frac { l} N$, $l \in \mathbb Z$
with $|l| \le N-1$, 
\begin{align*}
K_h(x) =K_h(\frac lN)= \sum_{ |k| \le \frac N2} \frac 1 {1+\beta^2 (2\pi k)^2} e^{2\pi i k 
\frac l N} - \frac 1 {1+\beta^2 (N\pi)^2} (-1)^l
\end{align*}
is real-valued.  On the other hand, for such $x$ we clearly have
\begin{align*}
K_h(x) &\ge \sum_{k \in \mathbb Z} \frac 1 {1+\beta^2 (2\pi k)^2}
e^{2\pi i kx} - \sum_{|k|\ge \frac N2}
\frac 1 {1+\beta^2 (2\pi k)^2} \notag \\
&> \frac 1 {\beta} e^{-\frac 1 {2\beta}} -\frac 1 {4\pi^2 \beta^2} \cdot 
2\cdot \frac 1 {\frac N2-1}.
\end{align*}
Thus $K_h(x)$ is positive if $N\ge 4+ \frac 1 {\pi^2 \beta} e^{\frac 1 {2\beta}} $ and
$x= \frac l N$ with $l\in \mathbb Z$, $|l|\le N-1$. The desired result then easily follows from
the discrete
convolution inequality.
\end{proof}

\begin{thm}[Strict maximum principle] \label{Nov2_thma}
Consider \eqref{Nov1} with initial data $U^0\in \mathbb R^N$. Assume $0<\tau \le \frac 12$. 
Suppose $\|U^0\|_{\infty} \le 1$. If $N\ge4+\frac 1 {\pi^2 \nu\sqrt{\tau}} e^{\frac 1 {2\nu \sqrt{\tau} } }$, then 
\begin{align*}
\|U^n\|_{\infty} \le 1, \qquad\forall\, n\ge 1.
\end{align*}
\end{thm}
\begin{proof}
This follows from an easy induction using Lemma \ref{Nov2} and Lemma 
\ref{lem_polycubic3}.
\end{proof}
\begin{rem}
Consider the kernel function
\begin{align*}
 \phi_N(x)= \op{Re} \Bigl( \sum_{-\frac N2 <k \le \frac N2}
 \frac 1 {1+\nu^2 (2\pi k)^2 \tau} e^{2\pi i k \cdot x } \Bigr).
 \end{align*}
 As was already pointed out in Remark \ref{Nov1re1}, in general $\phi_N$ can have $O(\log N)$ $L^1$ mass when $\tau \to 0$. On the other hand,  if we restrict to the grid points
 $x_j = \frac j N$ and consider the discrete $L^1$ mass
 \begin{align*}
 m_N= \frac 1 N \sum_{j=0}^{N-1} |\phi_N(\frac j N)|,
 \end{align*}
 one can have unit discrete $L^1$ mass for some special values of $N$ but this does not hold 
 in general. For example consider $N=2$, we have
 (below we denote $s= 4\pi^2 \nu^2 \tau$)
 \begin{align*}
 \phi_2(0)= 1+\frac 1 {1+s}, \quad  \phi_2(\frac 1 2)= 1-\frac 1 {1+s}>0, \qquad\forall\, s>0.
 \end{align*}
 Thus $m_2 =1$. On the other hand, for $N=4$, we have
 \begin{align*}
 \phi_4(0)=1+\frac 2{1+s}+\frac 1{1+4s},\quad \phi_4(\frac 14)=1-\frac 1{1+4s},
 \quad \phi_4(\frac 12)=1-\frac 2{1+s}+\frac 1{1+4s}, \quad
 \phi_4(\frac 34)=1-\frac 1{1+4s}.
 \end{align*}
 For $s>0$ sufficiently small, we have $\phi_4(\frac 12)<0$ and it follows that
 $m_4 >1$. In general, consider $N=4K$ with $K$ being an  integer, then
 \begin{align*}
 \phi_N(\frac 12) = \sum_{-\frac N2 <k \le \frac N2} \frac 1 {1+k^2 s} (-1)^k.
 \end{align*}
 Note that 
 \begin{align*}
 \frac d {ds} \phi_N(\frac 12) \Bigr|_{s=0} &=
 \sum_{-\frac N2+1 \le k\le \frac N2} k^2 (-1)^{k+1}
 =2\sum_{\substack{1\le k \le \frac N2-1 \\ \text{$k$ is odd} } }
 k^2-
 2\sum_{\substack{1\le k \le \frac N2-1 \\ \text{$k$ is even} } } k^2 -\bigl(\frac N2\bigr)^2 \notag \\
 & = 2 \sum_{l=1}^K \Bigl( (2l-1)^2 - (2l)^2 \Bigr) + 4K^2 =-2K = -\frac N2.
 \end{align*}
 Thus for small $s>0$ we have $\phi_N(\frac 12) <0$ and this renders $m_N>1$. More quantitatively
 we have for $s \le \alpha_1 N^{-2}$ ($\alpha_1>0$ is a sufficiently small absolute constant),
 \begin{align*}
 \phi_N(\frac 12) \le -\alpha_2 N\cdot s,
 \end{align*}
 where $\alpha_2>0$ is a small  absolute constant. 
 \end{rem}

Our next set of results will be concerned with generalized maximum principles.  We first
consider the ODE system \eqref{Nov1.1}.

\begin{thm}[Almost sharp maximum principle]
Consider \eqref{Nov1.1}. Suppose $u_{\op{init}}:\, \mathbb T\to \mathbb R$ satisfies $u_{\op{init}} \in H^{\frac 32} (\mathbb T)$ and
$\|u_{\op{init}}\|_{\infty} \le 1$. Take the initial data
$U^0$ such that $U^0_j= u_{\op{init}}(\frac j N)$ for $0\le j \le N-1$. 
 Then for $N\ge N_1=N_1(u_{\op{init}}, \nu)>0$, we have
\begin{align*}
\|U(t)\|_{\infty} \le \|u(t) \|_{\infty} \le  1 + C_1 \cdot N^{-\frac 12},
\end{align*}
where $C_1>0$ depends only on ($u_{\op{init}} $, $\nu$). In the above $u$ is the solution
to \eqref{Nov1.2} with $Q_N u_{\op{init}}$ as initial data.
\end{thm}
\begin{rem}
Since we are concerned with the generalized maximum principe,
the dependence of $N_1$ on $\nu$ will only be power-type which is much better than the
dependence in Theorem \ref{Nov2_thma}. These can be easily inferred from the computations
in the proof below. However we shall not state the explicit dependence here for the ease of 
notation. 
\end{rem}

\begin{proof}
Step 1. Energy estimate and global wellposedness. Concerning the system \eqref{Nov1.1},
it is not difficult to check that $\frac d {dt} E(U) \le 0$, where
\begin{align}
E(U) & = \frac 12 \nu^2 \| \nabla_h U\|_2^2 + \frac 14 \cdot \frac 1N \sum_{j=0}^{N-1}
(U_j^2-1)^2  \notag \\
& = \frac 12 \nu^2 \langle -\Delta_h U, U \rangle +\frac 14
\langle U^{.2} -1, U^{.2}-1 \rangle. \label{Nov18EU1}
\end{align}
It follows easily that the system \eqref{Nov1.1} admits a global solution. 

Step 2. $H^1$ Estimate of $u$. We first consider the initial data. Observe that
\begin{align*}
\frac 1 {4\pi^2} \| \partial_x Q_N u_{\op{init}} \|_2^2 & = \sum_{|k|< \frac N2} k^2
\Bigl|\sum_{l \in \mathbb Z} \widehat{u_{\op{init}}} (k+l N) \Bigr|^2  
 + 2 \cdot \frac {N^2} 4 \Bigl|\frac 12\sum_{l\in \mathbb Z} \widehat{u_{\op{init}}}(\frac N2+l N)
\Bigr|^2 \notag \\
& \lesssim \| \partial_x u_{\op{init}} \|_2^2 +  \| u_{\op{init} } \|_{H^{\frac 32}(\mathbb T)}^2
\lesssim \| u_{\op{init} } \|_{H^{\frac 32}(\mathbb T)}^2.
\end{align*}
Similarly $\| Q_N u_{\op{init}} \|_2 \lesssim \| u_{\op{init} } \|_{H^{\frac 32}(\mathbb T)}$.  Thus
\begin{align*}
\|Q_N u_{\op{init} } \|_2 +
\|Q_N u_{\op{init} } \|_{\infty} \lesssim \| Q_N u_{\op{init}} \|_{H^1(\mathbb T)} \lesssim \| u _{\op{init}} \|_{H^{\frac 32}(\mathbb T)}.
\end{align*}
Note that $\| U^0\|_{\infty} \le \| Q_N u_{\op{init} } \|_{\infty}$, and by 
\eqref{Nov1.001}
\begin{align*}
\| \nabla_h U^0\|_2^2 = \sum_{-\frac N 2 <k\le \frac N2} k^2
|\widetilde{U^0_k}|^2 \sim \| \partial_x Q_N u_{\op{init} } \|_2^2.
\end{align*}
Thus
\begin{align*}
E(U^0) \le  A_1=A_1(\| u_{\op{init}} \|_{H^{\frac 32} }, \nu).
\end{align*}
Now by using the inequality $ U_j^2 \le \frac 12 (U_j^2-1)^2 + \frac 32$, it is then clear that
\begin{align*}
\| u(t) \|_{H^1}^2  \lesssim  \sum_{-\frac N2 <k\le \frac N2}
(1+k^2)|\widetilde{U(t)_k}|^2  \lesssim 1+  E(U(t) ) \le 1+ E(U^0)
\le  1+ A_1, \qquad\forall\, t\ge 0.
\end{align*}

Step 3. $L^{\infty}$-estimate of $u$.  For the initial data, observe that
\begin{align*}
\| Q_N u_{\op{init} } - u_{\op{init}} \|_{\infty}
\lesssim \sum_{|k|\ge \frac N2} | \widehat{u_{\op{init} }} (k)| \lesssim N^{-\frac 12}
\| u_{\op{init}} \|_{H^1(\mathbb T)}.
\end{align*}
Thus
\begin{align*}
\| Q_N u_{\op {init}} \|_{\infty} \le 1+ O(N^{-\frac 12}),
\end{align*}
where the implied constant depends on ($u_{\op{init}}$, $\nu$). 

Now we consider the nonlinear term.  By using Step 2, we have $\|u(t)^3 \|_{H^1} \lesssim 1$
uniformly in $t\ge 0$.  Clearly we have
\begin{align*}
\| Q_N (u^3- u ) - (u^3 -u ) \|_{\infty} = O(N^{-\frac 12}). 
\end{align*}
Then the equation for $u$ can be rewritten as
\begin{align*}
\begin{cases}
\partial_t u = \nu^2 \Delta u +u-u^3 + O(N^{-\frac 12}), \\
\| u(0) \|_{\infty} \le 1 + O(N^{-\frac 12}).
\end{cases}
\end{align*}
A simple maximum principle argument then yields the desired result. 
\end{proof}

\begin{thm}[Almost sharp maximum principle, version 2]
Consider \eqref{Nov1.1}. Suppose $u_{\op{init}}:\, \mathbb T\to \mathbb R$ has Fourier support in 
$\{|k|\le \frac N2 \}$ with $\widehat{u_{\op{init}}}(\frac N2)
=\widehat{u_{\op{init}}}(-\frac N2)$ and
$\|u_{\op{init}}\|_{\infty} \le 1$. Note that $Q_N (u_{\op{init}})=u_{\op{init}}$.
 Take the initial data
$U^0$ such that $U^0_j= u_{\op{init}}(\frac j N)$ for $0\le j \le N-1$. 
 Then for $N\ge N_2=N_2(u_{\op{init}}, \nu)>0$, we have
\begin{align*}
\|U(t)\|_{\infty} \le \|u(t) \|_{\infty} \le  1 + C_2 \cdot N^{-\frac 14},
\end{align*}
where $C_2>0$ depends only on ($u_{\op{init}} $, $\nu$). In the above $u$ is the solution
to \eqref{Nov1.2} with $Q_N u_{\op{init}} = u_{\op{init}} $ as initial data.
\end{thm}
\begin{rem}
The dependence of $N_2$ on $\nu$ is only power-like.
\end{rem}
\begin{rem}
By Remark \ref{Nov1re1}, one cannot deduce a good bound
on $\|Q_N u_{\op{init}} \|_{\infty}$ for general $u_{\op{init}}$ assuming only
$\|u_{\op{init}}\|_{\infty} \le 1$.
\end{rem}
\begin{proof}
For simplicity of notation write $u_0 = u_{\op{init}}$.  Note that $u(t)$, $t\ge 0$ all have
frequency supported in $[-\frac N2, \frac N2]$. In particular $Q_N u(t) =u(t)$ for all $t\ge 0$. 

Step 1. Local in time estimate near $t=0$. By a contraction estimate,  for some sufficiently small
$t_0=t_0(u_0,\nu)>0$ we have 
\begin{align*}
& \| u(t) \|_{\infty} \le 1 +\beta_1 t^{\frac 12}, \quad \forall\, 0\le t\le t_0; \\
&  \sup_{0\le t\le t_0} t^{\frac 12} \| u(t) \| _{H^1} \le \beta_2,
\end{align*}
where the constants $\beta_1>0$, $\beta_2>0$ only depend on  ($u_0$, $\nu$). 
Now pick $t_1= N^{-\frac 1{2} } \le t_0$ and we clearly have
\begin{align*}
& \sup_{0\le t\le t_1} \|u(t) \|_{\infty} \le 1 + O(N^{-\frac 1{4} }); \\
& \| u(t_1) \|_{H^1} \le O(N^{\frac 1{4}}).
\end{align*}
For $t>t_1$, we can regard $u(t_1)$ as initial data and observe that by using energy conservation
(similar to the argument in the preceding theorem, but note that we no longer need the $H^{\frac 32}$ bound), we have
\begin{align*}
\sup_{t\ge t_1} \|u(t) \|_{H^1} \le O(N^{\frac 14}). 
\end{align*}
Furthermore, for $t\ge t_1$, assuming that $\|u(t) \|_{\infty} \le 2$, it is easy to check that
\begin{align*}
\| Q_N (u(t)^3) - u(t)^3 \|_{\infty} \le O(N^{-\frac 14}).
\end{align*}
One can then run a maximum principle argument to finish the proof (i.e. consider some $t_*>t_1$
as the first time that $\|u(t)\|_{\infty}$ exceeds the bound $1+ C\cdot N^{-\frac 14}$ and
show a contradiction). 
\end{proof}

\begin{thm}[Almost sharp maximum principle for \eqref{Nov1.1}, general rough initial data]
\label{N21_5}
Consider \eqref{Nov1.1}. Suppose $U^0 \in \mathbb R^N$ satisfies $\|U^0\|_{\infty} \le 1$.
If $N\ge N_3(\nu)>0$,  it holds that
\begin{align*}
\sup_{0\le t <\infty} \| U(t) \|_{\infty} \le 1+\epsilon_1,
\end{align*}
where $0<\epsilon_1<10^{-2}$ is an absolute constant. More precisely the following hold.
(Below we shall write $X=O(Y)$ if $|X| \le C Y$ where the constant $C$ only depends on $\nu$.)
\begin{enumerate}
\item For $t\ge T_0= T_0(\nu)>0$, 
\begin{align*}
& \| U(t) \|_{\infty} \le \|u(t) \|_{\infty} \le 1 + O(N^{-\frac 18} (\log N)^2); \\
& \| \partial_x u(t) \|_2 \le O(1), \quad E( U(t)) \le O(1),
\end{align*}
where $u$ solves \eqref{N21_5.e1}, and $E(U)$ was defined in \eqref{Nov18EU1}.

\item For $C_1=C_1(\nu)>0 $ and $C_1\cdot N^{-\frac 34} \le t \le T_0$, it holds that
\begin{align*}
\| U(t) \|_{\infty} \le \| u(t) \|_{\infty} \le 1+ O(N^{-\frac 18} (\log N)^2 ).
\end{align*}

\item For $C_2=C_2(\nu)>0$ and $ C_2 \cdot N^{-2} \log N \le t \le C_1 \cdot N^{-\frac 34}$,
\begin{align*}
\| U(t) \|_{\infty} \le \| u(t) \|_{\infty} \le 1+ O(N^{-\frac 3 {16}} ).
\end{align*}

\item For $0<t \le C_2 N^{-2} \log N$, we have
\begin{align*}
& \| u(t) -e^{\nu^2 t \partial_{xx}} v_0 \|_{\infty} \le O(N^{-\frac 12} (\log N)^{\frac 14} ), \\
& \sup_{0\le l \le N-1} | (e^{\nu^2 t \partial_{xx} } Q_N U^0)(t, \frac l N) |
\le (1+\frac 12 \epsilon_1) \|U^0\|_{\infty};\\
& \|U(t) \|_{\infty} \le (1+\frac 12 \epsilon_1) \|U^0\|_{\infty}
+ O(N^{-\frac 12} (\log N)^{\frac 14} ).
\end{align*}

\item If $\|U^0\|_{\infty} \le \frac 1 {1+\frac 12 \epsilon_1} $ (note that 
$\frac 1 {1+\frac 12 \epsilon_1} >0.995$), then for all $t\ge 0$, we have
\begin{align*}
\|U(t) \|_{\infty} \le 1 + O(N^{-\frac 18} (\log N)^2).
\end{align*}

\item Suppose $f:\, \mathbb T\to \mathbb R$ is continuous and $\|f\|_{\infty} \le 1$. 
If $(U^0)_l= f(\frac l N)$ for all $0\le l\le N-1$,  then
\begin{align*}
\sup_{t\ge 0} \|U(t) \|_{\infty} 
\le 1 + O(\omega_f(N^{-\frac 23}) ) + O(N^{-c}),
\end{align*}
where $c>0$ is an absolute constant, and $\omega_f$ is defined in \eqref{N18:1b.0}.

\item  Suppose $f:\, \mathbb T\to \mathbb R$ is $C^{\alpha}$-continuous (see
\eqref{N18:1b.1}) for some $0<\alpha<1$ and $\|f\|_{\infty} \le 1$.  
If $(U^0)_l= f(\frac l N)$ for all $0\le l\le N-1$,  then
\begin{align*}
\sup_{t\ge 0}\|U(t)\|_{\infty}
\le 1 + O(N^{-c_1}),
\end{align*}
where $c_1>0$ is a constant depending only on $\alpha$. 
\end{enumerate}
Moreover we have the following result which shows the sharpness of our estimates above.
There exists a  function $f$: $\mathbb T \to \mathbb R$, continuous at all of $\mathbb T
\setminus \{x_*\}$ for some $x_*\in \mathbb T$ (i.e. continuous at all $x\ne x_*$) and  has the bound $\|f\|_{\infty} \le 1$
 such that 
the following hold:
for a sequence of even numbers $N_m \to \infty$, $t_m =\nu^{-2}N_m^{-2}$ and $x_m
=j_m/N_m$ with $0\le j_m\le N_m-1$ ,  if  $\tilde U_m(t) \in \mathbb R^{N_m}$ solves \eqref{Nov1.1}
with $\tilde U_m(0) =v_m \in \mathbb R^{N_m}$ satisfying $(v_m)_j= f(\frac j {N_m})$ for all $0\le j \le N_m-1$.  Then
\begin{align*}
|\tilde U_m(t_m,x_m) | \ge 1+ \eta_*, \qquad\forall\, m\ge 1,
\end{align*}
where $\eta_*>0.001$ is an absolute constant.
\end{thm}
\begin{rem}
 One should compare our last result with the Fourier spectral Galerkin truncation case (see for
example Theorem \ref{17O8}) where we dealt with initial data bounded 
by one and proved a sharp maximum principle. There thanks to spectral localization the initial
data is smooth.  Here in the collocation case, our counterexample
shows that there is non-approximation if we work with mere $L^{\infty}$ initial data which is bounded by
one and everywhere continuous except at one point. This is certainly connected with the non-smooth
cut-off in the Fourier space.
\end{rem}
\begin{rem}
The dependence of $N_3$ on $\nu$ is only power-like.
\end{rem}

\begin{proof}
To simplify the notation we shall tacitly assume $\nu=1$ below. It is not difficult to keep
track of the explicit dependence of $\nu$  but we shall not do it here in order to simplify
the presentation. 
Let $u=u(t,x)$ solve the following system (below $\nu=1$):
\begin{align} \label{N21_5.e1}
\begin{cases}
\partial_t u = \nu^2 \partial_{xx} u +Q_N( u -u^3), \\
u\Bigr|_{t=0} =Q_N U^0 =:v_0,
\end{cases}
\end{align}
where $Q_N U^0$ was defined in \eqref{Nov1.01b}.  It is not difficult to check 
that $u(t, \frac j N)= U(t)_j$ for all $0\le j\le N-1$ and $t\ge 0$. 
Furthermore 
\begin{align*}
&\| v_0\|_{L_x^2(\mathbb T)} = \| U^0 \|_{l^2} 
= \Bigl( \frac 1 N \sum_j (U^0_j)^2  \Bigr)^{\frac 12} \le 1; \\
& \| v_0 \|_{L_x^{\infty} (\mathbb T)} \lesssim \log N.
\end{align*}
In the argument below we shall take $N\ge N_0$ where $N_0$ is  a sufficiently large absolute constant. The needed largeness of $N_0$ can be easily worked out from the argument.

Step 1. Local in time estimate. Introduce the norm
\begin{align*}
\|u\|_{X_T}:= \frac { \|u \|_{L_t^{\infty} L_x^{\infty} ([0,T]\times \mathbb T)}}
{\log N} + \| t^{0.5} \partial_x u \|_{L_t^{\infty} L_x^2( [0,T]\times \mathbb T)}
+ \| u \|_{L_t^{\infty} L_x^2 ([0,T] \times \mathbb T)}.
\end{align*}
By using the mild formulation  (note that $Q_N u = u$)
\begin{align*}
u(t) = e^{t \partial_{xx} } v_0 
+ \int_0^t e^{(t-s) \partial_{xx} } (u - u^3 ) ds 
-\int_0^t e^{(t-s) \partial_{xx} } ( Q_N - \op{Id} ) (u^3) ds,
\end{align*}
we have 
\begin{align*}
\|u (t) \|_{\infty}
&\lesssim \log N+ t \max_{0\le s \le t } \| u(s) \|_{\infty}
+ \int_0^t (t-s)^{-\frac 12} \| u^3 \|_1 ds 
+ \int_0^t N^{-\frac 12} \| \partial_x (u^3) \|_2 ds \notag \\
& \lesssim \log N + t \log N +  t^{\frac 12} \log N +  N^{-\frac 12} (\log N)^2 t^{0.5}
\lesssim \log N.
\end{align*}
Also  by using the inequality $\| u\|_3 \lesssim  \|u\|_2+
\| u\|_2^{\frac 56} \| \partial_x u \|_2^{\frac 16}$, we have
\begin{align*}
\| u(t) \|_2 
& \lesssim 1 + t +  \int_0^t (t-s)^{-\frac 14} \| u^3 \|_1ds 
+ \int_0^t   \| (Q_N -\op{Id} )(u^3) \|_{L^2(\mathbb T)} ds \notag \\
& \lesssim 1 + t +  \int_0^t (t-s)^{-\frac 14} \| u^3 \|_1ds 
+ \int_0^t   \| (Q_N -\op{Id} )(u^3) \|_{L^{\infty}(\mathbb T)} ds \notag \\
& \lesssim 1 + t +  \int_0^t (t-s)^{-\frac 14} \| u^3 \|_1ds 
+ \int_0^t  N^{-\frac 12} \| \partial_x (u^3 ) \|_2 ds \notag \\
& \lesssim 1 + t + t^{\frac 34} + \int_0^t (t-s)^{-\frac 14} s^{-\frac 14} ds 
+N^{-\frac 12} (\log N)^2 t^{0.5} \lesssim 1,
\end{align*}
and 
\begin{align*}
t^{0.5} \| \partial_x u (t) \|_2 & \lesssim 1+ t^{0.5} \int_0^t (t-s)^{-\frac 34}
\| u^3 \|_1 ds + t^{0.5} \int_0^t (t-s)^{-\frac 12} N^{-\frac 12} \| \partial_x (u^3) \|_2 ds 
\notag \\
& \lesssim 1+ t^{0.5} \int_0^t (t-s)^{-\frac 34}
s^{-\frac 14} ds + t^{0.5} \int_0^t (t-s)^{-\frac 12} N^{-\frac 12} (\log N)^2 s^{-\frac 12} ds 
\notag \\
& \lesssim  1+ t^{0.5} \lesssim 1.
\end{align*}
By carefully tracking the constants and similar estimates as done in the above, it is not difficult
to check that for $T=T_0>0$ ($T_0$ is a  sufficiently small absolute constant) we have 
contraction in the ball $\{ \|u \|_{X_T} \le C_1 \} $ where $C_1$ is an absolute constant.

Step 2. Refined estimate. For any $0<t\le T_0$, we have
\begin{align*}
\| u(t) - e^{t\partial_{xx} } v_0 \|_{\infty}
& \le \int_0^t \| e^{(t-s)\partial_{xx} } (u-u^3) \|_{\infty} ds +
\int_0^t \| e^{(t-s) \partial_{xx}} ( Q_N-\op{Id} ) (u^3) \|_{\infty} ds \notag \\
& \lesssim \int_0^t (t-s)^{-\frac 14} \| u(s) \|_2 ds 
+ \int_0^t (t-s)^{-\frac 12} \| u^3 \|_1 ds + 
\int_0^t N^{-\frac 12} \| \partial_x (u^3 ) \|_2 ds \notag \\
& \lesssim  t^{\frac 34}+t^{\frac 14} + N^{-\frac 12} (\log N)^2 t^{\frac 12}.
\end{align*}
Now take $T_1 = 100 N^{-2} \log N \le T_0$, we have 
\begin{align*}
\| u(T_1) -e^{T_1 \partial_{xx} } v_0 \|_{\infty} \lesssim N^{-\frac 12} (\log N)^{\frac 14}.
\end{align*}
On the other hand, since
\begin{align*}
(e^{t \partial_{xx}} v_0)(t,x)
= \frac 1 N \sum_{j=0}^{N-1} (U^0)_j
\op{Re}
\Bigl( \sum_{-\frac N2 <k \le \frac N2}
e^{-(2\pi k)^2 t + 2\pi i k (x-\frac j N) } \Bigr),
\end{align*}
it is clear that 
\begin{align}
\| e^{T_1 \partial_{xx} } v_0 \|_{\infty}
& \le \| \op{Re} \Bigl(\sum_{-\frac N2 <k \le \frac N2} e^{-(2\pi k)^2 T_1 + 2\pi i k\cdot x } \Bigr) \|_{\infty}
\notag \\
& \le 1 + \int_{|k| >\frac N2 -1} e^{-(2\pi k)^2 T_1}  dk \notag \\
& \le 1 + N^{-1}.  \notag 
\end{align}
Thus 
\begin{align*}
\|u (T_1) \|_{\infty} \le 1 + O(N^{-\frac 12} (\log N)^{\frac 14} ).
\end{align*}
By a similar estimate, we have for any $T_1 \le t \le T_2 =N^{-\frac 34} \le T_0$, 
\begin{align*}
\| u(t) \|_{\infty} \le 1 + O(t^{\frac 14}) \le 1 +O(N^{-\frac 3 {16}} ).
\end{align*}

Step 3. The regime $t\ge T_2$.  Note that at $t=T_2$, we have
\begin{align*}
& \| u(T_2)\|_{\infty} \le 1 + O(N^{-\frac 3{16}});\\
& \| \partial_x u(T_2) \|_{2} \le  O(T_2^{-\frac 12}) = O(N^{\frac 3 8}).
\end{align*}
It follows easily that for the system \eqref{Nov1.1},
\begin{align*}
E(U(t)) \le E(U(T_2) ) \le O(N^{\frac 38}), \quad \forall\, t\ge T_2,
\end{align*}
where $E(U)$ was defined in \eqref{Nov18EU1}. This in turn implies that 
\begin{align*}
\| \partial_x u(t) \|_2 \le O(N^{\frac 38} ), \quad \forall\, t\ge T_2.
\end{align*}
One should note that for $t\ge T_0$ we have $\| \partial_x u(t) \|_2 \lesssim 1$. 
Thus 
\begin{align*}
\|(Q_N-\op{Id})(u(t)^3 ) \|_{\infty} \lesssim N^{-\frac 12}
\cdot N^{\frac 38} (\log N)^2 \le O(N^{-\frac 18} (\log N)^2), \quad \forall\, t\ge T_2.
\end{align*}
By using a maximum principle argument, we then obtain 
\begin{align*}
\sup_{t\ge T_2} \| u(t) \|_{\infty} \le 1+ O(N^{-\frac 18} (\log N)^2).
\end{align*}

Step 4. The regime $t\le T_1$.  In this regime we observe that 
\begin{align*}
\| u(t) - e^{t \partial_{xx}} v_0 \|_{\infty} \lesssim N^{-\frac 12} (\log N)^{\frac 14},
\quad \forall\, 0\le t \le T_1.
\end{align*}
By Remark \ref{Nov1re1}, in general we have $\|v_0\|_{\infty} \lesssim \log N$ which is
optimal. Thus we need to restrict to the grid points $x_j = \frac j N$, $0\le j\le N-1$ and 
work with the function $U(t)$. Now denote  $U^{\op{lin} } (t) \in \mathbb R^N$ such that
\begin{align*}
U^{\op{lin}} (t)_j = (e^{t \partial_{xx} } v_0 )(t,\frac jN), \quad \forall\, 0\le j \le N-1.
\end{align*}
Clearly 
\begin{align*}
\| U(t) - U^{\op{lin}}(t) \|_{\infty} \lesssim N^{-\frac 12} (\log N)^{\frac 14},
\quad \forall\, 0\le t \le T_1.
\end{align*}
The remaining results then follow from Theorem \ref{N18:1}, Theorem \ref{N18:1a} 
and Theorem \ref{N18:1b} proved
in the appendix.
\end{proof}

\begin{thm}[Generalized $L^{\infty}$-stability for \eqref{Nov1a}, case
$\frac 12 \le \tau \le 1.99$]
Consider \eqref{Nov1a}. 
Suppose $U^0 \in \mathbb R^N$ satisfies 
$\| U^0\|_{\infty} \le 1$. Assume $\frac 12 \le \tau\le 1.99$. If $N\ge N_4(\nu)>0$ (the dependence of $N_4$ on $\nu$ is only 
power-like), then the following hold. 

\begin{itemize}
\item For all $n\ge 0$, we have
\begin{align} \label{N26_1}
\| U^n \|_{\infty} \le M_a = 
\frac 12 ( \frac 23 \cdot \frac {(1+\tau)^{\frac 32} } {\sqrt{3\tau}}
+\sqrt{\frac  {2+\tau}{\tau} } \Bigr).
\end{align}

\item If $\frac 12 \le \tau \le 0.86$, then for all $n\ge 0$, we have
\begin{align}
&\| U^n\|_{\infty} \le M_b = \frac 23 \cdot \frac {(1+\tau)^{\frac 32} } {\sqrt{3\tau}}
+ \eta_0, \qquad \eta_0=10^{-5};  \label{N26_2}\\
& E(U^{n+1}) \le E(U^n),     \label{N26_3}
\end{align}
where $E(U)$ was defined in \eqref{Nov18EU1}. Moreover, for $n\ge 1$,
\begin{align}
E(U^n) \le E(U^1) \le \mathcal E_1, \label{N26_4}
\end{align}
where $\mathcal E_1>0$ depends only on $\nu$.
\end{itemize}
\end{thm}
\begin{rem}
The cut-off $\tau=1.99$ is for convenience only. One can replace it by any number less
than $2$ but  $N_4$ will have to be adjusted correspondingly.
\end{rem}
\begin{proof}
Denote
\begin{align} \label{N26_4a}
&K_N(x) = \op{Re} \Bigl( \sum_{-\frac N2< k \le \frac N2} \frac 1 {1+\tau (2\pi k \nu )^2}
e^{2\pi i k \cdot x} \Bigr),  \notag \\
&K_{\infty}(x) = \sum_{k \in \mathbb Z} 
\frac 1 {1+\tau (2\pi k \nu )^2}
e^{2\pi i k \cdot x}.
\end{align}
For $\frac 12 \le \tau \le 2$ we have 
\begin{align*}
\| K_N - K_{\infty} \|_{L^{\infty} (\mathbb T)} \le  O(N^{-1})
\end{align*}
and
\begin{align*}
\frac 1 N \sum_{j=0}^{N-1} | (K_N-K_{\infty})(\frac jN) | \le O(N^{-1}).
\end{align*}
Denote $p(x) = (1+\tau) x -\tau x^3$. By \eqref{Nov1}, we have
\begin{align*}
\| U^{n+1} \|_{\infty} \le \| p(U^n) \|_{\infty} 
+ \frac{\op{const}} {N}\cdot \| p(U^n)\|_{\infty}.
\end{align*}
Then \eqref{N26_1} follows from Lemma \ref{lem_cubic_iteration} and induction. 

The proof of \eqref{N26_2}--\eqref{N26_3} is similar to that in the proof of Proposition
\ref{prop3.18a}. The main point is to use the inequality
\begin{align*}
\frac 12 +\frac 1{\tau} \ge \frac 3 2 M_b^2
\end{align*}
which holds for all $\frac 12 \le \tau\le 0.86$. 

Finally \eqref{N26_4} follows from \eqref{Nov1} and the simple estimate 
\begin{align*}
\| \Delta_h U^{n+1} \|_2 \lesssim \| p(U^n) \|_2  \lesssim 1.
\end{align*}
\end{proof}

\begin{thm}[Almost sharp maximum principle for \eqref{Nov1a}, case
$ N^{-0.3} \le \tau \le \frac 12$]
Consider \eqref{Nov1a}. 
Suppose $U^0 \in \mathbb R^N$ satisfies 
$\| U^0\|_{\infty} \le 1$. Assume $
N^{-0.3} \le \tau\le \frac 12$. If $N\ge N_5(\nu)>0$ (the dependence of $N_5$ on $\nu$ is only 
power-like), then the following hold. 

For all $n\ge 0$, we have
\begin{align} 
&\| U^n \|_{\infty} \le 1+ O(N^{-\frac 13});  \label{N26_7a} \\
& E(U^{n+1}) \le E(U^n), \label{N26_7b}
\end{align}
where $E(U)$ was defined in \eqref{Nov18EU1}. Moreover, for $n\ge \frac 1 {\tau}$,
\begin{align}
E(U^n) \le \mathcal E_2, \label{N26_8}
\end{align}
where $\mathcal E_2>0$ depends only on $\nu$.
\end{thm}
\begin{proof}
We define $K_N$ and $K_{\infty}$ as in \eqref{N26_4a}. Then for $
N^{-0.3} \le \tau\le \frac 12$, it is clear that
\begin{align*}
\frac 1 N\sum_{j=0}^{N-1}
|(K_N-K_{\infty})(\frac j N) | \le \| K_N -K_{\infty} \|_{L^{\infty}(\mathbb T)} \le  
O(N^{-1} {\tau}^{-1} ). 
\end{align*}
Set $\delta>0$ which we will take to be $O(N^{-\frac 13})$. Inductively assume $\|U^n\|_{\infty} \le 1+\delta$.
By \eqref{Nov1} and the preceding estimate, we have
\begin{align*}
\| U^{n+1} \|_{\infty} \le \max_{|x| \le 1+\delta} |p(x)| +
O(N^{-1} \tau^{-1} ).
\end{align*}
where $p(x) = (1+\tau) x - \tau x^3$. Observe that $p^{\prime}(1)=1-2\tau$ and
for $0<\tau \le \frac 12$, 
\begin{align*}
\max_{|x| \le 1+\delta} |p(x)| 
\le 1 + (1-2\tau) \delta + O(\delta^2).
\end{align*}
Since we assume $\delta = O(N^{-\frac 13})$, it follows easily that
\begin{align*}
& O(\delta^2) \ll \tau \delta, \\
& O(N^{-1} \tau^{-1} ) \ll \tau \delta, 
\end{align*}
and 
\begin{align*}
\|U^{n+1} \|_{\infty} \le 1+(1-2\tau) \delta +O(\delta^2)
+ O(N^{-1} \tau^{-1}  ) \le 1+ \delta.
\end{align*}
Thus \eqref{N26_7a} holds.

Now  \eqref{N26_7b} follows from the fact that for $\tau\le \frac 12$, 
\begin{align*}
\frac 1{\tau}+\frac 12 \ge \frac 52 \ge \frac 3 2 \max\{\|U^n\|_{\infty}^2, \;
\|U^{n+1}\|_{\infty}^2 \}, \qquad \forall\, n\ge 0.
\end{align*}
One should note again here that the restriction for $\tau$  for energy stability is quite mild, i.e. $\tau$ must be smaller than $1$ (provided $\|U^n\|_{\infty}$ is close to $1$).  The restriction
for $\tau\le \frac 12$ is was due to the fact that we need to have $\|U^n\|_{\infty} \le 1+
\text{small error}$.

Finally to show \eqref{N26_8} we shall use discrete smoothing estimates.  To this
end we assume with no loss that $\nu=\frac 1 {2\pi}$ and $A=(\op{I} - \nu^2 \tau \Delta_h 
)^{-1}$. Note that on the Fourier side $\widehat A(k) = (1+ k^2 \tau)^{-1}$ for
$-\frac N2<k \le \frac N2$. It is clear that if $\tau
\sim 1$ we have $E(U^1) \le O(1)$. Thus we only need to consider the regime
$0<\tau \ll 1$.  We shall show for all 
$1\le n \le 1/\tau$,
\begin{align} \label{N27_1}
\sqrt{ n \tau} \| \partial_h U^n \|_2 \le C_0,
\end{align}
where $C_0>0$ is an absolute constant (for general $\nu>0$, $C_0$ will depend on $\nu$; here we assumed $\nu=1/2\pi$ for simplicity).
 Here we use the notation $\partial_h$ to denote 
the Fourier multiplier $2\pi ik$ in the discrete Fourier transform formula.  In particular we have
for $U\in \mathbb R^N$, 
\begin{align*}
\| \partial_h U \|_2^2 = \sum_{-\frac N2 <k \le \frac N2} 
(2\pi k)^2 |\tilde U_k|^2.
\end{align*}

It is not difficult to check that for $n\ge 1$, 
\begin{align*}
U^n = A^n U^0 + \tau \sum_{j=0}^{n-1} A^{n-j}  \Bigl( f(U^{j}) \Bigr),
\end{align*}
where $f(z) = z-z^3$ and $f(U)= U-U^{.3}$. 

Observe that for any $j\ge 1$ and $\tau>0$, we have
\begin{align*}
\sup_{k \in \mathbb Z} 
\sqrt{j\tau} 2\pi |k| (1+\tau k^2)^{-j}
\le 2\pi\cdot \sup_{x\ge 0}  x (1+\frac {x^2} j)^{-j} \le C_1<\infty,
\end{align*}
where $C_1>0$ is an absolute constant.

Now first for $n=1$, we have 
\begin{align*}
\sqrt{\tau} \|\partial_h U^1\|_2 
& \le \sqrt{\tau} \| \partial_x A U^0\|_2 + \tau^{\frac 32}
\| \partial_x A  \Bigl( f(U^0) \Bigr) \|_2 \notag \\
& \lesssim \| U^0\|_2 + \tau \| f(U^0)\|_2 \le  C_0,
\end{align*}
where clearly we can take $C_0$ sufficiently large to fulfill the last bound.

Similarly for $n\ge 2$ with $n\tau \le 1$,
\begin{align*}
\sqrt{n\tau} \|\partial_h U^n\|_2 
& \le \sqrt{n\tau} \|\partial_x A^n U^0\|_2 
+ \tau \cdot \sqrt{n\tau}
\| \partial_x A^n \Bigl( f(U^0) \Bigr) \|_2
+ \tau \sqrt{n\tau} \sum_{j=1}^{n-1} \| \partial_x A^{n-j} \Bigl( f(U^{j} ) \Bigr)
\|_2 \notag \\
& \le C_1 \| U^0\|_2+ C_1 \| f(U^0)\|_2+  \sup_{1\le j\le n-1}
\|f(U^j)\|_2 \cdot C_1 \tau \cdot \sqrt{n\tau} \sum_{j=1}^{n-1} \frac 1 {\sqrt{(n-j) \tau} }
\le C_0.
\end{align*}
Thus \eqref{N27_1} is proved.
\end{proof}

Our next result is concerned with the regime $0<\tau \le N^{-0.3}$. Let $u^n=u^n(x):
\mathbb T \to \mathbb R$ solve
the system
\begin{align} \label{N28_0}
\begin{cases}
(\op{Id}-\nu \tau \Delta) u^{n+1} =  u^n + \tau Q_N ( f(u^n) ), \\
u^0= Q_N U^0,
\end{cases}
\end{align}
where $f(z)=z-z^3$. It is not difficult to check that $u^n$ and \eqref{Nov1a} are
equivalent in the sense that $u^n(\frac j N) = (U^n)_j$ for all $n\ge 0$ and
$0\le j\le N-1$.

\begin{thm}[Almost sharp maximum principle for \eqref{Nov1a}, case
$ 0<\tau \le N^{-0.3} $]
Consider \eqref{Nov1a}. 
Suppose $U^0 \in \mathbb R^N$ satisfies 
$\| U^0\|_{\infty} \le 1$. Assume $
0<\tau\le N^{-0.3} $. If $N\ge N_6(\nu)>0$ (the dependence of $N_6$ on $\nu$ is only 
power-like), then it holds that 
\begin{align*}
\sup_{n\ge 0} \|U^n \|_{\infty} \le 1 +\epsilon_1,
\end{align*}
where $0<\epsilon_1 <10^{-2}$ is an absolute constant. More precisely the following hold.
(Below we shall write $X=O(Y)$ if $|X|\le CY$ and the constant $C$ depends only
on $\nu$.)
\begin{enumerate}
\item $E(U^{n+1}) \le E(U^n)$ for all $n\ge 0$.
\item For some $T_0=T_0(\nu)>0$ sufficiently small and for all $n\ge T_0/\tau$, we have
\begin{align*}
&  \| \partial_x u^n \|_{L_x^2(\mathbb T)} \le O(1); \\
& \| U^n \|_{\infty} \le \| u^n\|_{\infty} \le 1+ O(N^{-\frac 3{40}}),
\end{align*}
where $u^n$ solves \eqref{N28_0}.
\item  For all $n\ge 2N^{-0.3}/\tau$, we have
\begin{align*}
\| U^n\|_{\infty} \le \| u^n \|_{\infty} \le 1+O(N^{-\frac 3 {40}}).
\end{align*}
\item For $1\le n \le 2N^{-0.3}/\tau$, we have 
\begin{align*}
& \|u^n - (\op{Id}-\nu \tau \partial_{xx})^{-n} u^0 \|_{\infty}
\le O(N^{-\frac 3 {40}}); \\
& \sup_{0\le l\le N-1}\Bigl| 
\Bigl( (\op{Id}-\nu \tau \partial_{xx})^{-n} u^0 \Bigr)(\frac l N)
\Bigr| \le (1+\frac 1 2 \epsilon_1) \|U^0\|_{\infty};  \\
& \|U^n\|_{\infty} \le (1+\frac 12 \epsilon_1)
\|U^0\|_{\infty} + O(N^{-\frac 3 {40}}).
\end{align*}
\item If $\|U^0\|_{\infty} \le \frac 1 {1+\frac 12 \epsilon_1}$ 
(note that $1/(1+\frac 12 \epsilon_1)>0.995$), then 
\begin{align*}
\sup_{n\ge 0} \|U^n\|_{\infty} \le 1 + O(N^{-\frac 3{40}}).
\end{align*}

\item Suppose $f:\, \mathbb T\to \mathbb R$ is continuous and $\|f\|_{\infty} \le 1$. 
If $(U^0)_l= f(\frac l N)$ for all $0\le l\le N-1$,  then
\begin{align*}
\sup_{n\ge 0} \|U^n\|_{\infty} 
\le 1 + O(\omega_f(N^{-\frac 23}) ) + O(N^{-c}),
\end{align*}
where $c>0$ is an absolute constant, and $\omega_f$ is defined in \eqref{N18:1b.0}.

\item  Suppose $f:\, \mathbb T\to \mathbb R$ is $C^{\alpha}$-continuous for some $0<\alpha<1$ and $\|f\|_{\infty} \le 1$.  
If $(U^0)_l= f(\frac l N)$ for all $0\le l\le N-1$,  then
\begin{align*}
\sup_{n\ge 0}\|U^n\|_{\infty}
\le 1 + O(N^{-c_1}),
\end{align*}
where $c_1>0$ is a constant depending only on $\alpha$. 
\end{enumerate}
Moreover we have the following result which shows the sharpness of our estimates above.
There exists a  function $f$: $\mathbb T \to \mathbb R$, continuous at all of $\mathbb T
\setminus \{x_*\}$ for some $x_*\in \mathbb T$ (i.e. continuous at all $x\ne x_*$) and  has the bound $\|f\|_{\infty} \le 1$
 such that 
the following hold:
for a sequence of even numbers $N_m \to \infty$, $\tau_m =\frac 14 \nu^{-2}N_m^{-2}$ and $x_m
=j_m/N_m$ with $0\le j_m\le N_m-1$ , if $\tilde U_m \in \mathbb R^{N_m}$ solves 
\begin{align*}
(\op{Id} -\tau \nu^2 \Delta_h ) \tilde U_m = V_m + \tau_m f(V_m),
\end{align*}
where $(V_m)_l = f(\frac l {N_m})$ for all $0\le l\le N_m-1$. Then 
\begin{align*}
|\tilde U_m(x_m) | \ge 1+ \eta_*, \qquad\forall\, m\ge 1,
\end{align*}
where $\eta_*>0.001$ is an absolute constant.
\end{thm}

\begin{proof}
We shall use the system \eqref{N28_0}.  Denote $B= (\op{Id}- \nu \tau \Delta)^{-1}$. 
Clearly we have (below $f_1(z)=-z^3$)
\begin{align} \label{N28_4}
u^n= B^{n-1} u^1+ \tau \sum_{j=1}^{n-1} B^{n-j}\Bigl( f(u^j) \Bigr)
+\tau \sum_{j=1}^{n-1} B^{n-j} (Q_N-\op{Id}) \Bigl( f_1(u^j) \Bigr), \; n\ge 1,
\end{align}
with 
\begin{align*}
\|u^0\|_{L^2(\mathbb T)} \le 1, \qquad \|u^0\|_{L^{\infty}(\mathbb T)}
\lesssim \log N.
\end{align*}
In the argument below we shall assume $N\ge N_6=N_6(\nu)$ where $N_6$
is a sufficiently large constant depending only on $\nu$. The needed largeness of
$N_6$ as well as the precise (power) dependence on $\nu$ can be worked out from the context
with more detailed computations. However for simplicity of notation we shall not 
present it here.

Step 1. Local estimate. We show that for $T_0=T_0(\nu)>0$ sufficiently small and all
$n\ge 1$ with $n\tau \le T_0$, it holds that
\begin{align} \label{N28_7}
\frac {\|u^n\|_{\infty}} {\log N}
+ \sqrt{n\tau} \|\partial_x u^n \|_2 + \|u^n\|_2 \le C_0,
\end{align}
where $C_0>0$ is a constant depending only on $\nu$. 

For $n=1$, we shall start with
\begin{align*}
U^1 =(\op{Id}-\tau \nu^2 \Delta_h)^{-1} U^0 +\tau (\op{Id}-\tau\nu^2 \Delta_h)^{-1}
(f(U^0) ).
\end{align*}
Clearly
\begin{align*}
&\sqrt{\tau} \|\partial_x u^1\|_2 \le \sqrt{\tau} \| \partial_h U^1 \|_2  \lesssim \|U^0\|_2 + 
\tau \| f(U^0)\|_2 \le \frac 13 C_0;\\
&\|u^1\|_2 \le \|U^1\|_2 \lesssim \|U^0\|_2+ \tau \| f(U^0)\|_2 \le
\frac 13 C_0.
\end{align*}
On the other hand, noting that $\Bigl( f(u^0) \Bigr)(\frac j N) =
\Bigl(f(U^0) \Bigr)_j$ for all $0\le j\le N-1$, we have
\begin{align*}
\| u^1 \|_{\infty} &\le  \|u^0\|_{\infty}+ \tau \| Q_N ( f(u^0) ) \|_{\infty} \notag \\
& = \|Q_N U^0\|_{\infty} +\tau \| Q_N ( f(U^0) ) \|_{\infty}
\lesssim \log N. 
\end{align*}
Thus 
\begin{align*}
\frac {\|u^1\|_{\infty} } {\log N} + \sqrt{\tau} \|\partial_x u^1 \|_2 
+ \|u^1 \|_2 \le C_0.
\end{align*}
Next for $n\ge 2$ with $n\tau \le T_0$ and for all $1\le j\le n-1$, we inductively assume that 
\begin{align*}
\frac {\|u^j\|_{\infty}} {\log N}
+ \sqrt{j\tau} \|\partial_x u^j \|_2 + \|u^j\|_2 \le C_0. 
\end{align*}
Then by \eqref{N28_4} and Lemma \ref{N29_apL1}, we have
(note that $j\tau \le T_0 \ll 1$)
\begin{align*}
\|u^n\|_{\infty}
&\le \|u^1\|_{\infty}
+ \tau b_1 \sum_{j=1}^{n-1}
((n-j)\tau)^{-\frac 12} \| f(u^j) \|_1
+ \tau \sum_{j=1}^{n-1}
\| (Q_N -\op{Id})(f_1(u^j) ) \|_{\infty} \notag \\
& \le b_2 \log N + b_3 \cdot \sqrt{n\tau} \cdot C_0(C_0+1) \log N 
+  \tau \sum_{j=1}^{n-1} N^{-\frac 12} b_4 \cdot \|f_1^{\prime}(u^j) \partial_x u^j\|_2
\notag \\
& \le b_2 \log N + b_3 \cdot \sqrt{n\tau} \cdot C_0(C_0+1) \log N 
+  \tau \sum_{j=1}^{n-1} N^{-\frac 12} b_4 \cdot  (\log N)^2 (j\tau)^{-\frac 12}
\cdot C_0^3
\notag \\
& \le \frac 13 C_0 \log N,
\end{align*}
where in the above we use $b_i$ to denote various constants depending only on $\nu$, and 
we take $T_0$ sufficiently small and $N$ sufficiently large.
Similarly, by using the inequality 
$\|u\|_3 \lesssim \|u\|_2+ \|u\|_2^{\frac 56} \| \partial_x u\|_2^{\frac 16}$, 
we have
\begin{align*}
\|u^n\|_{2}
&\le \|u^1\|_{2} 
+ \tau b_5 \sum_{j=1}^{n-1}
((n-j)\tau)^{-\frac 14} \| f(u^j) \|_1
+ \tau \sum_{j=1}^{n-1}
\| (Q_N -\op{Id})(f_1(u^j) ) \|_{2} \notag \\
& \le b_6+ b_7\tau \sum_{j=1}^{n-1}
((n-j)\tau)^{-\frac 14}
\cdot( C_0+C_0^3 +C_0^{\frac 52}
\cdot (j\tau)^{-\frac 14} )
+\tau \sum_{j=1}^{n-1}
\| (Q_N -\op{Id})(f_1(u^j) ) \|_{\infty} \notag \\ 
& \le  b_6+ b_8\cdot ( (n\tau)^{\frac 34}+(n\tau)^{\frac 12})
+ \tau \sum_{j=1}^{n-1} N^{-\frac 12} b_4 \cdot  (\log N)^2 (j\tau)^{-\frac 12}
\cdot C_0^3
\notag \\
& \le \frac 13 C_0,
\end{align*}
and 
\begin{align*}
&\sqrt{n\tau} \|\partial_x u^n\|_2  \notag \\
\le &\sqrt{n\tau}
\| \partial_x B^{n-1} u^1 \|_2
+ \sqrt{n\tau} \tau 
\sum_{j=1}^{n-1} \| \partial_x B^{n-j} ( f(u^j ) )\|_2
+ \sqrt{n\tau}\tau \sum_{j=1}^{n-1}
\| \partial_x B^{n-j} (Q_N-\op{Id}) (f_1(u^j) ) \|_2 \notag \\
 \le&\; b_9 +  \sqrt{n\tau} \tau \sum_{j=1}^{n-1}
 ((n-j)\tau)^{-\frac 34}  \| f(u^j) \|_1
 + \sqrt{n\tau} \tau \sum_{j=1}^{n-1} ((n-j)\tau)^{-\frac 12}
 N^{-\frac 12} b_4 \| f_1^{\prime}(u^j) \partial_x u^j \|_2  \notag \\
 \le & \frac 13 C_0.
\end{align*}
Hence \eqref{N28_7} is proved.

Step 2. Refined estimate. First we note that 
\begin{align*}
\|Q_N (f(u^0) )\|_{\infty} =\|Q_N (f(U^0) )\|_{\infty} \lesssim \log N.
\end{align*}
Since $\tau \le N^{-0.3}$, we obtain for any $n\ge 1$,
\begin{align} \label{N28_8}
\tau \|B^n Q_N( f(u^0) ) \|_{\infty} \lesssim N^{-0.3} \log N.
\end{align}
Now for any $1\le n \le T_0 /\tau$,  by using \eqref{N28_8}
and the fact that 
\begin{align*}
\| u^j\|_3^3 \lesssim \|u^j\|_2^3 + \| u^j\|_2^{\frac 52}
\| \partial_x u^j \|_2^{\frac 12} \lesssim 1+ (j\tau)^{-\frac 14}, \quad \forall\, 1\le j\le T_0/\tau,
\end{align*}
 we have
\begin{align*}
&\| u^n - B^n u^0 \|_{\infty} \notag \\
\le &\;
\tau \| B^n Q_N (f (u^0) ) \|_{\infty}
+ \tau \sum_{j=1}^{n-1} \| B^{n-j} f(u^j )\|_{\infty}
+ \tau \sum_{j=1}^{n-1}  \| (Q_N-\op{Id}) ( f_1(u^j) ) \|_{\infty} \notag \\
\lesssim &\; N^{-0.3} \log N
+ \tau \sum_{j=1}^{n-1} ( (n-j )\tau)^{-\frac 14}
\| u^j \|_2
+ \tau \sum_{j=1}^{n-1} ((n-j)\tau)^{-\frac 12} \| (u^j)^3 \|_1
+ N^{-\frac 12} (\log N)^2 (n\tau)^{\frac 12} \notag \\
\lesssim & N^{-0.3} \log N+ (n\tau)^{\frac 14}
+ (n\tau)^{\frac 34} \lesssim N^{-0.3} \log N + (n\tau)^{\frac 14}.
\end{align*}

Step 3. The regime $n\tau \ge 2N^{-0.3}$. Set $n_0= [2 N^{-0.3}/\tau]\ge 2$ where
$[x]$ denotes the usual integer part of the real number $x>0$.  Clearly
$n_0\tau \sim N^{-0.3}$ and 
\begin{align*}
\| u^{n_0} - B^{n_0} u^0\|_{\infty} \lesssim N^{-\frac 3 {40}}.
\end{align*}
On the other hand, since
\begin{align*}
(B^{n_0} u^0)(x) = \frac 1 N \sum_{j=0}^{N-1}
(U^0)_j \op{Re} \Bigl(
\sum_{-\frac N2 <k \le \frac N2}
(1+(2\pi k\nu)^2 \tau)^{-n_0} e^{2\pi i k(x-\frac j N) } \Bigr),
\end{align*}
we have
\begin{align*}
\| B^{n_0} u^0 \|_{\infty}
& \le 1 + \sum_{|k|>\frac N2-1}  (1+(2\pi k \nu)^2 \tau)^{-n_0} \notag \\
& \le 1 + \sum_{|k|>\frac N2-1}  \frac 1 { 1+ (2\pi \nu)^2 n_0 \tau k^2} \notag \\
& \le 1+ O( (n_0\tau N)^{-1}) = 1+ O(N^{-0.7}).
\end{align*}
Thus
\begin{align*}
\| u^{n_0} \|_{\infty} \le 1 + O(N^{-\frac 3 {40}} ).
\end{align*}
Note that 
\begin{align*}
\| \partial_x u^{n_0} \|_2 \lesssim \frac 1 {\sqrt{n_0 \tau} } \le O(N^{0.15}).
\end{align*}
Thus for $n_0 \le n \le T_0/\tau$, 
\begin{align*}
E(U^n) \le  E(U^{n_0}) \le O(N^{0.3}). 
\end{align*}
This implies for all $n_0\le n \le T_0/\tau$, 
\begin{align*}
&\| \partial_x u^n \|_2 \le O(N^{0.15} ), \\
& \| (Q_N -\op{Id}) ( f_1(u^n ) ) \|_{\infty} 
\le O(N^{-\frac 12} \cdot N^{0.15}) = O(N^{-0.35}).
\end{align*}
Now consider the system
\begin{align*}
(\op{Id}-\nu^2 \tau \Delta) u^{n+1} = u^n + \tau f(u^n) 
+\tau (Q_N-\op{Id})(f_1(u^n) ).
\end{align*}
If we assume $\|u^n\|_{\infty} \le 1+\delta$, then 
(note again that for $p(x)=x+\tau (x-x^3)$, $p^{\prime\prime}(x)<0$ for $x$ near $\pm 1$ so
that $O(\delta^2)$ does not appear below)
\begin{align*}
\|u^{n+1}\|_{\infty} \le 1+ (1-2\tau) \delta 
+\tau \cdot O(N^{-0.35}) \le 1+\delta,
\end{align*}
provided we take $\delta \ge O(N^{-0.25})$. 

Thus it is clear that for all $n_0 \le n \le T_0/\tau$,
\begin{align*}
\| u^n \|_{\infty} \le 1 + O(N^{-\frac 3 {40} }).
\end{align*}
For $n\ge T_0/\tau$, note that $\| \partial_x u^n \|_2 \lesssim 1$. By another maximum
principle argument, we obtain
\begin{align*}
\sup_{n\ge n_0} \|u^n\|_{\infty} \le 1 +O(N^{-\frac 3{40}}).
\end{align*}
Note that we did not optimize the exponent  (in $N$) here for simplicity of presentation. These
can certainly be improved by more refined estimates but we do not dwell on this issue here.

Step 4. The regime $ n\tau \le 2N^{-0.3}$.  In this regime we observe that
\begin{align*}
\| u^n - B^n u^0 \|_{\infty} \le O(N^{-\frac 3 {40}}). 
\end{align*}
The main contribution of the $L^{\infty}$-estimate of $u^n$ then comes from the linear
flow piece. In this case we have to restrict to the grid points $x_j =j/N$ for $0\le j\le N-1$
and work with $U_{\op{lin} }^n \in \mathbb R^N$ such that
\begin{align*}
(U_{\op{lin}}^n )_j = (B^n u^0)(\frac j N), \quad\forall\, 0\le j\le N-1.
\end{align*}
The remaining results then follow from Theorem \ref{N21:1} proved in the appendix and
a simple perturbation argument.
\end{proof}

\section{Strang splitting with Fourier collocation for Allen-Cahn}
In this section we consider the Strang splitting with Fourier collocation for Allen-Cahn
on $\mathbb T$. The PDE is again
\begin{align*}
\partial_t u = \nu^2 \Delta u + f(u),\qquad f(u)=u-u^3.
\end{align*}
We shall adopt the same notation as in Section \ref{S:DFT1}.  In particular we slightly
abuse the notation and denote for $U\in \mathbb R^N$
\begin{align*}
f(U)= U - U^{\cdot 3}.
\end{align*}

We consider the time splitting
as follows. Let $\tau>0$ be the time step.  Consider the ODE
\begin{align*}
\begin{cases}
\frac d {dt} U = f(U), \\
U\Bigr|_{t=0} = a \in \mathbb R^N.
\end{cases}
\end{align*}
We define the solution operator $\mathcal N_{\tau}:\, \mathbb R^N
\to \mathbb R^N$ as the map $a \to U(\tau)$.  Thanks to
the explicit form of $f(U)$, we have ( $a=(a_0,\cdots, a_{N-1})^T$ and with no loss
we shall assume that $N\ge 2$ )
\begin{align*}
(\mathcal N_{\tau} a)_j = U(\tau)_j = \frac {a_j}
{ \sqrt{ e^{-2\tau} + (1-e^{-2\tau} ) (a_j)^2} },
\qquad j=0, 1,\cdots, N-1.
\end{align*}

Define $\mathcal S_{\tau}= e^{\frac {\tau}2 \nu^2 \Delta_h}$. 
Then $U^{n+1}$, $U^n \in \mathbb R^N$
are related via the relation:
\begin{align} \label{Dec5:0}
U^{n+1}= 
\mathcal S_{\tau} \mathcal N_{\tau} \mathcal S_{\tau} U^n, 
\qquad n\ge 0.
\end{align}

For $U\in \mathbb R^N$, we recall 
\begin{align*}
\|U\|_2^2 = \frac 1 N \sum_{j=0}^{N-1} U_j^2 = \sum_{-\frac N2<k
\le \frac N2} 
|\tilde U_k |^2.
\end{align*}
Denote 
\begin{align*}
\| \partial_h U \|_2^2 = \sum_{-\frac N2 <k \le \frac N2}
(2\pi k)^2 | \tilde U_k|^2.
\end{align*}

\begin{lem} \label{Dec5:1}
Consider \eqref{Dec5:0}. We have
\begin{align} \label{Dec5:1a}
\sup_{n\ge 0} \|U^n \|_2 \le \max\{ 1, \|U^0\|_2 \}.
\end{align}
Furthermore assuming $\|U^0\|_2\le 1$,  there are $\tau_0=\tau_0(\nu)>0$,
$t_0=t_0(\nu)>0$  sufficiently small such that the following hold:
If  $\tau\ge \tau_0$, then
\begin{align} \label{Dec5:1b}
\sup_{n\ge 1} \| \partial_h U^n \|_2 \le C_1,
\end{align}
where $C_1>0$ depends only on $\nu$.
If $0<\tau<\tau_0$, then 
\begin{align} \label{Dec5:1c}
\sup_{1\le n\le t_0/\tau} \sqrt{n\tau} \| \partial_h U^n \|_2+ 
 \sup_{n\ge t_0/\tau} \|\partial_h U^n \|_2 \le  C_2,
\end{align}
where $C_2>0$ depends only on ($\nu$, $t_0$).

Suppose $\|U^0\|_{\infty} \le 1$, then for all $\tau>0$, we have
\begin{align} \label{Dec5:1d}
\sup_{n\ge 1} \| U^n \|_{\infty} \le C_3,
\end{align}
where $C_3>0$ depends only on $\nu$. Also for any $t_1\ge \tau$ and
$n\tau \le t_1$, we have
\begin{align} \label{Dec5:1e}
\sup_{n\le t_1 /\tau} \| U^n - \mathcal S_{2n\tau} U^0\|_{\infty}
\le C_4 t_1,
\end{align}
where $C_4>0$ depends only on $\nu$.
\end{lem}
\begin{proof}
For the basic $L^2$ bound we can just use the ODE. Clearly
\begin{align*}
\frac 12 \frac d {dt} ( \| U \|_2^2) & = \|U\|_2^2 - \|U \|_4^4 \notag \\
& \le \|U \|_2^2 - \|U\|_2^4. 
\end{align*}
Thus $\|U(t) \|_2 \le \max\{1, \|U(0)\|_2\}$. This immediately yields 
\eqref{Dec5:1a}.

For \eqref{Dec5:1b}, we note that since $\tau \gtrsim 1$, it follows easily that
\begin{align*}
\| \partial_h U^{n+1} \|_2 \lesssim \| \mathcal N_{\tau} \mathcal
S_{\tau} U^n \|_2 \lesssim 1.
\end{align*}

For \eqref{Dec5:1c},  we first rewrite 
\begin{align}
U^{n+1} & = \mathcal S_{2\tau} U^n
+ \mathcal S_{\tau} ( \mathcal N_{\tau} -\op{Id} ) \mathcal S_{\tau} U^n \notag \\
& = :\mathcal S_{2\tau} U^n + \tau \mathcal S_{\tau}  f_n, \label{Dec5:2}
\end{align}
where
\begin{align*}
f_n = \frac { \mathcal  N_{\tau} - \op{Id} }{\tau} \mathcal S_{\tau} U^n.
\end{align*}
Now by using the inequality
\begin{align*}
\Bigl| \frac x {\sqrt{e^{-2\tau} + (1-e^{-2\tau} ) x^2}} -x 
\Bigr|\lesssim \tau (|x|+|x|^3), \qquad \forall\, 0<\tau \lesssim 1, \; x \in \mathbb R,
\end{align*}
We have 
\begin{align*}
|f_n| \lesssim  | \mathcal S_{\tau} U^n|+ |\mathcal S_{\tau} U^n|^3.
\end{align*}
By \eqref{Dec5:2}, we have
\begin{align}
U^n &= \mathcal S_{2n\tau} U^0 + \tau \sum_{j=0}^{n-1}   \mathcal S_{(2(n-j)-1)\tau} f_j
\notag \\
& =  \mathcal S_{2n\tau} U^0 + \tau \mathcal S_{(2n-1)\tau} f_0+ 
 \tau \sum_{j=1}^{n-1}   \mathcal S_{(2(n-j)-1)\tau} f_j, 
\quad\forall\, n\ge 1. \label{Dec5:3}
\end{align}
Note that in the above, the empty summation when $n=1$ is defined to be zero.
For $1\le j\le n-1$, note that 
\begin{align*}
\| f_j \|_1 & \lesssim \| \mathcal S_{\tau} U^j \|_1 + \| \mathcal S_{\tau} U^j \|_3^3 \notag \\
& \lesssim \| \mathcal S_{\tau} U^j \|_2 + \|\mathcal S_{\tau} U^j \|_{\infty} \| U^j \|_2^2  \notag \\
& \lesssim 1+ \|\mathcal S_{\tau} U^j \|_{\infty}.
\end{align*}
To bound $\|\mathcal S_{\tau} U^j \|_{\infty}$, we observe that for $U\in \mathbb R^N$,
\begin{align*}
\| U\|_{\infty} &\le \| Q_N U\|_{L_x^{\infty}(\mathbb T)}
\lesssim \|Q_N U\|_{L_x^2(\mathbb T)}+ \|Q_N U \|_{L_x^2(\mathbb T)}^{\frac 12} \| 
\partial_x Q_N U \|_{L_x^2(\mathbb T)}^{\frac 12} \notag \\
& \lesssim  \| U\|_2+ \| U \|_2^{\frac 12} \| \partial_h U \|_2^{\frac 12}.
\end{align*}
Thus for $1\le j\le n-1$, 
\begin{align*}
\| f_j\|_1 \lesssim 1 + \| \partial_h U^j \|_2^{\frac 12}.
\end{align*}
For $j=0$, we have
\begin{align*}
\| f_0 \|_1 &\lesssim 1 + \| \mathcal S_{\tau} U^0\|_{\infty} \lesssim 1+
\| \partial_h S_{\tau} U^0\|_2^{\frac 12} \notag \\
& \lesssim  1+ \tau^{-\frac 12}.
\end{align*}
Denote the kernel corresponding to $\mathcal S_{l\tau}$ as $K_{l\tau}$ ($l\ge 1$), then clearly
\begin{align*}
&\| K_{l\tau} \|_{2} \lesssim  (l \tau)^{-\frac 14}, \quad
\| \partial_h K_{l\tau} \|_2 \lesssim (l \tau)^{-\frac 34}, \quad 
\forall\, l\tau \lesssim 1 ;\\
&\| \partial_h  \mathcal S_{l\tau} U \|_2 \lesssim  (l \tau)^{-\frac 34} \|U \|_1, \qquad
\forall\, U \in \mathbb R^N, 
\; l\tau \lesssim 1.
\end{align*}
We then obtain the recurrent relation,
\begin{align*}
F_n=\sqrt{n\tau} \| \partial_h U^n\|_2
\lesssim 1 + \tau \sqrt{n\tau}
\sum_{j=1}^{n-1} ((n-j)\tau)^{-\frac 34}
( 1+ (j\tau)^{-\frac 14} F_j^{\frac 12}),
\end{align*}
where $F_j = (j\tau)^{\frac 12} \| \partial_h U^j\|_2$. 
Let $n_0$ be the first integer such that $n_0\tau\ge t_0$. 
Our desired result for $U^{n_0}$ then easily follows
from a simple induction procedure. Note that we need  $n_0\tau \le t_0 +\tau \ll 1$. 
For general $U^n$ we need to estimate it using $U^{n_0}$ where $(n-n_0)\tau =O(t_0)$. 

Next for \eqref{Dec5:1d}, we only need to consider the regime
$0<\tau \ll 1$ and $n\tau \le t_1$ for some $t_1 \ll 1$. By Theorem \ref{N18:1} and
using \eqref{Dec5:3}, we have
\begin{align*}
\| U^n\|_{\infty} \lesssim \|U^0\|_{\infty}
+ \tau \sum_{j=0}^{n-1} (\|U^j \|_{\infty} + \|U^j \|_{\infty}^3).
\end{align*}
A simple induction argument then yields the desired estimate.  Finally the estimate \eqref{Dec5:1e}
follows easily from \eqref{Dec5:1d}.
\end{proof}

\begin{thm}[Almost sharp maximum principle for \eqref{Dec5:0}, case
$ N^{-0.4} \le \tau <\infty$]
Consider \eqref{Dec5:0}. 
Suppose $U^0 \in \mathbb R^N$ satisfies 
$\| U^0\|_{\infty} \le 1$. Assume $
N^{-0.4} \le \tau<\infty$. If $N\ge N_1(\nu)>0$ (the dependence of $N_1$ on $\nu$ is only 
power-like), then  for all $n\ge 0$, we have
\begin{align}  \notag
&\| U^n \|_{\infty} \le 1+ O(N^{-0.5}).
\end{align}
\end{thm}
\begin{proof}
Since $\tau \ge N^{-0.4}$, by Theorem \ref{N18:1}, we have
\begin{align*}
\|\mathcal S_{\tau} U \|_{\infty} \le   (1+O(N^{-1} )) \|U \|_{\infty},
\qquad\forall\, U \in \mathbb R^N.
\end{align*}

On the other hand, if $\|U\|_{\infty} \le 1+\epsilon_1$ where $\epsilon_1 \ll 1$, then 
it is not difficult to check that for $\tau>0$,
\begin{align*}
\| \mathcal N_{\tau} U \|_{\infty} \le 1 + e^{-2\tau} \cdot   \Bigl|\|U\|_{\infty}-1\Bigr|
+
e^{-4\tau} \cdot C_1 \cdot \Bigl|\|U\|_{\infty}-1\Bigr|^2,
\end{align*}
where $C_1>0$ is an absolute constant. 

Our main induction hypothesis is $\|U^n\|_{\infty} \le 1 +\delta$, where $\delta=O(N^{-0.5})$.
Then it suffices to verify the inequality:
\begin{align*}
(1+O(N^{-1}) ) \Bigl(1+ e^{-2\tau} (O(N^{-1}) + \delta)+
e^{-4\tau} C_1 \cdot (O(N^{-2})  + \delta^2) \Bigr)
\le 1+\delta.
\end{align*}
This amounts to showing 
\begin{align*}
O(N^{-1}) +O(\delta^2) \le (1- e^{-2\tau} ) \delta.
\end{align*}
Since $\delta=O(N^{-0.5})$ and $\tau\ge N^{-0.4}$, this last inequality is obvious for $N$
sufficiently large.
\end{proof}

\begin{lem} \label{leDec8:1}
Assume $N\ge 2$. Suppose $u$ is a smooth solution to the equation:
\begin{align*}
\begin{cases}
\partial_t u = Q_N ( f(u) ), \\
u(0) =a,
\end{cases}
\end{align*}
where $f(z)=z-z^3$, and the initial data $a$ satisfies $\| a \|_{H^1(\mathbb T)} \le A_0$ for some 
constant $A_0\ge 1$. We also assume that $\widehat{a}$ is supported in $\{k:\, -\frac N2 <k\le \frac N2\}$ (in particular this implies $Q_N a = a$).
Then for some $t_0 >0$ depending only on $A_0$ (one can take $t_0 \sim A_0^{-2}$),
it holds that 
\begin{align*}
\sup_{0\le t \le t_0} \| u(t) \|_{H^1} \le \alpha A_0,
\end{align*}
where $\alpha>0$ is an absolute constant.
\end{lem}
\begin{proof}
Denote $J=\{k\in \mathbb Z:\, -\frac N2<k\le \frac N2\}$. 
Observe that for any $k \in J$, we have
\begin{align*}
|\hat u (t, k)|
\le |\hat a (k) | + \int_0^t |\hat u(s, k) | ds
+ \sum_{|l|\le 2}
\int_0^t \sum_{k_1, k_2 \in J}
|\hat u(s,k+lN -k_1-k_2)| |\hat u(s, k_1) | |\hat u(s,k_2) | ds.
\end{align*}
Clearly for $|l|\le 2$ and $k, k_1, k_2 \in J$, we have
\begin{align*}
|k_1|+|k_2| + |k+lN-k_1-k_2|
\ge |k+lN| \ge \frac 14|k|.
\end{align*}
Define
\begin{align*}
A(t) = 1+\sum_{k\in J} |\hat u(t, k) |^2 \cdot (1+ k^2).
\end{align*}
It is then not difficult to show that 
\begin{align*}
A(t) \le A(0) + \int_0^t  C \cdot A(s)^{\frac 32} ds,
\end{align*}
where $C>0$ is an absolute constant. The desired result then follows easily by taking $t$
sufficiently small.
\end{proof}

\begin{thm}[Almost sharp maximum principle for \eqref{Dec5:0}, case
$ 0< \tau <N^{-0.4} $]
Consider \eqref{Dec5:0}. 
Suppose $U^0 \in \mathbb R^N$ satisfies 
$\| U^0\|_{\infty} \le 1$. Assume $
0< \tau<N^{-0.4}$. If $N\ge N_2(\nu)>0$ (the dependence of $N_2$ on $\nu$ is only 
power-like), then 
\begin{align}  \notag
\sup_{n\ge 0} \| U^n \|_{\infty} \le 1+ \epsilon_1,
\end{align}
where $0<\epsilon_1<10^{-2}$ is an absolute constant. More precisely the following hold.
\begin{enumerate}
\item For some $T_0=O(N^{-0.3})$ sufficiently small and for all $n\ge T_0/\tau$, we have
\begin{align*}
& \| U^n \|_{\infty} \le \| u^n\|_{\infty} \le 1+ O(N^{-0.29}),
\end{align*}
where $u^n$ solves \eqref{Dec8:2}.
\item For $1\le n \le T_0/\tau$, we have 
\begin{align*}
& \|U^n - \mathcal S_{2n\tau} U^0 \|_{\infty}
\le O(N^{-0.3}); \\
& \|
 \mathcal S_{2n\tau} U^0 
\|_{\infty} \le (1+\frac 1 2 \epsilon_1) \|U^0\|_{\infty};  \\
& \|U^n\|_{\infty} \le (1+\frac 12 \epsilon_1)
\|U^0\|_{\infty} + O(N^{-0.3}).
\end{align*}
\item If $\|U^0\|_{\infty} \le \frac 1 {1+\frac 12 \epsilon_1}$ 
(note that $1/(1+\frac 12 \epsilon_1)>0.995$), then 
\begin{align*}
\sup_{n\ge 0} \|U^n\|_{\infty} \le 1 + O(N^{-0.3}).
\end{align*}

\item Suppose $f:\, \mathbb T\to \mathbb R$ is continuous and $\|f\|_{\infty} \le 1$. 
If $(U^0)_l= f(\frac l N)$ for all $0\le l\le N-1$,  then
\begin{align*}
\sup_{n\ge 0} \|U^n\|_{\infty} 
\le 1 + O(\omega_f(N^{-\frac 23}) ) + O(N^{-c}),
\end{align*}
where $c>0$ is an absolute constant, and $\omega_f$ is defined in \eqref{N18:1b.0}.

\item  Suppose $f:\, \mathbb T\to \mathbb R$ is $C^{\alpha}$-continuous for some $0<\alpha<1$ and $\|f\|_{\infty} \le 1$.  
If $(U^0)_l= f(\frac l N)$ for all $0\le l\le N-1$,  then
\begin{align*}
\sup_{n\ge 0}\|U^n\|_{\infty}
\le 1 + O(N^{-c_1}),
\end{align*}
where $c_1>0$ is a constant depending only on $\alpha$. 
\end{enumerate}
Moreover we have the following result which shows the sharpness of our estimates above.
There exists a  function $f$: $\mathbb T \to \mathbb R$, continuous at all of $\mathbb T
\setminus \{x_*\}$ for some $x_*\in \mathbb T$ (i.e. continuous at all $x\ne x_*$) and  has the bound $\|f\|_{\infty} \le 1$
 such that 
the following hold:
for a sequence of even numbers $N_m \to \infty$, $\tau_m =\frac 14 \nu^{-2}N_m^{-2}$
, $n_m\ge 1$,  and $x_m
=j_m/N_m$ with $0\le j_m\le N_m-1$, if $\tilde U_m \in \mathbb R^{N_m}$ satisfies
\begin{align*}
\tilde U_m = \underbrace{e^{\frac 1 2 \tau_m \Delta_h} \mathcal N_{\tau_m} e^{\frac 12
\tau_m \Delta_h} \cdots 
e^{\frac 1 2 \tau_m \Delta_h} \mathcal N_{\tau_m} e^{\frac 12
\tau_m \Delta_h}}_{\text{iterate $n_m$ times}}  a_m,
\end{align*}
where $a_m \in \mathbb R^{N_m} $ satisfies 
$(a_m)_l = f(\frac l {N_m})$ for all $0\le l\le N_m-1$. Then 
\begin{align*}
|\tilde U_m(x_m) | \ge 1+ \eta_*, \qquad\forall\, m\ge 1,
\end{align*}
where $\eta_*>0.001$ is an absolute constant.
\end{thm}
\begin{proof}
We consider $u^{n+1}$, $u^n:\, \mathbb T\to \mathbb R$ defined via the relation:
\begin{align} \label{Dec8:2}
u^{n+1} = e^{\frac {\tau} 2 \partial_{xx} } V_{\tau} e^{\frac {\tau} 2 \partial_{xx}} u^n,
\end{align}
where $V_{\tau}$ is the time-$\tau$ solution operator to the problem:
\begin{align*}
\partial_t u = Q_{N} ( f(u) ).
\end{align*}
It is not difficult to check that if $u^0 = Q_N U^0$, then $u^n = Q_N U^n$ for all $n\ge 1$. 

Choose $n_0\ge 1$ such that $n_0 \tau \sim N^{-0.3}$. By Lemma \ref{Dec5:1}, we have
\begin{align*}
\| U^{n_0} -\mathcal S_{2n_0 \tau} U^0\|_{\infty}
\le O(N^{-0.3}).
\end{align*}
Thus
\begin{align*}
\| u^{n_0} - Q_{N} ( \mathcal S_{2n_0\tau} U^0 ) \|_{\infty} 
\le O(N^{-0.3} \log N).
\end{align*}
Now 
\begin{align*}
\Bigl( Q_N (\mathcal S_{2n_0\tau} U^0) \Bigr)(x) & = \frac 1 N \sum_{j=0}^{N-1} 
U_j^0 \op{Re}\Bigl(
\sum_{-\frac N2 <k \le \frac N2}
e^{-(2\pi k \nu)^2 2n_0 \tau} e^{2\pi i k\cdot (x- \frac j N) } \Bigr).
\end{align*}
Since $\|U^0\|_{\infty} \le 1$ and $n_0\tau \sim N^{-0.3}$, it easily follows that
\begin{align*}
\| Q_{\mathcal S_{2n_0\tau}} U^0 \|_{\infty} \le 1+ \sum_{|k|\ge \frac N2}
e^{-(2\pi k \nu)^2 2n_0\tau} \le 1 +O(N^{-10}).
\end{align*}
Thus 
\begin{align*}
\| u^{n_0} \|_{\infty} \le 1 + O(N^{-0.3} \log N).
\end{align*}

Then for $n\ge n_0$,  we consider the problem
\begin{align*}
\begin{cases}
\partial_t u = Q_N (f(u) ) = f(u) + (Q_N-\op{Id} ) (f(u) ), \\
u(0) =e^{\frac{\tau}2 \partial_{xx}} u^n=e^{\frac{\tau} 2 \partial_{xx}} Q_N U^n.
\end{cases}
\end{align*}
By Lemma \ref{leDec8:1} (note that $\tau <N^{-0.4}$ can be made sufficiently small
by taking $N$ large), we have
\begin{align*}
\sup_{0\le t \le \tau} \| u(t) \|_{H^1}  =O(1).
\end{align*}
By using Lemma \ref{Dec5:1}, we have 
\begin{align*}
\| (Q_N- \op{Id} ) (f(u) ) \|_{\infty}
&\lesssim N^{-\frac 12}.
\end{align*}
If we assume that $\|u(0)\|_{\infty} \le 1+\delta_0$ for some $\delta_0
=O(N^{-0.29})$, then
a simple maximum principle argument  yields that 
\begin{align*}
\sup_{0\le t \le \tau} \| u(t) \|_{\infty} \le 1+\delta_0.
\end{align*}
Consequently
\begin{align*}
\sup_{n\ge n_0} \|u^n \|_{\infty} \le 1 + O(N^{-0.29}).
\end{align*}

Finally we consider the regime $1\le n\le n_0$. 
Clearly we have 
\begin{align*}
&\sup_{1\le n\le n_0} \| U^{n} -\mathcal S_{2n\tau} U^0\|_{\infty}
\le O(N^{-0.3}), \notag \\
& \sup_{1\le n\le n_0} \| u^{n} - Q_{N} ( \mathcal S_{2n\tau} U^0 ) \|_{\infty} 
\le O(N^{-0.3} \log N).
\end{align*}
We are then in a situation similar to that in Theorem \ref{N21_5}. The argument is then
similar and thus omitted. 
\end{proof}
\begin{rem}
We should point out that, by using the argument in the preceding proof,
it is also possible to obtain the result $\|U^n\|_{\infty} \le \|u^n\|_{\infty}
\le 1 + O(N^{-c})$ for
the case $N^{-0.4} \le \tau \le \tau_1$, where $\tau_1 \ll 1$ is a constant determined
in Lemma \ref{leDec8:1} ($\tau_1=t_0$ in the notation therein). We briefly
sketch the argument as follows. Again
 consider $u^{n+1}$, $u^n:\, \mathbb T\to \mathbb R$ defined via the relation:
\begin{align*}
u^{n+1} = e^{\frac {\tau} 2 \partial_{xx} } V_{\tau} e^{\frac {\tau} 2 \partial_{xx}} u^n,
\end{align*}
where $V_{\tau}$ is the time-$\tau$ solution operator to the problem:
\begin{align*}
\partial_t u = Q_{N} ( f(u) ).
\end{align*}
Denote $v^1 = e^{\frac{\tau}2 \partial_{xx} } u^0$.  Clearly since 
$\tau \ge N^{-0.4}$, 
\begin{align*}
\| v^1 \|_{\infty} \le 1+ \sum_{|k|\ge \frac N2}
e^{-(2\pi k \nu)^2 \tau} \le 1 +O(N^{-10}).
\end{align*}
Then we consider the problem
\begin{align*}
\begin{cases}
\partial_t u = Q_N (f(u) ) = f(u) + (Q_N-\op{Id} ) (f(u) ), \\
u(0) =e^{\frac{\tau} 2 \partial_{xx}} Q_N U^n.
\end{cases}
\end{align*}
By using Lemma \ref{Dec5:1} and Lemma \ref{leDec8:1}, we have 
\begin{align*}
\| (Q_N- \op{Id} ) (f(u) ) \|_{\infty}
&\lesssim N^{-\frac 12} \| \partial_x u \|_2 (1+ \|u \|_{\infty}^2) \notag \\
& \lesssim  N^{-\frac 12} (1+ (\log N)^2)
+ N^{-\frac 12} (1+(\log N)^2) \cdot \tau^{-\frac 12}.
\end{align*}
In the above last inequality, the first term is due to the case $\tau \sim 1$,
and the second term accounts for the case $\tau \ll 1$.   Since we assumed
$\tau\ge N^{-0.4}$, we then obtain
\begin{align*}
\| (Q_N- \op{Id} ) (f(u) ) \|_{\infty}  \le O(N^{-0.29}). 
\end{align*}
A maximum principle argument then yields that $\|u^n\|_{\infty}
\le 1 +O(N^{-c})$. 
\end{rem}
\begin{rem}
Denote $\theta=e^{-2\tau}$ and for $u:\, \mathbb T\to \mathbb R$, 
\begin{align*}
(T_{\theta} u )(x) = \frac {u(x)} { \sqrt{\theta +(1-\theta) u(x)^2}}, \quad x \in \mathbb T.
\end{align*}
Clearly we have $Q_{N} \mathcal N_{\tau} U = Q_{N} T_{\theta} Q_N U $ for
any $U\in \mathbb R^N$. It is not difficult
to check that for any $f \in H^1(\mathbb T)$, we have
\begin{align*}
\| Q_N f \|_{H^1(\mathbb T)} \lesssim \| f \|_{H^1(\mathbb T)}.
\end{align*}
Note that the function $g(z)=z/\sqrt{\theta+(1-\theta)z^2 }$ satisfies
\begin{align*}
\sup_{\theta\gtrsim 1} \| g^{\prime}\|_{L^{\infty}(\mathbb R)} \lesssim 1.
\end{align*}

It follows that for $0<\tau\lesssim 1$ (i.e. $\theta \gtrsim 1$), 
\begin{align*}
\| Q_N \mathcal N_{\tau} U \|_{H^1(\mathbb T)}
\lesssim  \| T_{\theta} Q_N U \|_{H^1(\mathbb T)}
\lesssim \| Q_N U \|_{H^1(\mathbb T)} \lesssim 
\| \partial_h U\|_2+ \|U\|_2.
\end{align*}
In particular this implies that we can improve the uniform-in-time $H^1$ estimate
in Lemma \ref{leDec8:1} to  any finite
$t_0\sim 1$. A further interesting issue is to understand the regime $t_0\gg 1$. However
in this regime one should not expect uniform-in-time $H^1$ bounds. A counterexample is
as follows. Take $\phi(x)= \cos 2\pi A x$ where $A\gg 1$ is an integer. Take $0< \theta \ll A^{-10}$
and $u=\sqrt{\theta} \phi$.  Then clearly $\|u \|_{H^1(\mathbb T)} \ll 1$ while
\begin{align*}
\| \partial_x T_{\theta} u \|_{L_x^2(\mathbb T)} \gtrsim A \gg 1.
\end{align*}
By taking $N$ sufficiently large, we obtain
\begin{align*}
\| \partial_x Q_N T_{\theta} Q_N u \|_{L_x^2(\mathbb T)} \gtrsim A \gg 1.
\end{align*}
Thus in the regime $t_0\gg 1$ one does not have uniform-in-time $H^1$-norm bounds.
\end{rem}

\section{Maximum principle for spectral Burgers}
Consider
\begin{align} \label{Dec2_0}
\begin{cases}
\partial_t u + \Pi_N ( u u_x) = \nu^2 u_{xx}, \quad (t,x) \in (0,\infty) \times \mathbb T; \\
u\Bigr|_{t=0} = u^0= \Pi_N f,
\end{cases}
\end{align}
where $f \in L^{\infty}(\mathbb T)$. 

\begin{thm}
Let $\nu>0$ and $u$ be the solution to \eqref{Dec2_0}. 
We have for all $N\ge 2$, 
\begin{align*}
\sup_{t\ge 0} \|u(t) \|_{\infty}
\le \|u^0\|_{\infty} + \alpha \cdot N^{-c}, 
\end{align*}
where $c>0$ is an absolute constant, and $\alpha>0$ depends only on 
($\|f\|_{\infty}$, $\nu$). 
\end{thm}
\begin{rem}
The dependence of $\alpha$ on $1/\nu$ is at most power-like for $0<\nu \lesssim 1$.
\end{rem}
\begin{rem}
If one work with $u^0=\tilde P_N f$, where $\tilde P_N$ is a nice Fourier projector
such as the Fejer's kernel which satisfies $\|\tilde P_N f \|_{\infty} \le \|f\|_{\infty}$ and
$\operatorname{supp}(\tilde P_N f ) \subset \{ k:\, -\frac N2 <k \le \frac N2\}$,  then
we obtain in this case 
\begin{align*}
\sup_{t\ge 0} \|u(t) \|_{\infty}
\le \|f\|_{\infty} + O(N^{-c}).
\end{align*}
Such pre-processing of initial data is quite easy to implement in practice.
\end{rem}
\begin{proof}
Denote $\bar u$ as the average of $u$ on $\mathbb T$ which is clearly preserved in time.
Let $v = u - \bar u$.  Then 
\begin{align*}
\frac 12 \frac d {dt} \Bigl( \| v \|^2_2 \Bigr)
= - \nu^2 \| \partial_x v \|_2^2  \le  - 4 \pi^2 \nu^2 \|v \|_2^2.
\end{align*}
Thus 
\begin{align*}
\| v(t) \|_{2} \le  e^{-4\pi^2 \nu^2 t}  \|v(0)\|_2, \qquad\forall\, t\ge 0.
\end{align*}
By using this and standard smoothing estimates, we then obtain
\begin{align*}
&\| v(t) \|_{\infty}  \le C_1 t^{-\epsilon_0} e^{-c_1 t \nu^2}, \quad\forall\, t> 0; \\
& \| \partial_x v(t) \|_{2} \le C_1 t^{-\frac 12} e^{-c_1 t \nu^2},
\qquad\forall\, t>0; \\
&\| |\partial_x|^{\frac 18} \partial_x v \|_{\infty} 
\le C_1  t^{-\epsilon_0-\frac 7{16}}  e^{-c_1 t \nu^2}, \quad\forall\, t>0. 
\end{align*}
where the constant $C_1>0$ depend only on ($\|f\|_{\infty}$, $\nu$),   $c_1>0$ is an absolute constant,
and $0<\epsilon_0 \ll1$ is a small absolute constant.  In the above the pre-factor
$t^{-\epsilon_0}$ is due to the fact that we work with $\| v(0) \|_p \lesssim
\|f\|_p$ for some $p$ sufficiently
large (instead of using $\|v(0) \|_{\infty}$). It follows that
\begin{align*}
\| (\op{Id} - \Pi_N) ( u u_x) \|_{\infty}
=\| (\op{Id} - \Pi_N ) ( (\bar u +v) v_x) \|_{\infty}
\le C_2 N^{-c} t^{-\epsilon_1} e^{- c_1 t \nu^2}, \qquad \forall\, t>0, 
\end{align*}
where $c>0$ is an absolute constant, $C_2>0$ depends only on 
($\|f\|_{\infty}$, $\nu$), and $0<\epsilon_1<1$ is another absolute constant.  Thus we may
rewrite the original equation as
\begin{align*}
\partial_t u +u u_x = \nu^2 u_{xx} + F,
\end{align*}
where  $F$ obeys
\begin{align*}
\| F\|_{\infty} \le C_2 N^{-c} t^{-\epsilon_1} e^{-c_1 t \nu^2}, \qquad\forall\, t>0.
\end{align*}
Since by spectral truncation $u^0$ is clearly smooth, the function $u$ is also smooth. 
Consider
\begin{align*}
g(t,x) = u(t,x) -\Bigl( \|u^0\|_{\infty} + a(t) \Bigr),
\end{align*}
where 
\begin{align*}
a(t) = C_2 N^{-c} \int_0^t s^{-\epsilon_1} e^{-c_1 s \nu^2} ds.
\end{align*}
Clearly for any $t_0>0$ and $t\ge t_0$, we have $g$ is smooth and
\begin{align*}
\partial_t g +u g_x  \le \nu^2 g_{xx},  \qquad t\ge t_0.
\end{align*}
A simple maximum principle argument then yields 
\begin{align*}
\max_{x\in \mathbb T} g(t,x) \le \max_{x\in \mathbb T} g(t_0,x), \qquad \forall\, t\ge t_0.
\end{align*}
Sending $t_0\to 0+$ then yields the desired estimate for the upper bound. The proof
for the lower bound is similar. We omit the details. 
\end{proof}

Consider
\begin{align} \label{Dec2_2}
\begin{cases}
\frac{u^{n+1}-u^n} 
{\tau}  + \Pi_N ( u^n \partial_x u^n) = \nu^2 
\partial_{xx} u^{n+1}, \quad (t,x) \in (0,\infty) \times \mathbb T; \\
u\Bigr|_{t=0} = u^0= \Pi_N f,
\end{cases}
\end{align}
where $f \in L^{\infty}(\mathbb T)$. 

\begin{thm}
Let $\nu>0$ and $u^n$ be the solution to \eqref{Dec2_2}. 
We have for all $0<\tau<\tau_0=\tau_0(\nu, \|f\|_{\infty})$ and $N\ge 2$, 
\begin{align*}
\sup_{n\ge 0} \|u^n \|_{\infty}
\le \|u^0\|_{\infty} + \alpha_1 \cdot N^{-\frac 12} +\alpha_2 \tau^{0.49}, 
\end{align*}
where $\alpha_1, \alpha_2>0$ depend only on 
($\|f\|_{\infty}$, $\nu$). 
\end{thm}
\begin{rem}
For convenience we do not spell out the explicit dependence of $\tau_0$ on $\nu$.
\end{rem}
\begin{proof}
To ease the notation we shall set $\nu=1$. 
Rewrite \eqref{Dec2_2} as
\begin{align} \label{Dec10:1}
u^{n+1} = A u^n - \frac 12 \tau A \Pi_N \partial_x (  (u^n)^2),
\end{align}
where $A=(\op{Id}-\tau \partial_{xx} )^{-1}$.  

Denote $v^n = u^n - \bar u$, where (below we denote $\int =\int_{\mathbb T} dx$)
\begin{align*}
\bar u = \int u^0 = \int f.
\end{align*}
Then clearly 
\begin{align} \label{Dec3:1}
\frac{v^{n+1}-v^n}
{\tau} + \Pi_N ( u^n \partial_x v^n) =  \partial_{xx} v^{n+1}.
\end{align}

We shall proceed in several steps. 

Step 1. Local smoothing estimates. 
Let $t_0>0$ be chosen sufficiently
small (depending only on $\|f\|_{\infty}$) so that the nonlinear part is dominated by the linear part when $n\tau\le t_0$. 
Choose $1\le n_0 \in \mathbb Z$ such that $n_0 \tau \sim t_0$. By using
\eqref{Dec10:1} and discrete smoothing estimates, we have
\begin{align}
&\sup_{1\le n \le n_0}( \sqrt{n\tau} \| \partial_x u^n \|_2 + 
\|u^n\|_{2}) \le C_1, \\
& \| \partial_x u^{n_0} \|_{\infty} +\| u^{n_0}\|_{\infty} \le C_1,
\end{align}
where $C_1>0$ depends only on $\|f\|_{\infty}$.

Step 2. $L^2$ estimate. Multiplying both sides of \eqref{Dec3:1} by $v^{n+1}$ and
integrating, we obtain
\begin{align*}
\frac 1 {2\tau} 
( \| v^{n+1}\|_2^2 - \| v^n \|_2^2
+ \| v^{n+1} -v^n \|_2^2)
+  \| \partial_x v^{n+1}\|_2^2
+ N_1 \le 0,
\end{align*}
where
\begin{align*}
N_1 & = \int \Pi_N (u^n \partial_x v^n) v^{n+1} 
= \int u^n \partial_x v^n ( v^{n+1} -v^n )   \notag \\
& =
\int u^n (\partial_x v^{n+1}) (v^{n+1}-v^n) 
+ \int u^n \partial_x (v^n - v^{n+1}) \cdot (v^{n+1} -v^n) dx.
\end{align*}
Clearly it holds that 
\begin{align*} 
|N_1|&\le \|u^n\|_{\infty} \cdot \| \partial_x v^{n+1} \|_2
\| v^{n+1} -v^n \|_2 + \frac 12 \| \partial_x u^n \|_{\infty}
\| v^{n+1} -v^n \|_2^2 \notag \\
&\le \; \frac 12 \| \partial_x v^{n+1} \|_2^2
+( \frac 1 {2}
\| u^n \|_{\infty}^2 + \frac 12 \| \partial_x u^n \|_{\infty})
\cdot \|v^{n+1}-v^n\|_2^2.
\end{align*}
Thus if 
\begin{align} \label{Dec10:2}
\frac 1 {\tau}
\ge 
\| u^n \|_{\infty}^2 + \|\partial_x u^n \|_{\infty},
\end{align}
then 
\begin{align*}
(1+ 4\pi^2 \tau) \| v^{n+1}\|_2^2
\le \|v^{n+1} \|_2^2+
  \tau \| \partial_x v^{n+1} \|_2^2 \le \|v^n\|_2^2.
\end{align*}
By a similar estimate and under the same condition \eqref{Dec10:2}, we also obtain
\begin{align*}
\| u^{n+1} \|_2 \le \|u^n \|_2.
\end{align*}

Step 3. Induction. For $n\ge n_0$, we inductively assume
\begin{align*}
\| \partial_x u^n \|_{\infty} \le  A_1 \tau^{-\frac 34},
\quad \|u^n\|_{\infty} \le A_1 \tau^{-\frac 14},
\quad \| u^n\|_{2} \le C_1,
\end{align*}
where $A_1$ is a suitably large constant. Throughout the argument we shall 
assume $\tau>0$ is sufficiently small, in particular it has to satisfy
\begin{align*}
\tau^{-1} \ge A_1^2 \tau^{-\frac 12} + A_1 \tau^{-\frac 34}
\end{align*}
and some additional mild constraints in the argument below. Now clearly by Step $2$
we have $\|u^{n+1} \|_2 \le C_1$.  By using \eqref{Dec10:1}, we have
(below $\beta_i$ denotes absolute constants)
\begin{align*}
\| \partial_x u^{n+1} \|_{\infty}
& \le \tau^{-\frac 34} \cdot \beta_1 \cdot C_1
+ \beta_2 \| \Pi_N ( (u^n)^2) \|_{\infty} \notag \\
& \le \tau^{-\frac 34} \beta_1 C_1
+ \beta_3 \Bigl(
1+ \|u^n\|_{\infty}^2 
\log ( 10 + \|\partial_x u^n \|_{\infty} \| u^n \|_{\infty} ) \Bigr) \notag \\
& \le \tau^{-\frac 34} \beta_1 C_1
+ \beta_3 \Bigl( 1+ \tau^{-\frac 12} A_1^2
\log( 10+ A_1^2 \tau^{-1} ) \Bigr) \notag \\
& \le A_1 \tau^{-\frac 34},
\end{align*}
where in the last step we choose $A_1\ge 2\beta_1 C_1$  and take $\tau$ sufficiently small.
By a similar estimate, it is not difficult to check that
\begin{align*}
\| u^{n+1}\|_{\infty} \le A_1 \tau^{-\frac 14}.
\end{align*}
This then completes the induction proof and we obtain 
\begin{align*}
& \| u^{n+1} \|_2 \le \|u^n\|_2 \le \|u^{n_0} \|_2, \qquad \forall\, n\ge n_0; \\
& (1+4\pi^2 \tau) \|v^{n+1}\|_2^2 \le  \|v^n\|_2^2, \qquad \forall\, n\ge n_0.
\end{align*}
It follows that
\begin{align*}
\|v^n\|_2 \le C_2 e^{-\beta n\tau}, \qquad \forall\, n\ge 1,
\end{align*}
where $\beta>0$ is an absolute constant, and $C_2>0$ depends only on $\|f\|_{\infty}$.

Step 4. Time-global estimates.  By using the estimates derived in Step 1 to Step 3 and
further bootstrapping estimates, we then obtain
\begin{align*}
& \| v^n \|_{\infty} \le C_3\cdot (n\tau)^{-\epsilon_1} e^{-\gamma_1 n\tau}, \qquad \forall\, n\ge 1;\\
& \| \partial_x u^n \|_{\infty} \le C_3 \cdot (n\tau)^{-\frac 12-\epsilon_2}
e^{-\gamma_2 n \tau}, \qquad \forall\, n\ge 1;\\
& \| \partial_{xx} u^n \|_{\infty} \le C_3 \cdot (n\tau)^{-1-\epsilon_3}
e^{-\gamma_3 n \tau}, \qquad \forall\, n\ge 1;\\
& \| \partial_{xxx} u^n \|_{\infty}
\le C_3 \cdot (n\tau)^{-\frac 32 -\epsilon_4} e^{-\gamma_4 n \tau}, \quad\forall\, n\ge 1;
\end{align*}
where $0<\epsilon_1, \epsilon_2,\epsilon_3,\epsilon_4\ll 1$ are absolute constants,
$\gamma_i>0$ are absolute constants, and $C_3>0$ depends only on $\|f\|_{\infty}$. 
By using \eqref{Dec2_2}, we have
\begin{align*}
&\tau \|u^n \partial_x (u^{n+1} -u^n) \|_{\infty} \le C_5 \cdot \tau^2 
\cdot (n\tau)^{-\frac 32-\epsilon_5}, \qquad \forall\, n\ge 1,
\end{align*}
where $0<\epsilon_5\ll 1$ is an absolute constant, and $C_5>0$ depends only
on $\|f\|_{\infty}$.  Also
\begin{align*}
\tau \| (\op{Id}-\Pi_N) ( u^n \partial_x u^n) \|_{\infty}
\le C_6 \cdot \tau \cdot N^{-\frac 12}
\cdot (n\tau)^{-\frac 34-\epsilon_6} e^{-\gamma_5 n\tau}, \qquad\forall\, n\ge 1,
\end{align*}
where $0<\epsilon_6\ll 1$ is an absolute constant, $\gamma_5>0$ is an
absolute constant, and $C_6>0$ depends only on $\|f\|_{\infty}$. 
Now for $n\ge 1$ we rewrite the equation for $u^{n+1}$ as
\begin{align*}
(\op{Id} -\tau \partial_{xx}) u^{n+1}
=u^n +\tau u^n \partial_x u^{n+1} +\tau u^n \partial_x (u^n -u^{n+1})
+ \tau (\op{Id}-\Pi_N)(u^n \partial_x u^n).
\end{align*}
Clearly by a maximum principle argument, we have
\begin{align*}
\|u^{n+1} \|_{\infty}
&\le \|u^n\|_{\infty} +
\tau \| u^n \partial_x (u^{n+1}-u^n) \|_{\infty}
+\tau \| (\op{Id}-\Pi_N) (u^n \partial_x u^n ) \|_{\infty}  \notag \\
&\le \|u^n \|_{\infty} +C_5 \cdot \tau^2 
\cdot (n\tau)^{-\frac 32-\epsilon_5}
+ C_6 \cdot \tau \cdot N^{-\frac 12}
\cdot (n\tau)^{-\frac 34-\epsilon_6} e^{-\gamma_5 n\tau}, \qquad\forall\, n\ge 1.
\end{align*}
Iterating in $n$,  we obtain
\begin{align*}
\| u^n \|_{\infty}
\le \|u^1 \|_{\infty}
+ C_7 \tau^{\frac 12 -\epsilon_5}  +C_8 N^{-\frac 12}, \qquad\forall\, n\ge 1,
\end{align*}
where $C_7, C_8>0$ are constants depending only on $\|f\|_{\infty}$.  Finally
we observe that
\begin{align*}
\| u^1 \|_{\infty}
& \le \|u^0\|_{\infty}
+\frac 12\| \tau \partial_x A \Pi_N ( (u^0)^2) \|_{\infty} \notag \\
& \le \|u^0\|_{\infty} + C \tau^{\frac 12-\epsilon_6} \|f \|_{\infty}^2,
\end{align*}
where $0<\epsilon_6\ll 1$ is an absolute constant, and $C>0$ is an absolute constant. 
Our desired result then follows easily.
\end{proof}

We now consider the Fourier collocation case. 
Let $U(t) \in \mathbb R^N$, $N\ge 2$ solve
\begin{align} \label{Dec12:1}
\begin{cases}
\frac d {dt} U + \frac 12 \partial_h ( U^{\cdot 2} )=\Delta_h U,  \quad
t>0, \\
U\Bigr|_{t=0} = U^0 \in \mathbb R^N,
\end{cases}
\end{align}
where for convenience we have set the viscosity coefficient $\nu=1$, and the operators
$\partial_h$, $\Delta_h$ correspond to the Fourier multipliers
$2\pi i k$, $-4\pi^2 k^2$ for $-\frac N2<k\le \frac N2$ respectively in the DFT formula.  

Set $u=Q_N U$. We obtain the following reformulation of \eqref{Dec12:1}:
\begin{align} \label{Dec12:2}
\begin{cases}
\partial_t u +\frac 12 \partial_x Q_N (u^2) = \partial_{xx} u, \quad
(t,x) \in (0,\infty)\times \mathbb T; \\
u\Bigr|_{t=0} =u^0= Q_N U^0.
\end{cases}
\end{align}
\begin{rem}
At this point we should point out a subtle technical difficulty associated with
the analysis of \eqref{Dec12:1} and the equivalent system \eqref{Dec12:2}.
Namely in general we have
\begin{align*}
\langle \partial_h ( U^{\cdot 2} ), U \rangle  \ne 0,
\end{align*}
or in terms of $u$:
\begin{align*}
\int_{\mathbb T} Q_N(u^2) \partial_x u dx \ne 0.
\end{align*}
An example can be constructed as follows. Let $N=6 k_0 \ge 6$ where $k_0$ is an integer. 
Set $m= N/3$ and  $u(x) = \sin 2\pi m x$.  It is not difficult to check that
\begin{align*}
Q_N(u^2) = \frac 12 - \frac 12 \cos2\pi m x. 
\end{align*}
Clearly then
\begin{align*}
\int_{\mathbb T} Q_N(u^2) \partial_x u dx \ne 0.
\end{align*}
\end{rem}

\begin{thm} \label{TmDec12:1}
Consider \eqref{Dec12:1} with $N\ge 2$. Suppose $f$ is a real-valued function with $f=Q_Nf$.
Let $(U^0)_j= f(\frac j N)$ for all $0\le j\le N-1$.
Then for all $N\ge N_0=N_0(\|f\|_{2})$, we have
\begin{align*}
\sup_{t\ge 0} \|U(t) \|_{\infty}
\le \sup_{t\ge 0} \|u(t) \|_{\infty} \le \|f \|_{\infty} + \gamma_1 \cdot N^{-\frac 12} 
\cdot (1+ \| \partial_x f \|_2^2),
\end{align*}
where $\gamma_1>0$ depends only on $\|f\|_{2}$,
and $u$ solves \eqref{Dec12:2}.
\end{thm}
\begin{proof}
Step $1$. $L^2$ estimate. Set $v= u-\bar u$ and note that $\bar u$ is clearly preserved
in time. Then
\begin{align*}
\partial_t v + \frac 12 \partial_x Q_N( v^2 +2 \bar u v) = \partial_{xx} v.
\end{align*}
Multiplying both sides by $v$ and integrating, we obtain
\begin{align*}
\frac 12 \frac d {dt} \| v\|_2^2
&= \frac 12 \langle (Q_N-\op{Id}) (v^2 +2v \bar u), \partial_x v) - \| \partial_x v\|_2^2 \notag \\
&\le \frac 12 \| (Q_N-\op{Id}) (v^2+2v\bar u)\|_2 \| \partial_x v\|_2 
-\| \partial_x v \|_2^2.
\end{align*}
Now note that for $g$ with Fourier support localized in $\{k:\; |k| \le 10N\}$ we have
\begin{align*}
\| (Q_N-\op{Id}) g\|_2 \lesssim N^{-1} \| \partial_x g \|_2 \lesssim N^{-\frac 12} 
\|\partial_x g\|_1.
\end{align*}
Thus
\begin{align*}
\| (Q_N-\op{Id}) (v^2+2v\bar u)\|_2 \|
\lesssim N^{-1} |\bar u| \|\partial_x v\|_2 + N^{-\frac 12} \|v\|_2 \| \partial_x v\|_2.
\end{align*}
We then obtain
\begin{align*}
\frac d {dt} \|v\|_2^2 \le (  \alpha_0 N^{-1}
|\bar u| +
\alpha_0 N^{-\frac 12} \| v\|_2 - 2) \| \partial_x v\|_2^2,
\end{align*}
where $\alpha_0>0$ is an absolute constant.  Clearly for $N\ge N_0=N_0(\|f\|_2)$ sufficiently
large, we have
\begin{align*}
\|v(t) \|_2 \le \|v(0) \|_2 e^{- 2\pi^2 t }, \qquad \forall\, t\ge 0.
\end{align*}

Step 2. Higher order estimates.  Observe that for $t\gtrsim 1$, by smoothing estimates we have
\begin{align*}
\| \partial_x u(t) \|_{H^3(\mathbb T)} \le C \cdot e^{-\beta_1 t},
\end{align*}
where the constant $C>0$ depends only on $\|f\|_2$, and $\beta_1>0$ is an absolute constant.
Now we only need to focus on the local in time estimate. 
We rewrite the equation for $u$ as
\begin{align} \label{Dec15:1}
\partial_t u + u \partial_x u = \frac 12
\partial_x ( \op{Id} -Q_N) (u^2) + \partial_{xx} u.
\end{align}
Clearly for $T>0$, we have
\begin{align}
\| \partial_x u \|_{L_t^{\infty}L_x^2([0,T]\times \mathbb T)}
&\lesssim \| \partial_x u_0\|_2 + T^{\frac 14} \|\partial_x u \|_{L_t^{\infty}L_x^2([0,T]\times
\mathbb T)} \cdot \|u\|_{L_t^{\infty} L_x^2}  \notag \\
& \qquad \qquad + 
T^{\frac 14} \| |\partial_x|^{\frac 12} (Q_N-\op{Id})(u^2) \|_{L_t^{\infty}L_x^2([0,T]
\times \mathbb T)}. \label{Dec15:2}
\end{align}
Since $u$ is spectrally localized to $\{|k| \le N \}$, we have
\begin{align*}
\| |\partial_x|^{\frac 12} (Q_N-\op{Id})(u^2)  \|_{L_x^2}
& \lesssim N^{-\frac 12} \| \partial_x (u^2) \|_{L_x^2} \notag \\
& \lesssim N^{-\frac 12} \| \partial_x u \|_{L_x^2} \| u \|_{L_x^{\infty}}
\lesssim \| \partial_x u \|_{L_x^2} \| u \|_{L_x^2}.
\end{align*}
Plugging this estimate into \eqref{Dec15:2}, we obtain for $T_0=T_0(\|f\|_2)$ sufficiently small that
\begin{align*}
\| \partial_x u \|_{L_t^{\infty} L_x^2([0,T_0]\times \mathbb T)} \lesssim
\| \partial_x f \|_2.
\end{align*}
Now on the same time interval $[0,T_0]$, we have
\begin{align*}
\| t^{\frac 14}  \partial_x u \|_{L_t^{\infty} L_x^{\infty} ([0,T_0]\times \mathbb T)}
&\lesssim \| \partial_x f \|_2 +  T_0^{\frac 14}
\cdot \|\partial_x u \|_{L_t^{\infty} L_x^{2}([0,T_0]\times \mathbb T)} \cdot \|u \|_{L_t^{\infty} L_x^{\infty}
([0,T_0]\times \mathbb T)} 
 \notag \\
 & \le  C \cdot \|\partial_x f \|_2 \cdot (1+ \| \partial_x f \|_2^{\frac 12}),
\end{align*}
where $C>0$ depends only on $\|f \|_2$.  
Next take $\epsilon_0<1$ to be an absolute constant which is sufficiently close to $1$.
We have
\begin{align*}
\| t^{\frac {\epsilon_0}2}  \partial_x |\partial_x|^{\epsilon_0} 
u \|_{L_t^{\infty} L_x^{2} ([0,T_0]\times \mathbb T)}
&\lesssim \| \partial_x f \|_2 +  T_0^{\frac 14}
\cdot \|\partial_x u \|_{L_t^{\infty} L_x^{\infty}([0,T_0]\times \mathbb T)} \cdot \|u \|_{L_t^{\infty} L_x^{2}
([0,T_0]\times \mathbb T)} 
 \notag \\
 & \le  C \cdot \|\partial_x f \|_2 \cdot (1+ \| \partial_x f \|_2^{\frac 12}),
\end{align*}
where $C>0$ depends only on $\|f\|_2$.  Now observe that (for
the second inequality below we used the fact that 
$u$ is spectrally localized to $\{|k| \le N\}$)
\begin{align*}
\| |\partial_x|^{\epsilon_0} \partial_x (u^2) 
\|_{L_x^2}
& \lesssim \| u\|_{L_x^2}^2 + \| |\partial_x|^{\epsilon_0}
\partial_x u \|_{L_x^2} \cdot \| u \|_{L_x^{\infty}}; \\
\| |\partial_x|^{\epsilon_0} \partial_x (Q_N-\op{Id}) (u^2)
\|_{L_x^2}
& \lesssim N^{-(1-\epsilon_0)} \| \partial_x^2 ( u^2) \|_{L_x^2} \notag \\
& \lesssim N^{-(1-\epsilon_0)} (\| \partial_{xx} u \|_{L_x^2}
\| u\|_{L_x^{\infty}} + \| \partial_x u \|_{L_x^2} \| \partial_x u \|_{L_x^{\infty}} ) \notag \\
& \lesssim \| |\partial_x|^{\epsilon_0} \partial_x u \|_{L_x^2}
\cdot (\| u\|_{L_x^2} + \| u \|_{L_x^2}^{\frac 12} \| \partial_x u \|_{L_x^2}^{\frac 12} )
+ \| \partial_x u \|_{L_x^2}^2.
\end{align*}
By using the preceding estimates,  we then compute
\begin{align*}
\| t^{\frac 14} \partial_x^2 u
\|_{L_t^{\infty}L_x^2([0,T_0]\times \mathbb T)}
& \lesssim \| \partial_x f \|_2 + C+ C\cdot \|\partial_x f \|_2 (1+\| \partial_x f \|_2) \notag \\
& \le C + C \|\partial_x f \|_2^2,
\end{align*}
where $C>0$ depends only on $\|f\|_2$. 

Step 3. Conclusion. 
By using the estimate from Step $2$, we have
for any $t>0$,
\begin{align*}
\| \partial_x ( \op{Id}-Q_N ) (u^2) \|_{\infty}
& \lesssim N^{-\frac 12} \| \partial_x^2 ( u^2) \|_2 \notag \\
& \le  C_2  (1+ \| \partial_x f \|_2^2) \cdot N^{-\frac 12} t^{-\frac 14} e^{-\beta t},
\end{align*}
where $\beta>0$ is an absolute constant, and $C_2>0$ depends only on $\|f \|_{2}$. 
Note that
\begin{align*}
\| u^0 \|_{\infty} \le \| f\|_{\infty} + \| (\op{Id} - Q_N) f \|_{\infty}
\le \| f \|_{\infty} +  N^{-\frac 12} \cdot \gamma \cdot \| \partial_x f\|_{2},
\end{align*}
where $\gamma>0$ is an absolute constant.
The desired result then follows from a maximum principle argument.
\end{proof}

\begin{thm} \label{thm7.4}
Consider \eqref{Dec12:1} with $N\ge 2$. Then for all $N\ge N_1=N_1(\|U^0\|_{\infty})$, we have
\begin{align*}
\sup_{t\ge 0} \| U(t) \|_{\infty} \le (1+\epsilon_1) \| U^0\|_{\infty},
\end{align*}
where $0<\epsilon_1 <10^{-2}$ is an absolute constant. More precisely the following hold.
(Below we write $X=O(Y)$ if $|X|\le CY$ and the constant $C$ depends only on $\|U^0\|_{\infty}$.)

\begin{enumerate}
\item For any $t>0$, we have
\begin{align*}
\frac d {dt} ( \|U(t) \|_2^2) = \frac d {dt} ( \|u(t) \|_2^2) \le 0,
\end{align*}
where $u$ solves \eqref{Dec12:2}. In particular for any $0\le t_1 <t_2 <\infty$,
we have
\begin{align*}
\| U(t_2) \|_2 \le \|U(t_1) \|_2.
\end{align*}

\item For $t\ge T_1= N^{-\frac 25}$, we have
\begin{align*}
\sup_{t\ge T_1} \| U(t) \|_{\infty} \le \sup_{t\ge T_1} \|u(t) \|_{\infty}
\le \|U^0\|_{\infty} +O(N^{-0.1}).
\end{align*}

\item For $0<t \le T_1$, we have
\begin{align*}
& \| u(t) -e^{t \partial_{xx}} Q_N U^0\|_{\infty} \le O(N^{-0.1} ), \\
& \sup_{0\le l \le N-1} | (e^{ t \partial_{xx} } Q_N U^0)(t, \frac l N) |
\le (1+\frac 12 \epsilon_1) \|U^0\|_{\infty};\\
& \|U(t) \|_{\infty} \le (1+\frac 12 \epsilon_1) \|U^0\|_{\infty}
+ O(N^{-0.1} ).
\end{align*}

\item Suppose $f:\, \mathbb T\to \mathbb R$ is continuous. 
If $(U^0)_l= f(\frac l N)$ for all $0\le l\le N-1$,  then
\begin{align*}
\sup_{t\ge 0} \|U(t) \|_{\infty} 
\le \|f\|_{\infty} + O(\omega_f(N^{-\frac 23}) ) + O(N^{-c}),
\end{align*}
where $c>0$ is an absolute constant, and $\omega_f$ is defined in \eqref{N18:1b.0}.

\item  Suppose $f:\, \mathbb T\to \mathbb R$ is $C^{\alpha}$-continuous (see
\eqref{N18:1b.1}) for some $0<\alpha<1$.  
If $(U^0)_l= f(\frac l N)$ for all $0\le l\le N-1$,  then
\begin{align*}
\sup_{t\ge 0}\|U(t)\|_{\infty}
\le \|f\|_{\infty} + O(N^{-c_1}),
\end{align*}
where $c_1>0$ is a constant depending only on $\alpha$. 
\end{enumerate}
Moreover we have the following result which shows the sharpness of our estimates above.
There exists a  function $f$: $\mathbb T \to \mathbb R$, continuous at all of $\mathbb T
\setminus \{x_*\}$ for some $x_*\in \mathbb T$ (i.e. continuous at all $x\ne x_*$) and  has the bound $\|f\|_{\infty} \le 1$
 such that 
the following hold:
for a sequence of even numbers $N_m \to \infty$, $t_m =N_m^{-2}$ and $x_m
=j_m/N_m$ with $0\le j_m\le N_m-1$ ,  if  $\tilde U_m(t) \in \mathbb R^{N_m}$ solves \eqref{Dec12:1}
with $\tilde U_m(0) =v_m \in R^{N_m}$ satisfying $(v_m)_j= f(\frac j {N_m})$ for all $0\le j \le N_m-1$.  Then
\begin{align*}
|\tilde U_m(t_m,x_m) | \ge 1+ \eta_*, \qquad\forall\, m\ge 1,
\end{align*}
where $\eta_*>0.001$ is an absolute constant.

\end{thm}
\begin{proof}
We proceed in several steps.

Step 1.  We note that statement (1) follows from the same proof as in Theorem \ref{TmDec12:1}.

Step 2. Local in time estimate. To establish the remaining results, we first perform a local in time estimate. First we show that for some $T_0=T_0(\|U^0\|_2)$ sufficiently small, we have 
contraction using the norm
\begin{align*}
\|u \|_{L_t^{\infty} L_x^2 ([0,T_0]\times \mathbb T)}
+ \| t^{\frac 12} \partial_x u \|_{L_t^{\infty} L_x^2([0,T_0]\times \mathbb T)}.
\end{align*}
Indeed by using the estimate
\begin{align*}
\| |\partial_x|^{\frac 12} (Q_N-\op{Id}) (u^2) \|_{L_x^2}
& \lesssim N^{-\frac 12} \| \partial_x u \|_{L_x^2} \| u\|_{L_x^{\infty}} \notag \\
& \lesssim \| \partial_x u \|_{L_x^2} \cdot \|u \|_{L_x^2},
\end{align*}
we have 
\begin{align*}
\| t^{\frac 12} \partial_x u \|_{L_t^{\infty} L_x^2([0,T_0]\times
\mathbb T)}
& \lesssim \| u(0) \|_2 + T_0^{\frac 14} \| t^{\frac 12} \partial_x u\|_{L_t^{\infty}
L_x^2([0,T_0]\times \mathbb T)} 
\cdot \|u\|_{L_t^{\infty} L_x^2}. 
\end{align*}
Thus for $T_0=T_0(\|U^0\|_{2})$ sufficiently small, we have
\begin{align*}
\|u \|_{L_t^{\infty} L_x^2 ([0,T_0]\times \mathbb T)}
+ \| t^{\frac 12} \partial_x u \|_{L_t^{\infty} L_x^2([0,T_0]\times \mathbb T)}
\lesssim \| U^0 \|_2.
\end{align*}
Now for any $0<T\le T_0$, we have
\begin{align*}
\| u(T) -e^{T\partial_{xx}} Q_N U^0\|_{L_x^{\infty}}
& \lesssim \| t^{\frac 12} \partial_x u \|_{L_t^{\infty} L_x^2 ([0,T_0]\times
\mathbb T)} \cdot \| u\|_{L_t^{\infty}L_x^2([0,T_0]\times \mathbb T)}  \notag \\
& \lesssim \| U^0\|_2^2.
\end{align*}
For $0<t \le T_0$, we write 
\begin{align*}
u(t) = e^{t \partial_{xx} }Q_N U^0 +h_1=: \Theta(t) +h_1,
\end{align*}
where $\| h_1 \|_{L_t^{\infty} L_x^{\infty}} \lesssim \|U^0 \|_2^2$. Note that
both $\Theta$ and $h_1$ are spectrally localized to $\{ |k|\le N \}$. It follows that
for $0<t\le T_0$, 
\begin{align*}
\| Q_N ( u(t)^2) \|_{L_x^2(\mathbb T)}
& \lesssim \| Q_N ( h_1^2) \|_{L_x^2(\mathbb T)}
+ \| Q_N (h_1 \Theta(t) ) \|_{L_x^2(\mathbb T)} + \| Q_N ( \Theta(t)^2) \|_{L_x^2(\mathbb T)} \notag \\
& \lesssim  \|U^0\|_2^4 + \|U^0\|_2^2  \| \Theta(t) \|_{L_x^2(\mathbb T)}
+ \| e^{t \Delta_h} U^0 \|_{2} \| e^{t\Delta_h} U^0\|_{\infty} \notag \\
& \lesssim  \| U^0\|_2^4 + \|U^0\|_2^3 + \|U^0\|_2 \|U^0\|_{\infty}.
\end{align*}
By using this improved estimate, we then obtain for any $0<T\le T_0$,
\begin{align*}
\| u(T) -e^{T\partial_{xx}} Q_N U^0\|_{L_x^{\infty}}
\lesssim T^{\frac 14} 
\Bigl( \| U^0\|_2^4 + \|U^0\|_2^3 + \|U^0\|_2 \|U^0\|_{\infty} \Bigr).
\end{align*}

Step 3. The regime $t\ge T_1=N^{-\frac 25}$.  Clearly
\begin{align*}
\| u(T_1) \|_{\infty} &\le \| e^{T_1 \partial_{xx} } Q_N U^0 \|_{\infty} + O(N^{-0.1}) \notag \\
& \le \|U^0\|_{\infty} + O(N^{-0.1}),
\end{align*}
where we have used the fact that
\begin{align*}
\sum_{|k|\ge \frac N2} e^{-(2\pi k)^2 T_1} \le O(N^{-10}).
\end{align*}
By Step 2 we also have
\begin{align*}
\| \partial_x u (T_1) \|_2  = O(N^{-\frac 25}).
\end{align*}
By Theorem \ref{TmDec12:1}, it follows that
\begin{align*}
\sup_{t\ge T_1} \|U(t)\|_{\infty}
\le \sup_{t\ge T_1} \|u(t) \|_{\infty} 
\le \|U^0\|_{\infty} + O(N^{-0.1}).
\end{align*}

Step 4. The regime $t<T_1$. Note here that our nonlinear solution is dominated by the linear
part. The situation is then similar to that in Theorem \ref{N21_5} with minor changes.
We omit the details.
\end{proof}

\begin{rem}
We should mention that by using the machinery developed above, it is also possible to give a complete
rigorous analysis of the fully discrete scheme
\begin{align*}
\frac {U^{n+1}-U^n} {\tau} +\frac 12 \partial_h( (U^n)^{\cdot 2} )= \Delta_h U^{n+1},
\quad n\ge 0.
\end{align*}
Yet another variant is
\begin{align*}
\frac {U^{n+1}-U^n} {\tau} +U^n \partial_h U^n= \Delta_h U^{n+1},
\quad n\ge 0,
\end{align*}
where $U^n \partial_h U^n$ takes point-wise product of $U^n$ and $\partial_h U^n$.
We shall address these and similar other schemes elsewhere.
\end{rem}

\section{Two dimensional Navier-Stokes}

Consider on the torus $\mathbb T^2 = [0,1)^2$ the two dimensional Navier-Stokes
system expressed in the vorticity form:
\begin{align} \label{Dec17:0}
\begin{cases}
\partial_t \omega + \Pi_N( u \cdot \nabla \omega) = \Delta \omega, \\
\omega \Bigr|_{t=0} = \omega^0 = \Pi_N f,
\end{cases}
\end{align}
where $f\in L^{\infty}(\mathbb T^2)$ and has mean zero.
Here and below we assume $\omega$ has mean zero which is clearly preserved in time. The
velocity $u$ is connected to the vorticity $\omega$ through the Biot-Savart law:
$u=\nabla^{\perp} \Delta^{-1} \omega$.

\begin{lem} \label{lmDec17:1}
Let the dimension $d\ge 1$. 
Suppose $g: \, \mathbb T^d \to \mathbb R$  is in $C^{\gamma}(\mathbb T^d)$ for some
$0<\gamma<1$. Then for all $N\ge 2$, it holds that
\begin{align*}
\| (\op{Id} - \Pi_N) g \|_{\infty} \le C_1\cdot N^{-\gamma} (\log N)^d \| g \|_{C^{\gamma}},
\end{align*}
where $C_1>0$ depends only on ($d$, $\gamma$).
\end{lem}
\begin{proof}
By using  smooth Fourier cut-offs, we can rewrite
\begin{align*}
(\op{Id}-\Pi_N) g = (\op{Id} - \Pi_N) P_N g  + P_{\gtrsim N} g,
\end{align*}
where $P_N$, $P_{\gtrsim N}$ are smooth Fourier projectors localized
to the regimes $\{ |k| \sim N \}$ and $ \{ |k|\gtrsim N \}$ respectively.
The result then follows
from Theorem \ref{thm2.18_00} together with the fact that 
\begin{align*}
\| (\op{Id}-\Delta) P_N g \|_{\infty} \lesssim N^{2-\gamma} \| g\|_{C^{\gamma}}.
\end{align*}
\end{proof}

\begin{thm}
Let  $\omega$ be the solution to \eqref{Dec17:0}. 
We have for all $N\ge 2$, 
\begin{align*}
\sup_{t\ge 0} \|\omega(t) \|_{\infty}
\le \|\omega^0\|_{\infty} + \alpha \cdot N^{-\frac 12}, 
\end{align*}
where  $\alpha>0$ depends only on 
$\|f\|_{\infty}$. 
\end{thm}
\begin{proof}
First we compute the $L^2$-norm. Clearly
\begin{align*}
\frac 1 2 \frac d {dt} \int_{\mathbb T^2} \omega^2 dx 
= - \| \nabla \omega \|_2^2 \le -(2\pi)^2 \| \omega \|_2^2.
\end{align*}
Thus
\begin{align*}
\| \omega(t) \|_2 \le \| f\|_2 \cdot e^{-4\pi^2 t }, \qquad t\ge 0.
\end{align*}
Using this and further smoothing estimates, we obtain
\begin{align*}
\| u(t)\cdot \nabla \omega(t) \|_{C^{0.52}} \le  C_1 t^{-\epsilon_1} e^{-c_1 t}, \qquad\forall\,
t>0,
\end{align*}
where $C_1>0$ depends only on $\|f \|_{\infty}$, and $0<\epsilon_1<1$, $c_1>0$ are
absolute constants. By Lemma \ref{lmDec17:1},  it follows that 
\begin{align*}
\| (\op{Id} - \Pi_N) (u(t) \cdot \nabla \omega(t) ) \|_{\infty}
\le C_2 N^{-\frac 12} t^{-\epsilon_1} e^{-c_1t}, \qquad\forall\, t>0,
\end{align*}
where $C_2>0$ depends only on $\|f\|_{\infty}$.  Our desired result then follows from
a maximum principle argument similar to that in Burger's case.
\end{proof}
Consider
\begin{align} \label{Dec17:2}
\begin{cases}
\frac{\omega^{n+1}-\omega^n} 
{\tau}  + \Pi_N ( u^n \cdot \nabla \omega^n) = 
\Delta \omega^{n+1}, \quad (t,x) \in (0,\infty) \times \mathbb T^2; \\
\omega\Bigr|_{t=0} = \omega^0= \Pi_N f,
\end{cases}
\end{align}
where $f \in L^{\infty}(\mathbb T^2)$ and has mean zero. 

\begin{thm}
Let  $\omega^n$ be the solution to \eqref{Dec17:2}. 
We have for all $0<\tau<\tau_0=\tau_0(\|f\|_{\infty})$ and $N\ge 2$, 
\begin{align*}
\sup_{n\ge 0} \|\omega^n \|_{\infty}
\le \|\omega^0\|_{\infty} + \alpha_1 \cdot N^{-\frac 12} +\alpha_2 \tau^{0.4}, 
\end{align*}
where $\alpha_1, \alpha_2>0$ depend only on  $\|f\|_{\infty}$. 
\end{thm}
\begin{proof}
Rewrite \eqref{Dec17:2} as
\begin{align} \label{Dec17:3}
\omega^{n+1} = A \omega^n - \tau A \Pi_N  (  u^n \cdot \nabla \omega^n),
\end{align}
where $A=(\op{Id}-\tau \Delta)^{-1}$.  

We shall proceed in several steps.

Step 1. $L^2$ estimate. Multiplying both sides of \eqref{Dec17:2} by $\omega^{n+1}$ and
integrating, we obtain
\begin{align*}
\frac 1 {2\tau} 
( \| \omega^{n+1}\|_2^2 - \| \omega^n \|_2^2
+ \| \omega^{n+1} -\omega^n \|_2^2)
+  \| \nabla \omega^{n+1}\|_2^2
+ N_1 \le 0,
\end{align*}
where
\begin{align*}
N_1 & = \int \Pi_N (u^n \cdot \nabla \omega^n) \omega^{n+1} 
= \int u^n \cdot\nabla \omega^n ( \omega^{n+1} -\omega^n )   \notag \\
& =
\int (u^n \cdot\nabla \omega^{n+1}) (\omega^{n+1}-\omega^n) 
+ \int u^n  \cdot \nabla(\omega^n - \omega^{n+1})  (\omega^{n+1} -\omega^n) dx.
\end{align*}
The second term above vanishes thanks to incompressibility.  Then 
\begin{align*} 
|N_1|&\le \|u^n\|_{\infty} \cdot \| \nabla \omega^{n+1} \|_2
\| \omega^{n+1} -\omega^n \|_2  \notag \\
&\le \; \frac 12 \| \nabla \omega^{n+1} \|_2^2
+ \frac 1 {2}
\| u^n \|_{\infty}^2 
\cdot \|\omega^{n+1}-\omega^n\|_2^2.
\end{align*}
Thus if 
\begin{align} \notag
\frac 1 {\tau}
\ge 
\| u^n \|_{\infty}^2,
\end{align}
then 
\begin{align*}
(1+ 8\pi^2 \tau) \| \omega^{n+1}\|_2^2
\le \|\omega^{n+1} \|_2^2+
  \tau \| \nabla \omega^{n+1} \|_2^2 \le \|\omega^n\|_2^2.
\end{align*}

Step 2. Induction.

For $n=0$, we have
\begin{align*}
\|u^0 \|_{\infty} \le C_1 \|\omega^0\|_4 \le C_2 \| f\|_4,
\end{align*}
where $C_1>0$, $C_2>0$ are absolute constants.  Clearly if 
\begin{align*}
\frac 1 {\tau} \ge  C_2 \|f\|_4,
\end{align*}
then 
\begin{align*}
(1+8\pi^2 \tau)\|\omega^1 \|_2^2 \le \|\omega^0\|_2^2.
\end{align*}
Now for $n\ge 0$, we assume 
\begin{align*}
(1+8\pi^2 \tau) \|\omega^{j+1} \|_2^2 \le \| \omega^{j} \|_2^2, \qquad \forall\, 
1\le j\le n.
\end{align*}
Then for $n+1$ by using \eqref{Dec17:3}, we have
\begin{align*}
\|\omega^{n+1}\|_4 
& \lesssim \| |\nabla|^{\frac 12} A \omega^n \|_2
+ \tau \| \Delta A (u^n \omega^n) \|_{\frac 43} \notag \\
& \lesssim \tau^{-\frac 14} \| f \|_2
+ \| u^n \|_4 \| \omega^n \|_2 \notag \\
& \le C_3 \tau^{-\frac 14} \| f \|_2 + C_4 \|f\|_2^2,
\end{align*}
where $C_3>0$, $C_4>0$ are absolute constants. Thus if
\begin{align*}
\frac 1 {\tau} \ge  C_3 \tau^{-\frac 14} \| f \|_2 + C_4 \|f\|_2^2,
\end{align*}
then clearly we have
\begin{align*}
(1+8\pi^2 \tau) \|\omega^{n+2} \|_2^2 \le \| \omega^{n+1} \|_2^2.
\end{align*}

This then completes the induction proof and we obtain  for sufficiently small $\tau$, 
\begin{align*}
& (1+8\pi^2 \tau) \|\omega^{n+1}\|_2^2 \le  \|\omega^n\|_2^2, \qquad \forall\, n\ge 0.
\end{align*}
It follows that
\begin{align*}
\|\omega^n\|_2 \le \gamma_1 e^{-\beta n\tau}, \qquad \forall\, n\ge 1,
\end{align*}
where $\beta>0$ is an absolute constant, and $\gamma_1>0$ depends only on $\|f\|_{\infty}$.

Step 3. Smoothing estimates. 
Let $t_0>0$ be chosen sufficiently
small (depending only on $\|f\|_{\infty}$) so that the nonlinear part is dominated by the linear part when $n\tau\le t_0$. 
Choose $1\le n_0 \in \mathbb Z$ such that $n_0 \tau \sim t_0$. By using
\eqref{Dec17:3} and discrete smoothing estimates, we have
\begin{align}
&\sup_{1\le n \le n_0}( (n\tau)^{1.6}
\| |\nabla|^{3.2} \omega^n \|_{20}+ (n\tau)^{0.9}\| |\nabla|^{1.8} \omega^n \|_{20}+ 
\|\omega^n\|_{20}) \le D_1,  \notag\\
& \sup_{n\ge n_0} \| \nabla^4 \omega^{n} \|_{\infty} \le D_1 e^{-\beta_1 n\tau}, \notag
\end{align}
where $D_1>0$ depends only on $\|f\|_{\infty}$, and $\beta_1>0$ is
an absolute constant. It follows that
\begin{align*}
\| |\nabla|^{1.8} \omega^n \|_{20}
\le D_2 \cdot (n\tau)^{-0.9} e^{-\beta_2 n\tau}, \qquad\forall\, n\ge 1,
\end{align*}
where $D_2>0$ depends only on $\|f \|_{\infty}$, and $\beta_2>0$ is an absolute
constant. By Lemma \ref{lmDec17:1},  we  obtain
\begin{align*}
\tau \| (\op{Id}-\Pi_N) ( u^n \cdot \nabla  \omega^n) \|_{\infty}
\le D_3 \cdot \tau \cdot N^{-\frac 12}
\cdot (n\tau)^{-0.9} e^{-\beta_2 n\tau}, \qquad\forall\, n\ge 1,
\end{align*}
where  $D_3>0$ depends only on $\|f\|_{\infty}$. 
By using \eqref{Dec17:2}, we obtain for $n\ge 1$, 
\begin{align*}
\tau \| u^n \cdot \nabla (\omega^{n+1} - \omega^n ) \|_{\infty}
& \lesssim \tau^2 \| u^n \|_{\infty} \| \nabla ( \Pi_N ( u^n \cdot \nabla \omega^n ) )
\|_{\infty}
+ \tau^2 \|u^n\|_{\infty} \| \nabla^3 \omega^{n+1} \|_{\infty} \notag \\
& \le  D_4 \tau^2\cdot (n\tau)^{-1.6},
\end{align*}
where  $D_4>0$ depends only on $\|f\|_{\infty}$. 

Step 4. Time-global estimates.  
Now for $n\ge 1$ we rewrite the equation for $\omega^{n+1}$ as
\begin{align*}
(\op{Id} -\tau \Delta) \omega^{n+1}
=u^n +\tau u^n \cdot \nabla \omega^{n+1} +\tau u^n \cdot \nabla (\omega^n -\omega^{n+1})
+ \tau (\op{Id}-\Pi_N)(u^n \cdot \nabla  \omega^n).
\end{align*}
Clearly by a maximum principle argument and the estimates derived in earlier steps, we have
\begin{align*}
\|\omega^{n+1} \|_{\infty}
&\le \|\omega^n\|_{\infty} +
\tau \| u^n \cdot \nabla (\omega^{n+1}-\omega^n) \|_{\infty}
+\tau \| (\op{Id}-\Pi_N) (u^n \cdot \nabla \omega^n ) \|_{\infty}  \notag \\
&\le \|\omega^n \|_{\infty} +D_4 \cdot \tau^2 
\cdot (n\tau)^{-1.6}
+ D_3 \cdot \tau \cdot N^{-\frac 12}
\cdot (n\tau)^{-0.9} e^{-\beta_2n\tau}, \qquad\forall\, n\ge 1.
\end{align*}
Iterating in $n$,  we obtain
\begin{align*}
\| \omega^n \|_{\infty}
\le \|\omega^1 \|_{\infty}
+ D_5 \tau^{0.4}  +D_6 N^{-\frac 12}, \qquad\forall\, n\ge 1,
\end{align*}
where $D_5, D_6>0$ are constants depending only on $\|f\|_{\infty}$.  Finally
we observe that
\begin{align*}
\| \omega^1 \|_{\infty}
& \le \|\omega^0\|_{\infty}
+\| \tau  A \Pi_N \nabla  \cdot ( u^0 \omega^0) \|_{\infty} \notag \\
& \le \|\omega^0\|_{\infty} + C \tau^{0.4} \|f \|_{\infty}^2,
\end{align*}
where $C>0$ is an absolute constant. 
Our desired result then follows easily.
\end{proof}
\begin{rem}
We should mention that similar to the Burgers case, we can also develop the corresponding
analysis for the Fourier collocation method applied to the Navier-Stokes system and
many other similar fluid models. All these and related issues will be investigated elsewhere.
\end{rem}

\appendix
\section{Some auxiliary estimates}
\begin{lem} \label{aOc1}
Consider for $z>0$, 
\begin{align*}
F(z) = e^{-z} \int_0^{\infty}
e^{-zt} (t+\frac {t^2} 2)^{-\frac 12} dt.
\end{align*}
Then for $0<z\le \frac 14$, we have
\begin{align*}
F(z) = c_1-c_2 \log z + f_1(z),
\end{align*}
where $c_1, c_2>0$ are constants, and
\begin{align*}
& |f_1(z)| \lesssim z |\log z|, \quad |f_1^{\prime}(z)| \lesssim |\log z|.
\end{align*}
\end{lem}
\begin{proof}[Proof of Lemma \ref{aOc1}]
We shall use the notation ``NICE" to denote those terms which can be included
in $c_1$ and $f_1$. 
It is not difficult to check that
\begin{align*}
F(z) &= e^{-z} \int_0^{\infty} e^{-t} (tz + \frac {t^2} 2)^{-\frac 12} dt \notag \\
&=e^{-z}\int_0^1 e^{-t} (tz + \frac {t^2} 2)^{-\frac 12} dt +\operatorname{NICE} \notag \\
&=e^{-z}\int_0^1 (tz + \frac {t^2} 2)^{-\frac 12} dt +\operatorname{NICE} \notag \\
&=\int_0^1 (tz + \frac {t^2} 2)^{-\frac 12} dt +\operatorname{NICE}.
\end{align*}
Now by a simple change of variable, we have
\begin{align*}
\int_0^1 (2 tz +t^2)^{-\frac 12} dt
&= \int_0^{\frac 1 z} (2 t+ t^2)^{-\frac 12 } dt \notag \\
&= \int_0^{3} (2 t+ t^2)^{-\frac 12 } dt 
+\int_3^{\frac 1 z} t^{-1} \cdot (1+\frac 2 t)^{-\frac 12} dt \notag \\
&=  \int_0^{3} (2 t+ t^2)^{-\frac 12 } dt 
+ \int_3^{\frac 1z} t^{-1} dt 
+ \int_3^{\infty} t^{-1}\cdot ( (1+\frac 2 t)^{-\frac 12}-1) dt \notag \\
&\qquad - \int_{\frac 1z}^{\infty} t^{-1}\cdot ( (1+\frac 2 t)^{-\frac 12}-1) dt
\end{align*}
Note that $\int_0^3 (2t+t^2)^{-\frac 12} dt =2\sinh^{-1}(\sqrt{\frac 32} )\approx
2.0634>\log 3\approx 1.0986$.
The desired result then clearly follows.
\end{proof}

\begin{proof}[Proof of Statement (3) of Theorem \ref{thm2.18_00}, 2D case]
For simplicity we take $\beta=1$.  We first consider the case $d=2$.
It suffices for us to work with the expression
\begin{align*}
I(x) = \sum_{n \in \mathbb Z^2}
\int_{|y|_{\infty} <\frac 12}
F(x-y +n) \frac {\sin M\pi y_1}{\sin \pi y_1}
\frac {\sin M\pi y_2} {\sin \pi y_2} dy_1 dy_2,
\end{align*}
where $M=2N+1$, and for $z\ne 0$,
\begin{align*}
F(z)=e^{-|z|} \int_0^{\infty} e^{-|z| t} (t+\frac {t^2} 2)^{-\frac 12} dt.
\end{align*}
We shall show that for $|x|_{\infty} \le \frac 12$, $I(x)$ has a strictly positive lower bound
if $M$ is sufficiently large, namely
\begin{align*}
I(x) \gtrsim \log(1+ \frac 1 {|x|} ), \quad\forall\, x \ne 0.
\end{align*}

Case 1: $|n|_{\infty}\ge 2$. It is not difficult to check that for $|x|_{\infty} \le \frac 12$,
$|y|_{\infty}<\frac 12$, the function
\begin{align*}
F_1(y) =\Bigl( \sum_{|n|_{\infty} \ge 2} 
F(x-y+n) \Bigr) \cdot \frac{y_1}{\sin \pi y_1} \cdot \frac {y_2}{\sin \pi y_2}
\end{align*}
is $C^{\infty}$ smooth and has bounded derivatives of all orders.  We then have
\begin{align}
  & \int_{|y|_{\infty}<\frac 12} F_1(y)
  \frac{\sin M\pi y_1} {y_1} \frac {\sin M \pi y_2} {y_2} dy \notag \\
  =&\; \int_{|y|_{\infty}<\frac 12 M}
  F_1(\frac y M) \frac {\sin \pi y_1} {y_1} \frac {\sin \pi y_2} {y_2} dy \notag \\
  =&\; O(M^{-1}  {(\log M)^2} ) + \frac 1 {\pi^2} 
  \int_{|y|_{\infty}<\frac 12 M}
  F_1(\frac y M) \frac {1-\cos \pi y_1} {y_1^2} \cdot \frac {1- \cos \pi y_2}{y_2^2} dy \notag \\
  =&\; O(M^{-1}  {(\log M)^2} ) + \frac 1 {\pi^2}  F_1(0) 
  \int_{|y|_{\infty}<\frac 12 M}
  \frac {1-\cos \pi y_1} {y_1^2} \cdot \frac {1- \cos \pi y_2}{y_2^2} dy  \notag \\
  =&\;  O(M^{-1}  {(\log M)^2} ) + \frac 1 {\pi^2}  F_1(0)
 \Bigl (\int_{\mathbb R}
  \frac {1-\cos \pi y_1} {y_1^2} dy_1 \Bigr)^2. \label{aOc2}
\end{align}  
Thus the contribution of this piece is acceptable for us.

Case 2: $|n|_{\infty}\le 1$. There are five terms $n=(\pm 1, \pm 1)$ and $n=(0,0)$. 
We consider only the case $n=(0,0)$ as the others can be similarly treated. 
Let $\chi \in C_c^{\infty}(\mathbb R^2)$ be such that $0\le \chi(z) \le 1$ for all $z$,
$\chi(z)=1$ for $|z|\le \frac 1{50}$ and $\chi(z)=0$ for $|z|\ge \frac 1{49}$. By Lemma
\ref{aOc1}, we  decompose $F$ as
\begin{align*}
F(z)&= F(z) \chi(z) + F(z) (1-\chi(z) ) \notag \\
&=-c_2\chi(z) \log |z|  + f_1(|z|) \chi(z) + c_1 \chi(z) +F(z) (1-\chi(z)).
\end{align*}
By a computation similar to that in \eqref{aOc2}, it is not difficult to check that the
contribution due to the term $c_1 \chi(z) + F(z) (1-\chi(z) )$ is acceptable for us. 
We now consider the contribution due to the term $f_1(|z|) \chi(z)$. For this
we need to work with the expression
\begin{align*}
I_1=\int_{|y|_{\infty} <\frac 12}
f_1( |y- x | )\underbrace{\chi(y -x) \cdot \frac {y_1} {\sin \pi y_1} 
\cdot \frac {y_2} {\sin \pi y_2}}_{=:\chi_1(y,x)} \cdot \frac {\sin M\pi y_1 \sin M\pi y_2} {y_1 y_2} dy_1 dy_2. 
\end{align*}
We then write
\begin{align*}
I_1 = \int_{|y|_{\infty}<\frac 12 M}
f_1(|\frac yM-x|) \chi_1(\frac y M,x)
\cdot \frac {\sin \pi y_1 \sin \pi y_2} {y_1 y_2} dy.
\end{align*}
By repeating a similar integration by parts argument in \eqref{aOc2} and
using Lemma \ref{aOc1}, we then only need to control
the main error term (the boundary terms are easily controlled)
\begin{align*}
\operatorname{er}_1:=\int_{|y|_{\infty}<\frac 12 M}
\frac 1 M \cdot | \log|\frac {y-Mx} M | |\cdot \frac 1 {(1+|y_1|)(1+|y_2|)} dy.
\end{align*}
By the simple 1D inequality (below $|z_1| \lesssim M$, and one can use the fact that
$|\log |y_1-z_1|| \lesssim \log (2+|y_1|) + \log |z_1|$ for $|y_1-z_1|>1$):
\begin{align*}
\int_{|y_1|<\frac 12 M} |\log |y_1-z_1| | \cdot \frac 1 {1+|y_1|} dy_1
\lesssim (\log M)^2,
\end{align*}
we have 
\begin{align*}
\operatorname{er}_1 \lesssim M^{-1} (\log M)^4.
\end{align*}
Thus the contribution of $I_1$ is acceptable for us. 
Finally we consider the contribution due to the term $-\chi(z) \log |z| $. For this
we need to work with the expression
\begin{align}
I_2 &=-\int_{|y|_{\infty} <\frac 12}
\log|y-x| \cdot \underbrace{\chi(y -x) \cdot \frac {y_1} {\sin \pi y_1} 
\cdot \frac {y_2} {\sin \pi y_2}}_{=:\chi_1(y,x)} \cdot \frac {\sin M\pi y_1 \sin M\pi y_2} {y_1 y_2} dy_1 dy_2 \notag 
\end{align}
Consider $0<|x| \ll 1$. We have
\begin{align}
I_2&=-\int_{|y|_{\infty}<\frac 12 M}
\log |\frac yM-x|\cdot \chi_1(\frac y M,x)
\cdot \frac {\sin \pi y_1 \sin \pi y_2} {y_1 y_2} dy\notag \\
&= -\int_{|y|_{\infty}<\frac 12 M}
\log | y-M x|\cdot \chi_1(\frac y M,x)
\cdot \frac {\sin \pi y_1 \sin \pi y_2} {y_1 y_2} dy\notag \\
&\quad + \log M \int_{|y|_{\infty}<\frac 12 M} \chi_1(\frac y M,x)
\cdot \frac {\sin \pi y_1 \sin \pi y_2} {y_1 y_2} dy\notag \\
=&\;-\int_{|y|_{\infty}<\frac 12 M}
\log | y-M x|\cdot \chi_1(\frac y M,x) \chi(y-Mx) 
\cdot \frac {\sin \pi y_1 \sin \pi y_2} {y_1 y_2} dy\label{aOc2.0a} \\
&\quad-\int_{|y|_{\infty}<\frac 12 M}
(\log | y-M x| )(1-\chi(y-Mx) )\cdot \chi_1(\frac y M,x)
\cdot \frac {\sin \pi y_1 \sin \pi y_2} {y_1 y_2} dy \label{aOc2.0b} \\
&\quad + (\log M)  \chi_1(0,x) \int_{|y|_{\infty}<\frac 12 M}
 \frac {\sin \pi y_1 \sin \pi y_2} {y_1 y_2} dy\label{aOc2.0c} \\
&\quad + \log M \int_{|y|_{\infty}<\frac 12 M}( \chi_1(\frac y M,x)
-\chi_1(0,x) )
\cdot \frac {\sin \pi y_1 \sin \pi y_2} {y_1 y_2} dy. \label{aOc2.0d}
\end{align}
For \eqref{aOc2.0a}, since $\chi(y-Mx)$ is compactly supported, it follows easily that 
\begin{align*}
|\eqref{aOc2.0a} | \lesssim 1.
\end{align*}
For \eqref{aOc2.0b}, by successive integration by parts, we have 
\begin{align*}
\eqref{aOc2.0b} &= -(\log | M x| )(1-\chi(Mx) )\cdot \chi_1(0,x)
\cdot \int_{|y|_{\infty}<\frac 12 M}
\frac {(1-\cos\pi y_1) (1-\cos \pi y_2)} {\pi^2 y_1^2 y_2^2} dy +O(1).
\end{align*}
We can rewrite \eqref{aOc2.0c} as
\begin{align*}
\eqref{aOc2.0c}=(\log M)  \chi_1(0,x) \int_{|y|_{\infty}<\frac 12 M}
 \frac {(1-\cos\pi y_1) (1-\cos \pi y_2)} {\pi^2 y_1^2 y_2^2} dy +O(M^{-1} (\log M)^5).
 \end{align*}
 For \eqref{aOc2.0d}, one can use integration by parts to obtain that
\begin{align*}
|\eqref{aOc2.0d}| \lesssim M^{-1} (\log M)^5.
\end{align*}
Collecting the estimates, it is then not difficult to check that for $0<|x| \ll 1$ and $M$ sufficiently large, 
\begin{align*}
I_2 \ge  - c_1\log |x|,
\end{align*}
where $c_1>0$ is an absolute constant.

Next consider the case $|x|\sim 1$. In this case we note that $M|x| = O(M)$. In
estimating \eqref{aOc2.0a}, we note that when $|y-Mx| \lesssim 1$, we have $|y|=O(M)$ which
implies that
\begin{align*}
|\eqref{aOc2.0a}| =O(M^{-1}).
\end{align*}
On the other hand, it is not difficult to check that 
\begin{align*}
&\eqref{aOc2.0b}  \notag\\
=& -(\log | M x| )(1-\chi(Mx) )\cdot \chi_1(0,x)
\cdot \int_{|y|_{\infty}<\frac 12 M}
\frac {(1-\cos\pi y_1) (1-\cos \pi y_2)} {\pi^2 y_1^2 y_2^2} dy +O(M^{-1}
(\log M)^5).
\end{align*}
Here when estimating the error terms, we used the fact that for $M|x| =O(M)$,
\begin{align*}
\int_{1\le |y_1|, |y_2|<M}
\frac 1 {1+|y-Mx|} \frac 1 {(1+|y_1|)(1+|y_2|)} dy  \lesssim 
M^{-1} (\log M)^2.
\end{align*}
Collecting the estimates, we  obtain that for the case $|x|\sim 1$,
\begin{align*}
I_2 \ge -c_2 (\log |x|) \chi(x) + O(M^{-1} (\log M)^5).
\end{align*}
where $c_2$ is an absolute constant and we assume $M$ is sufficiently large.
Thus for all $x\ne 0$ with $|x|_{\infty}<\frac 12$, we have
\begin{align*}
I_2 \ge - c_3 (\log |x|) \chi(x) + O(M^{-1} (\log M)^5).
\end{align*}
One need not worry about the case $\chi(x)=0$ (i.e. when $|x|\ge \frac 1{50}$), since
 by \eqref{aOc1}, the main order is a constant.

Collecting all the estimates from Case 1 and Case 2, we then obtain the desired conclusion for
2D.
\end{proof}
\section{Estimate of $e^{\nu^2 t\partial_{xx} } Q_N U^0$}
Recall that 
\begin{align*}
(e^{ \nu^2 t \partial_{xx}} Q_N U^0 )(t,x)
=\op{Re} \Bigl( \sum_{-\frac N2<k \le \frac N2} \Bigl( \frac 1 N \sum_{j=0}^{N-1} U^0_j e^{-2\pi i \frac {k j}N }
e^{-(2\pi k)^2 \nu^2 t} 
\Bigr)
e^{2\pi i k\cdot x} \Bigr).
\end{align*}
We need to bound $\sup_{0\le l\le N-1} | (e^{\nu^2 t \partial_{xx} } Q_N U^0 )(t,\frac lN)|$ under the assumption that
$\| U^0 \|_{\infty} \le 1$.
Clearly we have
\begin{align*}
\sup_{0\le l\le N-1} | (e^{\nu^2 t \partial_{xx} } Q_N U^0 )(t,\frac lN)| \le 
\frac 1 N \sum_{j=0}^{N-1} \Bigl| \op{Re} \Bigl( { \sum_{-\frac N2 < k \le \frac N2}
e^{-4\pi^2\nu^2 t k^2  +2\pi i k \cdot \frac j N}}  \Bigr)\Bigr| =: A_{N,t}.
\end{align*}

Below we shall write $X= O(Y)$ if $|X|\le CY$  and $C>0$ is a  constant depending only on $\nu$.
\begin{thm} \label{N18:1}
It holds that 
\begin{align*}
A_{N,t} \le 
\begin{cases}
1+ O(N^{-1}), \qquad \text{if $N^2 \nu^2 t \ge 100 \log N$}; \\
1+ \epsilon_a + O(N^{-\frac 13} \sqrt{\log N}), \qquad \text{if $N^2 \nu^2 t <100 \log N$}.
\end{cases}
\end{align*}
Here $0<\epsilon_a <0.01$ is an absolute constant. Furthermore, we have
for $t=t_0=\frac 14 N^{-2}\nu^{-2}$,
\begin{align*}
A_{N, t_0} >1.001+O(N^{-\frac 13} \sqrt{\log N}).
\end{align*}
\end{thm}
The rest of this section is devoted to the proof of Theorem \ref{N18:1}. To ease the notation
we shall assume $\nu=1$. 

First observe that if $N^2 t \ge 100 \log N$, then
\begin{align*}
A_{N,t} & \le 1 + \sum_{|k|\ge \frac N2} e^{-4 \pi^2 t k^2} \notag  \\
& \le 1 + \int_{|k|\ge \frac N2-1} e^{-4\pi^2 tk^2} d k \le 1 +O(N^{-1}).
\end{align*}

Thus we only need to consider the regime $N^2 t <100 \log N$. Denote 
$k_0 = 2N \sqrt {t}$ and note that $k_0 \le 20 \sqrt{\log N}$. 
Denote 
\begin{align*}
Q_j = \sum_{-\frac N2 < k \le \frac N2}
e^{-4\pi^2t k^2  +2\pi i k \cdot \frac j N}.
\end{align*}

It is easy to check that
\begin{align*}
\op{Re}\Bigl( \sum_{-\frac N2<k\le \frac N2} e^{2\pi i k\cdot x} \Bigr)
= \frac {\sin N\pi x} {\tan \pi x}.
\end{align*}
Then
\begin{align}
Q_j&= \frac 1 {\sqrt{4\pi t}} \int_{\mathbb R}e^{-\frac { (\frac jN-y)^2} {4t} }
\frac {\sin N \pi y} { \tan \pi y} dy \notag \\
&= \frac 1 {\sqrt{\pi} } \int_{\mathbb R} e^{-y^2}
\frac {\sin 2N \sqrt t \pi y} { \tan \pi (\frac jN +2\sqrt t y) } dy (-1)^j\notag \\
&= \frac 1 {\sqrt{\pi} } \int_{\mathbb R} e^{-y^2}
\frac {\sin \pi k_0 y} { \tan \pi (\frac {j+k_0y} N ) } dy (-1)^j\label{N18:2} \\
&= \frac 1 {\sqrt{\pi} k_0} \int_{\mathbb R}
e^{-\frac {y^2}{k_0^2}} \frac {\sin \pi y} {\tan \pi \frac {j +y} N} dy (-1)^j. 
\label{N18:3}
\end{align}
We shall use either \eqref{N18:2} or \eqref{N18:3}.  Note that 
\begin{align*}
|Q_j| =|Q_{N-j}|, \qquad \forall\, 1\le j\le N-1.
\end{align*}

Case 1: \texttt{The regime $N^{\frac 13} < j <N-N^{\frac 13} $}. 
This is the tail piece.
Observe that
 \begin{align*}
 \int_{|y| \ge 100 \sqrt{\log N} }
 e^{-y^2} \Bigl| \frac {\sin \pi k_0y } { \tan \pi \frac {j +k_0y} N}
 \Bigr| dy \le  N^{-5}. 
 \end{align*}
 Then we only need to estimate
 \begin{align*}
 &\Bigl| \int_{0<y\le 100 \sqrt{\log N}}
 e^{-y^2} \sin \pi k_0 y ( \cot \pi {\frac {j+k_0y} N} -\cot \pi \frac{j-k_0 y} N ) dy \Bigr|
 \notag \\
 \lesssim &\;  \frac {N^2} {j_1^2} \cdot \frac {k_0} N = N j_1^{-2} k_0, \qquad j_1=\min
 \{j, N-j\}.
 \end{align*}
  Then
  \begin{align*}
  \frac 1 N \sum_{N^{\frac 13} < j< N-N^{\frac 13}}
  |Q_j| \lesssim \sum_{ N^{\frac 13} <j <N-N^{\frac 13} }
  j_1^{-2} k_0 \lesssim N^{-\frac 13} \sqrt{\log N}.
  \end{align*}
  Thus the regime $N^{\frac 13} <j <N-N^{\frac 13}$ is OK for us.

  Case 2: \texttt{The regime $0\le j \le N^{\frac 13}$}. This is the main piece. Note 
  that since $Q_j = Q_{N-j}$, the regime $ N-N^{\frac 13}\le j \le N-1$ does not need any further
  analysis.  Now for $0\le j \le N^{\frac 13}$, by using \eqref{N18:3}, we have 
  \begin{align*}
  |Q_j| = \frac 1 {\sqrt{\pi} k_0}\Bigl|
  \int_{\R} e^{-\frac {(y+j)^2}{k_0^2} } \frac {\sin \pi y} {\tan \pi \frac {y} N } dy\Bigr|.
  \end{align*}
  Observe that 
  \begin{align*}
k_0^{-1} \Bigl |\int_{|y|\ge N^{\frac 12}}
e^{-\frac {(y+j)^2} {k_0^2} } \cdot 
\frac {\sin \pi y} {\tan \pi \frac y N} d y \Bigr| \lesssim e^{-c N}.
\end{align*}
Note that for $|x| \ll 1$, we have
\begin{align*}
|\cot x - \frac 1 x| \lesssim x.
\end{align*}
We obtain
\begin{align*}
\frac 1 N \sum_{j\le N^{\frac 13}} k_0^{-1}  \Bigl| \int_{|y| \le N^{\frac 12}} 
e^{- \frac{(y+j)^2}{k_0^2}} \cdot \frac {|y|} N dy\Bigr| \lesssim  \frac 1 N \cdot N^{\frac 13} \cdot N^{-\frac 12}
\lesssim N^{-\frac 76}. 
\end{align*}
We then obtain the main piece
\begin{align*}
&\frac 1 N \sum_{j\le N^{\frac 13}} k_0^{-1}
\Bigl| \int_{|y|\le N^{\frac 12}} e^{- \frac{(y+j)^2}{k_0^2}}
\cdot \frac {\sin \pi y} y N dy  \Bigr| \notag \\
= &\sum_{j\le N^{\frac 13}} k_0^{-1}
\Bigl| \int_{\R } e^{- \frac{(y+j)^2}{k_0^2}}
\cdot \frac {\sin \pi y} y dy   - \int_{|y|>N^{\frac 12}} 
e^{- \frac{(y+j)^2}{k_0^2}}
\cdot \frac {\sin \pi y} y dy \Bigr|   \notag \\
=& \biggl( \sum_{j\le N^{\frac 13}} k_0^{-1}
\Bigl| \int_{\R } e^{-\frac{(y+j)^2}{k_0^2}}
\cdot \frac {\sin \pi y} y dy  \Bigr|  \biggr) + 
O(e^{-cN}). 
\end{align*}
Thus
\begin{align*}
\frac 1 N \sum_{j\le N^{\frac 13}} |Q_j|
= \pi^{-\frac 32} k_0^{-1}\sum_{j\le N^{\frac 13}}
\Bigl| \int_{\R } e^{-\frac{(y+j)^2}{k_0^2}}
\cdot \frac {\sin \pi y} y dy  \Bigr| + O(N^{-\frac 76}).
\end{align*}
Fix $j$ and consider
\begin{align*}
F_j(1)= \pi^{-\frac 32}k_0^{-1}  \int_{\R} e^{-\frac {(y+j)^2 } {k_0^2} } \frac {\sin \pi  y} { y} dy =
\int_0^1 F_j^{\prime}( s) d s,
\end{align*}
where
\begin{align*}
F_j^{\prime}(s) &= \pi^{-\frac 12} k_0^{-1}  \op{Re} \Bigl( \int_{\R} e^{-\frac {(y+j)^2}{k_0^2}}
 e^{i \pi s  y} dy \Bigr) \notag \\
&= e^{-\frac 14 \pi^2 k_0^2 s^2} \cos \pi s  j.
\end{align*}
It follows that
\begin{align*}
  \pi^{-\frac 32}k_0^{-1} \int_{\R } e^{-\frac{(y+j)^2}{k_0^2}}
\cdot \frac {\sin \pi y} y dy =\int_0^1 e^{-\frac 14 \pi^2 k_0^2 s^2} \cos \pi   j s ds.
\end{align*}

Collecting the estimates, 
we then obtain for $k_0 \le 20 \sqrt{\log N}$,
\begin{align*}
 & \frac 1 N \sum_{j=0}^{N-1} |Q_j| \notag \\
 =& O(N^{-\frac 13} \sqrt{\log N}) + \sum_{|j| \le N^{\frac 13} } \Bigl| \; \underbrace{\int_0^1 e^{-\frac 14 \pi^2 k_0^2 s^2} \cos \pi   j s ds
  }_{=:\beta_j} \; \Bigr|.
  \end{align*}

Here $\beta_j$ depends on $k_0$ but we suppress this explicit dependence for simplicity of
notation.
By using successive integration by parts (three times), it is not difficult to check that
uniformly for all $k_0 \le 20 \sqrt{\log N}$, 
\begin{align*}
|\beta_j | \lesssim j^{-2}, \qquad \forall\, |j|\ge N^{\frac 13}.
\end{align*}
Thus we can rewrite
\begin{align*}
  \frac 1 N \sum_{j=0}^{N-1} |Q_j| 
 = O(N^{-\frac 13} \sqrt{\log N}) + \sum_{j \in \mathbb Z }  |\beta_j|.
  \end{align*}
We now focus on analyzing the quantity
\begin{align*}
A(k_0)= \sum_{j \in \mathbb Z}
|\beta_j |.
\end{align*}

 First we consider $k_0=1$.
 For $k_0=1$, a rigorous numerical computation gives
  \begin{align*}
  &\beta_0\approx 0.54934; \\
  &\beta_1=\beta_{-1} \approx 0.219568; \\
  &\beta_2=\beta_{-2} \approx 0.0031275\\
  &\beta_3=\beta_{-3} \approx 0.00413676.
  \end{align*}
  It follows that
  \begin{align*}
  A(1) \ge \sum_{|j|\le 3} \beta_j \ge 1.002.
  \end{align*}

Next we shall give more precise bounds on $A(k_0)$ depending on the 
size of $k_0$. 

  \begin{rem}
   Observe that
  \begin{align*}
  \frac 12 \int_{-1}^1 e^{-\frac 14 \pi^2 k_0^2 s^2}
  \cos j\pi s ds =  \int_{|s|\le \frac 12}
  e^{-\pi^2 k_0^2 s^2} \cos 2\pi j s ds.
   \end{align*}

  We may then rewrite
   \begin{align*}
 A(k_0)= \sum_{j \in \mathbb Z } {\Bigl|\int_{|s|\le \frac 12} e^{- \pi^2 k_0^2 s^2} \cos 2\pi   j s ds
  \Bigr|}.
  \end{align*}
  Interestingly, this is exactly the $l^1$ norm of the Fourier coefficients of the function
  $e^{-\pi^2 k_0^2 x^2} $ on
  the periodic torus $\mathbb T=[-\frac 12, \frac 12)$.
\end{rem}

  \begin{rem}
 Interestingly, if we consider
 \begin{align*}
 S=\sum_{j\in \mathbb Z} \Bigl| \int_{|s|\le \frac 12} e^{-\beta |s|} 
 \cos 2\pi js d s \Bigr| =: \sum_{j\in \mathbb Z} |P_j|,
 \end{align*}
  we can obtain the closed form of $P_j$.  Indeed we have
  \begin{align*}
  P_j=\int_{|s|\le \frac 12} e^{-\beta |s| } \cos 2\pi j s ds 
  & = 2 \op{Re} \Bigl( \int_0^{\frac 12} e^{-s(\beta-2\pi i j)} ds  \Bigr) \notag \\
  & =2 \op{Re}\left( \frac {e^{-s(\beta-2\pi i j)} }{ -(\beta-2\pi i j) } \Bigr|_{s=0}^{\frac 12} \right)
  \notag \\
  &=2   \beta \frac {e^{-\frac 12\beta} (-1)^j - 1} {-(\beta^2+(2\pi j)^2 )} 
  \notag \\
  &= 2\beta \frac {1-e^{-\frac 12\beta} (-1)^j} { \beta^2 + (2\pi j)^2} \ge 0.
  \end{align*}
  Thus $S=1$. 
  \end{rem}

  \subsection{Estimate of $A(k_0)$ when $k_0$ is large}
  We first give a bound on $A(k_0)$ when $k_0\ge 1$. Rewrite
  \begin{align*}
  A(k_0)=  \sum_{j \in \mathbb Z} {\Bigl|\int_{|s|\le \frac 12} e^{- a s^2} \cos 2\pi   j s ds
  \Bigr|},
  \end{align*}
  where $a=\pi k_0^2\ge \pi$.
  Denote 
  \begin{align*}
  f_0(s) = \sum_{n \in \mathbb Z} e^{-a(s+n)^2}.
  \end{align*}
  Clearly 
  \begin{align*}
  \int_{|s|\le \frac 12} f_0(s) \cos 2\pi js ds =\int_{\R} e^{-as^2} \cos 2\pi js ds \ge 0,
  \end{align*}
  and thus
  \begin{align*}
  \sum_{j\in \mathbb Z}
  \Bigl| \int_{|s| \le \frac 12} f_0(s) \cos 2\pi js ds 
  \Bigr| = \sum_{j\in \mathbb Z}
   \int_{|s| \le \frac 12} f_0(s) \cos 2\pi js ds  = f_0(0) =\sum_{n \in \mathbb Z}
   e^{-a n^2}.
   \end{align*}
   
 Now for $|n|\ge 1$ and $|j|\ge 1$, we have
 \begin{align*}
    & - \int_{|s|\le \frac 12} e^{-a(s+n)^2} \cos 2\pi js ds \notag \\
=& \frac  1 {2\pi j} \int_{|s|\le \frac 12} \frac d {ds}
\left( e^{-a(s+n)^2} \right) \sin 2\pi j s ds \notag \\
=& (2\pi j)^{-2} \frac d{ds} \left( e^{-a(s+n)^2} \right) (-\cos 2\pi js)
\Bigr|_{s=-\frac 12}^{\frac 12}
+ (2\pi j)^{-2}
\int_{|s|\le \frac 12} \frac {d^2}{ds^2}
\left( e^{-a(s+n)^2} \right) \cos 2\pi j s ds.
\end{align*}

Now consider the function $e^{-x^2}$. Clearly
\begin{align*}
(e^{-x^2} )^{\prime\prime} = e^{-x^2} (4x^2-2) \ge 0, \qquad\forall\, |x| \ge \frac 1 {\sqrt 2}.
\end{align*}
It follows that $\frac {d^2}{ds^2} e^{-a (s+n)^2}    \ge 0$ since $(n+s)^2 \ge \frac 14$ and
$a\ge \pi$.  Thus
\begin{align*}
&(2\pi j)^{-2} \left|
\int_{|s|\le \frac 12} \frac {d^2}{ds^2}
\left( e^{-a(s+n)^2} \right) \cos 2\pi j s ds
\right| \le (2\pi j)^{-2} 2a \left|e^{-a(n+\frac 12)^2} (n+\frac 12)- e^{-a(n-\frac 12)^2} (n-\frac 12) \right|;\\
& (2\pi j)^{-2} \left| \frac d{ds} \left( e^{-a(s+n)^2} \right) (-\cos 2\pi js)
\Bigr|_{s=-\frac 12}^{\frac 12}
\right|
\le  
(2\pi j)^{-2} 2a \left|e^{-a(n+\frac 12)^2} (n+\frac 12)- e^{-a(n-\frac 12)^2} (n-\frac 12) \right|.
\end{align*}

 It follows that
 \begin{align*}
 A(k_0) &\le \sum_{n \in \mathbb Z} e^{-a n^2} + \sum_{|j|\ge 1, |n|\ge 1}
 2 (2\pi j)^{-2} 2a \left|e^{-a(n+\frac 12)^2} (n+\frac 12)- e^{-a(n-\frac 12)^2} (n-\frac 12) \right| 
  + \int_{|s|\le \frac 12} \sum_{|n|\ge 1} e^{-a(n+s)^2} ds \notag \\
  & \le 1 +2 \sum_{n\ge 1} e^{-an^2} + \int_{|s|\ge \frac 12} e^{-as^2} ds +
  \frac a 3
  \sum_{|n|\ge 1} \left|e^{-a(n+\frac 12)^2} (n+\frac 12)- e^{-a(n-\frac 12)^2} (n-\frac 12) \right| 
  \notag \\
  & \le 1+ 2 \sum_{n\ge 1} e^{-an^2} + \frac 2 {\sqrt a} \int_{s\ge \frac 12 \sqrt a}
  e^{-s^2} ds+
   \frac a 3 e^{-\frac a 4}, \qquad a=\pi k_0^2 \ge \pi.
 \end{align*}
  This bound is particularly effective when $a$ is large.

  \begin{rem}
  We should mention that it is possible to extract more information on the
  Fourier coefficients $\beta_j$ in the regime $k_0\gtrsim 1$. 
 Consider for $b=\frac 14 \pi k_0^2\ge \frac 12$,
 \begin{align*} 
 A(k_0)= \sum_{j \in \mathbb Z} |\beta_j|, \qquad \beta_j = \int_0^1 e^{-bs^2}
  \cos \pi j s ds.
  \end{align*}
 We rewrite
 \begin{align*}
 \beta_j &= \int_0^{\infty} e^{-bs^2}  \cos \pi j s ds
 - \int_1^{\infty} e^{-bs^2} \cos \pi js ds  \notag \\
 &=\frac 12 \sqrt{\pi} b^{-\frac 12}
 e^{-\frac {\pi^2 j^2}{4b} } - \int_1^{\infty} e^{-bs^2} \cos \pi js ds .
 \end{align*}
  Clearly for $j\ne 0$, 
  \begin{align*}
   & - \int_1^{\infty} e^{-bs^2} \cos \pi js ds  \notag \\
 =&\; (j\pi)^{-2}  \Bigl( 
 \partial_s (e^{-bs^2} ) \bigr|_{s=1} \cdot (-1)^j
 +\int_1^{\infty}
 \partial_{ss} (e^{-bs^2} ) \cos j\pi s ds \Bigr).
 \end{align*}
  Observe that $\partial_{ss} e^{-bs^2}=e^{-bs^2} 2b( 2bs^2-1) \ge 0$ for $s\ge 1$, $b\ge
  \frac 12$. Thus $\beta_j\ge 0$ if $j$ is odd.
\end{rem}

  \subsection{Estimate of $A(k_0)$ when  $k_0$ is small}
In this subsection we give the precise form of $A(k_0)$ when $k_0$ is small. First we rewrite 
 \begin{align*}
  A(k_0)= \sum_{j \in \mathbb Z} |\beta_j|, \qquad \beta_j=
  \int_0^1 e^{-b s^2} \cos \pi   j s ds,
  \end{align*}
  where $b= \frac 14 \pi^2 k_0^2 \le \frac 12$. 
Clearly  for $j\ge 1$, 
\begin{align*}
 -\beta_j=& -\int_0^1 e^{-b s^2} \cos \pi j s ds  \notag \\
 =&\; (j\pi)^{-1} \int_0^1  \partial_s (e^{-b s^2}) \sin j\pi s ds \notag \\
 =& \; (j\pi)^{-2} \partial_s(e^{-b s^2} ) (-\cos j \pi s) \Bigr|_{s=0}^1
 + (j\pi)^{-2} 
 \int_0^1 \partial_{ss} ( e^{-b s^2}) \cos j\pi s ds.
 \end{align*}
  Now denote  $g(s)= (j\pi)^{-2} (-1) \partial_s (e^{-b s^2})$ and note that $g^{\prime}(s)\ge 0$
  whenever $0<b\le \frac 12$ and $0\le s\le 1$. Then
  \begin{align*}
  \beta_j &= g(1) (-1)^{j+1} + \int_0^1 g^{\prime}(s)  \cos j\pi s ds.
  \end{align*}
  It follows that $\beta_j\le 0$ for $j$ even, and $\beta_j\ge 0$ for $j$ odd.  Note that $\beta_0>0$.
  
  We then write
  \begin{align*}
  A(k_0)& = \beta_0 +\sum_{\text{$j$ odd} } \beta_j -\sum_{ \text{ $0\ne j$ even}}
  \beta_j \notag \\
  &=2\beta_0+ \sum_{j \in \mathbb Z} \beta_j - 2 (\beta_0+ \sum_{\text{$0\ne j$ even }} \beta_j )\notag \\
  &= 2\int_0^1 e^{-bs^2} ds +1 - (1+e^{-b} ) \notag \\
  & =2 \int_0^1 e^{-bs^2} ds -e^{-b}.
  \end{align*}
  Thus for $0<b=\frac 14 \pi^2 k_0^2\le \frac 12$, 
   \begin{align*}
 A(k_0)
 = 2 \int_0^1 e^{-bs^2} ds -e^{-b}.
  \end{align*}
  
  \begin{rem}
  We explain how to rigorously compute $\sum_{l \in \mathbb Z} \beta_l$ and
   $\sum_{l \in \mathbb Z}  \beta_{2l} $ as follows. Firstly we rewrite
   \begin{align*}
   \beta_l = \int_{|s|\le \frac 12} f(s) \cos 2 \pi l s ds, 
   \end{align*}
   where $f(s) =e^{-as^2}$. 
 Recall the Dirichlet kernel
 \begin{align*}
 D_N(x)= \sum_{|j|\le N} e^{2\pi i j x} = \frac {\sin(2N+1)\pi x}{ \sin \pi x}.
 \end{align*}
 Clearly
 \begin{align*}
 \int_{|x| \le \frac 12} (f(x) -f(0)) D_N(x) dx  
 =& \frac{f(x)-f(0)} {\sin \pi x} \frac 1 {2N+1} 
(-\cos (2N+1)\pi x) \Bigr|_{x=-\frac 12}^{\frac 12}  \notag \\
&\quad+\frac 1 {2N+1} \int_{|x|\le \frac 12}
\frac d {dx}
\Bigl( \frac{f(x)-f(0)} {\sin \pi x} 
\Bigr) \cos (2N+1) \pi x dx = O(\frac 1N).
\end{align*}
Thus \begin{align*}
\int_{|x| \le \frac 12} f(x) D_N(x) dx \to f(0), \quad \text{as $N\to \infty$}.
\end{align*}
Alternatively if one does not need the quantitative convergence rate,  one can directly obtain
\begin{align*}
\lim_{N\to \infty} \int_{|x|\le \frac 12} \frac {f(x)-f(0)} {\sin \pi x}  \sin(2N+1)\pi x dx=0
\end{align*}
since $(f(x)-f(0))/\sin \pi x \in L^1((-\frac 12, \frac 12) )$.

Next we consider  
\begin{align*}
\int_{|x|\le \frac 12} f(x) D_N (2x )dx 
&= \frac 12\int_{|x| \le 1} f(\frac x 2) D_N(x) dx \notag \\
&= \frac 12 \int_{|x|\le \frac 12} 
f(\frac x 2) D_N(x) dx +  \int_{\frac 12 \le x \le 1} f(\frac x2) D_N(x) dx \notag \\
&= \frac 12 \int_{|x|\le \frac 12} 
f(\frac x 2) D_N(x) dx + \int_0^{\frac 12} f(\frac {1-x} 2)
D_N(x) dx. 
\end{align*}
Now consider 
\begin{align*}
f(\frac {1-x}2) = f( \frac 12) +\tilde f(x).
\end{align*}
Clearly
\begin{align*}
\int_0^{\frac 12} \tilde f(x) D_N(x) dx 
& = \int_0^{\frac 12} \frac {\tilde f(x) }{\sin \pi x} \sin(2N+1)\pi x dx \notag \\
&= \frac{\tilde f(x) }{\sin \pi x} \frac 1 {2N+1} \frac 1{\pi} (-\cos (2N+1)\pi x) 
\Bigr|_{0}^{\frac 12} +O(\frac 1 N) = O(\frac 1 N).
\end{align*}
Thus we obtain
\begin{align*}
  \int_{|x|\le \frac 12} f(x) D_N (2x )dx 
= \frac 12 (f(0)+f(\frac 12) ) +O(\frac 1 N).
\end{align*}
 \end{rem}
 \begin{rem}
 More generally, suppose $a<b$ and let $(n_j)_{j=0}^L$ be all the possible integer points
 in the interval $[a,b]$ listed in the ascending order. If $a$ is an integer then $n_0=a$ and similarly
 if $b$ is an integer then $n_L=b$.  Suppose $f$ is a continuous function on $[a,b]$ such that
 \begin{align*}
 \frac { f(x) - \tilde f(x) } {\sin \pi x } \in L^1((a,b)),
 \end{align*}
 where
 \begin{align*}
 \tilde f(x) = \sum_{j=0}^L f(n_j) \phi(x-n_j).
 \end{align*}
 Here $\phi \in C_c^{\infty}(\mathbb R)$ is such that $\phi(x) \equiv 1$ for $|x|\le \frac 14$
 and $\phi(x)\equiv 0$ for $|x| \ge \frac 13$. 
If there are no integers in $[a,b]$ then we define $\tilde f(x)\equiv 0$.  Then we have
\begin{align*}
&\lim_{N\to \infty} \int_{a}^b f(x) \frac {\sin (2N+1)\pi x} {\sin \pi x} dx  \notag \\
=&
\begin{cases}
0, \qquad \text{if $[a,b]$ contains no integers}; \\
\sum_{j=0}^L f(n_j), \qquad \text{ if $[a,b]$ contain some integers, but $a$ and $b$ are not integers}; \\
\frac 12 f(a) + \sum_{j=1}^L f(n_j), \qquad
\text{ if $a\in \mathbb Z$, $b\notin \mathbb Z$}; \\
\frac 12 f(b)+ \sum_{j=0}^{L-1} f(n_j), \qquad \text{ if $b \in \mathbb Z$, but $a\notin Z$};\\
 \frac 12 (f(a)+f(b)) + \sum_{j=1}^{L-1} f(n_j),
 \qquad \text{ if $a\in \mathbb Z$ and $b\in \mathbb Z$}.
 \end{cases}
 \end{align*}
 Note that
 the above derivation is a natural extension of the usual identity for Dirac comb:
 \begin{align*}
 \sum_{j \in \mathbb Z} \cos 2 \pi j s =\lim_{N\to \infty}
 \sum_{|j|\le N} \cos 2\pi j s = \sum_{n \in \mathbb Z} \delta (s+n),
 \qquad \text{in $\mathcal S^{\prime}(\mathbb R)$}.
 \end{align*}
 \end{rem}

For a function $f: \; \mathbb T \to \mathbb R$, we  denote 
\begin{align*}
(e^{ \nu^2 t \partial_{xx}} Q_N f )(t,x)
=\op{Re} \Bigl( \sum_{-\frac N2<k \le \frac N2} \Bigl( \frac 1 N \sum_{j=0}^{N-1} f(\frac j N) e^{-2\pi i \frac {k j}N }
e^{-(2\pi k)^2 \nu^2 t} 
\Bigr)
e^{2\pi i k\cdot x} \Bigr).
\end{align*}

\begin{thm} \label{N18:1a}
Let $\nu>0$ and identify $\mathbb T=[0,1)$.
There exists a  function $f$: $\mathbb T \to \mathbb R$ with $\|f\|_{\infty} \le 1$ and satisfying
$f$ is continuous at $[0, \frac 14) \bigcup (\frac 14, 1)$ (i.e. $f$ is continuous 
at all $\frac 14 \ne x \in \mathbb T$)
 such that 
the following hold:
for a sequence of even numbers $N_m \to \infty$, $t_m =\frac 14\nu^{-2}N_m^{-2}$ and $x_m
=j_m/N_m$ with $0\le j_m\le N_m-1$ , we have
\begin{align} \label{N19:1}
\Bigl| ( e^{\nu^2 t \partial_{xx} } Q_{N_m} f  )(t_m, x_m ) \Bigr|\ge 1+\epsilon_*,
\end{align}
where $\epsilon_*>0$ is an absolute constant.  

Let $u_m(t) \in \mathbb R^{N_m}$ solve
\begin{align*} 
\begin{cases}
\frac d {dt} u_m = \nu^2 \Delta_h u_m +u_m - (u_m)^{.3}; \\
u_m\Bigr|_{t=0} =  v_m,
\end{cases}
\end{align*}
where $(v_m)_j= f(\frac j {N_m})$ for all $0\le j \le N_m-1$.  Then
\begin{align} \label{N19:2}
|u_m(t_m,x_m) | \ge 1+\frac 1 2 \epsilon_*, \qquad\forall\, m\ge 1.
\end{align}
\end{thm}
\begin{rem}
The point $x_*= \frac 14$ is for convenience only and we may replace it by any other
point $x_*\in \mathbb T$.
\end{rem}

\begin{proof}
We begin by noting that \eqref{N19:2} follows from \eqref{N19:1} and a perturbation argument
(as in Theorem \ref{N21_5}).
We now show \eqref{N19:1}. With no loss we assume $\nu=1$.  Suppose $f$ satisfies $\|f\|_{\infty} \le 1$. Then for $N\ge 10$, $\tau_N =\frac 14N^{-2}$ and $0\le l \le N-1$, 
by using the proof of Theorem \ref{N18:1}, we have
\begin{align*}
(e^{t \partial_{xx} } Q_N f)( \tau_N , \frac l N)
& = \sum_{j=0}^{N-1} f(\frac {l-j}N) \cdot \frac 1 N \op{Re}
\Bigl( \sum_{-\frac N2 <k \le \frac N2} e^{ -4\pi^2 k^2 \tau_N} e^{2\pi i k\frac j N} \Bigr) \notag \\
& =O(N^{-\frac 13} \sqrt{\log N})
+ f(\frac l N) c_0 + \sum_{1\le | j |\le N^{\frac 13} } f(\frac {l +j }N)  c_j,
\end{align*}
where 
\begin{align*}
c_j = \int_0^1 e^{-\frac 1{16} \pi^2 s^2} \cos (\pi js )ds (-1)^j. 
\end{align*}
By the computation in the proof of Theorem \ref{N18:1}, we have
\begin{align*}
\sum_{|j|\le 3} |c_j| \ge 1.002.
\end{align*}
Thus if $f$ satisfies $f(\frac {l+j} N) = \op{sgn}(c_j)$ for $|j|\le 3$,  and 
$f(\frac {l+j} N) =0$ for $4\le |j|\le N^{\frac 13}$, then 
\begin{align*}
(e^{t \partial_{xx} } Q_N f)( \tau_N , \frac l N)
\ge 1.002 +O(N^{-\frac 13 } \sqrt{\log N}).
\end{align*}
Now for $m=1,2,3,\cdots$, define 
\begin{align*}
y_m = \frac 1 4 + \frac 1 {m+10}.
\end{align*}
We may then choose large  $N_m$ and $0\le l_m\le N_m-1$ such that 
\begin{align*}
\Bigl| \frac {l_m} {N_m} - y_m \Bigr| \ll  \frac 1 {m^2}.
\end{align*}
By a suitable interpolation, it is then obvious to define the function $f$ such that $f(\frac {l_m+j} {N_m} ) = \op{sgn} (c_j)$ for $|j|\le 3$ and for all $m$ satisfying the desired properties.
\end{proof}
To state the next result we need some notation. Suppose $f: \, \mathbb T \to \mathbb R$ is 
continuous. Then we denote the modulus of continuity $\omega_f: [0,1] \to [0,\infty)$ as
\begin{align} \label{N18:1b.0}
\omega_f( \delta) =\max_{|x-y|\le \eta} |f(x)-f(y)|, 
\end{align}
where $|x-y|$ denotes the distance of $x$, $y \in \mathbb T$.  For $0<\alpha<1$, we say $f:\, \mathbb T\to
\mathbb R$ is $C^{\alpha}$-continuous if 
\begin{align} \label{N18:1b.1}
\sup_{0<|x-y|\le 1} \frac { |f(x)-f(y)|} {|x-y|^{\alpha}} <\infty.
\end{align}

\begin{thm} \label{N18:1b}
Let  $\nu>0$. Suppose $f:\, \mathbb T\to \mathbb R$ is continuous and $\|f\|_{\infty} \le 1$. 
Then
\begin{align*}
\sup_{t\ge 0}\sup_{0\le l \le N-1} | (e^{\nu^2 t \partial_{xx} } Q_N f)(t, \frac l N) |
\le 1 + O(\omega_f(N^{-\frac 23}) ) + O(N^{-c}),
\end{align*}
where $c>0$ is an absolute constant. If $f:\, \mathbb T\to \mathbb R$ is $C^{\alpha}$-continuous 
for some $0<\alpha<1$ and $\|f\|_{\infty} \le 1$. 
Then
\begin{align*}
\sup_{t\ge 0}\sup_{0\le l \le N-1} | (e^{\nu^2 t \partial_{xx} } Q_N f)(t, \frac l N) |
\le 1 + O(N^{-c_1}),
\end{align*}
where $c_1>0$ is a constant depending only on $\alpha$. 
\end{thm}
\begin{proof}
With no loss we assume $\nu=1$.  Suppose $f$ satisfies $\|f\|_{\infty} \le 1$.
By Theorem \ref{N18:1} we only need to consider the regime $N^2 t <100 \log N$. 
For $0\le l \le N-1$ and $N^2 t <100 \log N$, 
by using the proof of Theorem \ref{N18:1}, we have
\begin{align*}
(e^{t \partial_{xx} } Q_N f)( t , \frac l N)
& = \sum_{j=0}^{N-1} f(\frac {l-j}N) \cdot \frac 1 N \op{Re}
\Bigl( \sum_{-\frac N2 <k \le \frac N2} e^{ -4\pi^2 k^2 t } e^{2\pi ik \frac j N} \Bigr) \notag \\
& =O(N^{-\frac 13} \sqrt{\log N})
+ f(\frac l N) c_0 + \sum_{1\le | j |\le N^{\frac 13} } f(\frac {l +j }N)  c_j,
\end{align*}
where 
\begin{align*}
c_j = \int_0^1 e^{-\frac 14 \pi^2 N^2 t s^2} \cos (\pi js )ds (-1)^j. 
\end{align*}
Now observe that for $|j| \le N^{\frac 13}$, 
\begin{align*}
&| f( \frac {l +j } N) -f (\frac l N) | \le \omega_f (N^{-\frac 23} ), \quad \text{if $f$ is continuous}, \\
& | f(\frac {l+j} N) -f (\frac l N) | \le O(N^{-\frac 23 \alpha}), \quad \text{if $f$ is $C^{\alpha}$
continuous}.
\end{align*}
Thus in the main order, the function $f$ can be taken as a constant value (the contributions of
$f$ for   $j$ in the
regime $N^{\frac 13}<j<N-N^{\frac 13}$ is unimportant).  Since
\begin{align*}
\sum_{j=0}^{N-1} \frac 1 N \op{Re}
\Bigl( \sum_{-\frac N2 <k \le \frac N2} e^{ -4\pi^2 k^2 t } e^{2\pi ik \frac j N} \Bigr)
=1,
\end{align*}
the desired result then follows easily.
\end{proof}

\section{Estimate of $(\op{Id} - \nu^2 \tau \partial_{xx} )^{-n} Q_N U^0$ and related
estimates}
\begin{lem} \label{N29_apL1}
Let $\tau>0$ and $s\ge 1$. Then  for any $1\le p <q\le \infty$, we have
\begin{align} \label{N29_apL2}
\| (\op{Id} - \tau \nu^2 \partial_{xx} )^{-s} f \|_{L_x^q(\mathbb T)}
\le C_{p,q,\nu} \cdot \Bigl(1+ (s\tau )^{-\frac 12(\frac 1p-\frac 1q)}  \Bigr)
\| f \|_{L_x^p(\mathbb T)},
\end{align}
where $C_{p,q, \nu}>0$ depends only on ($p$, $q$, $\nu$). Also
\begin{align} \label{N29_apL3}
\| \partial_x (\op{Id} - \tau \nu^2 \partial_{xx} )^{-s} f \|_{L_x^2(\mathbb T)}
\le C_2 \cdot \Bigl(1+ (s\tau )^{-\frac 34}  \Bigr)
\| f \|_{L_x^1(\mathbb T)},
\end{align}
where $C_2>0$ depends only on $\nu$.
\end{lem}
\begin{proof}
Denote by $K$ the kernel corresponding to $(\op{Id}-\tau \nu^2 \partial_{xx} )^{-s}$.  
Clearly $\|K\|_{L_x^1(\mathbb T)} =1$ and 
\begin{align*}
\| K\|_{L_x^{\infty}(\mathbb T)} \le 1+ \sum_{0\ne k \in \mathbb Z} (1+ \frac {\alpha} s k^2)^{-s},
\end{align*}
where $\alpha=s \tau (2\pi \nu)^2$.  Now note that for $s\ge 1$ and $x\ge 0$, we have
\begin{align} \label{N29_apL4}
(1+\frac x s)^{-s} \le (1+x)^{-1}.
\end{align}
It follows that
\begin{align*}
&\|K\|_{L_x^{\infty}(\mathbb T)} \lesssim 1+ \alpha^{-\frac 12}; \\
& \|K\|_{L_x^r(\mathbb T)} \lesssim 1+\alpha^{-\frac 12(1-\frac 1r)}, \qquad\forall\, 1\le r\le \infty.
\end{align*}
Then \eqref{N29_apL2}  follows from Young's inequality. Now for 
\eqref{N29_apL3}, we first note that  (by using \eqref{N29_apL4})
\begin{align*}
\| \partial_x (\op{Id} - \tau \nu^2 \partial_{xx} )^{-s} f \|_{L_x^2(\mathbb T)}
\lesssim (\tau s)^{-\frac 12}
\| (\op{Id} - s\tau \nu^2 \partial_{xx})^{-\frac 12} f\|_{L_x^2(\mathbb T)}.
\end{align*}
Denote $T= (\op{Id}-s\tau \nu^2 \partial_{xx})^{-\frac 12}$. Since
\begin{align*}
\| Tf \|_2^2 = \langle f , T^2 f \rangle.
\end{align*}
The desired estimate follows easily.
\end{proof}

We recall for $n\ge 1$,
\begin{align*}
( (\op{Id} - \nu^2 \tau \partial_{xx} )^{-n} Q_N U^0 )(x)
=\op{Re} \Bigl( \sum_{-\frac N2<k \le \frac N2} \Bigl( \frac 1 N \sum_{j=0}^{N-1} U^0_j e^{-2\pi i \frac {k j}N }
(1+\nu^2\tau (2\pi k)^2 )^{-n} 
\Bigr)
e^{2\pi i k\cdot x} \Bigr).
\end{align*}
We need to examine $\sup_{0\le l\le N-1} | ((\op{Id} - \nu^2 \tau \partial_{xx} )^{-n}  Q_N U^0 )(\frac lN)|$ under the assumption that
$\| U^0 \|_{\infty} \le 1$.
Clearly we have
\begin{align*}
\sup_{0\le l\le N-1} | ((\op{Id} - \nu^2 \tau \partial_{xx} )^{-n} Q_N U^0 )(\frac lN)| \le 
\frac 1 N \sum_{j=0}^{N-1} \Bigl| \op{Re} \Bigl( { \sum_{-\frac N2 < k \le \frac N2}
(1+ (2\pi \nu k)^2 \tau)^{-n} 
e^{2\pi i k \cdot \frac j N}}  \Bigr)\Bigr| =: B_{N,n}.
\end{align*}

Below we shall write $X= O(Y)$ if $|X|\le CY$  and $C>0$ is a  constant depending only on $\nu$.
\begin{thm} \label{N21:1}
Let $\nu>0$, $\tau>0$ and $n\ge 1$.
It holds that 
\begin{align*}
B_{N,n} \le 
\begin{cases}
1+O(N^{-1}), \qquad \text{if $\tau\ge \frac 12$}; \\
1+ O(N^{-0.7}), \qquad \text{if $0<\tau<\frac 12$ and $ n\tau \nu^2 \ge N^{-0.3} $}; \\
1+ \epsilon_a + O(N^{-\frac 13} \sqrt{\log N}), \qquad \text{if 
$0<\tau<\frac 12$ and $ n\tau \nu^2\le N^{-0.3} $}.
\end{cases}
\end{align*}
Here $0<\epsilon_a <0.01$ is an absolute constant.

If
$\tau = \frac 1 {4 N^2 \nu^2}$, then  for $n=1$ we have
\begin{align*}
B_{N, 1} >1.001+O(N^{-c}),
\end{align*}
where $c>0$ is an absolute constant.

 Suppose $f:\, \mathbb T\to \mathbb R$ is continuous and $\|f\|_{\infty} \le 1$.  Then
\begin{align} \label{N21:1.0a}
\sup_{0\le l \le N-1} \Bigl| 
( (\op{Id}-\nu^2 \tau \partial_{xx} )^{-n} Q_N f)( \frac l N) \Bigr| 
\le 1 + O(\omega_f(N^{-\frac 23}) ) + O(N^{-c}),
\end{align}
where $c>0$ is an absolute constant, and $\omega_f$ is defined in \eqref{N18:1b.0}.

Suppose $f:\, \mathbb T\to \mathbb R$ is $C^{\alpha}$-continuous for some $0<\alpha<1$ 
and $\|f\|_{\infty} \le 1$.  
If $(U^0)_l= f(\frac l N)$ for all $0\le l\le N-1$,  then
\begin{align} \label{N21:1.0b}
\sup_{0\le l \le N-1} \Bigl| 
( (\op{Id}-\nu^2 \tau \partial_{xx} )^{-n} Q_N f)( \frac l N) \Bigr| 
\le 1 + O(N^{-c_1}),
\end{align}
where $c_1>0$ is a constant depending only on $\alpha$. 

There exists a  function $f$: $\mathbb T \to \mathbb R$, continuous at all of $\mathbb T
\setminus \{x_*\}$ for some $x_*\in \mathbb T$ (i.e. continuous at all $x\ne x_*$) and  has the bound $\|f\|_{\infty} \le 1$
 such that 
the following hold:
for a sequence of even numbers $N_m \to \infty$, $\tau_m =\frac 14\nu^{-2}N_m^{-2}$ and $x_m
=j_m/N_m$ with $0\le j_m\le N_m-1$ , we have 
\begin{align} \label{N21:1.0c}
\Bigl| ( (\op{Id} - \nu^2 \tau \partial_{xx} )^{-1} Q_N f) (x_m)  \Bigr| \ge 1+ \eta_*, \qquad\forall\, m\ge 1,
\end{align}
where $\eta_*>0.001$ is an absolute constant.

\end{thm}
We now give the proof of Theorem \ref{N21:1}. To ease the notation
we shall assume $\nu=1$. We discuss several cases.

Case 1: $\tau\ge \frac 12$. Observe that  for all $n\ge 1$,
\begin{align*}
\sum_{|k|\ge \frac N2}
\frac 1 { (1+ (2\pi k)^2 \tau)^{n} } & \lesssim \int_{ |k|\ge \frac N2 -1} \frac 1 {k^2}  d k\notag \\
& \le O(N^{-1}).
\end{align*}
Thus in this case
\begin{align*}
|B_{N,n} | \le 1+ O(N^{-1}).
\end{align*}

Case 2: $0<\tau<\frac 12$ and $n\tau \ge N^{-0.3}$.  We have
\begin{align*}
\sum_{|k|\ge \frac N2}
\frac 1 { (1+ (2\pi k)^2 \tau)^{n} } & \lesssim \int_{ |k|\ge \frac N2 -1} \frac 1 {1+ n\tau k^2}  d k\notag \\
& \le O(N^{-1}(n\tau)^{-1} ) \le O(N^{-0.7}).
\end{align*}
Thus in this case
\begin{align*}
|B_{N,n} | \le 1+ O(N^{-0.7}).
\end{align*}

Case 3: $0<\tau<\frac 12$ and $n\tau<N^{-0.3}$.  In this case we shall use Theorem
\ref{N18:1}.  Note that for $a= 1+\beta k^2$ ($\beta>0$ is a constant) and $n\ge 1$, we have
\begin{align} \label{N21:2.0c}
a^{-n} = \frac 1 {(n-1)!} \int_0^{\infty} e^{-sa} s^{n-1} ds.
\end{align}
The desired estimate then follows from this and Theorem \ref{N18:1}.

By a similar reasoning, we note that \eqref{N21:1.0a}--\eqref{N21:1.0b} follow from \eqref{N21:2.0c} and Theorem
\ref{N18:1b}.

Next we shall derive more detailed estimates for $B_{N,1}$. 
We begin with  an identity: for $\tau>0$, 
\begin{align*}
\int_{\R} \frac 1 {1+(2\pi k)^2 \tau} e^{2\pi i kx} dk = \frac 1 {2\sqrt {\tau} } e^{-\frac {|x|} 
{\sqrt {\tau} } }.
\end{align*}
  
We denote $B_{N,1}$ as $B_N$:
\begin{align*}
B_N=\frac 1 N \sum_{j=0}^{N-1} |Q_j|,
\end{align*}
where
\begin{align*}
Q_j &= \sum_{-\frac N2 < k\le \frac N2} \frac 1 { 1+(2\pi k)^2 \tau} e^{2\pi i k \frac jN} \notag \\
& =\frac 1 {2\sqrt{\tau}} \int_{\R} e^{-\frac {|y- \frac j N |} {\sqrt{\tau} } }
\frac {\sin N\pi y} {\tan \pi y} dy \notag \\
& =(-1)^j \frac 1 {2\sqrt{\tau}} \int_{\R} e^{-\frac {|y|} {\sqrt{\tau} } }
\frac {\sin N \pi y} { \tan \pi( \frac j N +y) } dy \notag \\
&=(-1)^j \frac 12 \int_{\R} e^{-|y|} 
\frac{\sin k_1 y} { \tan \pi( \frac {j +k_1y } N)} dy.
\end{align*}
Here we denote $k_1= N\sqrt{\tau}$.

 Case 1: $k_1 \ge N^{\frac {1+\epsilon_0}2}$ where $0<\epsilon_0\ll 1$ is a small
 constant. Then $N\tau>N^{\epsilon_0}$ and
 \begin{align*}
 B_N \le 1 + \sum_{|k|\ge \frac N2} \frac 1 {1+(2\pi k)^2 \tau} 
 \le 1 + O(N^{-\epsilon_0}).
 \end{align*}
  
  Case 2: $k_1 < N^{\frac {1+\epsilon_0} 2}$. Note that we may restrict the integral
  in $dy$ to the regime $|y| \le 100 \log N$ since the contribution due to the other 
  regime is negligible. 
  
  First consider $ N^{\frac 12+2\epsilon_0} \le j \le N-N^{\frac 12+2\epsilon_0}$. Clearly
  for $j_1 =\min\{j, |N-j| \}$, we have
  \begin{align*}
  \int_{0<y< 100 \log N} e^{-y} | \cot(\frac {j+k_1 y} N) -
  \cot(\frac {j-k_1 y} N ) | dy 
  \lesssim   \frac { N^2 } {j_1^2} 
  \cdot \frac {k_1} N
  \end{align*}
  Then
  \begin{align*}
  \frac 1 N \sum_{N^{\frac 12+2\epsilon_0} \le j \le N- N^{\frac 12+2\epsilon_0}}
  |Q_j| \lesssim  k_1 \cdot N^{-\frac 12-2\epsilon_0}  \le O(N^{-\epsilon_0}).
  \end{align*}
  
  Next consider $0\le j< N^{\frac 12+2\epsilon_0}$.  We rewrite
  \begin{align*}
  Q_j = \frac 1{2k_1} \int_{\R} e^{-\frac {|y+j|} {k_1} }
  \frac {\sin \pi y } {\tan \pi\frac y N} dy
  \end{align*}
  In the above integral, observe that the regime $|y| \ge N^{\frac 23}$ is negligible. 
  On the other hand for $|y| < N^{\frac 23}$, we observe that
  \begin{align*}
  |\cot (\frac y N) - \frac N y | \lesssim \frac {|y|} N.
  \end{align*}
  Clearly the contribution due to the error term $O(\frac {|y|} N) \le O(\frac {|y|} {k_1} )
  \cdot \frac {k_1} N$ is bounded by
  \begin{align*}
  \frac 1 N \sum_{j\le  N^{\frac 12 +2\epsilon_0} } \frac {k_1} N \lesssim O(N^{-c}).
  \end{align*}
  
 We then obtain 
 \begin{align*}
 B_N = O(N^{-c})+\sum_{|j| \le N^{\frac 12 + 2\epsilon_0} }
 \frac 1 {2k_1} \frac 1 {\pi} 
 \Bigl| \int_{\R} e^{-\frac {|y+j|} {k_1} } \frac {\sin \pi y} y dy \Bigr|
 \end{align*}
 
 Now we have
 \begin{align*}
 F_j(1) = \frac 1 {2k_1 \pi} \int_{\R} e^{-\frac {|y+j|} {k_1} }
 \frac {\sin \pi y} y dy= \int_0^1 F_j^{\prime}(s) ds,
 \end{align*}
  where
  \begin{align*}
  F_j^{\prime}(s) &= \frac 1 {2k_1} \int_{\R} e^{-\frac {|y+j|} {k_1} }
  \cos (\pi s y) dy \notag \\
  &= (\cos \pi js )  \frac 1 {2k_1} \int_{\R} e^{-\frac {|y|} {k_1} }
  e^{i \pi s y  } dy \notag \\
  & = (\cos \pi js ) \cdot \frac 1 {1+ (k_1 \pi s)^2}.
  \end{align*}
  Thus we obtain 
  \begin{align*}
  B_N &= O(N^{-c}) + \sum_{|j| \le N^{\frac 12+2\epsilon_0} }
  \Bigl| \int_0^1 \frac 1 {1+(k_1 \pi s)^2} \cos \pi  j s d s \Bigr| \notag \\
  &= O(N^{-c}) + \sum_{j \in \mathbb Z }
  \Bigl| \int_0^1 \frac 1 {1+(k_1 \pi s)^2} \cos \pi  j s d s \Bigr|
  \end{align*}
  
  Now for $k_1=\frac 12$, we denote 
  \begin{align*}
  \beta_j = \int_0^1 \frac 1 {1+(k_1 \pi s)^2} \cos \pi  j s d s.
  \end{align*}
  By rigorous numerical computation, we have
  \begin{align*}
  &\beta_0 \approx 0.639093; \\
  & \beta_1 =\beta_{-1} \approx  0.165881; \\
  &\beta_2=\beta_{-2} \approx 0.00883796;\\
  &\beta_3=\beta_{-3} \approx 0.00691018;\\
  &\beta_4=\beta_{-4} \approx -0.00220351.
  \end{align*}
  In particular
  \begin{align*}
  \sum_{|j|\le 3} |\beta_j| > 1.0023.
  \end{align*}

Finally we show how to derive \eqref{N21:1.0c}.  With no loss we assume $\nu=1$. 
Suppose $f$ satisfies $\|f\|_{\infty} \le 1$.  Then for $N\ge 10$, $\tau=\frac 14 N^{-2}$
and $0\le l\le N-1$, we have
\begin{align*}
( (\op{Id} -  \tau \partial_{xx} )^{-1} Q_N f)(\frac l N)
& = \sum_{j=0}^{N-1} f(\frac {l-j} N)
\cdot \frac 1 N \op{Re}
\Bigl( \sum_{-\frac N2<k\le \frac N2} (1+(2\pi)^2 \tau k^2)^{-1} e^{2\pi i k
\frac j N } \Bigr) \notag \\
& = O(N^{-c}) +  f(\frac lN) \beta_0 + \sum_{1\le |j| \le N^{\frac 12+2\epsilon_0} }
f(\frac {l+j} N) \beta_j,
\end{align*}
where
\begin{align*}
\beta_j = \int_0^1 \frac 1 {1+(\frac 12 \pi s)^2} \cos \pi j s ds.
\end{align*}
It is then clear to choose $f$ similar to the proof in Theorem \ref{N18:1a}. We omit
further details.

\frenchspacing
\bibliographystyle{plain}

\end{document}